\documentclass[a4paper,10pt,leqno,english]{article}
\usepackage{tikz}
\usetikzlibrary{positioning}
\usepackage[T1]{fontenc}
\usepackage[utf8]{inputenc}
\usepackage[a4paper, margin=2cm]{geometry}
\usepackage{babel}
\usepackage{amsthm,amsmath,amsfonts,amssymb}
\usepackage{dsfont}
\usepackage{enumitem,pifont,physics,comment}
\usepackage{hyperref}
\usepackage{float}
\newtheorem{theorem}{Theorem}
\newtheorem{proposition}{Proposition}[section]
\newtheorem{lemma}{Lemma}[section]
\newtheorem{corollary}{Corollary}[section]
\newtheorem{example}{Example}[section]
\makeatletter
\renewcommand\theequation%
{\thesection.\arabic{equation}}
\@addtoreset{equation}{section}
\makeatother
\newtheorem{definition}{Definition}[section]
\newtheorem{hypothesis}{Hypothesis}[section]

\newtheorem{conjecture}{Conjecture}[section]
\newtheorem{deflemm}{Definition-Lemma}[section]
\newtheorem{remark}{Remark}[section]
\newtheorem{notation}{Notation}[section]

\usepackage[backend=biber,style=alphabetic]{biblatex}
\usepackage{csquotes}
\usepackage{amsmath,amsthm,amssymb,amscd,color, xcolor,mathtools,url,tikz}
\usepackage{tikz}
\usetikzlibrary{positioning}

\usepackage{titlesec}

\def\R{\mathbb{R}}
\def\N{\mathbb{N}}
\def\Z{\mathbb{Z}}
\def\Q{\mathbb{Q}}
\def\C{\mathbb{C}}

\let\eps\varepsilon
\let\epsilon\varepsilon

\newcommand{\scal}[2]{\langle #1, #2 \rangle}

\let\oldsum\sum
\renewcommand{\sum}{\displaystyle\oldsum}
\let\oldprod\prod
\renewcommand{\prod}{\displaystyle\oldprod}
\let\oldinf\inf
\renewcommand{\inf}{\displaystyle\oldinf}
\let\oldsup\sup
\renewcommand{\sup}{\displaystyle\oldsup}

\let\leq\leqslant
\let\geq\geqslant
\newlength{\oldparindent}
\newcommand{\myindent}{\hspace{\oldparindent}}

\usepackage{etoolbox}

\providecommand{\keywords}[1]
{
  \small	
  \textbf{Keywords---} #1
}

\providecommand{\MSC}[1]
{
  \small	
  \textbf{Mathematics Subject Classification (2020)---} #1
}

\title{Bounds for quasimodes with polynomially narrow bandwidth on surfaces of revolution}
\author{Ambre Chabert\thanks{Département de Mathématiques et Applications, Ecole Normale Supérieure, UMR 8553, 45 rue d’Ulm. 75230 Paris Cedex 05, France. Université Paris Cité and Sorbonne Université, CNRS, IMJ-PRG, F-75013 Paris, France. Corresponding author. Email adress : ambre.chabert@ens.fr}}

\begin{document}
\maketitle

\begin{abstract}
    Given a compact surface of revolution with Laplace-Beltrami operator $\Delta$, we consider the spectral projector $P_{\lambda,\delta}$ on a polynomially narrow frequency interval $[\lambda-\delta,\lambda + \delta]$, which is associated to the self-adjoint operator $\sqrt{-\Delta}$. For a large class of surfaces of revolution, and after excluding small disks around the poles, we prove that the $L^2 \to L^{\infty}$ norm of $P_{\lambda,\delta}$ is of order $\lambda^{\frac{1}{2}} \delta^{\frac{1}{2}}$ up to $\delta \geq \lambda^{-\frac{1}{32}}$. We adapt the microlocal approach introduced by Sogge for the case $\delta = 1$, by using the Quantum Completely Integrable structure of surfaces of revolution introduced by Colin de Verdière. This reduces the analysis to a number of estimates of explicit oscillatory integrals, for which we introduce new quantitative tools. This is the first sharp result in the case $\delta \ll 1$ beyond the case of locally symmetric surfaces (torus, sphere, arithmetic hyperbolic surfaces).
\end{abstract}


\MSC{58J50, 35J05, 58J40, 37J35, 47A30}

\tableofcontents

\section{Introduction}\label{secIntro}

\myindent Let $(M,g)$ be a smooth complete Riemannian manifold of dimension $d$. Let $\Delta$ be the Laplace-Beltrami operator. For $\lambda,\delta \geq 0$, let $P_{\lambda,\delta}$ be the spectral projector of $\sqrt{-\Delta}$ on the frequency interval $[\lambda-\delta,\lambda + \delta]$, which can be defined, through standard functional calculus, as
\begin{equation}\label{defPlambdadelta}
    P_{\lambda,\delta} := \mathds{1}_{[\lambda - \delta, \lambda + \delta]} \left(\sqrt{-\Delta} \right).
\end{equation}

\myindent For $2\leq p \leq \infty$, the question of estimating the operator norm of $P_{\lambda,\delta}$, seen as an operator from $L^2(M)$ to $L^p(M)$, has become in the last forty years a major subject of interest in the field of Spectral Geometry. While being of interest in itself, it is also a relaxation of the more classical problem of estimating the $L^p$ norm of $L^2$ normalized eigenfunctions, in the case of a compact manifold. Indeed, this would correspond to the case $\delta = 0$, which is out of reach in most settings. On the contrary, the case $\delta = 1$ is fully understood since the seminal works of Sogge in the 1980s, as we will see below, and this yields the well-known optimal upper bound on the $L^p$ norm of eigenfunctions in the compact case (see below). Hence, the case $\delta > 0$ but very small can be seen as an intermediate problem, for which promising results have appeared recently. In that spirit, $P_{\lambda,\delta}$ can be interpreted as the spectral projector on $L^2$-normalized \textit{quasimodes} of frequency $\lambda$, with a bandwidth $[\lambda-\delta,\lambda +\delta]$. Moreover, this problem enables a unified approach for compact and non-compact manifolds, since eigenfunctions of the Laplacian are not always $L^2$ normalized for the latter, and, in this spirit, it can be viewed as a generalization of the Stein-Thomas Fourier restriction problem (see \cite{stein1993harmonic,tomas1975restriction}). Finally, in the compact case, the interest of this problem is that it captures both the question of the \textit{norm of eigenfunctions}, and of the \textit{distribution of eigenvalues}.

\myindent For compact manifolds, there are essentially only three results available in the regime where the size of the frequency interval $\delta$ is \textit{polynomially small} compared to the frequency $\lambda$, i.e. $\delta$ can be as small as $\lambda^{-\kappa}$ for some $\kappa > 0$ fixed. First, the case of the sphere is immediate, since there is no dependence in $\delta$ for $\delta$ small. On the contrary, improvements have been obtained in the polynomially thin regime only in two locally symmetric settings, namely the flat tori (Bourgain and others, see below), and the arithmetic surfaces (Iwaniec and Sarnack, see below).

\subsection{Results}\label{subsec1Intro}

\myindent In this article, we derive a method for the study of the spectral projector $P_{\lambda,\delta}$ on a frequency interval of size $\delta$ \textit{polynomially} small compared to $\lambda$, for (a class of) surfaces of revolution (see Paragraph \ref{subsubsec12CdV} for a formal definition). As an example of application, we give the first improved upper bound, to our knowledge,  on the operator norm of the spectral projector $P_{\lambda,\delta}$, in the regime where $\delta$ is polynomially small compared to $\lambda$, with a quantitative control, in the context of compact manifolds with positive curvature. Since our goal is to give a new method for studying this kind of problems, we will not aim for optimality, neither in the result nor in the hypotheses. Rather, we will work under many simplifying hypotheses, and postpone until the end a discussion on possible improvements. The typical model that we have in mind is that of \textit{ellipsoids of revolution} (see Paragraph \ref{subsubsec13CdV} for a formal definition). Hence, we will always assume that surfaces of revolution are \textit{simple}, i.e. they have only one equator, and that they are \textit{symmetric} with respect to this equator. 

\myindent The type of results that we are able to prove with our approach is the following theorem, the proof of which will be the main object of this article.

\begin{theorem}[Main result]\label{mainthm}
    In the set of simple symmetric smooth surfaces of revolution, there is an open set $\mathfrak{S}$, which contains all the ellipsoids of revolution which are close enough to, but not equal to the round sphere, such that, if $\mathcal{S} \in \mathfrak{S}$, $\eps > 0$, and if $K_{\eps}$ is the set of those points of $\mathcal{S}$ which are at a distance at least $\eps$ from both poles of $\mathcal{S}$, then there holds, for any $\lambda \geq 1$,
    \begin{equation}
        \forall \delta \geq \lambda^{-\frac{1}{32}}, \qquad \left\|P_{\lambda,\delta}\right\|_{L^2(\mathcal{S}) \to L^{\infty}(K_{\eps})} \lesssim_{\mathcal{S},\eps} \lambda^{\frac{1}{2}} \delta^{\frac{1}{2}}.
    \end{equation}
    Moreover, there is an explicit $\kappa > 0$ (see Corollary \ref{lowerbounds}) such that, in the range $\delta \geq \lambda^{-\kappa}$, this bound is optimal, in the sense that there holds
    \begin{equation}
        \forall \delta \geq \lambda^{-\kappa}, \qquad \left\|P_{\lambda,\delta}\right\|_{L^2(\mathcal{S}) \to L^{\infty}(K_{\eps})} \gtrsim_{\mathcal{S},\eps} \lambda^{\frac{1}{2}} \delta^{\frac{1}{2}}.
    \end{equation}
\end{theorem}

\myindent In particular, this immediately yields a polynomial improvement compared to the general upper bound of the $L^{\infty}$ norms of eigenfunctions (see below). This polynomial gain was predicted by Bourgain \cite{bourgain1993eigenfunctiona} in the context of generic completely integrable manifolds. To our knowledge, the present article gives the first full rigorous derivation of such a result relying on complete integrability. 

\begin{corollary}\label{mainThmCor}
    Under the hypotheses and notations of Theorem \ref{mainthm}, if $\lambda^2 \geq 0$ is an eigenvalue of $-\Delta$, and $\phi_{\lambda}$ is an $L^2$-normalized eigenfunction associated to $\lambda^2$, there holds
    \begin{equation}
        \|\phi_{\lambda}\|_{L^{\infty}(K_{\eps})} \lesssim_{\mathcal{S},\eps} \lambda^{\frac{1}{2} - \frac{1}{64}}.
    \end{equation}
\end{corollary}

\myindent Before giving the main ideas for the proof, let us comment the results.
\begin{itemize}
\item It is necessary to exclude small disks around both poles, since there are no possible improvements in the $\delta$ small regime around them due to zonal eigenfunction concentration, see \eqref{zonalconcentration}.
\item Similarly, there are no possible improvements to the case $\delta =1$ for $L^p$ norms, in the small $p$ regime, around the equator, due to Gaussian beams quasimodes, see \eqref{gaussianbeams}. In particular, we focus on the $p = \infty$ case, and we refer to \cite{chabert2026l4} for a discussion of polynomially improved upper bounds for other values of $p$ away from the poles and the equator, which can be derived using separation of variables, see also \cite{chabert2025infty}. It is indeed very delicate to give estimates near the equator, as we will see in the proof.
    \item The exponent $\frac{1}{32}$ is not optimal, and we are in fact able to improve it greatly. However, in the spirit of \cite{bourgain1993eigenfunctiona}; our first goal is to prove that Sogge's optimal estimate for the spectral projectors (see Theorem \ref{thmSogge}) extends to the polynomially thin regime, and for all target frequency $\lambda$, that is there exists a constant $\kappa > 0$ such that 
    \begin{equation}
    \forall \delta \geq \lambda^{-\kappa}, \qquad \|P_{\lambda,\delta}\|_{L^2(\mathcal{S}) \to L^{\infty}(K_\eps)} \simeq_{\mathcal{S}, \eps} \lambda^{1/2} \delta^{1/2}.
    \end{equation}
    Moreover, our second goal is to prove that, with a more careful approach than in \cite{bourgain1993eigenfunctiona}, we can reach \textit{quantitatively} small polynomially thin frequency intervals. Now, optimizing the constant $\kappa$, while being possible, as we will discuss in Section \ref{subsec2Further}, would greatly increase the length of the present article, and wouldn't be as much of a theoretical improvement since it is mostly a question of detailing computations. Moreover, the precise type of calculations would be very much specific to surfaces of revolution, while, even though we have this precise model in mind, our proof aims to be adaptable to similar settings, and rather rely on new abstract results for oscillatory integrals such as Theorem \ref{mixedVdCABZ}.
    \item In Section \ref{subsec2Further}, we will discuss the (conjectured) optimal constant $\kappa > 0$ for which the theorem holds, and explain how reaching such fine scales illustrates the role of unstable closed geodesics of $\mathcal{S}$, which are usually not seen in spectral projectors estimates. 
    \item In Section \ref{subsec3Further}, we will discuss the hypotheses that we introduce in order to define the set $\mathfrak{S}$. Indeed, similarly to the choice of a non optimal exponent, we choose not ti give the proof for the largest possible $\mathfrak{S}$, and rather briefly explain what needs to be changed in the proof for more general surfaces of revolution. This will culminate in a conjecture for the equivalent of Theorem \ref{mainthm} for general quantum completely integrable manifolds (see definition below). 
    \item In Section \ref{subsec4Further}, we will give a conjecture regarding the pointwise Weyl law, which we are convinced should follow from similar methods than the ones presented in this article, as we will argue in a future work.
\end{itemize}

\subsection{Method of proof}\label{subsec3Intro}

\myindent The idea of the proof for the case $\delta = 1$ (seeTheorem \ref{thmSogge}), and of many works using the microlocal approach, can be resumed as follows.

\myindent i. First, observe that bounding $P_{\lambda,\delta}$ is essentially equivalent to bounding the smoothed version
\begin{equation}
    P^{\flat}_{\lambda,\delta} := \chi\left(\frac{\sqrt{-\Delta} - \lambda}{\delta}\right),
\end{equation}
where $\chi$ is a smooth cutoff function. 

\myindent ii. Second, one can then express the smoothed projector $P^{\flat}_{\lambda,\delta}$ using the group of unitary operators $t\mapsto e^{it\sqrt{-\Delta}}$, i.e. using the Fourier transform
\begin{equation}\label{Psharpfirstformula}
    P^{\flat}_{\lambda,\delta} = \delta \int \hat{\chi}(\delta t) e^{i\lambda t} e^{it\sqrt{-\Delta}} dt.
\end{equation}

\myindent iii. Third, the interest is now that the Schwartz kernel of $e^{it\sqrt{-\Delta}}$ is very well known, thanks to the existence of \textit{parametrices}, which connect this group of operators to the \textit{geodesic flow} on the cotangent bundle of $M$. Thus, one can use a parametrix to reduce the problem of estimating the Schwartz kernel $P^{\flat}_{\lambda,\delta} (x,y)$ into a problem of \textit{oscillatory integrals}.

\quad

\myindent Now, the problem is that one has to control $t\mapsto e^{it\sqrt{-\Delta}}$ roughly for $t \in [-\delta^{-1}, \delta^{-1}]$. Thus, if we want to be able to choose $\delta = \delta(\lambda)$ very small when $\lambda \to \infty$, we have to deal with larger and larger times. Quantifying this approach is possible, but leads only to $\delta$ \textit{logarithmically small} compared to $\lambda$. In particular, when we are dealing with a manifold for which geodesics pass many times by their starting point, or near their starting point, it seems very hard to go up to the size of the frequency interval $\delta$ polynomially small compared to $\lambda$.

\quad

\myindent In order to avoid this difficulty, we use the additional structure which is given by \textit{Quantum Complete Integrability}. We will start in Section \ref{secCdV} with a detailed presentation of Colin de Verdière's construction in \cite{de1980spectre}, with some developments. In Sections \ref{subsec1CdV} and \ref{subsec2CdV}, we start with basic definitions which set the problem, and with a simple geometric description of the geodesics of a simple symmetric surface of revolution.

\myindent The general idea of \cite{de1980spectre} is to consider the joint spectral study of two commuting pseudodifferential operators, namely $P_1 = \sqrt{-\Delta}$ and $P_2$ the infinitesimal generator of the rotation around the axis of $\mathcal{S}$, i.e. $P_2 = \frac{1}{i} \frac{\partial}{\partial \theta}$ if $\theta \in S^1$ is the angular coordinate on $\mathcal{S}$. On the level of their principal symbols $p_1,p_2$, this corresponds to the well-known setting of the \textit{complete integrability} of the geodesic flow, in the sense of Hamiltonian Geometry. Indeed, the geodesic flow on the cotangent bundle is the Hamiltonian flow of the principal symbol of $\sqrt{-\Delta}$, which is the norm of cotangent vectors. In this setting, we are able to have a simple representation of the geodesics through the introduction of \textit{action-angle coordinates} given by the well-known Liouville-Arnold theorem.

\myindent The observation of Colin de Verdière is that, in the case of \textit{simple} surfaces of revolution, action-angle coordinates are defined \textit{globally} on the cotangent bundle $T^*\mathcal{S}$. Precisely, this means that there exist a smooth diffeomorphism $\mathcal{G} : \R^2 \to \R^2$ such that, if we define
\begin{equation}
    (q_1,q_2)(x,\xi) := (\mathcal{G}(p_1(x,\xi),p_2(x,\xi))),
\end{equation}
then $q_1,q_2$ are smooth functions on $T^*\mathcal{S}$, such that their Hamiltonian flows are $2\pi$-periodic. Moreover, there actually holds that $q_2 = p_2$ is the principal symbol of the operator $P_2$ (which already has a $2\pi$-periodic Hamiltonian flow). This construction is presented in Section \ref{subsec3CdV}.

\myindent Now, an important part of \cite{de1980spectre} is that this complete integrability structure of the geodesic flow on the cotangent bundle of a simple surface of revolution fully transfers to the level of pseudodifferential operators. Precisely, there holds the following, which is part of Theorem \ref{CdVthm} given in Section \ref{subsec4CdV}.
\begin{proposition}\label{reducedCdVthmQ1Q2}
    There exists two commuting pseudodifferential operators of order one $Q_1,Q_2$, whose principal symbols are $q_1,q_2$, and such that, on the one hand,
    \begin{equation}\label{exp2ipiQieqIdlocal}
        e^{2i\pi Q_k} = (-1)^k Id \qquad k = 1,2.
    \end{equation}

    \myindent On the other hand there exists a classical elliptic symbol $F$ of order $2$ on $\R^2$ such that
    \begin{equation}
        F \sim F_2 + F_0 + F_{-1} + ...
    \end{equation}
    (the homogeneous component of order $1$ equals zero as the subprincipal symbol equals zero), and such that
    \begin{equation}
        -\Delta = F(Q_1,Q_2).
    \end{equation}
\end{proposition}

\myindent Finally, in Section \ref{subsec5CdV}, we will give some additional geometric properties.

\quad

\myindent The main idea of the proof is thus to use Proposition \ref{reducedCdVthmQ1Q2} to twist the standard microlocal approach roughly introduced above, as we will present in Section \ref{secMicro}. First, we express the spectral projector $P_{\lambda,\delta}$ in terms of the groups of unitary operators $s \mapsto e^{isQ_1}$ and $t\mapsto e^{itQ_2}$. Indeed, we will construct in Section \ref{subsec1Micro} an adapted smoothed spectral projector of the form 
\begin{equation}
    P^{\sharp}_{\lambda,\delta} = f_{\lambda,\delta} (Q_1,Q_2),
\end{equation}
where
\begin{equation}
    f_{\lambda,\delta}(q_1,q_2) = \left(d\mu_{\lambda} \ast \rho_{\delta}\right)(q_1,q_2),
\end{equation}
where, if $d\mu$ is the superficial measure on the curve $\{F_2 = 1\}$, and $\rho$ is a smooth bump function, $d\mu_{\lambda} := \lambda d\mu(\lambda^{-1} \cdot)$ and $\rho_{\delta} = \delta^{-1}\rho(\delta^{-1}\cdot)$. This is a smooth bump function around the set
\begin{equation}
    \{(q_1,q_2) \in \R^2 \ \text{such that} \ \lambda- \delta \leq \sqrt{F_2(q_1,q_2)} \leq \lambda + \delta\}.
\end{equation}

\myindent Now, the interest is that there holds, similarly to \eqref{Psharpfirstformula},
\begin{equation}
    P^{\sharp}_{\lambda,\delta} = \lambda \delta \int_{\R^2} \hat{d\mu}(\lambda(s,t)) \hat{\rho}(\delta(s,t)) e^{is Q_1} e^{itQ_2} ds dt.
\end{equation}

\myindent The point is that, in this formula, since the groups $s \mapsto e^{isQ1}$ and $t\mapsto e^{itQ_1}$ are $2\pi$-periodic (up to a sign), the issue of quantitative control of long-term parametrices mentioned above disappears, and we only need to deal with parametrices on times $O(1)$. Intuitively, this means that all the delicate behavior of geodesics for large time is fully encoded in the decaying and oscillatory behavior of $\hat{d\mu}(\lambda \cdot)$. However, this function is now a completely explicit oscillatory integral. Moreover, the theory of parametrices allows to express the Schwartz kernel of $(s,t) \mapsto e^{isQ_1} e^{itQ_2}$ as an oscillatory integral, at least locally. We will construct well-chosen parametrices in Section \ref{subsec2Micro}. 

\myindent Overall, we will thus be able to reduce the problem of bounding $P^{\sharp}_{\lambda,\delta}$, and thus $P_{\lambda,\delta}$, to the problem of estimating a countable number of sufficiently explicit oscillatory integrals. The major part of this article, in terms of volume, is thus devoted to these estimates. In order to give intuition, and also to avoid the necessity to read all the technical computations at first reading, we will present a precise overview of the analysis in Section \ref{subsec3Micro}. Moreover, we will introduce the classical and new oscillatory integral estimates that we need in Section \ref{subsec4Micro}, which culminates in the new result of Theorem \ref{mixedVdCABZ}.

\myindent Sections \ref{secHormPhase} to \ref{secAntipod}, which are much more technical, are fully devoted to the derivation of a rigorous quantitative analysis for the different oscillatory integrals which appear in the analysis. In order to shorten the already long proof, we have chosen to give the full proofs only for one type of integrals, which are also the most difficult to bound, in Sections \ref{secHormPhase} and \ref{secHormQuant}. For the two other types, we have thus chosen to give the main steps of the proof without explaining all the details in Sections \ref{secBicharact} and \ref{secAntipod}.

\myindent Finally, in Section \ref{secFurther}, we will come back to various parts of the proof, in order to motivate further results and conjectures. In Section \ref{subsec1Further}, we will explore \textit{lower bounds} on the $L^2 \to L^{\infty}$ norm of spectral projectors. In Section \ref{subsec2Further}, we will discuss the possible improvements of Theorem \ref{mainthm} (and thus of the other results that the analysis should yield), in particular away from the equator. Indeed, we have chosen to present a method of proof which is the most likely to extend to other settings. Hence, in many estimates, we are actually able to prove much better that what we state, by using specially adapted methods. Thus, we will motivate a conjecture for the \textit{optimal} version of Theorem \ref{mainthm}, i.e. the optimal $\kappa > 0$ such that the estimate of the theorem holds for $\delta > \lambda^{-\kappa}$, at least outside of the equator. As we will explain, this exponent is quite surprising, and we see the difference with the case of the flat torus, which stems from the particular geometry of \textit{non stable closed geodesics} of surfaces of revolution. In Section \ref{subsec3Further}, we will discuss the many simplifying hypotheses that we will introduce in the analysis, and explain how to remove them, and how much more work is needed for the removal of each. This will culminate in the discussion of the possible extensions of Theorem \ref{mainthm} in other Quantum Completely Integrable geometries. Finally, we will mention in Section \ref{subsec4Further} that the same analysis should yield a polynomial quantitative improvement on the remainder of the pointwise Weyl law.

\subsection{Earlier Works}\label{subsec2Intro}

\myindent We briefly recall some of the landmark results of Spectral Geometry, in order to contextualise the question of estimating the norm of spectral projectors. 

\subsubsection{The global and pointwise Weyl laws}

\myindent Consider $(M,g)$ a smooth \textit{compact} Riemannian manifold of dimension $d$, let $0 \leq \lambda_0^2 \leq \lambda_1^2 \leq ...$ be the eigenvalues of $-\Delta$, repeated with multiplicity. A long line of research has been to study the \textit{asymptotic distribution} of the eigenvalues $(\lambda_i)_{i\geq 0}$, and its dependency on the geometry of $M$. This is usually measured by the \textit{eigenvalue counting function}
\begin{equation}
    N(\lambda) := \#\{j\geq 0 \ \text{such that} \ \lambda_j \leq \lambda\}.
\end{equation}

\myindent Then, the \textit{Weyl law} (named after the pioneering article \cite{weyl1912asymptotische}) states that
\begin{equation}\label{GlobWeyl}
    N(\lambda) = (2\pi)^{-d} \omega_d \upsilon(M) \lambda^d + R(\lambda),
\end{equation}
where $R(\lambda) = o(\lambda^d)$, $\omega_d$ is the volume of the unit ball in $\R^d$, and $\upsilon(M)$ is the volume of $M$ for the metric $g$.

\myindent Weyl conjectured that the remainder term $R(\lambda)$ was of order $O\left(\lambda^{d-1}\right)$. This was established with a logarithmic loss by Richard Courant in 1922, and without loss for compact closed manifolds by Levitan in \cite{levitan1953asymptotic}. This estimate cannot be improved in the general case, as if one takes $M = S^d$ the $d$-dimensional sphere, the remainder term $R(\lambda)$ can be shown to be exactly of order $\lambda^{d-1}$, due to the \textit{multiplicity} of eigenvalues. However, it is expected that the remainder term is $o\left(\lambda^{d-1}\right)$ for "generic" manifolds. Thus, the theory has been refined to the study of the remainder term $R(\lambda)$ and its interaction with the geometry of $M$. For example, for the $d$-dimensional flat torus $\mathbb{T} := \R^d /(2\pi \Z)^d$, the eigenvalues are the $|n|^2$, for $n \in \Z^d$. Hence, the problem of estimating $R(\lambda)$ is essentially the $d$-dimensional Gauss problem, i.e. the estimation of the number of points with integer coordinates in a ball of radius $\lambda$. For $d = 2$, it is thus conjectured since \cite{hardy1917average} that $R(\lambda) = O(\lambda^{\frac{1}{2} + \eps})$ for any $\eps > 0$, and \cite{hardy1917dirichlet} establishes that the bound $O(\lambda^{\frac{1}{2}})$ doesn't hold. The best bound today is $O(\lambda^{\frac{131}{208}})$, proved in \cite{huxley2002integer}. 

\myindent From the 1950s onward, the theory of \textit{microlocal analysis} and in particular Hörmander's theory of Fourier Integral Operators (\cite{hormander2009analysis,duistermaat1972fourier,hormander1971fourier}) enabled a new approach of the Weyl law. Indeed, modern presentations rely on the \textit{pointwise Weyl law}. Consider the spectral projector
\begin{equation}
    P_{\lambda} := \mathds{1}_{[0,\lambda]}(\sqrt{-\Delta}).
\end{equation}

\myindent Then, the \textit{pointwise counting function} is 
\begin{equation}\label{defNlambdax}
    N(\lambda,x) = P_{\lambda}(x,x),
\end{equation}
where $P_{\lambda}(x,y)$ is the Schwartz kernel of $P_{\lambda}$, which is a well-defined smooth function of $x$. Indeed, let $(\phi_j)_{j\geq 0}$ be an orthonormal basis of eigenfunctions. Then, thanks to elliptic regularity, the $\phi_j$ are smooth, and there holds
\begin{equation}
    P_{\lambda}(x,x) = \sum_{\lambda_j \leq \lambda} |\phi_j(x)|^2.
\end{equation}

\myindent In particular, since the $(\phi_j)$ are normalized in $L^2$, we find that, if $d\upsilon(x)$ is the volume form induced by the metric $g$ on $M$,
\begin{equation}
    N(\lambda) = \int_M N(\lambda,x) d\upsilon(x).
\end{equation}

\myindent The pointwise Weyl law, stated in the seminal work of Hörmander \cite{hormander1968spectral}, is that
\begin{equation}\label{locWeyl}
    N(\lambda,x) = (2\pi)^{-d}c(x) \lambda^d + R(\lambda,x),
\end{equation}
where $c(x)$ is the volume of the unit ball in the cotangent space at $x$ $T_xX^*$, and $R(\lambda,x) = O\left(\lambda^{d-1}\right)$ uniformly over $M$. The global Weyl law \eqref{GlobWeyl} is thus a straightforward consequence of \eqref{locWeyl}. A textbook presentation of the pointwise Weyl law can be found in \cite[Chapter IV]{sogge2017fourier}.

\myindent The microlocal approach has revealed the deep connection between the distribution of eigenvalues, and the length of the \textit{periodic geodesics} of $M$, see \cite{colin1973spectre,colin1973spectreb}. In particular, Duistermaat and Guillemin proved in the seminal article \cite{duistermaat1975spectrum} that the remainder term for the (global) Weyl law is $o\left(\lambda^{d-1}\right)$ if the set of \textit{periodic geodesics} has measure zero in the cotangent bundle $T^*M$, which was generalized by Ivrii in \cite{ivrii1980second}.

\myindent The question of the distribution of eigenvalues has since known some major developments, and has been generalized to other settings, see the monograph of Ivrii \cite{ivrii2013microlocal}. Let us mention that logarithmic improvements on the remainder have been obtained recently under quite general hypotheses in \cite{canzani2023weyl}.

\subsubsection{The $L^p$ norm of eigenfunctions}

\myindent A second line of research is the study of the \textit{eigenfunctions} $(\phi_j)$ (which are not unique due to the multiplicity of eigenvalues), and, in particular, the asymptotic study of the $L^p$ norms, $p > 2$, of eigenfunctions. Indeed, this is a good way of measuring the asymptotic distribution of mass of eigenfunctions. The classical Hörmander bound, from the already mentionned \cite{hormander1968spectral}, is that 
\begin{equation}\label{hormupperboundeigen}
    \|\phi_j\|_{L^{\infty}(M)} \lesssim \lambda_j^{\frac{d-1}{2}}.
\end{equation}

\myindent This has been generalized to the following much stronger result of Sogge in \cite{sogge1988concerning}, concerning the spectral projector $P_{\lambda,1}$ (see \eqref{defPlambdadelta}).
\begin{theorem}[Sogge]\label{thmSogge}
    If $M$ is a complete Riemannian manifold with bounded geometry, of dimension $d\geq 2$, then for any $2\leq p \leq \infty$,
    \begin{equation}
        \|P_{\lambda,1}\|_{L^2(M)\to L^p(M)} \lesssim \lambda^{\gamma(p)}.
    \end{equation}
    
    \myindent Moreover, there exists a constant $R_0$ such that, for all $\lambda_0$, there exists a $\lambda$ with $|\lambda - \lambda_0|\leq R_0$ such that
    \begin{equation}
        \|P_{\lambda,1}\|_{L^2(M)\to L^p(M)} \gtrsim \lambda^{\gamma(p)}, \qquad 2\leq p \leq \infty.
    \end{equation}
    
    \myindent Here, 
    \begin{equation}\label{defgammap}
        \gamma(p) = \max\left[\frac{d-1}{2} - \frac{d}{p}, \frac{d-1}{2}\left(\frac{1}{2} - \frac{1}{p}\right)\right].
    \end{equation}
\end{theorem}

\myindent In the case where $M$ is compact, this yields immediately the following bound on the eigenfunctions
\begin{equation}\label{Soggeupperboundeigen}
    \|\phi_{\lambda}\|_{L^p} \lesssim \lambda^{\gamma(p)}.
\end{equation}

\myindent This upper bound is optimal in the general case, since, for the sphere $S^d$, one can saturate the bound as follows. If we define the Stein-Thomas exponent $p_{ST} := \frac{2(d+1)}{d-1}$, which discriminates between the two regimes for $\gamma(p)$ defined by \eqref{defgammap}, then

i. For $p_{ST}\leq p \leq \infty$, the eigenfunctions which saturate the upper bound \eqref{Soggeupperboundeigen} focus on a point. This is achieved by the \textit{zonal} spherical harmonics, which concentrate on the Poles of $S^d$.

ii. For $2\leq p \leq p_{ST}$, the eigenfunctions which saturate the upper bound \eqref{Soggeupperboundeigen} focus on a \textit{stable closed geodesic}. This is achieved by the \textit{highest-weight} spherical harmonics, which concentrate on the equator of $S^d$.

\myindent In generic settings, however, it is expected that the upper bound \eqref{Soggeupperboundeigen} can be improved. For the $L^{\infty}$ norm, it was proved in particular in \cite{sogge2002riemannian} that, for generic metrics, $\|\phi_j\|_{L^{\infty}(M)} = o\left(\lambda_j^{\frac{d-1}{2}}\right)$. Moreover, the converse question, namely finding conditions on $M$ for \textit{maximal eigenfunction growth}, i.e. the existence of a sequence of eigenfunctions saturating the bound \eqref{hormupperboundeigen}, has been studied in many works, such as \cite{sogge2002riemannian,sogge2011blowup,sogge2016focal,sogge2016focalb,sogge2001riemannian,canzani2019growth,canzani2021eigenfunction}.

\quad

\myindent With additional assumptions on the geometry of $M$, the upper bound can even be \textit{quantitatively} improved.

\myindent In the case of the regular flat torus, the results of \cite{cooke1971cantor,zygmund1974fourier} yield that, on $\mathbb{T}^2$, $\|\phi_j\|_{L^4} \lesssim 1$, and the case $p > 4$ is still open today. For $d > 2$, Bourgain conjectured in \cite{bourgain1993eigenfunction} that there holds on $\mathbb{T}^d$
\begin{equation}
    \|\phi_j\|_{L^p} \lesssim \lambda_j^{\frac{d}{2} - 1 - \frac{d}{p}} \qquad p > \frac{2d}{d-2}.
\end{equation}

\myindent Some cases of this conjecture were proved in \cite{bourgain2013moment,bourgain2013improved,bourgain2015new}, and with the seminal $l^2$ decoupling estimate of Bourgain and Demeter \cite{bourgain2015proof}.

\myindent For \textit{arithmetic surfaces}, it is conjectured in \cite{iwaniec1995norms} that there should hold $\|\phi_{\lambda_j}\|_{L^{\infty}} \lesssim \lambda_j^{\eps}$. For recent progresses on that conjecture, see \cite{buttcane2017fourth,humphries2018equidistribution,humphries2022p}.

\myindent For more general manifolds with \textit{nonpositive curvature}, many works have obtained explicit \textit{logarithmic} improvements on the upper bound of $L^p$ norms of eigenfunctions, such as \cite{berard1977wave,hassell2015improvement,hezari2016lp,blair2017refined,blair2018concerning,blair2019logarithmic}. Moreover, in the context of \textit{magnetic Laplacians} on hyperbolic surfaces, one observes both saturation of $L^p$ bounds in the low energy regime, and polynomial improvement on $L^p$ norms of eigenfunctions in the critical energy regime, as discussed in \cite{chabert2026zonalstatesimprovedlinfty}, while the high energy regime is similar to the standard laplacian on hyperbolic surfaces.

\myindent In the context of \textit{completely integrable} manifolds (see Paragraph \ref{subsubsec31CdV} for a definition), it is expected that, generically, there should be some \textit{polynomial} improvements on the upper bound of $L^p$ norms of eigenfunctions. Bourgain claimed in  \cite{bourgain1993eigenfunctiona} that, for $M$ completely integrable, and for $
U$ a generic subset of $M$, there should hold $\|\phi_j\|_{L^{\infty}(U)} \lesssim \lambda_j^{\frac{d-1}{2} - \eps}$ for some $\eps > 0$ depending on $M$. In particular, our article proves this result in the case of (some) surfaces of revolution. In the setting of general completely integrable manifolds, the converse question of exhibiting sequences of eigenfunctions with $L^p$ norm quantitatively polynomial in the eigenvalue has been studied in \cite{toth2002riemannian,toth2003lp,toth2003norms}.

\subsubsection{Spectral projectors on thin frequency intervals}

\myindent Few results are available in the recent line of research on the problem of bounding the spectral projectors $P_{\lambda,\delta}$ for $0 < \delta \ll 1$, see the review of Germain \cite{germain2023l2}. For a general flat torus $\R^d/ \Lambda$, where $\Lambda$ is a lattice, it is conjectured in \cite{germain2022boundsBook} that there holds
\begin{equation}\label{conjGermain}
    \|P_{\lambda,\delta}\|_{L^2 \to L^p} \lesssim \lambda^{\frac{d-1}{2} - \frac{d}{p}}\delta^{\frac{1}{2}} + \left(\lambda \delta\right)^{\frac{d-1}{2}\left(\frac{1}{2} - \frac{1}{p}\right)} \qquad 2\leq p \leq \infty, \ \delta > \lambda^{-1}.
\end{equation}

\myindent Several cases of the conjecture have been established in \cite{germain2022bounds,demeter2024l2,hickman2020uniform}. Moreover, results have been obtained for non-compact manifolds, for which we recall that spectral projectors are well-defined even if eigenfunctions are not, see \cite{germain2022bounds2,germain2023spectral}.

\myindent Regarding lower bounds, it is mentioned in Germain's review \cite{germain2023l2} that, from the existing literature (cited in said review) on quasimodes supported on stable closed geodesic (see Colin de Verdière's article \cite{de1977quasi} for the construction of quasimodes), the existence of a stable closed geodesic implies that, for all $N > 0$, along a sequence $\lambda_j$,
\begin{equation}\label{gaussianbeams}
    \|P_{\lambda_j,\delta}\|_{L^2(
    M) \to L^p(M)} \gtrsim_N \lambda_j^{\frac{d-1}{2}\left(\frac{1}{2} - \frac{1}{p}\right)} \qquad \delta > \lambda_j^{-N}.
\end{equation}

\myindent In particular, coming back to Theorem \ref{thmSogge}, there is no possible amelioration for $\delta$ small in the case $2\leq p \leq p_{ST}$. Thus, since, in our study of surfaces of revolution, we include the equator, which is a stable closed geodesic, we focus on the case $p = \infty$.

\subsubsection{Earlier results for surfaces of revolution}

\myindent In our study of surfaces of revolution, we will strongly use the fact that, due to the symmetry of revolution, they are \textit{completely integrable}, and, precisely, that they are in the class of \textit{quantum completely integrable} (QCI) manifolds, as defined in the recent article \cite{eswarathasan2024pointwise}, where a microlocalized Weyl law is established for such manifolds. One can date the study of QCI manifolds back to Colin de Verdière's article \cite{de1980spectre}, where, among many deep results, it is proved the following theorem.
\begin{theorem}[Colin de Verdière]\label{thmCdVRemainders}
    For $\mathcal{S}$ a generic simple surface of revolution, the global Weyl remainder satisfies
    \begin{equation}
        R(\lambda) = O(\lambda^{\frac{2}{3}}).
    \end{equation}
\end{theorem}

\myindent This estimate has been refined in \cite{bleher1994distribution} into an explicit asymptotic expansion of the remainder term $R(\lambda)$, for a simple surface of revolution satisfying a twist hypothesis, or more generally a Diophantine hypothesis (see definitions in the cited article). We will come back in the following on the precise meaning of the QCI hypothesis. A similar estimate than Theorem \ref{thmCdVRemainders} has been proved in \cite{colin2010remainder} for the Euclidean disk.

\myindent Another way of seeing the symmetry of revolution for surfaces of revolution is that it is an example of manifolds which are invariant under a group action (here, by the circle $S^1 := \R/(2\pi\Z)$). Donelly studied in \cite{donnelly2001bounds} the problem of growth of eigenfunctions for such manifolds. It is proved moreover in \cite{donnelly1978g} that surfaces of revolution always exhibit \textit{maximal eigenfunction growth} due to concentration at the poles. Indeed, if $\mathcal{S}$ is a surface of revolution, choosing a basis of eigenfunctions which are decomposed on (the equivalent of) spherical harmonics, say $(\phi_j)$, yields that, for $P$ a pole of $\mathcal{S}$, there always hold
\begin{equation}\label{zonalconcentration}
    \sup_{\lambda_j\leq \lambda} |\phi_j(P)| \gtrsim \lambda^{\frac{1}{2}}.
\end{equation}

\myindent Hence, if we wish to prove non trivial improvements on $\|P_{\lambda,\delta}\|_{L^2 \to L^{\infty}}$ for $\delta < 1$, it is necessary to exclude small disks around the poles, as in Theorem \ref{mainthm} and Corollary \ref{mainThmCor}.

\quad

\myindent To conclude this section, let us mention that one can directly prove Corollary \ref{mainThmCor} away from the equator, with a better exponent, and more generally an upper bound on $P_{\lambda,\delta}$ for a choice of $\delta$ polynomially small, by using the \textit{separation of variables}, as discussed in the closely related \cite{chabert2025infty}. Indeed, similarly to the usual diagonalization of the Laplacian on a sphere by the spherical harmonics, one can find a joint basis of eigenfunctions of the Laplacian and of the infinitesimal generator of rotations on $\mathcal{S}$ of the form
\begin{equation}
    \phi_{k,l}(\theta,\sigma) = e^{ik\theta} \Phi_{k,l}(\sigma), \qquad |k| \leq l, \ l \in \{0,1,...\}
\end{equation}
where the coordinates $(\theta,\sigma)$ are respectively the angular and normal coordinates (see Paragraph \ref{subsubsec11CdV} for a formal definition). Thus, the eigenfunction equation reduces to an ODE on $\Phi_{k,l}$, which can be studied through WKB approximation.

\myindent While this approach is more straightforward, shorter, and applies also to the Euclidean disk, in order to prove upper bounds on eigenfunctions, it is of a different interest than ours. Indeed, on the one hand, we are for the moment unable to apply the approach of \cite{chabert2025infty} near the equator. On the other hand, the separation of variables approach proves \textit{non optimal} upper bounds on the spectral projector, and, in particular, it is not possible to derive exact estimate of the spectral projector, or even nontrivial lower bounds from it. Moreover, the question of estimating spectral projectors, or quasimodes, has its own interest. Finally, this approach would not be likely to apply in higher dimensional or more abstract settings. We refer to \cite{schlag2010decay} for another example of results obtained using the separation of variables in another context.

\subsection{Acknowledgments}\label{subsec4Intro}

\myindent I would not have been able to write this article without the help of my PhD advisors, Professor Pierre Germain and Professor Isabelle Gallagher, to whom I wish to express my warmest thanks for the countless discussions that we had, and for continuously giving me feedback on my work and on this manuscript, and believing in me throughout the long time that it took to finish it. I am also very grateful to Professor Jean-Marc Delort, who kindly took of his time on several occasions to help me with certain technical parts of the proof. Finally, I also address sincere thanks to Professors Matthieu Léautaud and Thomas Duyckaerts, who took of their time to discuss some parts of this work.

\myindent I express my warmest thanks to Professor Yves Colin de Verdière and Professor Christopher Sogge, who kindly agreed to give me their opinion and some comments on the first version of the manuscript.

\myindent Preliminary work for this article was done during my stay at Imperial College London as a visiting student. I am financially supported by the Ecole Normale Supérieure, Paris, France.

\section{The Classical and Quantum Complete Integrability : a presentation of Colin de Verdière's construction and some developments}\label{secCdV}

\myindent In this section, we detail the idea of the construction used by Colin de Verdière in \cite{de1980spectre} to prove, among many other results, Theorem \ref{thmCdVRemainders}. As briefly mentioned in the introduction, the main point is to transfer the classical complete integrability of the geodesic flow, to a quantum version, at the level of pseudo differential operators, i.e. Theorem \ref{CdVthm}. We insist that \cite{de1980spectre} deals with more general manifolds than surfaces of revolution, hence we don't give a presentation of the article in its full generality.

\subsection{Notations, definitions, first results}\label{subsec1CdV}

\subsubsection{Notations and a reminder on cotangent bundles }\label{subsubsec11CdV}

\begin{notation}
    Let $M$ be a smooth manifold of dimension $d$. We use the following standard notations
    \begin{equation}
        \begin{split}
            &TM \ \text{is the tangent bundle over} \ M \\
            &T^*M \ \text{is the cotangent bundle over} \ M \\
            &T^*M\backslash \{0\} \ \text{is the bundle of nonzero cotangent vectors over}\ M \ \text{where the zero section is removed}.
        \end{split}
    \end{equation}
\end{notation}

\myindent Let $M$ be a smooth manifold of dimension $f$. The cotangent bundle $T^*M$ is a symplectic manifold for the canonical \textit{Liouville 2-form} $-d\lambda$ which is given on a coordinate patch $(x,\xi)$ by 
\begin{equation}\label{liouvform}
    -d\lambda = \sum_{i=1}^d dx^i \wedge d\xi^i.
\end{equation}

\myindent If $(M,g)$ is a Riemannian manifold, $T^*M$ is equipped with the \textit{norm of cotangent vectors} $p(x,\xi)$, given in local coordinates by 
\begin{equation}\label{normcot}
    p(x,\xi) := \sqrt{\sum_{i,j} g^{ij}(x) \xi_i \xi_j},
\end{equation}
where $g^{ij}(x)$ is the inverse of the matrix  $g(x) = (g_{ij}(x))_{i,j}$ of the scalar product on $T_xM$ in local coordinates.

\begin{notation}
    Let $(M,g)$ be a smooth Riemannian manifold. We define
    \begin{equation}
        \begin{split}
            &SM \ \text{is the unit tangent bundle over} \ M \\
            &S^*M \ \text{is the unit cotangent bundle over} \ M.
        \end{split}
    \end{equation}
\end{notation}
\begin{definition}\label{geodflow}
    Let $(M,g)$ be a compact Riemannian manifold and let $-d\lambda$ be the canonical 2-form on $T^*M$. The geodesic flow $\Phi_t$ on $T^*M$ is the Hamiltonian flow of the norm $p : T^*M \to \R$ defined by \eqref{normcot}, that is in local coordinates, $\Phi_t$ is the unique solution to the Cauchy problem
    \begin{equation}
    \begin{cases}
        \frac{\partial \Phi_t}{\partial t} = \begin{pmatrix} \nabla_{\xi} p \circ \Phi_t \\
       -\nabla_x p\circ \Phi_t \end{pmatrix}\\
       \Phi_0 = Id
    \end{cases}
    \end{equation}
\end{definition}

\myindent We recall that, by standard properties of Hamiltonian flow, $\Phi_t$ preserves the norm $p$. Moreover, since $p$ is homogeneous of degree $1$ in the $\xi$ variable, it is also standard that the geodesic flow is homogeneous of degree zero in the $\xi$ variable on $T^*M\backslash\{0\}$. In particular, it restricts naturally to a flow on $S^*M$.

\myindent The geodesic flow on the cotangent bundle is the natural way to consider geodesics from a microlocal point of view. Indeed, it is related to the standard geodesic flow on $M$ by the following lemma.
\begin{lemma}\label{eqgeod}
    For all $x \in M$, for all $v\in S_x M$, we denote the geodesic starting at $x$ with initial velocity vector $v$ by
    \begin{equation}
        t \mapsto \exp_x(tv) \in M.
    \end{equation}
    \myindent Let $\xi \in S_x^*M$ be canonically associated to $v$ by the isometric isomorphism $S_x M \to S_x^*M$ given by 
    \begin{equation}
        v \in S_x M \mapsto g(x) v \in S_x^*M,
    \end{equation}
    where $g(x)$ is the matrix $(g_{ij}(x))$ giving the scalar product on $T_xM$ in local coordinates. Then there holds
    \begin{equation}
        \forall t \in \R \ \exp_x(tv) = P(\Phi_t(x,\xi)),
    \end{equation}
    where $P : T^*M \to M$ is the bundle projection.
\end{lemma}

\myindent That this correspondence exists is a consequence of the equivalence between Euler-Lagrange systems and Hamiltonian systems under the \textit{Legendre transformation} that arises from Lagrangian mechanics as explained in \cite[section 3.8]{duistermaat1996fourier}. 

\quad

\myindent More generally, let $M$ be a smooth manifold and let $q$ be any \textit{elliptic symbol} on $T^*M$ i.e. $q : T^* M \backslash \{0\} \to \R$ is smooth, homogeneous of degree 1, and positive.

\begin{definition}\label{defbicharacteristics}
    We set $t\mapsto \Phi_t^q$ the Hamiltonian flow of $q$. More precisely, for $(x,\xi) \in T^*M\backslash \backslash\{0\}$, 
    \begin{equation}
        \Phi_t^q(x,\xi) =: (x(t),\xi(t))
    \end{equation}
    is the unique solution to the Cauchy problem
    \begin{equation}\label{explicithamiltonianflow}
        \begin{cases}
            &\dot{x} = \nabla_{\xi} q(x(t),\xi(t)) \in T_{x(t)} M \backslash\{0\} \\
            &\dot{\xi} = - \nabla_x q(x(t),\xi(t)) \in T_{x(t)}^*M\backslash\{0\} \\
            &(x(0),\xi(0)) = (x,\xi)
        \end{cases}.
    \end{equation}
    \myindent We call the curves $t\mapsto (x(t),\xi(t)) \in T^*M\backslash \{0\}$ the \textit{bicharacteristics} of $q$. Moreover, we call their projections $t \mapsto x(t)$, which are curves on $M$, the \textit{bicharacteristic curves} of $q$.
\end{definition}

\subsubsection{Notations and definitions for surfaces of revolution}\label{subsubsec12CdV}

\myindent We fix $\mathcal{S} = (S^2,g)$ a \textit{surface of revolution}. It can be described as a Riemannian structure on the two-sphere $S^2$ defined by a metric $g$ which is invariant under an isometric smooth action of $S^1 := \R /(2\pi\Z)$ on $\mathcal{S}$ with two fixed points, the North and South Poles $N$ and $S$, see Figure \ref{surfrevfig} from \cite{de1980spectre}.

\begin{figure}[h]
\includegraphics[scale = 0.8]{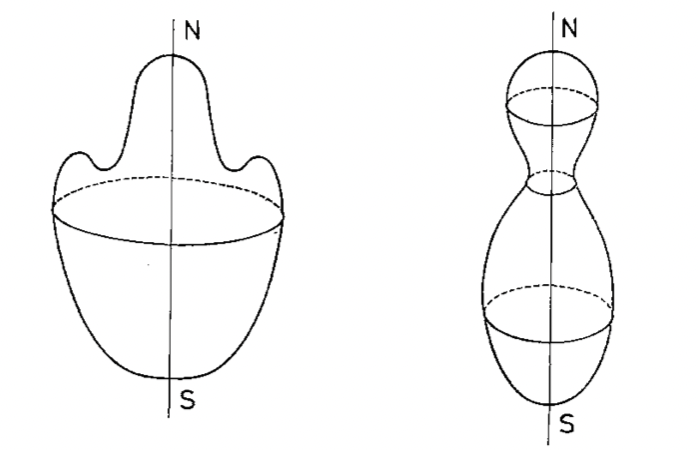}
\centering
\caption{Two examples of surfaces of revolution}
\label{surfrevfig}
\end{figure}

\myindent Outside of the poles, local coordinates are $\nu \in ]0,L[$ the normal coordinate, which is the geodesic distance between a given point and the \textit{South Pole} $S$ following a meridian, and $\theta \in S^1$ the angular coordinate, i.e. the polar angle between a point and an arbitrary fixed meridian. In order to have the correct orientation on $\mathcal{S}$, the local coordinates are $(\theta,\nu)$. We observe that, with our convention, $\nu = 0$ is the South Pole and $\nu = L$ is the North Pole. In the local coordinates $(\theta,\nu)$ the metric is
\begin{equation}
    g = (f(\nu))^2 d\theta^2 + d\nu^2,
\end{equation}
where $f$ is called the \textit{profile} of $\mathcal{S}$, since the equation 
\begin{equation}
    r= f(\nu) \qquad \nu \in [0,L]
\end{equation}
alternatively defines  $\mathcal{S}$ as a surface of $\R^3$ equipped with the induced metric, $r$ being the radial coordinate in usual cylindrical coordinates. In order for $\mathcal{S}$ to be smooth, we impose that $f$ is smooth on $[0,L]$, $f > 0$ on $]0,L[$, $f(0) = f(L) = 0$ and finally $f'(0) = f'(1) = 1$. Thus, for example, any unit speed curve joining $S$ to $N$ with $\theta = \theta_0$ constant is a geodesic curve of length $L$.

\myindent We now detail the additional structure that we need on the surface of revolution $\mathcal{S}$.
\begin{definition}\label{simplehypothesis}
    A smooth surface of revolution $\mathcal{S}$ is $\textit{simple}$ if it admits one and only one equatorial geodesic (the closed path $\nu = \nu_0$ constant, $\theta \in S^1$ is a geodesic). In terms of the profile $f$, this means that $f$ admits a unique non degenerate critical point $\nu_{max} \in ]0,L[$, where $f''(\nu_{max}) < 0$. The geodesic $\nu = \nu_{max}$, $\theta \in S^1$ is the \textit{equator} of $\mathcal{S}$, and we denote it by $\gamma_E$.
\end{definition}

\myindent For example, in Figure \ref{surfrevfig}, the surface of revolution on the left is simple, while the one on the right is not. For the sake of simplicity, we will moreover assume (without loss of generality) that $f(\nu_{max}) = 1$.
\begin{definition}\label{symmetrichypothesis}
    A smooth simple surface of revolution $\mathcal{S}$ is \textit{symmetric} if it is stable under the reflection symmetry with respect to its equator. In terms of the profile, this means on the one hand that $\nu_{max} = L/2$, and on the other hand that $f(\nu) = f(L - \nu)$ for all $\nu \in [0,L]$.
\end{definition}

\myindent For example, the two surfaces of revolution in Figure \ref{surfrevfig} are not symmetric, while the ellipsoids of revolution in Figure \ref{ellipsoids} are symmetric. Until the end, we assume that all the surfaces of revolution we mention are \textit{simple} and \textit{symmetric}. Thanks to this assumption, we can use the following notation.
\begin{notation}
    Let 
    \begin{equation}
        \sigma := \nu - \frac{L}{2} \qquad \nu \in [0,L]
    \end{equation}
    be the algebraic normal distance between a given point and the equator (following a meridian). Until the end, we describe $\mathcal{S}$ with local coordinates $(\theta,\sigma)$ instead of $(\theta,\nu)$. In particular, for any function $g(\nu)$, we will use the abuse of notation $g(\sigma) := g(\sigma + L/2)$. With this convention, the equator is described by the equation $\sigma = 0$, and the profile $f$ is an even function of $\sigma$ defined on $\left[-\frac{L}{2}, \frac{L}{2}\right]$.
\end{notation}
\begin{example}
    In the case where $\mathcal{S}$ is the usual sphere $S^2$, then $L = \pi$ and the profile is
    \begin{equation}
        f(\sigma) = \cos(\sigma).
    \end{equation}
\end{example}

\begin{definition}\label{antipodpoint} For any $x = (\theta,\sigma) \in \mathcal{S}$, we define $\bar{x}$ the \textit{antipodal} point of $x$ by 
\begin{equation}
    \bar{x} = (\theta + \pi, -\sigma).
\end{equation}
\myindent Alternatively, $\bar{x}$ is the point reached after following the meridian passing through $x$ for a time $L$. 

\myindent We can moreover extend this definition to the cotangent bundle, using the canonical extension of a diffeomorphism on a manifold to a diffeomorphism on its cotangent bundle. Denoting $(\bar{x},\bar{\xi})$ this extension for any $(x,\xi) \in T^* M$, there holds
\begin{equation}
    \begin{cases}
        &\text{if} \ x \notin \{N,S\} \ \text{and} \ (x,\xi) = (\theta,\sigma,\Theta,\Sigma) \ \text{then} \ (\bar{x},\bar{\xi}) := (\theta + \pi,-\sigma,  \Theta,-\Sigma) \\
        &\text{if} \ x \in \{N,S\} \ \text{then} \ (\bar{x},\bar{\xi}) := \Phi_{L}(x,\xi)
    \end{cases},
\end{equation}
where $(\theta,\sigma,\Theta,\Sigma)$ are the coordinates induced on the open set $T^*\mathcal{S} \backslash\{T_{N}^*\mathcal{S} \cup T_{S}^* \mathcal{S}\}$ by the coordinates $(\theta,\sigma)$.
\end{definition}

\myindent The local coordinates $(\theta,\sigma)$ extend canonically to local coordinates $(\theta,\sigma,\Theta,\Sigma)$ on the cotangent bundle $T^*\mathcal{S}\backslash\left(T_N^*\mathcal{S} \cup T_S^*\mathcal{S}\right)$. In those coordinates, the norm of cotangent vectors \eqref{normcot} is given by
\begin{equation}\label{normsurfrev}
    p\left(\theta,\sigma,\Theta,\Sigma\right) = \sqrt{\frac{\Theta^2}{f^2(\sigma)} + \Sigma^2}.
\end{equation}
\myindent In particular, we observe that $p$ is \textit{independent of the angular coordinate} $\theta$, as could be expected from the symmetry of revolution. Hence, until the end, we will drop the notation of the $\theta$ dependency and use the abuse of notation $p(\sigma,\Theta,\Sigma)$.

\subsubsection{The model of ellipsoids of revolution}\label{subsubsec13CdV}

\myindent In this paragraph, we define a particular class of surfaces of revolution, the \textit{ellipsoids of revolution}. Throughout the rest of this article, they will always be the model of surfaces of revolution that we have in mind. They are obtained by scaling the sphere $S^2$ along a privileged axis. Precisely, we define
\begin{definition}\label{defellipsoids}
    Let $a,b > 0$. We define $\mathcal{E}(a,b)$ the embedded surface of $\R^3$ defined by the equation, in Cartesian coordinates,
    \begin{equation}
        \frac{x^2}{a^2} + \frac{y^2}{a^2} + \frac{z^2}{b^2} = 1.
    \end{equation}
\end{definition}
\myindent In particular, with this definition, $\mathcal{E}(A,B)$ is a surface of revolution, with North Pole $N := (0,0,b)$ and South Pole $S = (0,0,-b)$. Moreover, it is an example of \textit{simple} surface of revolution in the sense of Definition \ref{simplehypothesis}, its equator being the circle $\{z = 0, \ x^2 + y^2 = a^2\}$. Finally, it is also an example of \textit{symmetric} surface of revolution in the sense of Definition \ref{symmetrichypothesis}. We don't define $\mathcal{E}(A,B)$ through its profile since it is not explicit and it involves special types of \textit{elliptic integrals}.

\myindent There are three important class of ellipsoids of revolution

\myindent i. The case $a = b$ simply gives a \textit{sphere}. We will say that this is a case of \textit{degenerate} ellipsoid of revolution

\myindent ii. The case $a > b$ is called an \textit{oblate} ellipsoid of revolution, and it corresponds to the blue manifold in Figure \ref{ellipsoids} from Wikipedia.

\myindent iii. Conversely, the case $a < b$ is called an \textit{oblong} ellipsoid of revolution, ant it corresponds to the yellow manifold in Figure \ref{ellipsoids}.

\begin{figure}[h]
\includegraphics[scale = 0.8]{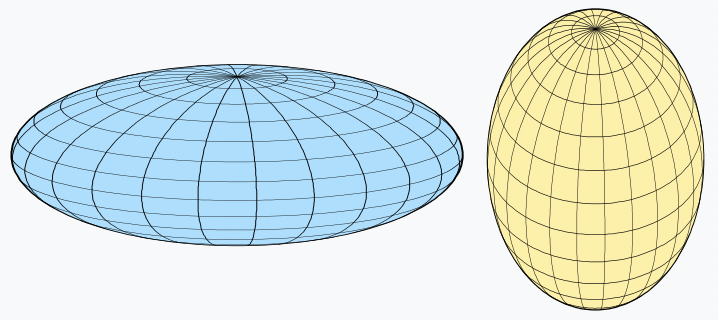}
\centering
\caption{Oblate and oblong ellipsoids of revolution}
\label{ellipsoids}
\end{figure}

\subsection{A simple geometric description of the geodesics}\label{subsec2CdV}

\myindent In this section, we introduce a short and very geometric understanding of the geodesics of a (simple and symmetric) surface of revolution. This presentation is directly adapted from the introduction of \cite{bleher1994distribution}, and, in particular, we use similar notations than the one in this article. Since this presentation is geometric, we won't focus on proofs, rather on intuition. In order to prove rigorous results, one should use the description of geodesics through the \textit{geodesic flow} on $T^{*} \mathcal{S} \backslash \{0\}$.

\quad 

\myindent An important consequence of the \textit{simplicity} of $\mathcal{S}$ (see Definition \ref{simplehypothesis}) is that \textit{every} geodesic (with nonzero velocity) intersects the equator $\gamma_E$. 

\myindent Now, fix $x_0 \in \gamma_E$ defined by  ~$\theta = 0$ (observe that choosing $x_0$ arbitrarily on $\gamma_E$ enables to \textit{define} the section $\theta = 0$). Using the rotational invariance of $\mathcal{S}$, it is enough to describe the geodesics starting at $x_0$ with a direction in $S_{x_0}\mathcal{S}$. We follow the presentation of \cite{bleher1994distribution}, and we reproduce Figure 1 of this article, see Figure \ref{figclairaut}.

\myindent Find 
\begin{equation}
    v_0 := v_{\theta} \frac{\partial}{\partial \theta} + v_{\sigma} \frac{\partial}{\partial \sigma} \in S_{x_0} \mathcal{S}
\end{equation}
a unit vector. A fundamental fact is that \textit{$f^2(\sigma)v_{\theta}$ is constant along the geodesic $t\mapsto \exp_{x_0}(tv)$} (see \cite{pressley2010elementary}[Proposition 9.3.2]). This means that either $v_{\theta} = 0$, and in that case the geodesic is \textit{vertical} i.e. it goes through the poles and is periodic with period $2L$ ; either $v_{\theta} \neq 0$ and in that case the geodesic \textit{doesn't} pass through any pole, thus the equation $f^2(\sigma) v_{\theta} = constant$ is defined in local coordinates for all times. That $f^2(\sigma) v_{\theta}$ is preserved is the \textit{Clairaut Theorem}, and its value is the \textit{Clairaut integral}. It is quite natural from a physical point of view. Indeed, following the general principle under which "geometric symmetries yields quantities conserved by the motion" (Noether's theorem), the Clairaut integral is the conservation law arising from $\mathcal{S}$ being invariant under rotation.

\begin{figure}[h]
\includegraphics[scale = 0.5]{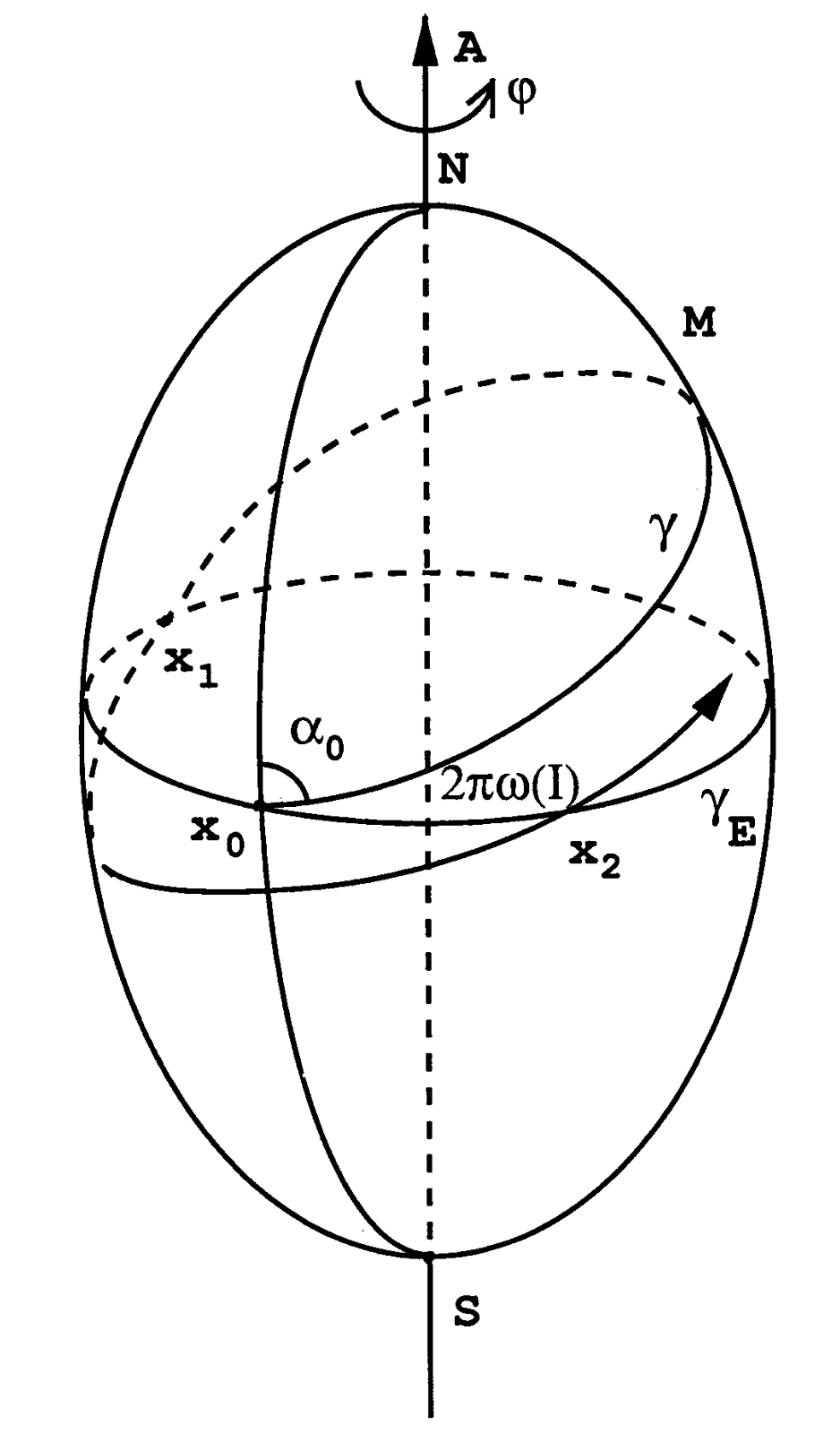}
\centering
\caption{Description of the geodesics by their Clairaut integral}
\label{figclairaut}
\end{figure}

\myindent Since it is an integral of the motion, we denote the Clairaut integral by $I \in [-1, 1]$. Observe that there holds
\begin{equation}
    I = \sin(\alpha_0),
\end{equation}
where $\alpha_0$ is the angle between the vertical axis and the velocity vector $v$ of the geodesic at its starting point $x_0$ (see Figure \ref{figclairaut}). Following the notations of \cite{bleher1994distribution}, we define, for $-1 \leq I \leq 1$, $\gamma(I)$ the geodesic starting from $x_0$ with Clairaut integral $I$, and which is pointing towards the North Pole (as in Figure \ref{figclairaut}).

\myindent Thanks to the existence of the Clairaut integral, one can prove that $\gamma(I)$ will oscillate (forward and backward in time), between two parallels defined by $\sigma_-(I) \leq 0 \leq \sigma_+(I)$, where $\sigma_{\pm}(I)$ are defined through the equation $I = f(\sigma)$. Precisely, forward in time, it first points towards the North Pole until reaching the parallel $\sigma = \sigma_+(I)$. Next, it goes down, pointing towards the South Pole, until reaching the parallel $\sigma = \sigma_-(I)$. In the case of a \textit{symmetric} surface of revolution (see Definition \ref{symmetrichypothesis}), there holds $\sigma_+(I) = -\sigma_-(I)$. More precisely, $\gamma(I)$ intersects a first time the equator at a point $x_1$, its direction being the symmetric to $v_0$ with respect to the equator. Finally, it goes back up again pointing towards the North Pole until intersecting a second time the equator at a point $x_2$. At this moment, its direction is exactly obtained from $v_0$ by a rotation. In particular, using the symmetry of revolution, the trajectory is obtained from the trajectory between $x_0$ and $x_2$ through a rotation. 

\myindent Following the notations of \cite{bleher1994distribution}, we define
\begin{equation}
    \tau(I) := |\gamma[x_0,x_2]|
\end{equation}
the length of $\gamma(I)$ between $x_0$ and $x_2$, and 
\begin{equation}\label{defomegaI}
    \omega(I) = \frac{1}{2\pi} (\theta(x_2) - \theta(x_0))
\end{equation}
the normalized \textit{phase shift} between $x_0$ and $x_2$, which is represented on Figure \ref{figclairaut}. Observe that it is a priori defined modulo $1$, but we may define it uniquely choosing a continuous branch which starts at $\omega(0) = 0$. It can moreover be smoothly extended up to $I = \pm 1$ by continuity. 
\begin{example}\label{tauomegsphere}
    In the case where $\mathcal{S}$ is the usual sphere $S^2$, then
    \begin{equation}
        \forall I \in [-1,1] \qquad \begin{cases}
            \tau(I) = 2\pi \\
            \omega(I) = 0
        \end{cases}.
    \end{equation}
\end{example}
\myindent Now, observe that $\gamma(I)$ is periodic if and only if it passes by $x_0$ again\footnote{In general, there is an important distinction between a \textit{periodic} geodesic, which is a geodesic curve passing through the same point twice with the same velocity vector, and a \textit{geodesic loop}, meaning a geodesic passing twice through the same point with velocity vectors which may differ. For example, in this situation, there could hold that $x_0 = x_1$, but then the velocity vector will be the symmetric of $v_0$ with respect to the horizontal, and in particular not equal to $v_0$. However, in the particular case of a \textit{symmetric} surface of revolution, it is straightforward that if such a situation happens, then the geodesic is periodic since its trajectories in both hemispheres are symmetric.}. Using the definition of $\omega(I)$, for $I \in (-1,1)$, this happens if and only if $\omega(I)$ is a \textit{rational number}. In particular, in view of the work of Sogge and Zelditch mentioned in the introduction, it natural, in order for Theorem \ref{mainthm} to hold, to impose hypotheses on $\mathcal{S}$ such that the set of those $I$ for which $\omega(I) \in \Q$ is well behaved (typically, at least of measure zero). The simplest possible hypothesis is the \textit{twist Hypothesis}, introduced by \cite{bleher1994distribution}.
\begin{hypothesis}[Twist Hypothesis]\label{twisthypothesis}
    The simple symmetric surface of revolution $\mathcal{S}$ satisfies the twist Hypothesis if 
 \begin{equation}
        \forall I \in [0,1] \qquad \omega'(I) \neq 0.
    \end{equation}
\end{hypothesis}
\myindent Indeed, this hypothesis guarantees that $|\omega'| \geq c > 0$ for $c$ some constant. Hence, it is straightforward to estimate the distribution of those $I$ such that $\omega(I)$ is rational, i.e. the distribution of directions of periodic geodesic (and of geodesic loops). As we will see in Paragraph \ref{subsubsec51CdV}, the twist hypothesis can also be reformulated in an analytical way, where the link with the distribution of the eigenvalues of $\mathcal{S}$ is more direct. Observe that the round sphere notably \textit{doesn't} satisfy the twist Hypothesis \ref{twisthypothesis}, thanks to Example \ref{tauomegsphere}.

\myindent Now, the class of simple symmetric surfaces of revolution satisfying the Twist Hypothesis is quite large. Indeed, there holds the following proposition (see \cite{bleher1994distribution}).
\begin{proposition}[Bleher]\label{oblongoblateomega'}
    The set of simple symmetric surfaces of revolution which satisfy the twist Hypothesis \ref{twisthypothesis} is open. Moreover, it contains all the ellipsoid of revolutions except the round sphere. Precisely, if $\mathcal{E}$ is an \textit{oblong} (resp \textit{oblate}) ellispoid of revolution, there holds $\omega'(I) > 0$ (resp $\omega'(I) < 0$).
\end{proposition}
\subsection{The complete integrability of the geodesic flow : explicit construction}\label{subsec3CdV}

\subsubsection{Complete integrability of the geodesic flow}\label{subsubsec31CdV}

\myindent The analysis of the previous section, which describes the geodesics on a surface of revolution, can be equivalently understood at the level of the geodesic flow on the cotangent bundle $T^*\mathcal{S}\backslash \{0\}$ (see Lemma \ref{eqgeod}), which we recall is the Hamiltonian flow of the norm $p$ of cotangent vectors for the Liouville form $-d\lambda$. Indeed, from the expression \eqref{normsurfrev}, we see that $p$ doesn't depend on the angular coordinate $\theta$. In particular, using the explicit definition of the geodesic flow \eqref{geodflow}, there holds along trajectories of $\Phi_t$
\begin{equation}
    \frac{\partial\Theta(t)}{\partial t} = -\frac{\partial p}{\partial \theta} \circ \Phi_t = 0,
\end{equation}
i.e. the coordinate $\Theta$ is an integral of the motion. Now, along a trajectory, at $p = 1$, there also holds thanks to \eqref{normsurfrev}
\begin{equation}
    \frac{\partial \theta(t)}{\partial t} = \frac{\partial p}{\partial \Theta} \circ \Phi_t = \frac{\Theta}{f^2(\sigma)},
\end{equation}
and we find that $f^2(\sigma) \dot{\theta} = \Theta$ is preserved. For $(x,\xi) = (\theta,\sigma,\Theta,\Sigma) \in T^*\mathcal{S}\backslash\{0\}$, we define its Clairaut integral by
\begin{equation}
    I = \frac{\Theta}{p(\sigma,\Theta,\Sigma)} \in [-1,1].
\end{equation}

\myindent Now, thanks to this integral of the motion, the geodesic flow $\Phi_t$ is completely integrable for a surface of revolution. We now detail the meaning of this assertion.

\myindent First, on $T^*\mathcal{S}$, we define on the one hand
\begin{equation}\label{defp1}
    p_1(x,\xi) := p(x,\xi).
\end{equation}
\myindent On the other hand, we define outside of the Poles
\begin{equation}\label{defp2}
    p_2(\theta,\sigma,\Theta,\Sigma) := \Theta,
\end{equation}
which is smoothly extended near the Poles by
\begin{equation}
    p_2 = x_1\xi_2 - x_2 \xi_1
\end{equation}
in normal coordinates around each Pole. Then, there holds the following two lemmas (see \cite{de1980spectre}[Lemma 6.2] and right after).
\begin{lemma}
    Let 
    \begin{equation}
        \vec{p} : (x,\xi) \in T^* \mathcal{S}\backslash \{0\} \mapsto (p_1,p_2)(x,\xi) \in \R^2 \backslash (0,0).
    \end{equation}
    \myindent Then the singular set $Z$ of $\vec{p}$ is exactly the cotangent bundle of the equator $\gamma_E$, seen as a 2-dimensional subbundle of $T^*\mathcal{S}\backslash \{0\}$, namely
    \begin{equation}\label{defZ}
        Z = \{(\theta,0,\Theta,0) \qquad \theta \in S^1, \ \Theta \in \R^*\}.
    \end{equation}
\end{lemma}
\myindent We also precise the image of $\vec{p}$.

\begin{lemma}
    The image of the application 
    \begin{equation}
        \vec{p} = (p_1,p_2) : T^*\mathcal{S}\backslash \{0\} \to \R^2 \backslash (0,0)
    \end{equation}
    is the open cone
    \begin{equation}\label{defGamma}
    \Gamma := \vec{p}(T^*\mathcal{S}\backslash\{0\}) = \{(\lambda,\mu) \in \R^2\backslash (0,0) \ |\mu|\leq\lambda\}.
    \end{equation}
    \myindent Moreover, $\vec{p}$ sends its singular set $Z$ onto $\partial\Gamma$.
\end{lemma}
\myindent Now, there holds that any simple surface of revolution is \textit{completely integrable} in the sense of Colin de Verdière \cite{de1977quasi}[Paragraph 4].
\begin{proposition}[Complete integrability of the geodesic flow]\label{completintegrab}
    On one hand, the function 
    \begin{equation}
        \vec{p} = (p_1,p_2) : \left(T^*\mathcal{S}\backslash \{0\}\right) \backslash Z \to \R^2 \backslash (0,0)
    \end{equation}
    is a proper submersion with connected compact fibers.
    
    \myindent On the other hand, there holds on $T^*\mathcal{S}$
    \begin{equation}
        \{p_1, p_2\} = 0,
    \end{equation}
    where $\{\cdot,\cdot\}$ is the Poisson bracket on the cotangent bundle.
\end{proposition}

\myindent It is well-known that complete integrability yields a foliation of the cotangent bundle by \textit{tori}, which are stable by both the geodesic flow and the Hamiltonian flow of $p_2$, namely by rotation. Indeed, let $a = (\lambda,\mu) \in \Gamma$. Then, 
\begin{equation}
    \Lambda_a := (\vec{p})^{-1}(a)
\end{equation}
is invariant by the Hamiltonian flows of both $p_1$ and $p_2$, and it is a connected compact set, diffeomorphic to a torus.

\myindent Though this is straightforward from an abstract point of view, we give a geometric interpretation of it, in this very explicit case. There holds
\begin{equation}\label{deff-1lambdamu}
    \forall (\lambda,\mu)\in \Gamma \ (\vec{p})^{-1}(\lambda,\mu) = \left\{\left(\theta,\sigma, \mu, \pm \sqrt{\lambda^2 - \frac{\mu^2}{f^2(\sigma)}}\right) \ \theta\in[0,2\pi]\  \sigma_-\left(\frac{\mu}{\lambda}\right) \leq \sigma \leq \sigma_+\left(\frac{\mu}{\lambda}\right)\right\},
\end{equation}
where we recall that $\sigma_+(I) = -\sigma_-(I) \geq 0$ are defined by
\begin{equation}
    f(\sigma_{\pm}(I)) = I.
\end{equation}

\myindent From this formula, there is a very simple geometrical visualization of the fact that $(\vec{p})^{-1}(\lambda,\mu)$ is a torus. Again, this fact can be proved instantly looking only at where the differentials of $p_1$ and $p_2$ are linearly independent. However we find it interesting, in an explicit case, to exhibit the foliation by tori.

\myindent i. in the degenerate case $\mu = \pm \lambda$ it is obvious that
\begin{equation}\label{expltorus}
    (\vec{p})^{-1}(\lambda,\pm \lambda) = \{(\theta, 0,\pm\lambda,0), \ \theta\in S^1\}
\end{equation}
is diffeomorphic to a circle.

\myindent ii. in the generic case $0 <|\mu| < \lambda$, let us observe that the projection of $(\vec{p})^{-1}(\lambda,\mu)$ on $\mathcal{S}$ is the set $C(\sigma_-,\sigma_+)$ of those points $(\theta,\sigma)$ with $\theta \in S^1$ and $\sigma_-\leq \sigma \leq \sigma_+$. Now, this is obviously diffeomorphic to a cylinder. Moreover, for every $x$ in the interior of $C(\sigma_-,\sigma_+)$, there are exactly two points of $(\vec{p})^{-1}(\lambda,\mu)$ in the cotangent space $T_x^*\mathcal{S}$ (depending on the sign of $\Sigma$). Finally, for points $x$ on the boundary of $C(\sigma_-,\sigma_+)$ (which is the reunion of two circles), there is exactly one point of $(\vec{p})^{-1}(\lambda,\mu)$ in $T_x^*\mathcal{S}$ (with $\Sigma = 0$). Thus, one can picture $(\vec{p})^{-1}(\lambda,\mu)$ as the torus embedded into $\R^3$ (i.e. the "geometrical torus"). For a visual proof moreover, one could use the natural isometry between the tangent bundle and the cotangent bundle and project in $\R^3$ to obtain indeed a geometrical torus.

\myindent iii. in the case $\mu = 0$, the image of a geometrical torus still holds, with the exception that the top and bottom circles are entirely contained in the cotangent spaces of the poles. Indeed, the intersection of $(\vec{p})^{-1}(\lambda,0)$ and $T_P^*\mathcal{S}$, when $P$ is a pole, is the circle of radius $\lambda$ in $T_P^*\mathcal{S}$.

\subsubsection{Explicit action-angle coordinates}\label{subsubsec32CdV}

\myindent One can understand locally the geodesic flow on a completely integrable manifold by the introduction of \textit{action-angle coordinates}, i.e. a symplectic change of variable through which the geodesic flow can be seen as simply following a straight line on a $2d$ regular torus. Precisely, a consequence of Proposition \ref{completintegrab} is the following theorem (see \cite{de1977quasi}[Theorem 4.1]).

\begin{theorem}[Action-angle coordinates]\label{torifoliation}
    Let $a \in Int(\Gamma)$. There exists an conic convex open neighborhood $U$ of $a$ in $\R^2 \backslash (0,0)$, and a symplectic diffeomorphism  $\chi$ from a conical open subset $\mathbb{T}^2\times C$ of the cotangent bundle $\mathbb{T}^2\times \R^2$ of $\mathbb{T}^2$, with coordinates $(x,\xi)$, onto $(\vec{p})^{-1}(U)$, such that $\chi$ is 1-homogeneous in the $\xi$ variable and moreover

    \myindent i. $\chi^{-1}$ sends the foliation $(\Lambda_a)_{a\in U}$ of $(\vec{p})^{-1}(U)$ onto the foliation $(T_{\xi})_{\xi \in C}$ where $T_{\xi} = \mathbb{T}^2 \times \{\xi\}$. We define $\xi(a)$ through $T_{\xi(a)} = \chi^{-1}(\Lambda_a)$.

    \myindent ii. In order to compute $\xi(a)$, find $(\gamma_i)$ the canonical basis of $\pi_1(\mathbb{T}^2)$ and let $(\gamma_{i,\xi})$ be the corresponding basis of $\pi_1(T_{\xi})$. There holds
    \begin{equation}\label{abstractformulaforq}
        \xi_i(a) = \frac{1}{2\pi}\int_{\gamma_{i,\xi(a)}} \lambda,
    \end{equation}
    where $\lambda$ is the Liouville 1-form on $T^*S$ i.e. in local coordinates $\lambda = \xi \cdot dx$.

    \myindent iii. $q \circ \chi(x,\xi) = K(\xi)$ is a smooth 1-homogeneous function on $C$. Thus, the geodesic flow is transformed by $\chi^{-1}$ into the flow $t\mapsto (x_0 + tK'(\xi_0),\xi_0)$.
\end{theorem}

\myindent A crucial observation of \cite{de1980spectre} is that, in the case of surface of revolution, one has a little bit better than complete integrability. Indeed, one can in fact extend Theorem \ref{torifoliation} \textit{globally}, i.e. define action-angle coordinates on the whole of $T^*\mathcal{S}\backslash \{0\}$. Moreover, the action coordinates are in fact given as smooth \textit{homogeneous} functions of $(p_1,p_2)$. Precisely, there holds the following.

\begin{proposition}[Colin de Verdière]\label{q1q2fctofp1p2}
    There exists a smooth diffeomorphism $G$ from $\Gamma$ onto itself, which is homogeneous of degree $1$, such that if
    \begin{equation}
        (q_1,q_2) := G(p_1,p_2),
    \end{equation}
    then the Hamiltonian flows of $q_1,q_2$, seen as function on $T^*\mathcal{S}\backslash \{0\}$, are $2\pi$ periodic (and they commute).

    \myindent Moreover, $q_2 = p_2$.
\end{proposition}

\begin{proof}
    The idea of the proof of \cite{de1980spectre} is to use the explicit formula \eqref{abstractformulaforq}. Indeed, thanks to this formula, in order to define $q_1,q_2$ on $Int(\Gamma)$, it is enough to find a smooth basis, say $\gamma_1,\gamma_2$, of $\pi_1(\Lambda_a)$ for each leaf $\Lambda_a$, $a =(p_1,p_2) \in Int(\Gamma)$, and to set
    \begin{equation}
        q_i = \frac{1}{2\pi}\int_{\gamma_i} \xi \cdot dx.
    \end{equation}
    \myindent Now, one can build $\gamma_1,\gamma_2$ for the leaf $\Lambda_{1,0}$, where one can explicitly choose for $\gamma_1$ a meridian lifted on $T^*\mathcal{S}$, and for $\gamma_2$ the unit circle of the cotangent space at the North Pole. Then, one can use that the fibration $(\Lambda_a)$ over $Int(\Gamma)$ is \textit{a priori} trivial since this set is connected, in order to extend smoothly $\gamma_1,\gamma_2$. Finally, one concludes by checking that $q_1,q_2$ extend to the singular set $Z$ (see \ref{defZ}) by introducing local action-angle coordinates. The fact that $q_2 = p_2$ is straightforward from this procedure.
\end{proof}

\myindent While this method is is very efficient, it unfortunately doesn't give a precise description of $q_1$, or, rather, of its Hamiltonian flow. However, since we have microlocal analysis in mind, we will need to understand very well the bicharacteristics of $q_1$. Hence, we give an alternative point of view, which differs slightly on the construction of the basis of $\pi_1(\Lambda_a)$.

\myindent To build this basis, find $(\lambda,\mu)\in Int(\Gamma)$. We fix the starting point on the equator pointing toward $N$
\begin{equation}
    O_{\lambda,\mu} := (0, 0, \mu, \sqrt{\lambda^2 - \mu^2}).
\end{equation}

\myindent We recall that $t\mapsto \Phi_t$ is geodesic flow and $t\mapsto \Phi^{p_2}_t$ is the Hamiltonian flow of $p_2$ (which is merely the rotation around the axis between the poles). We know that $(\vec{p})^{-1}(\lambda,\mu)$ is given by $\{\Phi_t \Phi^{p_2}_s(O_{\lambda,\mu})\ (t,s)\in\R^2\}$. Moreover, $(\vec{p})^{-1}(\lambda,\mu)$ is naturally diffeomorphic to $\R^2/\Lambda$ where $\Lambda$ is the set of fixed points i.e. the lattice of those $(t,s)$ such that 
\begin{equation}
    \Phi_t \Phi^{p_2}_s(O_{\lambda,\mu}) = O_{\lambda,\mu}.
\end{equation}

\myindent We can find a basis of the lattice $\Lambda$. Indeed, $t = 0, s= 2\pi$ is always in $\Lambda$ (this corresponds to moving along the equatorial geodesic). Moreover, we recall that we defined $\omega(I)$ the phase shift between $t = 0$, and $t = \tau(I) > 0$ the first time the geodesic crosses again the equator towards the North Pole, where $I = \frac{\mu}{\lambda}$. Thus, another point of $\Lambda$ is $(\tau(I),-2\pi\omega(I))$ and we find that
\begin{equation}
    \Lambda = \Z\left(\tau\left(\frac{\mu}{\lambda}\right), - 2\pi\omega\left(\frac{\mu}{\lambda}\right)\right) \bigoplus \Z(0,2\pi).
\end{equation}

\myindent Thus, it is quite natural to set $\gamma_2$ projecting onto the equatorial geodesic i.e.
\begin{equation}
    \begin{split}
        \gamma_2 &:= \{t \mapsto \Phi^{p_2}_t(O_{\lambda,\mu}) \ 0\leq t \leq 2\pi\}\\
    &= \{\left(\theta, 0, \mu, \sqrt{\lambda^2 - \mu^2}\right) \ 0\leq \theta \leq 2\pi\},
    \end{split}
\end{equation}
and 
\begin{equation}\label{explicitgamma1}
    \gamma_1 := \{t \mapsto \Phi_{\frac{t}{2\pi}\tau(I)} \circ \Phi^{p_2}_{-t\omega(I)} (O_{\lambda,\mu}) \ 0\leq t \leq 2\pi\}.
\end{equation}

\myindent Geometrically, $\gamma_1$ oscillates one time around the axis while oscillating between $\sigma_+(I)$ and $\sigma_-(I)$. We recall that $\omega(0) = 0$, thus, for the torus $(\vec{p})^{-1}(1,0)$, $\gamma_1$ is actually a meridian geodesic of length $2L$, i.e. the same choice than in \cite{de1980spectre}.

\myindent Finally, one can conclude by defining $q_1$ and $q_2$ through the formula \eqref{abstractformulaforq}.

\myindent We observe that the explicit formula for $\gamma_i$ can be extended on the singular set $Z$ of $p$ as $\omega$ and $\tau$ both admit a smooth extension for $I = \pm 1$. Thus, $q_1$ and $q_2$ actually extend into smooth homogeneous functions on $T^*\mathcal{S}\backslash \{0\}$. Similarly, it is straightforward from the construction than $q_2 = p_2$. 

\quad

\myindent An immediate, yet important, corollary of the proposition is the following 

\begin{corollary}\label{q1elliptic}
    There holds
    \begin{equation}
        (q_1,q_2)(T^*\mathcal{S}\backslash \{0\}) = \Gamma,
    \end{equation}
    where $\Gamma$ is the cone defined by \eqref{defGamma}. In particular, $q_1$ is strictly positive on $T^*\mathcal{S}\backslash \{0\}$, hence $q_1$ is an \textit{elliptic} symbol on $T^*\mathcal{S}\backslash \{0\} $.
\end{corollary}

\myindent Thanks to the explicit formula for $\gamma_1$, one can find a closed formula for $q_1$.
\begin{lemma}\label{formulaofq1}
    There holds
    \begin{equation}\label{explicitG}
        \begin{split}
            q_1 &= G(p_1,p_2) \\
            &= \frac{1}{\pi}\int_{\sigma_+(I)}^{\sigma_-(I)}\sqrt{p_1^2 - \frac{p_2^2}{f^2(\sigma)}}d\sigma + |p_2|,
        \end{split}
    \end{equation}
    where $I = \frac{p_2}{p_1}$. In particular, $q_1$ is independent of the angular variable $\theta$.
\end{lemma}

\quad

\myindent Now, while the formula \eqref{explicitgamma1}, and hence the construction of $q_1$, can seem arbitrary, they are actually a canonical choice. Precisely, once we know that there exists some $(q_1,q_2)$ as in Proposition \ref{q1q2fctofp1p2}, and such that $q_2 = p_2$, then the Hamiltonian flow of $q_1$ can be computed explicitly, as we will prove in the following lemma. Now, once we know the Hamiltonian flow of $q_1$, then $q_1$ is fully determined up to a constant, which is fixed by the fact that, in Proposition \ref{q1q2fctofp1p2}, we impose the image of $(q_1,q_2)$. Thus, the following lemma yields an alternate point of view for the construction of $q_1$ than the more abstract point of view in \cite{de1980spectre}. While the latter has the advantage of conciseness, the former has the advantage of being more explicit. In particular, we will use crucially the following lemma in the rest of the present article, since we will need to have a precise control on the bicharacteristics of $q_1$.

\begin{lemma}[Hamiltonian flow of $q_1$]\label{explicitbicharac}
    Assume that $q_1$ is as in Proposition \ref{q1q2fctofp1p2}. Let $(x,\xi) \in T^*\mathcal{S}$, and let $I$ be its Clairaut integral. There holds
    \begin{equation}
        \Phi^{q_1}_t(x,\xi) = \Phi_{\frac{t}{2\pi}\tau(I)}\ \circ \ \Phi^{p_2}_{-t\omega(I)}(x,\xi).
    \end{equation}
\end{lemma}

\myindent As a consequence, formula \eqref{explicitgamma1} simply expresses the fact that $\gamma_1$ is the bicharacteristic of $q_1$ starting at $O_{\lambda,\mu}$. We now prove the lemma.

\begin{proof}
    It is enough, by continuity, to prove the formula when $(x,\xi)\notin Z$, i.e. when $I\notin \{-1,1\}$. Write $(\lambda,\mu) = (p_1(x,\xi),p_2(x,\xi))$, so that, in particular, $I = \frac{\mu}{\lambda}$. By definition, we know that the three curves
    \begin{equation}
        \begin{split}
            t &\mapsto \Phi_t(x,\xi)\\
            t &\mapsto \Phi^{q_1}_t(x,\xi) \\
            t &\mapsto \Phi^{q_2}_t(x,\xi)
        \end{split}
    \end{equation}
    are included in the set $T_{\lambda,\mu} := (p_1,p_2)^{-1}(\lambda,\mu)$. Now, using action-angle coordinates associated to $(q_1,q_2)$, $T_{\lambda,\mu}$ is diffeomorphic to a regular torus $\mathbb{T}^2$, on which following the Hamiltonian flow of $q_1$ (resp $q_2$) corresponds to travelling at constant speed on the vertical axis (resp horizontal axis). Moreover, in this representation, the geodesic flow is a straight line followed at a constant speed. In other words, the situation is given by Figure \ref{torusbichq1}, where the sides of the square are of length $2\pi$. 
    
\begin{figure}[h]
\includegraphics[scale = 0.5]{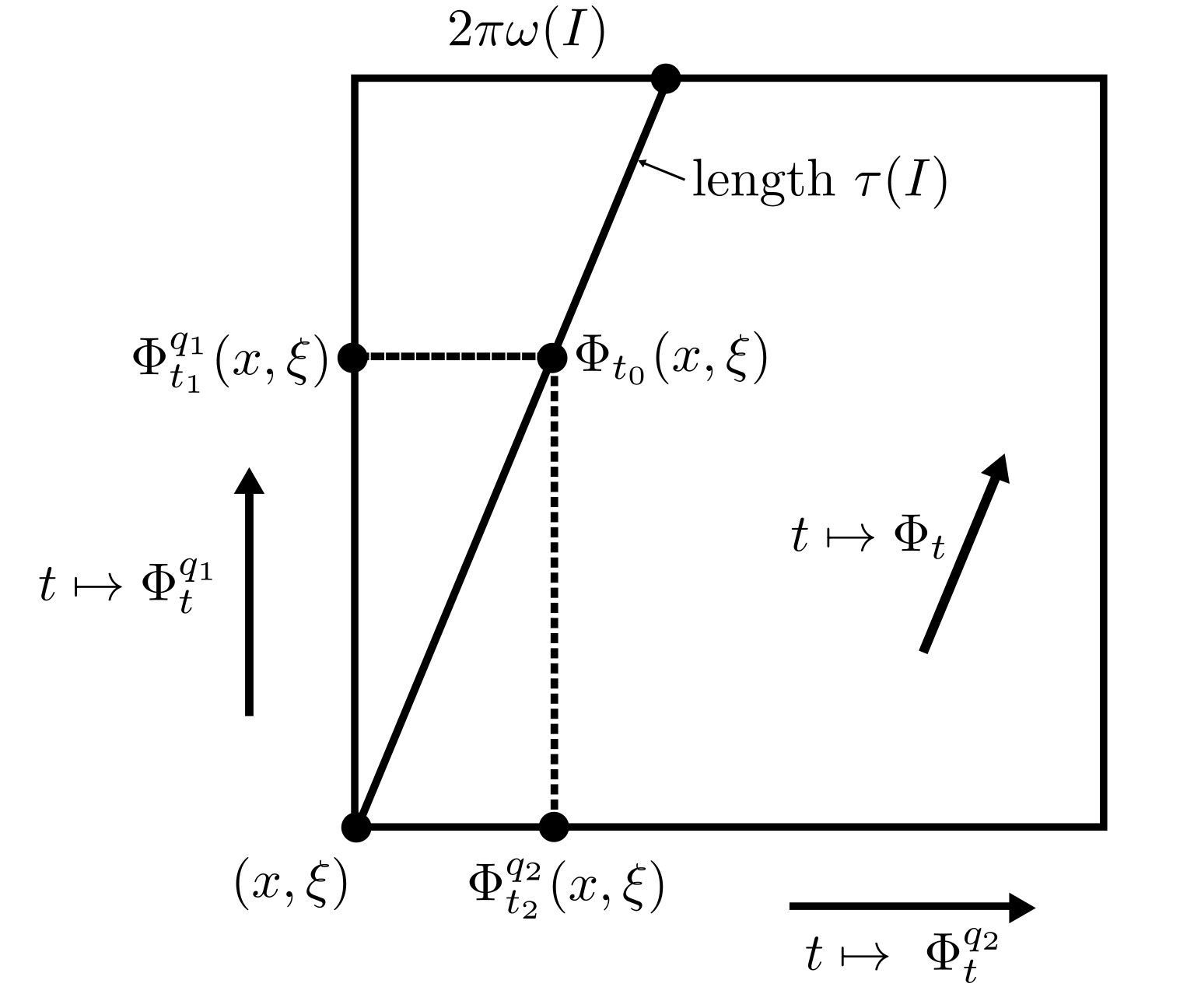}
\centering
\label{torusbichq1}
\caption{The geodesics in action-angle coordinates}
\end{figure}

    \myindent With the notations introduced on Figure \ref{torusbichq1}, this yields, in particular, that
    \begin{equation}
        \Phi_{t_0}(x,\xi) = \Phi_{t_1}^{q_1} \circ \Phi_{t_2}^{p_2}(x,\xi),
    \end{equation} 
    or, equivalently,
    \begin{equation}
        \Phi_{t_1}^{q_1}(x,\xi) = \Phi_{t_0}\circ\Phi_{-t_2}^{p_2}(x,\xi).
    \end{equation}
    
    \myindent Now, an application of the intercept theorem yields that
    \begin{equation}
        \frac{t_0}{\tau(I)} = \frac{t_1}{2\pi} = \frac{t_2}{\omega(I)},
    \end{equation}
    from which formula \eqref{explicitbicharac} follows.
\end{proof}

\myindent Thanks to this lemma, we find that the bicharacteristic curves of $q_1$ have the same behaviour than the geodesics of $\mathcal{S}$. Namely, for each of these curves, we can define the conserved Clairaut integral $I$. Moreover, the curve oscillates between the parallels $\{\sigma = \sigma_+(I)\}$ and $\{\sigma = \sigma_-(I)\}$, with exactly one oscillation for each revolution around the vertical axis. Hence, even though it is not fully explicit, we can really understand the Hamiltonian flow of $q_1$ as a \textit{periodized} version of the geodesic flow on $T^*\mathcal{S}$.

\subsection{Quantum Complete Integrability}\label{subsec4CdV}

\myindent As we announced in the introduction, the main interest of the previous analysis is that it can be lifted to the level of pseudodifferential operators. The presentation is adapted from \cite{de1980spectre}. Define
\begin{equation}
\begin{split}
    P_1 := \sqrt{-\Delta} \\
    P_2 := \frac{1}{i} \frac{\partial}{\partial \theta},
\end{split}
\end{equation}
so that, in particular, the principal symbol of $P_i$, $i = 1,2$, is $p_i$ (see \eqref{defp1} and \eqref{defp2}.

\myindent Writing explicitly the Laplacian in coordinates yields that 
\begin{equation}
    [P_1, P_2 = 0].
\end{equation}

\myindent This implies, though it is strictly stronger, that $\{p_1,p_2\} = 0$, i.e. (the key part of) the complete integrability of $\mathcal{S}$.

\quad

\myindent Now, the construction of $q_1$ and $q_2$ can similarly be "translated" into the framework of pseudo differential operators : let $\hat{\mathcal{A}}$ the algebra of operators formed by the $f(P_1,P_2)$ where $f$ is a classical symbol on $\R^2\backslash\{0\}$. Precisely, define the following commutative algebra of pseudodifferential operators (see Strichartz \cite{strichartz1972functional}) 
\begin{equation}
    \hat{\mathcal{A}} = \{f(P_1,P_2) \ f \in S^{\infty}_{cl}(\R^2\backslash\{0\})\},
\end{equation}
where $S_{cl}^{\infty}(\R^2\backslash\{0\})$ is the set of smooth functions $\R^2\backslash\{0\}\to \R\backslash\{0\}$ which can be asymptotically decomposed into homogeneous components of degree $m-j$ ($m\in \Z$ is fixed) i.e.
\begin{equation}
    f \sim f_m + f_{m-1} + ... + f_{m-j} + ... \ \text{in the sense that} \ f - \sum_{j=0}^{N-1} f_{m-j} = O(\|\xi\|^{m-N}).
\end{equation}

\myindent Then, the following theorem holds (see \cite{de1980spectre}[Theorem 6.1.]).

\begin{theorem}[Colin de Verdière]\label{CdVthm}
    Let $\mathcal{S}$ be a simple surface of revolution. Then, the algebra of operators $\hat{\mathcal{A}}$ admits two generators $Q_1,Q_2 \in \hat{\mathcal{A}}$ with principal symbols $q_1$ and $q_2$ defined by Proposition \ref{q1q2fctofp1p2}, such that moreover
    \begin{equation}\label{exp2ipiQieqId}
        e^{2i\pi Q_k} = (-1)^k Id, \qquad k = 1,2.
    \end{equation}

    \myindent In particular, there exists a classical elliptic symbol $F$ of order $2$ on $\R^2$ such that
    \begin{equation}\label{defFinThmCdV}
        F \sim F_2 + F_0 + F_{-1} + ...
    \end{equation}
    (the homogeneous component of order $1$ equals zero as the subprincipal symbol equals zero) and
    \begin{equation}
        -\Delta = F(Q_1,Q_2).
    \end{equation}

    \myindent In particular, there holds on $\mathcal{T}^*\mathcal{S} \backslash\{0\}$
    \begin{equation}\label{pfunctq}
        p_1 = \sqrt{F_2(q_1,q_2)}.
    \end{equation}

    \myindent Finally, there holds
    \begin{equation}
        Q_2 = P_2.
    \end{equation}
\end{theorem}

\myindent Observe that \eqref{exp2ipiQieqId} implies that the Hamiltonian flows of $q_1$ and $q_2$ are $2\pi$-periodic. Indeed, $e^{2i\pi Q_i}$ is a Fourier Integral Operator (FIO) whose canonical relation is given by the canonical transformation given by the Hamiltonian flow of the principal symbol $q_i$ at the time $2\pi$ (see \cite{hormander2009analysis}). One could be surprised that
\begin{equation}
    e^{2i\pi Q_1} = -Id
\end{equation}
(and not $+Id$) while the Hamiltonian flow of $q_1$ is $2\pi$-periodic. This is due to the phenomenon of \textit{Maslov indices}, of which we won't speak since it doesn't affect the presentation. We refer to the presentation give by Hörmander of the relation between the structure of the fibers of $(q_1,q_2)$ and the Maslov indices, see \cite[section 3]{hormander1971fourier}.

\begin{proof}
    We give a sketch of the proof, and refer to \cite{de1980spectre} for the details. 
    
    \myindent The first part is to construct approximate solutions to \eqref{exp2ipiQieqId}. Write $q_i = f_i(p_1,p_2)$ with $f_i$ a smooth homogeneous function. Define $T_i := f_i(P_1,P_2)$, $i = 1,2$. Then, the principal symbol of $T_i$ is $q_i$, thanks to the usual calculus of FIO, and the $2\pi$-periodicity of the Hamiltonian flows of the $q_i$ yields that
    \begin{equation}
        \exp(2i\pi (T_j-\mu_j)) = Id + C_j,
    \end{equation}
    where $\mu_j \in \frac{1}{4}\Z$ is a Maslov index and $C_j$ is a remainder, which is a pseudodifferential operator of order $-1$ (see \cite{de1980spectre}[Lemma 3.4.].

    \quad

    \myindent Now, as we will argue in the following, it is crucial that formula \eqref{exp2ipiQieqId} holds \textit{exactly}, and \textit{not} up to smoothing remainders, as is often the case in microlocal analysis. Hence, the second part, which is more technical, is to build an \textit{exact} suitable logarithm for $Id + C_j$, which can be done spectrally (and not iteratively with remainders of order more and more regularizing, as is often the case in microlocal analysis).
    
    \quad
    
    \myindent Next, one needs to prove that this logarithm is itself an element of $\hat{\mathcal{A}}$. This last property relies on the one hand on proving an equivalent property at the level of principal symbols, which is that any symbol commuting (in terms of Poisson bracket) with $p_1,p_2$ is a smooth homogeneous function of $p_1,p_2$(\cite{de1980spectre}[Proposition 2.3.]). On the other hand, this can be applied to successive approximations, of strictly decreasing orders. Finally, one checks that smoothing remainders themselves are in $\hat{\mathcal{A}}$. 

    \quad

    \myindent Finally, one needs to check that $Q_1,Q_2$ indeed generate the algebra $\hat{\mathcal{A}}$. This part is rather technical, and we refer to \cite{de1980spectre}.
\end{proof}

\begin{remark}\label{Q1elliptic}
    Observe that, since the principal symbol of $Q_1$ is $q_1$, which is an elliptic symbol on the cotangent bundle thanks to Corollary \ref{q1elliptic}, then $Q_1$ is itself an elliptic pseudodifferential operator.
\end{remark}

\begin{remark}
    The most difficult part of \cite{de1980spectre}, of which we won't speak, is actually the precise computation of the joint spectrum of $(Q_1,Q_2)$. Indeed, thanks to Theorem \ref{CdVthm}, this yields a description of the spectrum of $\sqrt{-\Delta}$, i.e. of the eigenvalues of $\mathcal{S}$. On this matter, see also \cite{verdiere1979spectre}.
\end{remark}

\subsection{Geometric properties}\label{subsec5CdV}

\subsubsection{The twist hypothesis}\label{subsubsec51CdV}

\myindent In this paragraph, we detail how the twist Hypothesis \ref{twisthypothesis} can be reformulated in terms of the function $F_2$ defined by Theorem \ref{CdVthm}.
\begin{definition}\label{deflilgammas}
    Let $F_2$ be given by Theorem \ref{CdVthm}. We define the curves
    \begin{equation}\label{deflilgamma}
        \gamma := \{F_2 = 1\},
    \end{equation}
    and 
    \begin{equation}\label{deflilgamma0}
        \gamma_0 := \gamma \cap \Gamma,
    \end{equation}
    where $\Gamma$ is defined by \eqref{defGamma}. 

    \myindent Moreover, we denote by $d\mu$ the superficial measure on the curve $\gamma$.
\end{definition}

\myindent Observe that only $\gamma_0$ has a geometrical meaning, and is fixed by $\mathcal{S}$. Indeed, thanks to \eqref{pfunctq}, it corresponds to the equation $\{p_1 = 1\}$. For $\gamma$, one may choose any smooth extension of $\gamma_0$ as long as it is a simple curve. We refer to Figure \ref{figlilgammas} in order to picture the curves $\gamma_0$ and $\gamma$.

\begin{figure}[h]
\includegraphics[scale = 0.5]{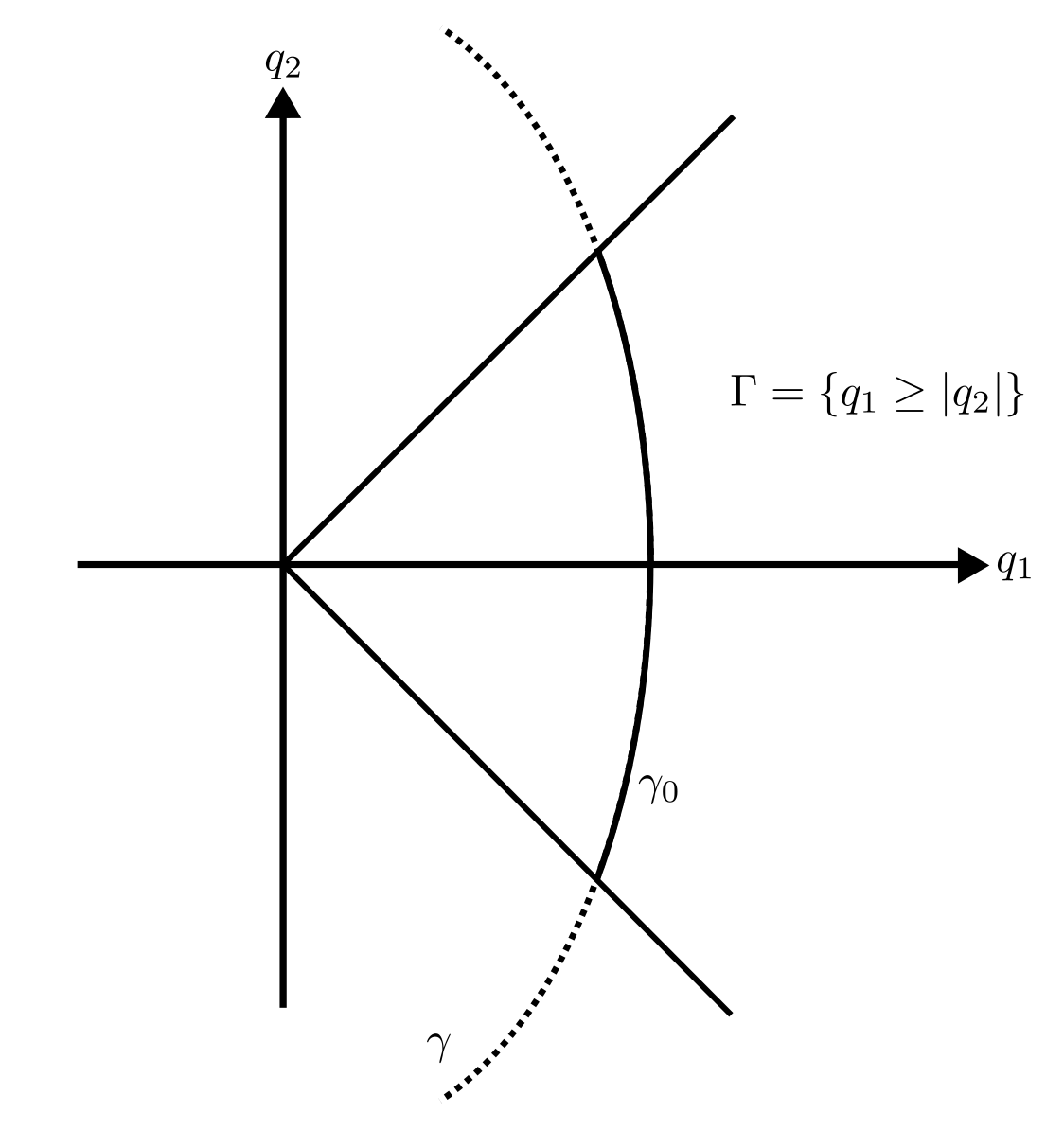}
\centering
\caption{The curves $\gamma$ and $\gamma_0$}
\label{figlilgammas}
\end{figure}

\myindent Now, $\gamma_0$ can be described as the graph of a function $q_1 = g(q_2)$. Indeed, by definition, there holds
\begin{equation}
    \gamma_0 = \{(q_1,q_2) \in \R^2 \ \text{such that} \ F_2(q_1,q_2) = 1 \ \text{and} \ |q_2| \leq q_1\}.
\end{equation}

\myindent Hence, using the homogeneity of $F_2$, there holds
\begin{equation}
    \gamma_0 = \left\{ \left(\frac{1}{F_2(1, I)}, I\right), \qquad I \in [-1,1] \right\}.
\end{equation}

\myindent \cite{bleher1994distribution}[Proposition 6.2] gave the following geometrical interpretation of the derivative of $g$ in terms of $\omega(I)$.

\begin{proposition}[Bleher]\label{g'eq-omega}
    There holds along the curve $\gamma_0$
    \begin{equation}
        g'(I) = \frac{d q_1}{dq_2}_{|q_2 = I} = -\omega(I).
    \end{equation}
\end{proposition}

\myindent As a consequence, critical points of $I \mapsto \omega(I)$ corresponds to inflexion points of $\gamma_0$. In particular, the twist Hypothesis \ref{twisthypothesis} is equivalent to the following assumption, which is closer to the framework of Colin de Verdière.

\begin{hypothesis}[twist Hypothesis V2]\label{twistV2}
    The curve $\gamma_0$ doesn't have any inflexion points, i.e. its curvature is nowhere zero.
\end{hypothesis}

\begin{remark}\label{defgeneric}
    The precise meaning of "generic" in Colin de Verdière's Theorem \ref{thmCdVRemainders} is that the curve $\gamma_0$ only has ordinary points of inflexion, i.e. its third derivative is nonzero if its curvature vanishes. Hence, it can be reformulated by saying that, generically, $\omega''(I) \neq 0$ whenever $\omega'(I) = 0$.
\end{remark}

\begin{remark}\label{curvaturegamma}
    Observe that the function $F_2$ given by Theorem \ref{CdVthm} is only constrained on the set $\Gamma$ defined by \eqref{defGamma}. In particular, if the twist Hypothesis \ref{twisthypothesis} holds, we may always choose $F_2$ such that the curve $\gamma$ has only ordinary points of inflexion. We will assume that this holds in the following.
\end{remark}

\subsubsection{Antipodal and exceptional refocalisations of bicharacteristic curves}\label{subsubsec52CdV}

\myindent In view of applying the analysis of Fourier Integral Operators to the semigroup $s \mapsto e^{isQ_1}$, it is crucial to understand the \textit{crossings} of the bicharacteristic curves of $q_1$ (see Definition \ref{defbicharacteristics}). We give the following definition
\begin{definition}[Crossings of bicharacteristic curves]\label{crossingDef}
    Let $x\in \mathcal{S}$ and let $(x,\xi), (x,\eta) \in S^*_x\mathcal{S}$ be two different directions. We say that the bicharacteristic curves starting at $(x,\xi)$ and $(x,\eta)$ \textit{cross} at a point $y \in \mathcal{S}\backslash\{x\}$ at the time $t \neq 0$ if
    \begin{equation}
        P(\Phi^{q_1}_t(x,\xi)) = P(\Phi^{q_1}_t(x,\eta)) = y,
    \end{equation}
    where we recall that $P : S^*\mathcal{S} \to \mathcal{S}$ is the fiber projection. 
    
    \myindent More generally, we say that these two curves \textit{intersect} at $y$ if there exist two times $s,t$ such that
    \begin{equation}
        P(\Phi_t^{q_1}(x,\xi)) = P(\Phi_s^{q_1}(x,\eta)) = y,
    \end{equation}
    which exactly corresponds to a geometrical intersection point of the bicharacteristic curves.
\end{definition}

\myindent The crossing of bicharacteristic curves corresponds to the well-known analysis of \textit{focal points} in Fourier Integral Operators theory, which is one of the very delicate part of the theory. In order to simplify the analysis, it is important that we limit the number of crossings of bicharacteristic curves. 

\myindent First, one cannot avoid that \textit{all} bicharacteristic curves starting at a point $x$ cross at the antipodal point $\bar{x}$ (see Definition \ref{antipodpoint}), which is thus \textit{conjugated} to $x$ for the Hamiltonian flow of $q_1$. We call this phenomenon the \textit{antipodal refocalisation} of bicharacteristics.

\begin{lemma}[Antipodal refocalisation]\label{antipodrefocus}
    Under the assumption that $\mathcal{S}$ is \textit{symmetric}, all the bicharacteristic curves of $q_1$ starting at $x$ cross at $\bar{x}$ at the time $\pi$. Precisely, with the Definition \ref{antipodpoint}, there holds
    \begin{equation}
        \forall (x,\xi) \in T^*\mathcal{S}\backslash \{0\}, \qquad \Phi_{\pi}^{q_1}(x,\xi) = (\bar{x},\bar{\xi}).
    \end{equation}
\end{lemma}

\begin{proof}
    This is a direct consequence of the explicit formula giving the bicharacteristics of $q_1$ (see Lemma \ref{explicitbicharac}), and of the fact that, thanks to the symmetry of revolution, the trajectory between $t = 0$ and $t = \pi$ is symmetric to the trajectory between $t = \pi$ and $t = 2\pi$.
\end{proof}

\myindent Now, in order to simplify the analysis, we will work under the hypothesis that this is the \textit{only} case of intersection of bicharacteristic curves (and thus of crossing). As it happens, this hypothesis is not really necessary to the analysis, since the framework actually extends to the case where there is no \textit{crossing} of bicharacteristic curves (except the antipodal refocalisation), but there may be some intersections, which is satisfied under the twist Hypothesis \ref{twisthypothesis}. We will discuss this further in Section \ref{secFurther}.

\begin{hypothesis}[Non intersection of bicharacteristic curves]\label{nonintersectHyp}
    The bicharacteristic curves of $q_1$ do not intersect (see Definition \ref{crossingDef}) except for the antipodal refocalisation.
\end{hypothesis}

\myindent While this hypothesis is restrictive, we give the following proposition to see that it is not empty.

\begin{proposition}\label{norefocusq1norm}
    Let $\mathcal{S}$ be a simple symmetric surface of revolution satisfying the twist Hypothesis \ref{twisthypothesis}. Assume that, for all $x \in \mathcal{S}$, the ball
    \begin{equation}
        \{\xi \in \R^2 \qquad q_1(x,\xi) \leq 1\}
    \end{equation}
    is strictly convex. Then $\mathcal{S}$ satisfies the non intersection of bicharacteristic curves Hypothesis \ref{nonintersectHyp}.
    
    \quad
    
    \myindent Moreover, this condition is satisfies in the following settings.
    
    \myindent i. If there holds, with the definition \eqref{defomegaI},
    \begin{equation}\label{caseomega'neg}
        \forall I \in [0,1], \qquad \omega'(I) < 0.
    \end{equation}
    
    \myindent This is satisfied, for example, by all oblate ellipsoids of revolutions, thanks to Proposition \ref{oblongoblateomega'}.
    
    \myindent ii. If the converse holds, but $\mathcal{S}$ is close enough to the round sphere, in the sense that
    \begin{equation}
        \forall I \in [0,1], \qquad 0 < \omega'(I) < 1.
    \end{equation}
    \myindent This is satisfied, for example, by those oblong ellipsoids of revolutions which are close enough to the round sphere.
\end{proposition}

\myindent We leave the proof of this proposition in the Appendix \ref{AppendixB}, since it is quite technical. 

\myindent We are now in position to define the set $\mathfrak{S}$ of simple symmetric surfaces of revolution on which Theorem \ref{mainthm} holds.
\begin{definition}\label{defmathfrakS}
    $\mathfrak{S}$ is the set of simple symmetric surfaces of revolution $\mathcal{S}$ such that
    
    \myindent i. Either 
    \begin{equation}
        \forall I \in [0,1], \qquad \omega'(I) < 0.
    \end{equation}
    
    \myindent ii. Either 
    \begin{equation}
        \forall I\in[0,1] \qquad 0 < \omega'(I) < 1.
    \end{equation}
    
    \myindent This is an open set in the set of simple symmetric surfaces of revolution, since $\omega'(I)$ has an explicit integral expression in terms of the profile $f$.
\end{definition}

\quad

\myindent Now, when $\mathcal{S}$ satisfies the non intersection of bicharacteristic curves Hypothesis \ref{nonintersectHyp}, similarly to the definition of the geodesic distance between two points on a Riemannian manifold, we may define the \textit{bicharacteristic length} between two points.

\begin{definition}\label{defpsi}
    Assume that $\mathcal{S}$ satisfies the non intersection of bicharacteristic curves Hypothesis \ref{nonintersectHyp}. For all $x,y \in \mathcal{S}$, we define $\psi(x,y)$ the bicharacteristic length between $x$ and $y$ by
    \begin{equation}
        \psi(x,y) := \inf \{t\geq 0 \ \text{such that there exists} \ (x,\xi) \in S_x^*\mathcal{S} \ \text{such that} \ P(\Phi_t^{q_1}(x,\xi)) = y\}.
    \end{equation}
\end{definition}

\myindent Thanks, for example, to the explicit description of the bicharacteristics of $q$ given by Lemma \ref{explicitbicharac}, it is straightforward that $\psi$ is well-defined and finite, since there is always at least once bicharacteristic curve joining $x$ to $y$. Now, thanks to the absence of exceptional intersections, we can actually see that $\psi(x,y)$ really behaves like the geodesic distance on a sphere.

\begin{lemma}\label{propertypsi}
    Assume that $\mathcal{S}$ satisfies the non intersection of bicharacteristic curves Hypothesis \ref{nonintersectHyp}. For all $x,y \in M$, there holds
    \begin{equation}
        \psi(x,y) \in [0,\pi].
    \end{equation}
    
    \myindent More precisely,
    \begin{equation}
        \begin{split}
            \psi(x,y) &= 0 \iff y = x \\
         \psi(x,y) &= \pi  \iff y = \bar{x},
        \end{split}
    \end{equation}
    and, if 
    \begin{equation}
        Z_{\psi} := \{(x,y) \in M \times M \qquad \text{such that} \qquad y = x \ \text{or} \ y = \bar{x}\},
    \end{equation}
    then $\psi$ is smooth on $(M\times M) \backslash Z_{\psi}$.
\end{lemma}

\begin{proof}
    Define, for $x\in M$, and for $t\geq 0$,
    \begin{equation}
        Q_x(t) := \{P(\Phi_t^{q_1}(x,\xi)) \qquad (x,\xi) \in S_x^*\mathcal{S}\}.
    \end{equation}
    
    \myindent Then, Hypothesis \ref{nonintersectHyp} yields that, for $t \in (0,\pi)$, $Q_x(t)$ is a circle (i.e. a 1-dimensional compact submanifold of $M$) which degenerates to $\{x\}$ (resp $\{\bar{x}\}$) as $t \to 0$ (resp $t \to \pi$). Moreover, the circles $Q_x(t), t \in [0,\pi]$ form a partition of $M$, which is smooth in $x$. The lemma is a direct consequence.
\end{proof}

\subsubsection{Miscellaneous}\label{subsubsec53CdV}

\begin{lemma}\label{dG1}
    Let $q_1 = G(p_1,p_2)$. Then $\partial_1 G$ is nowhere zero on $\Gamma$. In particular, it is uniformly bounded from below.
\end{lemma}

\begin{proof}
    In $Int(\Gamma)$, we compute, from \eqref{explicitG}
    \begin{equation}
        \partial_{p_1}G(p_1,p_2) = \frac{1}{\pi}\int_{\sigma_-(I)}^{\sigma_+(I)}\frac{p_1}{\sqrt{p_1^2 - \frac{p_2^2}{f^2(\sigma)}}}d\sigma > 0.
    \end{equation}

    \myindent The only subteltly is that no extra terms comes from derivating the bounds of the integral, as the integrand vanishes when $\sigma = \sigma_{\pm}(I)$.
    
    \myindent Now, using the Taylor expansion of $f$ near $\sigma = 0$, one can check that this converges to a nonzero constant on the boundary $\partial \Gamma$ of $\Gamma$.
    
    \myindent Finally, regarding the uniform bound, it follows from the fact that $\partial_1 G$ is homogeneous of degree $0$ on $\Gamma$, hence one needs only prove it on $\gamma_0$ (see \eqref{deflilgamma0}), where it follows from compactness.
\end{proof}

\begin{lemma}\label{deth'hnotzero}
    For $\vec{p} \in \gamma$, let $\vec{t}(\vec{p})$ be the unit tangent vector to $\gamma$ positively oriented at the point $\vec{p}$. Then there holds
    \begin{equation}
        \inf_{\vec{p}\in\gamma} \det(\vec{p},\vec{t}(\vec{p})) > 0.
    \end{equation}
\end{lemma}

\begin{proof}
    This is a straightforward consequence of the fact that $F_2$ is a smooth homogeneous function of degree 2 which doesn't vanish on $\R^2 \backslash (0,0)$. Indeed, by definition of $\gamma$, 
    \begin{equation}
        \vec{t}(\vec{p}) = \frac{\left(\vec{\nabla} F_2 (\vec{p})\right)^{\perp}}{\left|\left(\vec{\nabla} F_2 (\vec{p})\right)^{\perp}\right|},
    \end{equation}
    where, for any vector $v = \begin{pmatrix}
    v_1\\
    v_2
\end{pmatrix} \in \R^2$, we define until the end
\begin{equation}
    v^{\perp} := \begin{pmatrix}
        -v_2 \\
        v_1
    \end{pmatrix}
\end{equation}
the direct orthogonal to $v$.
    
    \myindent Now, using the Euler identity for homogeneous function,
    \begin{equation}
            \det(\vec{p},\vec{t}(\vec{p})) = \vec{p} \cdot (-(\vec{t})^{\perp}) = \vec{p} \cdot \frac{\vec{\nabla} F_2 (\vec{p})}{\left|\vec{\nabla} F_2 (\vec{p})\right|} = 2\frac{F_2(\vec{p})}{\left|\vec{\nabla} F_2 (\vec{p})\right|} = \frac{2}{\left|\vec{\nabla} F_2 (\vec{p})\right|}  > 0.
    \end{equation}
    
    \myindent The lemma follows from the compactness of $\gamma$.
\end{proof}

\section{Microlocal reformulation of the problem as an estimate of oscillatory integrals}\label{secMicro}

\subsection{Reduction to bounding integrals involving unitary groups only up to times $O(1)$}\label{subsec1Micro}

\myindent In this section, we explain how to twist the usual microlocal approach using the Schwartz kernel of unitary groups, and Colin de Verdière's construction, as briefly presented in Section \ref{secCdV}, in order to reformulate the problem, as explained in the introduction.

\subsubsection{Reduction to a well-chosen smoothed projector}\label{subsubsec11Micro}

\myindent Thanks to Theorem \ref{CdVthm}, we may rewrite the spectral projector $P_{\lambda,\delta}$ (see \eqref{defPlambdadelta}) using the operators $Q_1$ and $Q_2$ as
\begin{equation}\label{defchilambddelt}
    P_{\lambda,\delta} = \chi_{\lambda,\delta}(Q_1,Q_2),
\end{equation}
where $\chi_{\lambda,\delta}$ is the indicator function of the set
\begin{equation}\label{defEld}
    E_{\lambda,\delta} := \{(x,y)\in \R^2 \ \lambda-\delta \leq \sqrt{F(x,y)} \leq \lambda + \delta\}.
\end{equation}

\myindent Now, let $f_{\lambda,\delta}$ be a smoothed version of $\chi_{\lambda,\delta}$, in the sense that it is nonnegative and smooth on $\R^2$ and satisfies moreover $f_{\lambda,\delta} \geq c > 0$ on $E_{\lambda,\delta}$. We may relax the problem of bounding the $L^2 \to L^{\infty}$ norm of $P_{\lambda,\delta}$ to bounding the norm of the smoothed projector $f_{\lambda,\delta}(Q_1,Q_2)$. Keeping in mind that we eventually want to use the unitary groups of $Q_1$ and $Q_2$, i.e. $s\mapsto e^{isQ_1}$ and $t\mapsto e^{itQ_2}$, we choose $f_{\lambda,\delta}$ so that its Fourier transform will have a simple expression.

\begin{deflemm}\label{deffld}
    Let $d\mu_{\lambda}$ be the superficial measure on the curve $\{\sqrt{F_2} = \lambda\}$ and let $\rho_{\delta} := \delta^{-1}\rho(\delta^{-1}\cdot)$, where $\rho$ is a nonnegative Schwartz function on $\R^2$ such that $\rho \geq 1$ on $B(0,3)$. Set
    \begin{equation}\label{defflambdadelta}
    f_{\lambda,\delta}(x,y) = \left(d\mu_{\lambda} \ast \rho_{\delta}\right)(x,y).
    \end{equation}
    
    \myindent We define the smoothed spectral projector
    \begin{equation}
        P_{\lambda,\delta}^{\sharp} := f_{\lambda,\delta}(Q_1,Q_2).
    \end{equation}

    \myindent Then, whenever $\delta \gtrsim_{F} \lambda^{-1}$, and $\lambda$ is large enough, there holds for any compact subset $K$ of $\mathcal{S}$
    \begin{equation}
        \left\|P_{\lambda,\delta}\right\|_{L^2(\mathcal{S}) \to L^{\infty}(K)} \lesssim \left\|P^{\sharp}_{\lambda,\delta}\right\|_{L^2(\mathcal{S}) \to L^{\infty}(K)}.
    \end{equation}
\end{deflemm}

\begin{proof}
    We need only prove that $f_{\lambda,\delta} \geq c > 0$ on the set $E_{\lambda,\delta}$ introduced in \eqref{defEld}. Now, write
    \begin{equation}
        F \sim F_2 + F_0 + F_{-1} + ... 
    \end{equation}
    as in Theorem \ref{CdVthm}. Let $\tilde{F} := F - F_2$, which is a smooth bounded function on $\R^2$. Then, from the observation that
    \begin{equation}
        \sqrt{F} = \sqrt{F_2}\left(1 + O\left(\frac{\|\tilde{F}\|_{\infty}}{F_2}\right)\right),
    \end{equation}
    one finds that, as long as $\delta \gtrsim_{F} \lambda^{-1}$ and $\lambda\geq \lambda_0$ for a large enough $\lambda_0 > 0$, the set $E_{\lambda,\delta}$ is included in the set
    \begin{equation}
    F_{\lambda,\delta} := \{(x,y)\in\R^2 \ \lambda-2\delta \leq \sqrt{F_2(x,y)} \leq \lambda + 2\delta\}.
\end{equation}

\myindent Now, fix $(x,y) \in F_{\lambda,\delta}$. Then $\rho_{\delta}((x,y) - \cdot) \geq \delta^{-1}$ on the ball centered at $(x,y)$ of radius $3\delta$. Now, this ball intersects the curve $\{\sqrt{F_2} = \lambda\}$ on an arc of length at least $c\delta$ for some universal constant $c > 0$. Thus, there holds $(d\mu_{\lambda}\ast \rho_{\delta})(x,y) \geq c > 0$.
\end{proof}

\subsubsection{Integral formulation}\label{subsubsec12Micro}

\myindent We now need to bound the smoothed spectral projector
\begin{equation}
    P^{\sharp}_{\lambda,\delta} := f_{\lambda,\delta}(Q_1,Q_2).
\end{equation}

\myindent We start by expressing $P^{\sharp}_{\lambda,\delta}$ in terms of the unitary groups $s\mapsto e^{isQ_1}$ and $t \mapsto e^{itQ_2}$. There holds in $L^2(\mathcal{S})$ (recall that $f_{\lambda,\delta}$ is a Schwartz function and that $Q_1$ and $Q_2$ \textit{commute} together)
\begin{equation}
    P^{\sharp}_{\lambda,\delta} = f_{\lambda,\delta}(Q_1,Q_2) = (2\pi)^{-1} \int_{\R^2} \widehat{f_{\lambda,\delta}}(s,t) e^{isQ_1}e^{itQ_2}dsdt.
\end{equation}

\myindent Now, using the special form of $f_{\lambda,\delta}$ given by \eqref{deffld}, there holds
\begin{equation}
    \widehat{f_{\lambda,\delta}}(s,t) = \widehat{d\mu_{\lambda}}(s,t) \widehat{\rho_{\delta}} (s,t).
\end{equation}

\myindent Thanks to the homogeneity of $\sqrt{F_2}$, there holds that $d\mu_{\lambda} = \lambda d\mu(\lambda^{-1}\cdot)$, where we recall that $d\mu$ is the superficial measure on $\gamma$ (see Definition \ref{deflilgamma}). Hence, we may express the Fourier transform of $d\mu_{\lambda}$ (resp. $\rho_{\delta}$) in terms of the Fourier transform of $d\mu$ (resp. $\rho$). There holds finally
\begin{equation}
    \widehat{f_{\lambda,\delta}}(s,t) = \lambda \delta \hat{d\mu}(\lambda(s,t)) \hat{\rho}(\delta(s,t)),
\end{equation}
from which we may deduce that
\begin{equation}\label{intexprflambdadelta}
    P^{\sharp}_{\lambda,\delta} = \lambda \delta (2\pi)^{-2} \int_{\R^2} \hat{d\mu}\left(\lambda(s,t)\right)\hat{\rho}\left(\delta(s,t)\right) e^{isQ_1}e^{itQ_2} dsdt.
\end{equation}

\myindent Now, in order to bound the $L^2 \to L^{\infty}$ norm of $P^{\sharp}_{\lambda,\delta}$, it is classical to introduce its \textit{Schwartz kernel}, which can be expressed, thanks to \eqref{intexprflambdadelta}, via the Schwartz kernels of the operators $e^{isQ_1}$ and $e^{itQ_2}$. We recall the definition of Schwartz kernels.

\begin{definition}
    Let $M$ be a smooth Riemannian manifold of dimension $d$. Let $L : \mathcal{D}(M) \to \mathcal{D}'(M)$ be a continuous linear map, where $\mathcal{D}(M)$ is the set of smooth compactly supported functions on $M$, and $\mathcal{D}'(M)$ is the set of distributions on $M$. Then, the Schwartz kernel of $L$ is the unique distribution $\ell \in \mathcal{D}'(M\times M)$ such that, informally, for any $f \in \mathcal{D}(M)$,
    \begin{equation}
        Lf(x) = \int_M \ell(x,y) f(y) d\upsilon(y).
    \end{equation}
    
    \myindent In particular, we will use the abuse of notation $L(x,y)$ for the Schwartz kernel of $L$.
\end{definition}

\myindent Now, a useful fact, when computing $L^2 \to L^{\infty}$ norms of an operator $L$, is that it depends only on the \textit{diagonal values} of the Schwartz kernel of $LL^{*}$, where $L^{*}$ is the adjoint operator of $L$. Precisely, we give the following lemma.

\begin{lemma}\label{bddnormbykernel}
    Let $M$ be a smooth manifold, and let $L : L^2(M) \to L^2(M)\cap L^{\infty}(M)$ be a bounded linear operator. Let $K$ be any compact subset of $M$. Then, there holds
    \begin{equation}
        \left\|L \right\|_{L^2(M) \to L^{\infty}(K)} = \sup_{x \in K} (LL^{*}(x,x))^{\frac{1}{2}},
    \end{equation}
    where $LL^{*}(x,y)$ is the Schwartz kernel of the operator $LL^{\star}$.
\end{lemma}
\begin{proof}
    There holds
    \begin{equation}
    \begin{split}
        \left\|L \right\|_{L^2(M) \to L^{\infty}(K)} &= \sup_{f \in L^2(M), \|f\| = 1} \left(\sup_{x\in K} |f(x)|\right) \\
        &= \sup_{x \in K} \left(\sup_{f\in L^2(M), \|f\| = 1} |f(x)| \right) \\
        &= \sup_{x \in K} \left(\sup_{f\in L^2(M), \|f\| = 1} \int_M L(x,y) f(y) d\upsilon(y) \right) \\
        &= \sup_{x \in K} \left(\int_M |L(x,y)|^2 d\upsilon(y)\right)^{\frac{1}{2}} \\
        &= \sup_{x \in K} \left(\int_M L(x,y) \overline{L(x,y)}  d\upsilon(y)\right)^{\frac{1}{2}}\\
         &= \sup_{x \in K} \left(\int_M L(x,y) L^{*}(y,x)  d\upsilon(y)\right)^{\frac{1}{2}} \\
          &= \sup_{x \in K} (LL^* (x,x))^{\frac{1}{2}}.
    \end{split}
    \end{equation}
\end{proof}

\myindent In particular, applying this lemma to $L = P_{\lambda,\delta}$, which satisfies $P_{\lambda,\delta}^* = P_{\lambda,\delta}^2 = P_{\lambda,\delta}$, we find that
\begin{equation}
    \|P_{\lambda,\delta}\|_{L^2(\mathcal{S}) \to L^{\infty}(K)} = \sup_{x\in K} \left( P_{\lambda,\delta}(x,x) \right)^{\frac{1}{2}}.
\end{equation}

\myindent Combining this formula with the Definition-Lemma \ref{deffld}, we finally find that

\begin{lemma}\label{boundbyintegral}
    Let $K$ be a compact subset of $\mathcal{S}$. There holds
    \begin{equation}\label{firstintegralexpression}
    \begin{split}
        \left\|P_{\lambda,\delta}\right\|_{L^2(\mathcal{S}) \to L^{\infty}(K)} ^2&\lesssim \sup_{x\in K} P_{\lambda,\delta}(x,x) \\
        &\lesssim \lambda\delta \sup_{x \in K} \int_{\R^2} \hat{d\mu}(\lambda(s,t)) \hat{\rho}(\delta(s,t)) \left(e^{isQ_1} e^{itQ_2}\right)(x,x) ds dt.
    \end{split}
    \end{equation}
\end{lemma}

\myindent The point is now that
\begin{equation}
    (s,t) \mapsto e^{isQ_1} e^{itQ_2}
\end{equation}
is $4\pi$-periodic in both $s$ and $t$ (and even $2\pi$-periodic up to a sign coming from the Maslov index). Hence, as announced in the introduction, the decay due to long-term dynamics of the geodesic flow is \textit{not} encoded in any decay of the kernel $e^{isQ_1}e^{itQ_2}(x,x)$ as $|(s,t)| \to \infty$. However, since the curvature of $\gamma$ doesn't vanish at infinite order, there is \textit{a priori} decay of $(s,t) \mapsto \hat{d\mu}(\lambda(s,t))$, both as $|(s,t)|\to \infty$ and, for fixed $(s,t)$, as $\lambda \to \infty$ (see for example \cite{stein1993harmonic}[Theorem 2, Chapter VIII]). This is the crucial point of our method : we have entirely encoded the decay of the kernel $t \mapsto e^{it\sqrt{-\Delta}}(x,x)$, which is \textit{qualitatively} known, into the decay of the semi-explicit function $\hat{d\mu}$, which is \textit{quantitatively} known. This is ultimately why we are able to obtain the quantitative estimate \ref{mainthm}. Another point of view on this periodicity is that the expression \eqref{firstintegralexpression} involves the unitary groups of $Q_1$ and $Q_2$ only up to times $O(1)$, thanks to the periodicity, thus solving the problem of the usual microlocal approach which involves a unitary group on a time $O(\delta^{-1})$ which is arbitrarily large, as explained in the introduction.

\quad

\myindent Now, in order to quantitatively bound the integral in the right-hand side of \eqref{firstintegralexpression}, we need to find suitable expression for the Schwartz kernel
\begin{equation}
    \left(e^{isQ_1} e^{itQ_2}\right)(x,x).
\end{equation}

\myindent First, we observe that $Q_2 = P_2 = \frac{1}{i}\frac{\partial}{\partial \theta}$ (see Theorem \ref{CdVthm}) is the infinitesimal generator of the rotation around the axis of $\mathcal{S}$. Hence, the action of $e^{itQ_2}$ on a given function $f(\theta,\sigma) \in L^2(\mathcal{S})$ is explicit outside of the Poles : denote $R_t$ the rotation of angle $t$ given in local coordinate by 
\begin{equation}
    R_t(\theta,\sigma) := (\theta + t,\sigma).
\end{equation}

\myindent Then,
    \begin{equation}
        e^{itQ_2} f = f \circ R_{t}.
    \end{equation}

\myindent Hence, we find that
    \begin{equation}\label{kernelofeisQ1}
        \begin{split}
        e^{isQ_1}e^{itQ_2}(x,x) &= e^{isQ_1}(x, R_{-t}(x))\\
        &= e^{isQ_1}(R_{t}(x),x),
       \end{split}
    \end{equation}
where we use the symmetry of revolution for the last equality. In particular, we are ultimately concerned only with the kernel of the unitary group generated by $Q_1$. Moreover, thanks to the symmetry of revolution of $\mathcal{S}$, the value of \eqref{kernelofeisQ1} doesn't depend on the polar coordinate $\theta$ of $x$. Hence, without loss of generality, we can assume that $x = (0, \sigma)$ with $\sigma \in \left]-\frac{L}{2} + \eps, \frac{L}{2}-\eps\right[$ and $R_t(x) = (t,\sigma)$. We thus define, for $\sigma,s,t \in \left]-\frac{L}{2},\frac{L}{2}\right[\times S^1 \times S^1$
\begin{equation}\label{defUsigmast}
    U(\sigma,s,t) := e^{isQ_1}((t,\sigma),(0,\sigma)).
\end{equation}

\myindent Now, it is well-known that one can describe the kernel $e^{isQ_1}(x,y)$ locally around any $(s_0,x_0,y_0)$, via the introduction of \textit{parametrices} from the theory of Fourier Integral Operators. The following section is thus devoted to the construction of suitable parametrices.

\subsection{Construction of parametrices}\label{subsec2Micro}

\myindent In this section, we construct a suitable representation of the kernel \eqref{defUsigmast} in the form of an oscillatory integral with an appropriate phase. We will rely on the abstract theory of Fourier Integral Operators (FIO). Since this theory is only needed for the construction of the oscillatory integrals, we will be extremely succinct in our presentation. Hence, the reader who is not familiar with the theory of FIO may skip Paragraphs \ref{subsubsec21Micro} to \ref{subsubsec24Micro} on first reading, and use the Classification Proposition \ref{classification} in Paragraph \ref{subsubsec25Micro} as a black box. We will ourselves use many theorems from the theory of FIO as black boxes, and we won't try to give intuition on their proof, since it is not relevant to the rest of the presentation. Instead, for the reader who would like to understand deeper the construction, we refer them to the textbook presentation of FIO by Duistermaat \cite{duistermaat1996fourierb}, to the seminal article \cite{duistermaat1975spectrum}, or to the more exhaustive \cite{hormander2009analysis}.

\subsubsection{A reminder on parametrices}\label{subsubsec21Micro}

\myindent In this paragraph, we briefly recall some fundamental theorems of the theory of Fourier Integral Operators, which we will use. We refer to the presentation of Duistermaat \cite{duistermaat1996fourierb} for the proofs and developments.

\myindent Fix $M$ a smooth compact manifold of dimension $d$, and $Q$ an elliptic first order classical pseudo differential operator on $M$. The semigroup of operators $s\mapsto e^{isQ}$ has been wildly studied, with most notable contributions coming, among many others, from Hörmander, Duisteermat and Guillemin. A first fundamental result is that this semigroup is itself a Fourier Integral Operator, whose canonical relation is explicit. We give the following theorem from \cite{duistermaat1975spectrum}[Theorem 1.1.].

\begin{theorem}[Duistermaat-Guillemin]\label{DG75thm}
    Let $U(s) := e^{isQ}$. Then, $U$ is a Fourier Integral Operator of class $I^{-\frac{1}{4}}\left(\R\times M, M;C\right)$ defined by the canonical relation
    \begin{multline}\label{canonicrelation}
        C := \{(s,\tau),(x,\xi),(y,\eta) ;\ (x,\xi),(y,\eta)\in T^*M\backslash \{0\},\\ (s,\tau) \in T^*\R\backslash 0,\
        \tau - q(x,\xi) = 0, \ \Phi^q_s(x,\xi) = (y,\eta)\},
    \end{multline}
    where $\Phi^q_s$ denotes the Hamiltonian flow of the principal symbol $q$ of $Q$ on $T^*M\backslash \{0\}$ as in Definition \ref{defbicharacteristics}.
\end{theorem}

\myindent One can then deduce a local representation of the kernel of the operator $(s,x,y) \mapsto e^{isQ}(x,y)$ as an \textit{oscillatory integral}. For that purpose, we recall the following definitions, adapted from \cite{duistermaat1996fourierb}.

\begin{definition}
    Let $X$ be a smooth compact manifold of dimension $n$. Let $N\geq 1$, which is called the number of \textit{angle variables}. A smooth function $\phi(x,\theta)$ defined on an open conic set $U\times C$ of $X\times \R^N$ is called a non degenerate phase function if

    i. $\phi(x,\theta)$ is homogeneous of degree $1$ in the variable $\theta$.

    ii. $d_{\theta}\phi(x,\theta) = 0$, $(x,\theta) \in U\times C$ $\implies$ \ $d_{x,\theta}\frac{\partial \phi(x,\theta)}{\partial_{\theta_j}}$ are linearly independent for $j = 1,...,N$.
\end{definition}

\myindent We introduce moreover useful notations.
\begin{deflemm}
    Let $\phi(x,\theta)$ be a nondegenerate phase function defined in an open conic set $U\times C \subset X \times \R^N$. We define
    \begin{equation}
        C_{\phi} := \{(x,\theta) \in U\times C \qquad d_{\theta}\phi(x,\theta) = 0\},
    \end{equation}
    which is a smooth conic submanifold of $U\times C$ of dimension $n$.

    \myindent Moreover, we define
    \begin{equation}
        \Lambda_{\phi} := \{(x,d_x\phi(x,\theta)), \qquad (x,\phi) \in C_{\phi}\},
    \end{equation}
    which is an immersed n-dimensional lagrangian conic submanifold of $T^*X\backslash \{0\}$.
\end{deflemm}

\begin{definition}
    Let $X,Y$ be two smooth compact manifolds. Let $\Lambda$ be a lagrangian conic submanifold of $T^*X\times T^*Y$. We define
    \begin{equation}
        \Lambda' :=\{(x,\xi,y,-\eta); \qquad (x,\xi,y,\eta) \in \Lambda\}.
    \end{equation}
\end{definition}

\myindent We also recall the definition of symbols.

\begin{definition}\label{symbolestimate}
    Let $X$ be a smooth manifold. Let $N\geq 1$. A smooth function $a(x,\theta)$ defined on an open conic subset $U \times \R^N$ of $X\times \R^N$ is called a \textit{symbol} of order $\mu \in \R$ if for any compact subset $K$ of $U$, for all multi-indices $\alpha,\beta$, there holds
    \begin{equation}
        \left|\left(\frac{\partial}{\partial x}\right)^{\beta} \left(\frac{\partial}{\partial\theta}\right)^{\alpha} a(x,\theta) \right| \lesssim_{\alpha,\beta,K} (1 + |\theta|)^{\mu - |\alpha|},
    \end{equation}
    for $(x,\theta) \in K\times \R^N$.
\end{definition}

\myindent Now, thanks to the theory of equivalence of phase functions (see for example \cite{duistermaat1996fourierb}[Theorem 2.3.4]), one can prove after some work the following corollary of Theorem \ref{DG75thm}.
\begin{corollary}\label{defoscintbyphase}
    Let $(s_0,x_0,y_0) \in \R \times M \times M$. Let $\phi(s,x,y,\zeta)$ be a nondegenerate phase function with angle variable $\zeta \in \R^N$ for some $N$. Assume that the canonical relation $C$ and $\Lambda_{\phi}'$ coincide locally near $(s_0,x_0,y_0)$, i.e. that locally near $(s_0,x_0,y_0)$ there holds
    \begin{equation}\label{adaptdef}
        d_{\zeta}\phi(s,x,y,\zeta) = 0 \iff (s,d_s \phi),(x,d_x\phi), (y,-d_y\phi) \in C.
    \end{equation}

    \myindent Then locally around $(s_0,x_0,y_0)$, $e^{stQ}(x,y)$ is given, modulo a smoothing remainder, by an oscillatory integral of the form
    \begin{equation}
    I(s,x,y) = \int_{\R^N} e^{i\phi(s,x,y,\zeta)} f(s,x,y,\zeta) d\zeta,
\end{equation}
    where $f$ is a smooth symbol of order $\frac{d - N}{2}$.
\end{corollary}

\myindent As a consequence of that corollary, we will say that the nondegenerate phase function $\phi$ is \textit{adapted} to the canonical relation $C$ around $(s_0,x_0,y_0)$ if \eqref{adaptdef} holds locally.

\quad

\myindent Finally, we recall the following useful Lemma (see \cite{duistermaat1996fourierb}[Lemma 2.3.5]) regarding the \textit{minimum number} of angle variables which are needed.
\begin{lemma}\label{minnumberofdim}
    The number of $\zeta$ variables is greater or equal to the dimension $k$ of the intersection of the tangent space of the canonical relation $C$ and of the fibers of the cotangent bundle $T^*\R\times T^*M \times T^* M$.

    \myindent Moreover, one can find an adapted phase function with exactly $k$ angles variables.
\end{lemma}

\myindent Looking closely at the definition of the canonical relation \eqref{canonicrelation}, observe that this dimension is at most $d = dim (M)$, and in fact it corresponds locally to the dimension of the set of directions of bicharacteristic curves of $q$ joining $x$ to $y$ in a time $t$. Observe that this dimension can jump.

\quad

\myindent In the following, we apply this theory to the case $M = \mathcal{S}$, $Q = Q_1$. Observe that the theory does apply to $Q_1$ thanks to the crucial fact that it is an \textit{elliptic} operator (see Remark \ref{Q1elliptic}). Since we want a suitable parametrix for \eqref{defUsigmast}, we will focus on finding a nondegenerate phase adapted to the canonical relation $C$ near $(s_0,x_0,y_0) = (s_0,(t_0,\sigma),(0,\sigma))$ for some $(s_0,t_0) \in [-\pi,\pi]^2$.

\subsubsection{Hörmander's parametrix}\label{subsubsec22Micro}

\myindent The problem of finding a parametrix of $s\mapsto e^{isQ}$ for \textit{small times} $|s| \ll 1$ has been extensively studied, and is more classical than the general case. We give the microlocal parametrix introduced by Hörmander in \cite{hormander1968spectral}, in the form of Sogge \cite{sogge2017fourier}[Section 4.1], where one can find a textbook presentation of (a more general version of) the following theorem.

\begin{theorem}\label{Hormthm}
    Let $(M,g)$ be a smooth riemannian manifold of dimension $d$ and let $Q$ be a classical pseudodifferential operator of order $1$ which is elliptic and self-adjoint. Then there exists $\eps>0$ such that, when $|s| < \eps$,
    \begin{equation}
        e^{isQ} = I(s) + R(s),
    \end{equation}
    where the remainder $R(s)$ has kernel $R(s,x,y) \in \mathcal{C}^{\infty}([-\eps,\eps] \times M \times M)$ and the kernel $I(s,x,y)$ is supported in a small neighborhood of the diagonal in $M\times M$. Furthermore, let $\Omega \subset M$ be a sufficiently small local coordinate patch and let $\omega \subset \Omega$ be relatively compact. Then, when $(s,x,y) \in [-\eps,\eps]\times \omega \times \omega$, $I(s,x,y)$ takes the form
    \begin{equation}
        I(s,x,y) = (2\pi)^{-d} \int_{\R^d} e^{i\varphi(x,y,\xi)}e^{isq(y,\xi)} a(s,x,y,\xi) d\xi,
    \end{equation}
    where locally the cotangent space $T^*\omega$ is identified to $\omega \times \R^n$, $q(y,\xi)$ is the principal symbol of $Q$ and $\varphi$ is the unique (local) solution to the \textit{eikonal equation}
    \begin{equation}\label{eikeq}
        q(x,\nabla_x\varphi(x,y,\xi)) = q(y,\xi)
    \end{equation}
    satisfying the control
     \begin{equation}\label{asdevphi}
        \forall \alpha \in \N^n \\ \left|\left(\frac{\partial}{\partial \xi}\right)^{\alpha}\left[\varphi(x,y,\xi) - \scal{x-y}{\xi}\right]\right| \lesssim_{\alpha} |x-y|^2|\xi|^{1 - \alpha}.
    \end{equation}
    
    \myindent In particular, 
    \begin{equation}
        \phi_H(s,x,y,\xi) := sq(y,\xi) + \varphi(x,y,\xi)
    \end{equation}
    is a nondegenerate phase function which is adapted to the canonical relation $C$ around $(0,x,x)$.
    
    \myindent Finally, $a$ is a symbol in $S^0$, i.e.
    \begin{equation}\label{estsymb}
        \left|\left(\frac{\partial}{\partial \xi}\right)^{\alpha}\left(\frac{\partial}{\partial t}\right)^{\beta_1}\left(\frac{\partial}{\partial x}\right)^{\beta_2}\left(\frac{\partial}{\partial y}\right)^{\beta_3}a(t,x,y,\xi)\right| \lesssim_{\alpha,\beta} (1 + |\xi|)^{-|\alpha|},
    \end{equation}
    such that moreover
    \begin{equation}
        a(0,x,x,\xi) = 1 + O((1 + |\xi|)^{-1})
    \end{equation}
    uniformly in $x$.
\end{theorem}

\myindent This theorem applies to $M = \mathcal{S}$, $d = 2$, and $Q = Q_1$ since we recall that $q_1$, its principal symbol, is \textit{elliptic} (see Lemma \ref{q1elliptic}). We use it on the open set $\Omega = \mathcal{S} \backslash \{N,S\}$, which is a coordinate patch with coordinates $(\theta,\sigma)$, and $\omega = K_{\eps}$. Hence, there exists a $\eps_1 > 0$ such that for $\sigma \in \left]-\frac{L}{2} + \eps, \frac{L}{2} - \eps\right[$ and $|s|, |t| < \eps_1$, we can write
\begin{equation}\label{HormUsigmast}
\begin{split}
    U(\sigma,s,t) &= \int_{\R^2} e^{i\phi_H(\sigma,s,t,\xi)} f(\sigma,s,t,\xi) d\xi + R(\sigma,s,t) \\
    &=: I(\sigma,s,t) + R(\sigma,s,t),
\end{split}
\end{equation}
where $f$ is a smooth symbol uniformly of order zero, $R$ is a smooth function, and where we use the abuse of notation for the phase
\begin{equation}\label{hormparam}
    \begin{split}
    \phi_H(\sigma,s,t,\xi) &:= \phi_H(s,(t,\sigma),(0,\sigma),\xi) \\
    &= s q_1(\sigma,\xi) + \varphi((t,\sigma),(0,\sigma),\xi),
    \end{split}
\end{equation}

where $\varphi(x,y,\xi)$ is the unique local solution of the eikonal equation \eqref{eikeq} satisfying the local control \eqref{asdevphi}. We insist that, here, the variable $\xi$ coincides with the covariables $(\Theta,\Sigma)$ associated to $(\theta,\sigma)$.

\subsubsection{Antipodal Hörmander's parametrix}\label{subsubsec23Micro}

\myindent It is usual in the theory of FIO that the most delicate part of finding a parametrix is near the caustics. Now, thanks to the analysis of Paragraph \ref{subsubsec52CdV}, we see that, thanks to the twist Hypothesis \ref{twisthypothesis}, the only case of focal point is given by the \textit{antipodal refocalisation} (see Lemma \ref{antipodrefocus}). More precisely, since we are interested in a parametrix near $(t,\sigma),(0,\sigma)$ which are on the same parallel, the only case of focal point will happen for $\sigma = 0$ and $t = \pi$. Now, in order to find a parametrix near $\sigma = 0, s = t = \pi$, we observe that, actually, thanks to the antipodal refocalisation itself, the geometry near $(\pi,0)$ is exactly the same as near $(0,0)$. In particular, we can use (an adapted version of) the Hörmander parametrix itself. Precisely, we claim the following. We don't claim that we are the first to observe that fact, however we don't know of any work were this construction is explicitely done.

\begin{lemma}\label{antipodparam}
    Let $U$ be a neighborhood of $(x_0, \bar{x_0})$ where $x_0 = (0,0)$ and $\bar{x_0} = (\pi,0)$ is its antipodal point. Let
    \begin{equation}
        \phi_H(s,x,y,\xi) := sq_1(x,\xi) + \varphi(x,y,\xi)
    \end{equation} 
    be the Hörmander phase function near $s = 0$, $x = y = x_0$. Then, we claim that the \textit{antipodal Hörmander phase function} $\phi_{\pi}(s,x,y,\xi)$, defined for $s$ near $\pi$, $(x,y)\in U$ and $\xi \in \R^2 \backslash (0,0)$ by
    \begin{equation}\label{defantipodparam}
        \phi_{\pi}(s,x,y,\xi) := \phi_H(s - \pi,x,\bar{y},\xi),
    \end{equation}
    is a well-defined smooth nondegenerate phase function which is adapted to the canonical relation of $e^{isQ_1}(x,y)$ near $(s_0,x_0,y_0) = (\pi,x_0,\bar{x_0})$ 
\end{lemma}

\begin{proof}
    We denote, for $x\in\mathcal{S}$, $S(x) := \bar{x}$. It is obvious from the definition that $\phi_{\pi}$ is a nondegenerate phase function. Hence, we only need to prove that, locally, $\Lambda_{\phi_{\pi}}'$ and $C$ defined by \eqref{canonicrelation} coincide. Now,
    \begin{equation}
        d_{\xi} \phi_{\pi}(s,x,y,\xi) = d_{\xi}\phi_H(s - \pi,x,S(y),\xi).
    \end{equation}
    
    `\myindent Hence,
    \begin{equation}
        C_{\phi_{\pi}} = \{(s,x,y,\xi) \ \text{such that} \ (s - \pi,x,S(y),\xi) \in C_{\phi_H}\}.
    \end{equation}

    \myindent Moreover, since $S \circ S = Id$, and hence $dS(S(y)) = (dS(y))^{-1}$, there holds
    \begin{equation}
        d_{x,y} \phi_{\pi}(s,x,y,\xi) = \left( d_x \phi_H(s - \pi,x,S(y),\xi), (dS(y))^{-1} d_y \phi_H(s - \pi,x,S(y),\xi) \right).
    \end{equation}

    \myindent From which we may deduce that 
    \begin{equation}
        \Lambda_{\phi_{\pi}} = \{(s,\tau),(x,\xi),(y,(dS(y))^{-1}\cdot \eta) \qquad (s-\pi,\tau),(x,\xi),(S(y),\eta) \in \Lambda_{\phi_H}\}.
    \end{equation}

    \quad

    \myindent Thus, in order to prove that $\Lambda_{\phi_{\pi}}' = C$ locally, we need only prove that
    \begin{equation}
    \begin{split}
        &\{(s,\tau),(x,\xi),(y,\zeta) \in T^*\R \backslash \{0\} \times T^*M \backslash \{0\} \times T^*M \backslash \{0\}\ 
    \text{such that} \ \tau - p(x,\xi) = 0, \ (y,\zeta) = \Phi_s^q(x,\xi)\}\\
    &= \{(s,\tau),(x, \xi),(y,(dS(y))^{-1} \cdot \eta) \qquad \tau = p(x,\xi), \ (S(y),\eta) = \Phi^q_{s-\pi}(x,\xi)\}.
    \end{split}
\end{equation}

\myindent Now, this is a consequence of the fact that, for all $(z,\zeta) \in T^*\mathcal{S}\backslash \{0\}$, there holds
\begin{equation}
    \Phi^q_{\pi}(z,\zeta) = (S(z),dS(z)\cdot \zeta),
\end{equation}
which itself follows from the symmetries of $\mathcal{S}$.
\end{proof}

\myindent As a consequence, Corollary \ref{defoscintbyphase} applies : locally around $(\sigma,s,t) = (0,\pi,\pi)$, we can write
\begin{equation}\label{antipodUsigmast}
\begin{split}
    U(\sigma,s,t) &= \int e^{i\phi_{\pi}(\sigma,s,t,\xi)} f(\sigma,s,t,\xi) d\xi + R(\sigma,s,t) \\
    &=: I(\sigma,s,t) + R(\sigma,s,t),
\end{split}
\end{equation}
for some smooth symbol $f$ uniformly of order zero, and some smooth remainder $R$, and where we use the abuse of notation
\begin{equation}\label{antipodhormparam}
    \phi_{\pi}(\sigma,s,t,\xi) = \phi_{\pi}(s,(t,\sigma),(0,\sigma),\xi).
\end{equation}

\begin{remark}
    In order to guess that \eqref{defantipodparam} is a good candidate for a phase function adapted to the canonical relation $C$ near $(s,x,y) = (\pi, x_0, \bar{x_0})$, one can observe that
    \begin{equation}
        e^{i\pi Q_1} = i S,
    \end{equation}
    where $S : M \to M$ is the operator $x\mapsto \bar{x}$. This can be proved by looking at the action of $e^{i\pi Q_1}$ on the joint eigenfunctions of $(Q_1,Q_2)$.
\end{remark}

\subsubsection{The bicharacteristic length parametrix}\label{subsubsec24Micro}

\myindent Now, there only remains to deal with the case where $(\sigma_0,s_0,t_0) \in \left]-\frac{L}{2} + \eps, \frac{L}{2} - \eps \right[ \times S^1 \times S^1$ is away from $s = t = 0$ and from $\sigma = 0$, $s = t = \pi$. We will prove, in this section, that the kernel $U(\sigma,s,t)$ can be expressed in the following way. Similar constructions have already been done in the literature, see for example \cite{colin1974parametrix}.

\begin{proposition}\label{bicharactparamprop}
    Let $(\sigma_0,s_0,t_0) \in \left]-\frac{L}{2} + \eps, \frac{L}{2} - \eps \right[ \times S^1 \times S^1$ such that $(s_0,t_0) \neq (0,0)$ and $(\sigma_0,s_0,t_0) \neq (0,\pi,\pi)$. Then, there exists a neighborhood $I\times J \times K$ of $(\sigma_0,s_0,t_0)$ in $\left]-\frac{L}{2} + \eps, \frac{L}{2} - \eps \right[ \times S^1 \times S^1$ such that

    \myindent i. If
    \begin{equation}\label{eqofsing}
        |s_0| = \psi((\sigma_0,t_0),(\sigma_0,0)),
    \end{equation}
    where we recall that $\psi(x,y)$ is the bicharacteristic length between $x$ and $y$ defined by \eqref{defpsi}, then 
    \begin{equation}\label{bicharlengthparam}
        \phi_{bl}(s,x,y,r) := r(|s| - \psi(x,y)), \qquad (s,x,y,r) \in I\times J \times K \times \R_+^*,
    \end{equation}
    is a nondegenerate phase function adapted to the canonical relation $C$ around $(s_0,x_0,y_0)$ locally.

    \myindent ii. Otherwise, 
    \begin{equation}
        (\sigma,s,t) \mapsto U(\sigma,s,t)
    \end{equation}
    is a smooth function on $I\times J \times K$
\end{proposition}

\myindent We detail the proof, and explain more precisely how to \textit{guess} the phase function \eqref{bicharlengthparam}. Indeed, once we have the formula \eqref{bicharlengthparam}, it is straightforward to check that it yields a phase function adapted to the canonical relation $C$, but we wish to give (our) intuition on the construction, in the hope that it could be generalized. In particular, the following proof may seem unnecessarily long, but it actually doesn't use many properties of the bicharacteristic length function $\psi$, and, hence, it could be generalized locally in other settings, as we will discuss in Section \ref{secFurther}. Moreover, we will use some of the elements of this proof in order to study the bicharacteristic lentgh function, see Appendix \ref{AppendixD}.

\begin{proof}
The non interesction of bicharacteristic curves Hypothesis \eqref{nonintersectHyp} guarantees that, locally near $(s_0,x_0,y_0) := (s_0, (\sigma_0,t_0), (\sigma_0,0))$, the maximal dimension of the intersection of the canonical relation of $e^{isQ_1}(x,y)$ and of the fibers of the cotangent bundle is one. We even know more : depending on whether there exists a bicharacteristic curve of $q_1$ joining $x$ to $y$ in time $s$, this intersection is either empty, either exactly a half-line.

\quad

\myindent Now, observe first that, for $s_0 \in [-\pi,\pi]$, then 
\begin{equation}
    |s_0| = \psi(x_0,y_0)
\end{equation}
if and only if there exists a bicharacteristic curve of $q_1$ joining $x_0$ to $y_0$ in time $|s_0|$. Hence, in the case ii. of the proposition, one can find a small neighborhood $K\times U \times V$ of $(s_0,x_0,y_0)$ in $S^1 \times \mathcal{S} \times \mathcal{S}$ such that
\begin{equation}
    \forall (s,x,y) \in K \times U \times V \qquad \psi(x,y) \neq |s|.
\end{equation}

\myindent In other words, there is \textit{no} bicharacteristic curve of $q_1$ joining $x$ to $y$ in time $|s|$. Thanks to \eqref{canonicrelation}, we find that the intersection of $C$ and of the fiber $T_{s,x,y}^* S^1\times \mathcal{S} \times \mathcal{S}$ is \textit{empty}. Hence, classical theory of Fourier Integral Operators yield that $e^{isQ_1}(x,y)$ is \textit{smooth} on $K\times U \times V$. As a consequence, we may find small intervalls $I,J$ in $\left]-\frac{L}{2} + \eps, \frac{L}{2} - \eps \right[$ and $ S^1$ respectively such that for $(\sigma,t) \in I \times J$, and $s\in K$, $U(\sigma,s,t)$ is a smooth function.

\quad

\myindent Hence, we focus on the more interesting case i. where \eqref{eqofsing} holds, or, equivalently, where the exists a bicharacteristic curve joining $x_0$ to $y_0$ in time $|s_0|$. Thanks to the hypotheses, we know that $y_0 \notin \{x_0, \bar{x_0}\}$ so, thanks to Lemma \ref{propertypsi}, $|s_0| \in (0,\pi)$ and $\psi$ is smooth around $(s_0,y_0)$. We assume moreover that $s_0 > 0$ in order to clarify the notations, the proof being exactly similar in the case $s_0 < 0$. Now, from Hypothesis \ref{nonintersectHyp}, if $(s,x,y)$ is close to $(s_0,x_0,y_0)$ and such that there exists $(x,\xi) \in S_x^*\mathcal{S}$ such that
\begin{equation}
    P(\Phi_s^{q_1}(x,\xi)) = y,
\end{equation}
then, this $\xi$ is necessarily unique, and hence the intersection of $C$ and of the fiber $T^*_{s,x,y}(\R\times \mathcal{S}\times \mathcal{S})$ equals
\begin{equation}
    \{(s,\tau), (x,r\xi),(y,-r \eta), \qquad r > 0\},
\end{equation}
where we define
\begin{equation}
    \tau := q_1(x,\xi),
\end{equation}
and
\begin{equation}
    (y,\eta) := \Phi_s^{q_1}(x,\xi).
\end{equation}

\myindent As a consequence, Lemma \ref{minnumberofdim} yields that there exists a nondegenerate homogeneous phase function with only \textit{one} angle variable, say
\begin{equation}
    \phi(s,x,y,r),
\end{equation}
defined for $(s,x,y)$ near $(s_0,x_0,y_0)$ and $r \in \R_+^*$, and that there exists a smooth symbol uniformly of order $\frac{1}{2}$ say $f(s,x,y,r)$ such that locally, and modulo a smoothing remainder,
\begin{equation}\label{intexprtheoretic1D}
    e^{isQ_1}(x,y) = \int e^{i\phi(s,x,y,r)}f(s,x,y,r) dr.
\end{equation}

\myindent Now, the fact that there is only one angle variable reduces considerably the uncertainty on the phase function. Indeed, observe first that, by homogeneity,
\begin{equation}
    \phi(s,x,y,r) = r \phi(s,x,y,1),
\end{equation}
so $\Phi$ really is only a function of $(s,x,y)$. Moreover, observe that, thanks to \eqref{intexprtheoretic1D}, there holds
\begin{equation}\label{whyeikonaleqholds}
    \left(\frac{1}{i} \frac{\partial}{\partial s} - Q_1 \right) \left(\int e^{i\phi(s,x,y,r)}f(s,x,y,r) dr\right) \ \in \Psi^{-\infty}(\R\times \mathcal{S} ; \mathcal{S}),
\end{equation}
where $\Psi^{-\infty}$ is the class of smoothing operators. Now, at the main order (i.e. principal symbols), one can compute, using for example \cite{sogge2017fourier}[Theorem 3.2.3], that this ensures the following eikonal equation 
\begin{equation}\label{eikeqtimedependent}
    \partial_s \phi = q_1(x,\nabla_x \phi).
\end{equation}

\myindent In particular, we see that $\Phi$ is uniquely determined from the data of 
\begin{equation}
    \tilde{\psi}(x,y) := \phi(s_0,x,y,1).
\end{equation}

\myindent Moreover, the phase equation
\begin{equation}
    \Lambda_{\phi}' = C
\end{equation}
locally reads, thanks to \eqref{canonicrelation},
\begin{equation}\label{eqgenerator}
    \begin{split}
    \tilde{\psi}(x,y) = 0 \qquad \iff \qquad &\exists (x,\xi) \in S_x^*\mathcal{S} \ \text{such that} \ P(\Phi_{s_0}^{q_1}(x,\xi)) = y\\
    &\text{and in that case} \qquad \exists \lambda > 0 \ (x,\nabla_x\tilde{\psi}) = (x,\lambda \xi) \\
    & \text{and} \ (y,-\nabla_y \tilde{\psi}) = \Phi_{s_0}^{q_1}(x,\nabla_x\tilde{\psi}).
    \end{split}
\end{equation}

\myindent We call the smooth functions $\tilde{\psi}$ satisfying \eqref{eqgenerator} \textit{generators}, since each of them yields locally a phase function adapted to the canonical relation $C$. Now, our claim is that in fact, the second and third line of \eqref{eqgenerator} are redundant. Indeed, we claim that \textit{any} smooth $\tilde{\psi}$ such that on the one hand
\begin{equation}\label{zeroset}
    \tilde{\psi}(x,y) = 0 \qquad \iff \qquad \exists (x,\xi) \in S_x^*\mathcal{S} \ \text{such that} \ P(\Phi_{s_0}^{q_1}(x,\xi)) = y,
\end{equation}
and, on the other hand, such that for any $(x,y)$ such that $\tilde{\psi}(x,y) = 0$ then
\begin{equation}\label{nonzerograd}
    \nabla_{x,y}\tilde{\psi} \neq 0
\end{equation}
satisfies 
\begin{equation}
    \begin{split}
        &\exists \lambda > 0, \ (x,\nabla_x\tilde{\psi}) = (x,\lambda \xi) \\
    & \text{and} \ (y,-\nabla_y \tilde{\psi}) = \Phi_{s_0}^{q_1}(x,\nabla_x\tilde{\psi})
    \end{split}
\end{equation}
whenever $\tilde{\psi}(x,y) = 0$. Indeed, this equation is only an equation on the \textit{direction} of the gradient of $\tilde{\psi}$. Now, the zero set of $\tilde{\psi}$ is locally a three-dimensional submanifold of $\mathcal{S}\times \mathcal{S}$. Since the gradient of $\tilde{\psi}$ is necessarily orthogonal to its zero set, its direction is uniquely determined. Hence, since we know that there \textit{exists} at least one generator $\tilde{\psi}$, we deduce that for \textit{all} generator, the gradient $\nabla_{x,y}\tilde{\psi}$ has the same direction on the zero set (the sign uncertainty can be fixed since we assume that the gradient doesn't vanish on the zero set, which is a connected submanifold).

\myindent Hence, any $\tilde{\psi}$ satisfying the zero set equation \eqref{zeroset} and the nondegeneracy condition \eqref{nonzerograd} is a generator for an appropriate phase function \textit{up to a sign}.

\myindent Thanks to this liberty on the choice of a generator, we claim that there is a natural choice, so that we ultimately have a simple form for the phase function $\phi$. Indeed, if we want it to be similar to Hörmander's phase, i.e. of the form
\begin{equation}
    \text{a linear term in s} \qquad + \qquad \text{a remainder independent of s},
\end{equation}
then, looking at the eikonal equation \eqref{eikeqtimedependent}, the natural condition to impose is that
\begin{equation}
    \nabla_x \left(q_1(x,\nabla_x\tilde{\psi}(x,y))\right) = 0 \qquad \forall x,y,
\end{equation}
and we may further reduce to the easier equation
\begin{equation}\label{reducedeikonal}
    q_1(x,\nabla_x\tilde{\psi}(x,y)) = 1.
\end{equation}

\myindent Indeed, this guarantees that
\begin{equation}
    \phi(s,x,y,r) := r ((s-s_0) + \tilde{\psi}(x,y))
\end{equation}
is an adapted phase function, which is obviously quite nice. Now, we claim that there is locally up to a sign  a \textit{unique} generator $\tilde{\psi}$ satisfying \eqref{reducedeikonal}.

\myindent Indeed, we can apply the Hamilton-Jacobi theory, for example \cite{duistermaat1996fourierb}[Theorem 3.6.3], which can also be approached in a more constructive way since we are in a very explicit case.

\myindent Fix $y$ close to $y_0$. Then,
\begin{equation}
    Q := \{x \in \mathcal{S} \ \text{such that} \ \psi(x,y) = s_0\},
\end{equation}
where $\psi$ is the bicharacteristic length, is in a neighborhood of $x_0$ a submanifold of dimension 1. In particular, for each $x \in Q$, we can define smoothly a unit normal to $Q$ at $x$, namely a smooth
\begin{equation}
        (x,\xi_x) \in Q^{\perp} \qquad \text{such that}\qquad q(x,\xi_x) = 1.
\end{equation}

\myindent Now, consider 
\begin{equation}
    \Lambda := \{\Phi_u^{q_1}(x,\xi_x) \qquad x\in Q, \ |u| \ll 1\},
\end{equation}

which can be visualized as follows : locally, $Q$ is the (projection of the) \textit{wavefront set} at time $s_0$ of the bicharacteristics starting at $x_0$, and we build $\Lambda$ as the union, locally, of the trajectories, see Figure \ref{constructLambda}.
\begin{figure}[h]
\includegraphics[scale = 0.5]{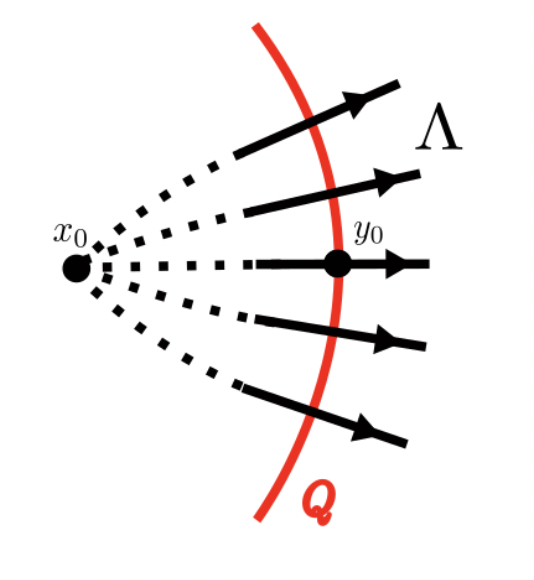}
\centering
\caption{The construction of $\Lambda$}
\label{constructLambda}
\end{figure}

\myindent Then, locally, this is a lagrangian submanifold of $T^*\mathcal{S}$ (see \cite{duistermaat1996fourierb}[Theorem 3.6.2]), and hence, usual theory of Lagrangian submanifolds yields that there exists a function such that 
\begin{equation}
    \Lambda = \{(x,d_x\tilde{\psi})\}
\end{equation}
locally around $x_0$ and moreover such that $\tilde{\psi}$ vanishes on $Q$. Now, it is straightforward to compute $\tilde{\psi}$. Indeed, fix $x \in Q$. We find that
\begin{equation}\label{incrbichleng}
    \begin{split}
    \frac{d}{du} \left(\tilde{\psi}\left(P\left(\Phi_u^{q_1}(x,\xi_x)\right)\right)\right)
    &= d_x \tilde{\psi}\left(P\left(\Phi_u^{q_1}(x,\xi_x)\right)\right) \cdot dP\left(\frac{d}{du} \Phi_u^{q_1}(x,\xi_x) \right) \\
    &= d_x \tilde{\psi}\left(P\left(\Phi_u^{q_1}(x,\xi_x)\right)\right) \cdot \left(d_{\xi} q_1\left(\Phi_u^{q_1}(x,\xi_x)\right)\right) \\
    &= d_x \tilde{\psi}\left(P\left(\Phi_u^{q_1}(x,\xi_x)\right)\right) \cdot \left(d_{\xi} q_1\left(P(\Phi_u^{q_1}(x,\xi_x)),d_x \tilde{\psi}\left(P\left(\Phi_u^{q_1}(x,\xi_x)\right)\right) \right)\right) \\
    &= q_1\left(d_x \tilde{\psi}\left(P\left(\Phi_u^{q_1}(x,\xi_x)\right)\right) \right) \\
    &= 1.
    \end{split}
\end{equation}

\myindent In other words, $\tilde{\psi}$ exactly measures, up to a constant, the increment in the bicharacteristic length, i.e. ultimately (since it vanishes on $Q$)
\begin{equation}
    \tilde{\psi}(x) = \psi(x,y) - s_0.
\end{equation}

\myindent In conclusion,
\begin{equation}
    \tilde{\psi}(x,y) := \pm(\psi(x,y) - s_0)
\end{equation}
appears naturally as the unique choice of generator satisfying \eqref{reducedeikonal}. There is still an uncertainty on the sign, but it can now be seen to be a $-$ sign, since we have an explicit candidate for the phase.
\end{proof}

\myindent Observe that, in the construction, the only important fact is that $\psi(x,y)$ measures locally the increment in the bicharacteristic length, in the sense of \eqref{incrbichleng}. In particular, the construction extends to give local phase functions for parametrices of groups of unitary operators $t\mapsto e^{itQ}$, at least near points $(t,x,y)$ such that the dimension of the intersection of the canonical relation \eqref{canonicrelation} and of the fibers of the cotangent bundle is locally at most one.

\myindent As a consequence of Proposition \ref{bicharactparamprop}, in the case i., Corollary \eqref{defoscintbyphase} applies and we can write locally around $(\sigma_0,s_0,t_0)$ 
\begin{equation}\label{bicharactUsigmast}
\begin{split}
    U(\sigma,s,t) &= \int_{\R_+^*} e^{i\phi_{bl}(\sigma,s,t,r)} f(\sigma,s,t,r) dr + R(\sigma,s,t) \\
    &=: I(\sigma,s,t) + R(\sigma,s,t),
\end{split}
\end{equation}
for some smooth symbol $f$ of order $\frac{1}{2}$ and for some smooth remainder $R$, and where we write
\begin{equation}\label{bicharactparam}
    \phi_{bl}(\sigma,s,t,r) := r(|s| - \psi((t,\sigma,),(0,\sigma)).
\end{equation}

\subsubsection{The resulting classification}\label{subsubsec25Micro}

\myindent As a consequence of Theorem \ref{Hormthm}, Lemma \ref{antipodparam} and Proposition \ref{bicharactparamprop}, and of the compactness of $\left[-\frac{L}{2} + \eps, \frac{L}{2} - \eps\right]\times S^1 \times S^1$, we now have a useful classification of the different expressions for \eqref{defUsigmast}. Before giving the proposition, we introduce some notation.

\begin{notation}
    Let $d \geq 1$. Let $U$ be a bounded subset of $\R^d$ which is invariant by reflexion with respect to a point $C \in U$, which we call its center. For any $k > 0$, we define $k U$ the set obtained from $U$ by a homothety of center $C$ and ratio $k$. For example, if $d= 1$ and $U$ is an interval $[C - a, C+a]$, then $k U = [C- k a, C + k a]$.
\end{notation}

\begin{proposition}\label{classification}
    There exists a covering $\mathfrak{Q}$ of $\mathcal{K} := \left[-\frac{L}{2} + \eps, \frac{L}{2} - \eps\right]\times S^1 \times S^1$ by open rectangular cuboids $\mathcal{Q}$ of the following form.

    \quad

    \myindent First, find a covering of $\mathcal{L} := \left[-\frac{L}{2} + \eps, \frac{L}{2} - \eps\right]$ by small enough open intervals $\mathcal{I}_i, i = 0,...,I$, such that $\mathcal{I}_0$ is centered on $\sigma = 0$ and $0 \notin 2\overline{\mathcal{I}_i}$ for $i \geq 1$. 

    \quad

    \myindent Second, for $i = 0,...,I$, find an open covering of $S^1 \times S^1$ by small enough open rectangles $\mathcal{R}_i^{\bullet}$, of the following form.
    
    \myindent i. Find $\mathcal{R}_i^{(H)}$ a rectangle centered on $(s,t) = 0$, which is small enough so that, if $\mathcal{Q}_i^{(H)} := \mathcal{I}_i \times \mathcal{R}_i^{(H)}$, then the Hörmander parametrix is defined on $2\mathcal{Q}_i^{(H)}$ i.e. for $(\sigma,s,t) \in 2\mathcal{Q}_i^{(H)}$, $U(\sigma,s,t)$ can be written in the form \eqref{HormUsigmast} with the phase \eqref{hormparam}.
    
    \myindent ii. Find $\mathcal{R}_0^{(\pi)}$ a rectangle centered on $(s,t) = (\pi,\pi)$, which is small enough so that, if $\mathcal{Q}_0^{(\pi)} := \mathcal{I}_0 \times \mathcal{R}_0^{(\pi)}$, then the antipodal Hörmander parametrix is defined on $2\mathcal{Q}_0^{(\pi)}$ i.e. for $(\sigma,s,t) \in 2\mathcal{Q}_i^{(\pi)}$, $U(\sigma,s,t)$ can be written in the form \eqref{antipodUsigmast} with the phase \eqref{antipodhormparam}.
    
    \myindent iii. Find $\mathcal{R}_i^{(j)}$, $j = 1,...,J_i$ a family of rectangles which form a covering of the curve $$\{(s,t), \ |s| = \psi((t,\sigma),(0,\sigma))\}$$ from which we remove $(0,0)$ (and $(\pi,\pi)$ in the case $i = 0$), which are small enough so that, if $\mathcal{Q}_i^{(j)} := \mathcal{I}_i \times \mathcal{R}_i^{(j)}$, then the bicharacteristic length parametrix is defined on $2\mathcal{Q}_i^{(j)}$ i.e. for $(\sigma,s,t) \in 2\mathcal{Q}_i^{(j)}$, $U(\sigma,s,t)$ can be written in the form \eqref{bicharactUsigmast} with the phase \eqref{bicharactparam}.
    
    \myindent iv. Finally, find $\mathcal{R}_i^{(\infty,j)}$, $j = 1,...,J^{(\infty)}_i$ a family of rectangles such that $U(\sigma,s,t) =: R(\sigma,s,t)$ is smooth and uniformly bounded on $2\mathcal{Q}_i^{(\infty,j)} := 2\mathcal{I}_i \times 2\mathcal{R}_i^{(\infty,j)}$.
    
    \quad 
    
    \myindent Third, the sets $\mathcal{R}_i^{\bullet}$ are chosen with the following compatibility relations : for all $i$, and for all $j = 1,...,J_i$, $(0,0) \notin 2 \overline{\mathcal{R}_i^{(j)}}$. Moreover, for $i = 0$, $(\pi,\pi) \notin \mathcal{R}_0^{(H)}$ and for $j = 1,...,J_0$, $(\pi,\pi) \notin 2 \overline{\mathcal{R}_0^{(j)}}$.
\end{proposition}

\myindent Visually, the meaning of this proposition is simply to partition, for each $\sigma$, the torus $S^1\times S^1$ into well-chosen rectangles in order to discriminate the different types of parametrices needed. Namely, for $\sigma \neq 0$, we know that there are three possible behaviour : near $(s,t) = 0$ we need the Hörmander parametrix; near the curve $\{|s| = \psi((t,\sigma),(0,\sigma))\}$ we need the bicharacteristic length parametrix; and, finally, $U(\sigma,\cdot,\cdot)$ is smooth outside of these zones. This is depicted on Figure \ref{visualpartition}.

\begin{figure}[h]
\includegraphics[scale = 0.5]{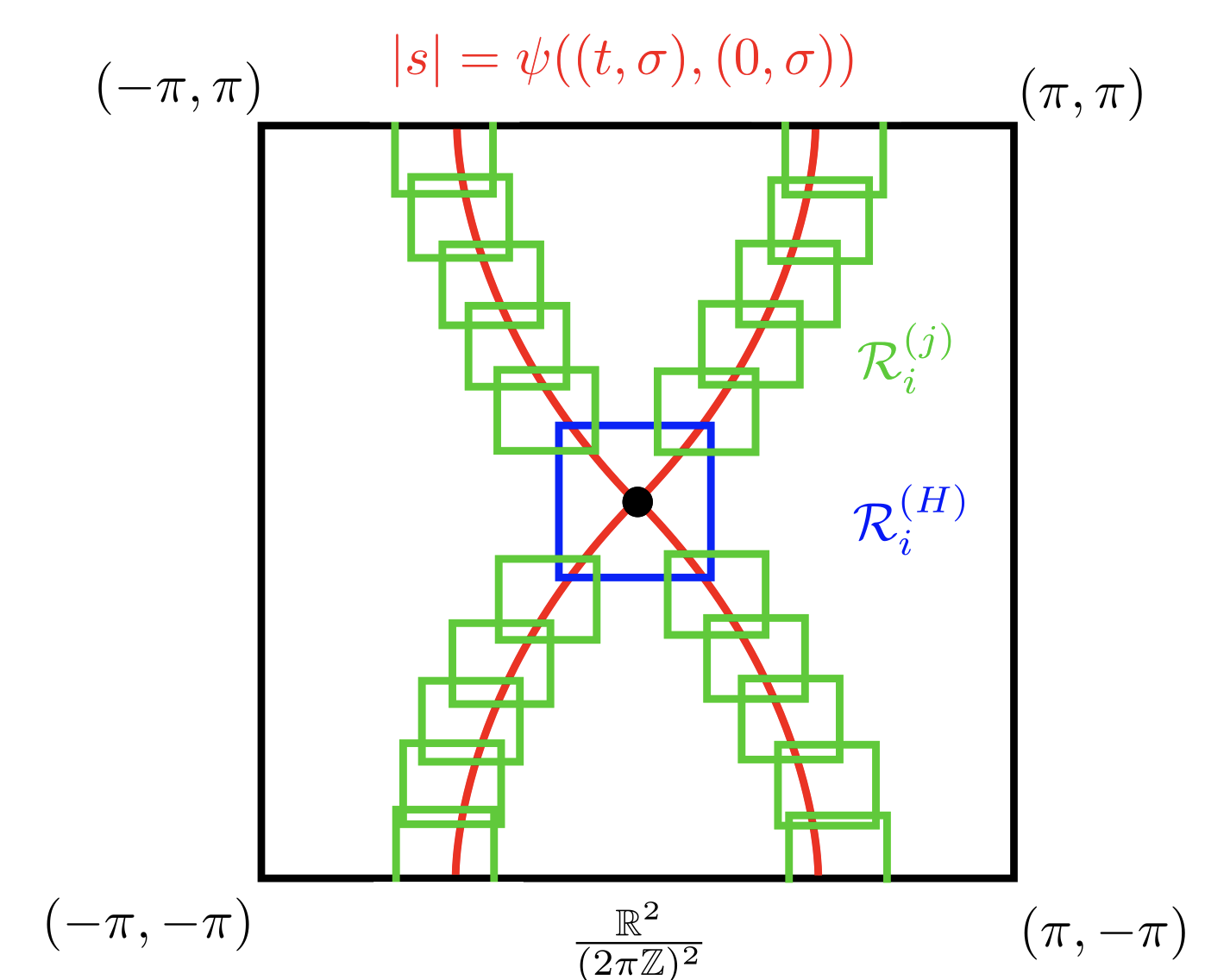}
\centering
\caption{The partition away from the equator}
\label{visualpartition}
\end{figure}

\myindent In the specific case $\sigma = 0$ (and hence nearby), we know that one more behaviour has to be taken into account due to the antipodal refocalisation. Hence, in that case, we need to add a rectangle centered around $(s,t) = (\pi,\pi)$ on which we need the antipodal Hörmander parametrix. This is depicted on Figure \ref{visualpartitionI0}.

\begin{figure}[h]
\includegraphics[scale = 0.5]{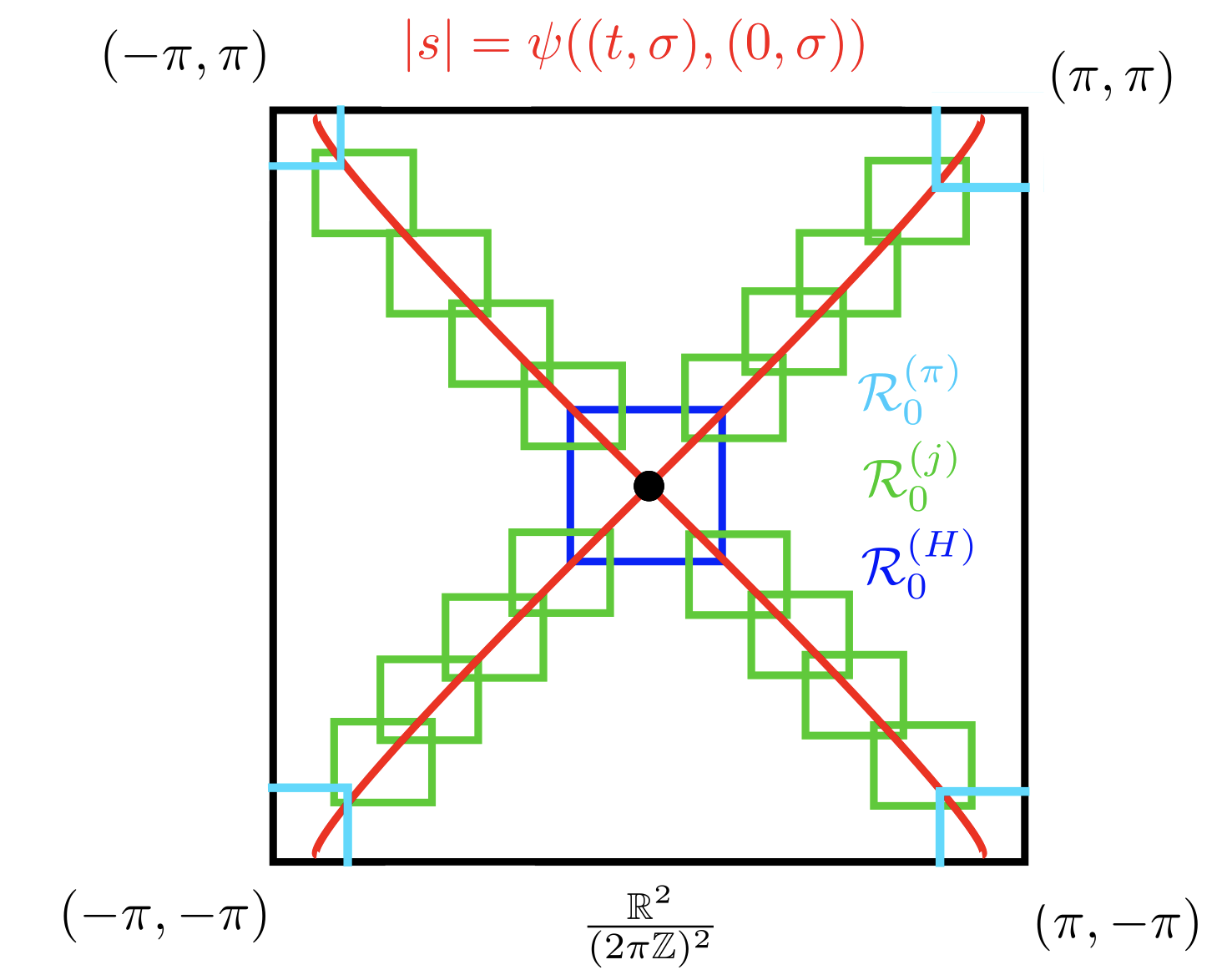}
\centering
\caption{The partition near the equator}
\label{visualpartitionI0}
\end{figure}

\quad

\myindent Now, we can obviously uplift this covering to an adapted periodic covering of $\R^2$. 

\begin{deflemm}
    Let $\pi : \R^2 \to \mathbb{T}^2$ be the canonical projection. Let $i \in \{0,...,I\}$ and $\bullet = (H)$, or $\bullet = (1),...,(J_i)$, or $\bullet = (\infty,1),...,(\infty,J^{(\infty)}_i)$ or $\bullet = (\pi)$ in the case $i = 0$. Then, $\pi^{-1}(\mathcal{R}_i^{\bullet})$ is a disjoint reunion of isometric open rectangles of $\R^2$, say $\tilde{\mathcal{R}}_{i,A,B}^{\bullet}$ which we can index uniquely by $(A,B) \in (2\pi\Z)^2$ such that

    \myindent i. $\tilde{\mathcal{R}}_{i,0,0}^{(H)}$ is centered at $(s,t) = (0,0)$.

    \myindent ii. More generally, there exists a unique $(s_i^{\bullet},t_i^{\bullet}) \in ]-\pi,\pi]^2$ and a rectangle $\tilde{R}_i^{\bullet}$ centered at $(s,t) = (0,0)$ such that
    \begin{equation}
        \tilde{\mathcal{R}}_{i,A,B}^{\bullet} = \tilde{\mathcal{R}}_i^{\bullet} + (s_i^{\bullet} + A, t_i^{\bullet} + B).
    \end{equation}
    
    \myindent We write 
    \begin{equation}
        \tilde{\mathcal{R}}_i^{\bullet} = \mathcal{J}_i^{\bullet} \times \mathcal{K}_i^{\bullet}
    \end{equation}
    for some intervals $\mathcal{J}_i^{\bullet},\mathcal{K}_i^{\bullet} \subset S^1$ centered at $0$.
\end{deflemm}

\myindent We can thus introduce a partition of unity adapted to this covering.

\begin{deflemm}
    For all $i \in \{0,...,I\}$ there exists a smooth partition of unity, say $\tilde{\chi}_{i,A,B}^{\bullet}, \ (A,B)\in (2\pi\Z)^2$, which is

    \myindent i. Adapted to the open covering of $\R^2$ by the rectangles $\tilde{\mathcal{R}}_{i,A,B}^{\bullet}$ i.e.
    \begin{equation}
        \tilde{\chi}_{i,A,B}^{\bullet} = \begin{cases}
            1 \qquad (s,t) \in \tilde{\mathcal{R}}_{i,A,B}^{\bullet} \\
            0 \qquad (s,t) \notin 2\tilde{\mathcal{R}}_{i,A,B}^{\bullet}
        \end{cases}.
    \end{equation}

    \myindent ii. Periodic, in the sense that, there exists a smooth nonnegative $\tilde{\chi}_i^{\bullet}$ which is supported in $2\tilde{\mathcal{R}}_i^{\bullet}$ and equals $1$ on $\tilde{\mathcal{R}}_i^{\bullet}$ such that
    \begin{equation}
        \tilde{\chi}_{i,A,B}^{\bullet}(s,t) = \chi_i^{\bullet}(s - s_i^{\bullet} - A, t - t_i^{\bullet} - B).
    \end{equation}
\end{deflemm}

\quad

\myindent Now, coming back to Lemma \ref{firstintegralexpression}, the interest is that we can write, for all $x = (\theta,\sigma) \in K_{\eps}$ such that $\sigma \in \mathcal{I}_i$,
\begin{equation}
    \begin{split}
        &\int_{\R^2} \hat{d\mu}(\lambda(s,t)) \hat{\rho}(\delta(s,t)) \left(e^{isQ_1} e^{itQ_2}\right)(x,x) ds dt \\
        &= \int_{\R^2} \hat{d\mu}(\lambda(s,t)) \hat{\rho}(\delta(s,t)) U(\sigma,s,t) ds dt \\
        &= \sum_{\bullet} \sum_{(A,B)\in (2\pi\Z)^2} \int \tilde{\chi}_{i,A,B}^{\bullet}(s,t) \hat{d\mu}(\lambda(s,t)) \hat{\rho}(\delta(s,t)) U(\sigma,s,t) ds dt \\
        &= \sum_{\bullet} \sum_{(A,B)\in(2\pi\Z)^2} (-1)^{\frac{A}{2\pi}} \tilde{\mathcal{I}}_{\lambda,\delta,i}^{\bullet}(\sigma,A,B),
    \end{split}
\end{equation}
where the sum $\sum_{\bullet}$ means that we sum on $\bullet = (H)$, $\bullet = (1),...,(J_i)$,$\bullet = (\infty,1),...,(\infty,J_i^{(\infty)})$, and, in the case where $i = 0$, additionally on $\bullet = (\pi)$, and where we define
\begin{multline}
    \tilde{\mathcal{I}}^{\bullet}_{\lambda,\delta,i}(\sigma,A,B) = \\
    \int \chi_i^{\bullet}(s,t) \hat{d\mu}\left(\lambda(s + s_i^{\bullet} + A, t + t_i^{\bullet} + B)\right) \hat{\rho}\left(\delta(s + s_i^{\bullet} + A, t + t_i^{\bullet} + B\right) U(\sigma, s + s_i^{\bullet}, t + t_i^{\bullet}) ds dt.
\end{multline}

\myindent Now, before giving the bounds and their dependence on $i,\bullet,A,B$, we first deal with the smooth remainders. Indeed, for any choice of $i,\bullet$, we know, thanks to the classification Proposition \ref{classification}, that the kernel $U$ can be decomposed on the support of $\chi_i^{\bullet}$ as 
\begin{equation}
    U(\sigma,s + s_i^{\bullet}, t+t_i^{\bullet}) = I(\sigma,s + s_i^{\bullet}, t + t_i^{\bullet}) + R(\sigma, s + s_i^{\bullet}, t + t_i^{\bullet}),
\end{equation}
where $I$ is an explicit oscillatory integral (depending on $i,\bullet$), and $R$ is a smoothing remainder, depending on $i,\bullet$. In the case where $\bullet = (\infty,j)$, we extend this decomposition by setting $I = 0$. Now, since there is ultimately a \textit{finite} number of choices for $i$ and for $\bullet$, the remainders are \textit{all} uniformly bounded in $L^{\infty}$, say by some constant $K > 0$. Thus, we can directly estimate the contribution to the integral of the smoothing remainders.
\begin{equation}
    \begin{split}
        &\sum_{\bullet}\sum_{(A,B) \in (2\pi \Z)^2} \left| \int \chi_i^{\bullet}(s,t) \hat{d\mu}\left(\lambda(s + s_i^{\bullet} + A, t + t_i^{\bullet} + B)\right) \hat{\rho}\left(\delta(s + s_i^{\bullet} + A, t + t_i^{\bullet} + B\right) R(\sigma, s + s_i^{\bullet}, t + t_i^{\bullet}) ds dt \right| \\
        &\leq K\sum_{\bullet}\sum_{(A,B) \in (2\pi \Z)^2} \int \chi_i^{\bullet}(s,t) \left|\hat{d\mu}\left(\lambda(s + s_i^{\bullet} + A, t + t_i^{\bullet} + B)\right) \hat{\rho}\left(\delta(s + s_i^{\bullet} + A, t + t_i^{\bullet} + B\right) \right| ds dt \\
        &\lesssim K \int_{\R^2} |\hat{d\mu}(\lambda(s,t))| |\hat{\rho}(\delta(s,t))| ds dt.
    \end{split}
\end{equation}

\myindent Now, thanks to remark \ref{curvaturegamma}, and to \cite{stein1993harmonic}[Theorem 2, Chapter VIII] , we know that, since $\gamma$ has only ordinary points of inflexion, there holds
\begin{equation}\label{decaydmuhat}
    |\hat{d\mu}(\lambda(s,t))| \lesssim \frac{1}{\lambda^{\frac{1}{3}} |(s,t)|^{\frac{1}{3}}}.
\end{equation}

\myindent Hence,
\begin{equation}
\begin{split}
    &\int_{\R^2} |\hat{d\mu}(\lambda(s,t))| |\hat{\rho}(\delta(s,t))| ds dt \\
    &\lesssim \lambda^{-\frac{1}{3}}\int_{\R^2} \frac{1}{|(s,t)|^{\frac{1}{3}}} |\hat{\rho}(\delta(s,t))| ds dt \\
    &\lesssim \lambda^{-\frac{1}{3}} \delta^{-\frac{5}{3}} \int_{\R^2} \frac{1}{|(s,t)|^{\frac{1}{2}}} |\hat{\rho}| ds dt \\
    &\lesssim \lambda^{-\frac{1}{3}} \delta^{-\frac{5}{3}}.
\end{split}
\end{equation}

\myindent In particular, we have proved the following lemma. We insist that it is crucial that there are only finitely many smoothing remainders, i.e. that they are all uniformly bounded. This is ultimately due to the fact that there is no smoothing remainder in \eqref{exp2ipiQieqId}.
\begin{lemma}\label{smoothingcontrib}
    The contribution of the smoothing remainders to the integral upper bound \eqref{firstintegralexpression} is of order $\lambda^{-\frac{1}{3}} \delta^{-\frac{5}{3}}$, in the sense that
    \begin{equation}
    \begin{split}
        \int_{\R^2} \hat{d\mu}(\lambda(s,t)) \hat{\rho}(\delta(s,t)) \left(e^{isQ_1} e^{itQ_2}\right)(x,x) ds dt &=
        \sum_{\bullet} \sum_{(A,B)\in(2\pi\Z)^2} (-1)^{\frac{A}{2\pi}} \mathcal{I}_{\lambda,\delta,i}^{\bullet}(\sigma,A,B) \\
        &+ O_{\eps}\left(\lambda^{-\frac{1}{3}} \delta^{-\frac{5}{3}}\right),
    \end{split}
    \end{equation}
    
    where we denote by $\mathcal{I}_{\lambda,\delta,i}^{\bullet}(\sigma,A,B)$ the oscillatory part of $\tilde{\mathcal{I}}_{\lambda,\delta,i}^{\bullet}(\sigma,A,B)$, i.e.
    \begin{multline}\label{defIlambdadeltaibullet}
        \mathcal{I}_{\lambda,\delta,i}^{\bullet}(\sigma,A,B) :=\\
        \int \chi_i^{\bullet}(s,t) \hat{d\mu}\left(\lambda(s + s_i^{\bullet} + A, t + t_i^{\bullet} + B)\right) \hat{\rho}\left(\delta(s + s_i^{\bullet} + A, t + t_i^{\bullet} + B\right) I(\sigma, s + s_i^{\bullet}, t + t_i^{\bullet}) ds dt.
    \end{multline}
\end{lemma}

\myindent Now, the classification Proposition \ref{classification} will enable us to express each $\mathcal{I}_{\lambda,\delta,i}^{\bullet}(\sigma,A,B)$ as an explicit oscillatory integral. The rest of this article, starting with Section \ref{subsec3Micro}., is entirely devoted to estimating quantitatively each of those pieces, since the analysis depends very much of $i, \bullet, A,B$. Precisely, we will prove the following proposition.

\begin{proposition}\label{intermediatethm}
    Up to refining the partition $\mathfrak{Q}$ given by the classification Proposition \ref{classification}, there exists a constant $M_0 > 0$ such that there holds the following for all $\lambda, \delta$ such that $\delta \gtrsim_{\mathcal{S},\eps,\mathfrak{Q}} \lambda^{-\frac{1}{3}}$, for $i \in \{0,...,I\}$, and for all $\sigma \in \mathcal{I}_i$. We set $M = |(A,B)|$.

    \myindent i. In the case where $\bullet, A,B = (H),0,0$, there holds
    \begin{equation}\label{estIH00}
        \mathcal{I}_{\lambda,\delta,i}^{(H)}(\sigma,0,0) = 2\hat{\rho}(0,0)c_W(\sigma) + O_{\mathcal{S},\eps,\mathfrak{Q}}(\lambda^{-1}),
    \end{equation}
    where $c_W(\sigma) > 0$ is the constant of the pointwise Weyl law,i.e.
    \begin{equation}
        c_W(\sigma) = \int_{p_1(\sigma,\xi) \leq 1} d\xi.
    \end{equation}

    \myindent ii. In the case where $\bullet = (H)$ and $(A,B) \neq (0,0)$, there holds
    \begin{equation}\label{estIHAB}
         \mathcal{I}_{\lambda,\delta,i}^{(H)}(\sigma,A,B) = O\left(\lambda^{-\frac{1}{4}} M^{5} + \lambda^{-\frac{1}{3}} M^7 + \lambda^{-\frac{1}{2}}M^{14}\right),
    \end{equation}
    which can be refined, in the case $i \neq 0$, to
    \begin{equation}\label{estIHABi0}
         \mathcal{I}_{\lambda,\delta,i}^{(H)}(\sigma,A,B) = O\left(\lambda^{-\frac{1}{4}} M^5 + \lambda^{-\frac{1}{3}}M^4 + \lambda^{-1} M^9\right).
    \end{equation}
    
    \myindent iii. In the case where $M\leq M_0$, and $\bullet = (1),...,(J_i)$,
    or $i = 0$ and $\bullet = (\pi)$ there holds
    \begin{equation}\label{IlowM}
         \mathcal{I}_{\lambda,\delta,i}^{(\bullet)}(\sigma,A,B) = O_{M_0,\mathcal{S},\eps,\mathfrak{Q}}(1).
    \end{equation}
    
    \myindent iv. In the case where $M\geq M_0$, and $\bullet = (1,...,J_i)$, there holds
    \begin{equation}\label{IBicharactlargeM}
        \mathcal{I}_{\lambda,\delta,i}^{(j)}(\sigma,A,B) = O_{M_0,\mathcal{S},\eps,\mathfrak{Q}}\left(\lambda^{-\frac{1}{3}} M^3 + \lambda^{-\frac{1}{2}} M^6\right),
    \end{equation}
    which can be refined, in the case $i \neq 0$, to
    \begin{equation}
         \mathcal{I}_{\lambda,\delta,i}^{(j)}(\sigma,A,B) = O\left(\lambda^{-\frac{1}{2}}M^{-\frac{1}{2}} +\lambda^{-\frac{3}{2}} \right).
    \end{equation} 
    
    \myindent v. In the case where $M\geq M_0$, $i= 0$, and $\bullet = (\pi)$, there holds
    \begin{equation}\label{IAntipodlargeM}
         \mathcal{I}_{\lambda,\delta,i}^{(\pi)}(\sigma,A,B) = O\left(\lambda^{-\frac{1}{4}} M^5 + \lambda^{-\frac{1}{3}} M^6 + \lambda^{-\frac{1}{2}}M^{14}\right).
    \end{equation}
\end{proposition}

\myindent We conclude this section by explaining how to conclude the proof of Theorem \ref{mainthm} with Proposition \ref{intermediatethm}. Choose $\rho$ such that $\hat{\rho}$ is compactly supported. Then, there holds, for $|(A,B)| \gtrsim \delta^{-1}$,
\begin{equation}
    \mathcal{I}_{\lambda,\delta,i}^{\bullet}(\sigma,A,B) = 0.
\end{equation}

\myindent In particular, keeping only the worst upper bounds in Proposition \ref{intermediatethm}, there holds
\begin{equation}\label{resummation}
    \begin{split}
        \sum_{\bullet} \sum_{(A,B) \in (2\pi \Z)^2} \left| \mathcal{I}_{\lambda,\delta,i} (\sigma,A,B)\right| &= \sum_{|(A,B)| \leq M_0} 0_{M_0,\mathcal{S},\eps,\mathfrak{Q}}(1)\\
        &+ \sum_{M_0 \leq |(A,B)| \lesssim \delta^{-1}} 0\left(\lambda^{-\frac{1}{4}} M^5 + \lambda^{-\frac{1}{3}} M^7 + \lambda^{-\frac{1}{2}} M^{14} + \lambda^{-1} M^6 \right) \\
        &= O_{M_0, \mathcal{S},\eps,\mathfrak{Q}}(1) + 0\left(\lambda^{-\frac{1}{4}} \delta^{-7} + \lambda^{-\frac{1}{3}} \delta^{-9} + \lambda^{-\frac{1}{2}} \delta^{-16} + \lambda^{-1} \delta^{-8}\right).
    \end{split}
\end{equation}

\myindent In particular, this identity, along with Lemma \ref{smoothingcontrib} and Lemma \ref{boundbyintegral}, yields that
\begin{equation}
    \left\|P_{\lambda,\delta} \right\|_{L^2(\mathcal{S}) \to L^{\infty}(K_{\eps})} \lesssim O_{M_0, \mathcal{S},\eps,\mathfrak{Q}}(1) + 0\left(\lambda^{-\frac{1}{4}} \delta^{-7} + \lambda^{-\frac{1}{3}} \delta^{-9} + \lambda^{-\frac{1}{2}} \delta^{-16} + \lambda^{-1} \delta^{-8}\right).
\end{equation}

\myindent In order to obtain, finally, Theorem \ref{mainthm}, one needs only observe that, for any $\tau,K > 0$, 
\begin{equation}
    \lambda^{-\tau} \delta^{-K} \lesssim 1  \iff \delta \gtrsim \lambda^{-\frac{\tau}{K}}.
\end{equation}

\myindent To conclude, let us insist on the fact that the constant $\frac{1}{32}$ in Theorem \ref{mainthm} is simply the worst term in the upper bound. Hence, any quantitative improvement in Proposition \ref{intermediatethm} instantly yields a quantitative improvement of that constant. We will come back to this in Section \ref{subsec2Further}.

\begin{remark}
    In the following, all the implicit constants may depend on $\mathcal{S}, \eps, \mathfrak{Q}$. Hence, we will not explicitly write this dependence, unless necessary.
\end{remark}

\subsection{The geometry of the Oscillatory Integral Analysis}\label{subsec3Micro}

\myindent This section is a (very) detailed introduction to the quantitative analysis of Sections \ref{secHormPhase},\ref{secHormQuant},\ref{secBicharact},\ref{secAntipod}. We will explain the \textit{idea} behind the proof of the estimates in Proposition \ref{intermediatethm}, i.e. how to view them as estimates on Oscillatory Integrals, and what are the interesting geometric features of those integrals. Indeed, there is an unavoidable technicity in the following sections, and one could get lost in the (long) computations without a general geometric vision of the problem.

\subsubsection{Reformulation of the problem as an estimate of Oscillatory Integrals}\label{subsubsec31Micro}

\myindent Thanks to the classification Proposition \ref{classification}, we have reduced the problem to bounding, for some $(s_0,t_0) \in ]-\pi,\pi]^2$, and $(A,B) \in (2\pi\Z)^2$, an integral of the form
\begin{equation}\label{intexprofprob}
    \mathcal{I}(\lambda,\delta,A,B) := \int \chi(s,t) \hat{d\mu}(\lambda(s + s_0 + A,t + t_0 + B)))\hat{\rho}(\delta(s + s_0 + A, t + t_0 + B)) I(\sigma, s_0 + s, t_0 + t) ds dt,
\end{equation}
where $\chi$ has a very small support. 

\myindent In order to reduce the size of the expressions, we will define up to the end
\begin{equation}\label{defM|AB|}
\begin{split}
    &(\tilde{A}, \tilde{B}) := (s_0 + A, t_0 + B) \\
    &M := |(A,B)|.
\end{split}
\end{equation}

\myindent Now, many terms in the expression \eqref{intexprofprob} can be expressed as oscillatory integrals. Indeed, we know, thanks to the classification Proposition \ref{classification}, that for some smooth nondegenerate phase function, $\phi$ and some smooth symbol $f$,
\begin{equation}\label{oscintexprofkernel}
    I(\sigma,s_0+s,t_0+t) = \int e^{i\phi(\sigma,s,t,\eta)} f(\sigma,s,t,\eta) d\eta,
\end{equation}
where $\eta \in \R^2$ and $f$ is of order $0$ (in the case $s_0 = t_0 = 0)$ or $i = 0$ and $s_0 = t_0 = \pi$) or $\eta \in \R_+^*$ and $f$ is of order $\frac{1}{2}$ (in all other cases). We recall that there is a dependency of $\phi$ on $(s_0,t_0)$ (but not on $(A,B)$), which we do not write since it is fixed in each case.

\quad

\myindent Moreover, it is crucial for the analysis to take into account the oscillations of $\hat{d\mu}$. Now, \eqref{decaydmuhat} gives decay, and one can actually even extract an asymptotic expansion of $\hat{d\mu}$ (see \cite{stein1993harmonic}), but we rather go back to its very definition, which expresses it naturally as an oscillatory integral. Indeed, let us parameterize the curve $\gamma$ defined by Lemma \ref{deflilgamma} by $u \mapsto h(u)$, where, if $\ell$ is the lentgh of $\gamma$, $u \in \R/\ell\Z$ is the curvilinear abscissa along $\gamma$ (i.e. $|h'(u)| = 1$), positively oriented (we turn anti clockwise), such that $h(0) \in \R_+^* (1,0)$ is horizontal. Then, there holds by definition
\begin{equation}\label{oscintexprofhatdmu}
    \hat{d\mu}(\lambda(s + \tilde{A}, t + + \tilde{B})) = \int_{\R / \ell\Z} e^{-i\lambda\scal{h(u)}{(s + \tilde{A}, t+ \tilde{B})}}|\det(h'(u),h(u))|  du.
\end{equation}

\quad

\myindent Replacing the terms in \eqref{intexprofprob} by their expressions as oscillatory integrals \eqref{oscintexprofkernel} and \eqref{oscintexprofhatdmu}, there holds
\begin{equation}
    \mathcal{I}(\lambda,\delta,A,B) = \int e^{i\left(-\lambda\scal{h(u)}{(s + \tilde{A}, t + \tilde{B})} + \phi(\sigma,s,t,\eta) \right)} f(\sigma,s,t,\eta) \chi(s,t) \hat{\rho}(\delta(s + \tilde{A}, t + \tilde{B})) |\det(h'(u),h(u))|  dsdtdud\eta.
\end{equation}

\myindent Now, it is natural to apply the analysis of oscillatory integrals to this expression, in order to find a bound as $\lambda \to \infty$. In order to read it in a more usual setting, it is convenient to change variables
\begin{equation}
    \eta \mapsto \lambda \eta,
\end{equation}
such that, using the homogeneity of $\Phi$ in the variable $\eta$,
\begin{equation}\label{oscintexpr}
    \mathcal{I}(\lambda,\delta,A,B) = \lambda^{N} \int e^{i\lambda \Psi(\sigma,A,B,u,s,t,\eta)}  f(\sigma, s, t, \lambda \eta) \chi(s,t) \hat{\rho}(\delta(s + \tilde{A}, t+\tilde{B})) |\det(h'(u),h(u))|  dsdtdud\eta,
\end{equation}
where $N = 1$ or $N = 2$ is the number of angle variables, and where the phase $\Psi$ is defined by
\begin{equation}\label{abstractdefofPsi}
    \Psi(\sigma,A,B,u,s,t,\eta) := -\scal{h(u)}{(s + s_0 + A, t+ t_0 + B} + \phi(\sigma,s,t,\eta).
\end{equation}

\subsubsection{Strategy of the analysis}\label{subsubsec32Micro}

\myindent Now that we have reduced the problem to bounding the oscillatory integral \eqref{oscintexpr}, we know that the behaviour is governed by the phase $\Psi$ defined by \eqref{abstractdefofPsi}. Hence, our strategy will be to have a very detailed understanding of the phase $\Psi$, depending on its form for each case of the classification Proposition \ref{classification}. In this paragraph, we give the strategy to analyse the phase without proofs, in order to have a guideline for the following more rigorous sections. Recall that we wish to obtain, ultimately, a bound of the form
\begin{equation}\label{formupperboundlocal}
    |\mathcal{I}(\lambda,\delta,A,B)| \lesssim \lambda^{-\tau} M^{K},
\end{equation}
for some constants $\tau, K > 0$, which may depend on $i, \bullet$, at least for $(A,B) \neq (0,0)$.

\quad

\myindent The first ingredient of the analysis is a count on the powers of $\lambda$ : first, in the expression \eqref{oscintexpr}, the symbol $f$ is of order $\frac{2 - N}{2}$, so there loosely holds
\begin{equation}
    f(\sigma,s,t,\lambda\eta) \sim \lambda^{\frac{2 - N}{2}}.
\end{equation}

\myindent Hence, when bounding \eqref{oscintexpr}, because there is an additional power $\lambda^N$ in front of the integral, one already has to account for a power $\lambda^{\frac{2+ N}{2}}$. 

\myindent Now, usually, an oscillatory integral 
\begin{equation}\label{generaloscintlocal}
    \int_{\R^d} e^{i\lambda \Phi(y)} a(y) dy
\end{equation}
is of order $O(\lambda^{-\frac{d}{2}})$ if the phase $\Phi$ is non degenerate, where the implicit constants depends on $a$ and $\Phi$ (see Theorem \ref{hormstatphase}). Counting the dimension in \eqref{oscintexpr}, we find that we have $3+N$ variables of integration. Hence, if the phase was nondegenerate as a function of $(u,s,t,\eta)$, we would finally have a bound on the integral of order
\begin{equation}\label{boundifnondeg}
   O_{\Psi,b}( \lambda^{\frac{2+ N}{2}} \lambda^{-\frac{3 + N}{2}}) = O_{\Psi,b}( \lambda^{-\frac{1}{2}}),
\end{equation}
where the symbol
\begin{equation}\label{defb}
     b:= \lambda^{\frac{2-N}{2}} f(\sigma, s, t, \lambda \eta) \chi(s,t) \hat{\rho}(\delta(s + \tilde{A}, t+\tilde{B})) |\det(h'(u),h(u))|
\end{equation}
depends on $\lambda$ but, thanks to the symbol estimate \eqref{symbolestimate} on $f$, it is bounded along with all its derivatives independently of $\lambda$.

\myindent The point is that \eqref{boundifnondeg} is too good compared to what we want to prove. Hence, there is some space for degeneracy of $\Psi$. For example, if we can find \textit{one} variable so that, if it is fixed, $\Psi$ is nondegenerate seen as a phase function of the $N + 2$ remaining variables, the powers of $\lambda$ would add up to $\mathcal{I}(\lambda,\delta,A,B)$ being \textit{bounded} independently of $\lambda$. Since, obviously, this bound would still depend on $A,B$, this wouldn't allow for the resummation process presented in the end of Paragraph \ref{subsubsec25Micro}. The core of the analysis will be that, although $\Psi$ is unavoidably \textit{not} in general a non degenerate phase function in the full $(u,s,t,\eta)$ variables, it is of a special type of degenerate phase functions with what we call a \textit{finite type degeneracy} (see \cite{stein1993harmonic}). Roughly, we will be able to "isolate" one variable, say $x$, such that, once it is fixed, the $(N+2)$ dimensional phase function in the remaining "free" variables is a nondegenerate stationary phase function. Hence, one can apply the usual stationary phase estimates. Moreover, there are still some oscillations in the "remaining" variable $x$, which are of finite type, in the sense that a derivative of order higher than two, say $p$, doesn't vanish in that direction. Now, the Van der Corput Lemma will thus yield an additional $O(\lambda^{-\frac{1}{p}})$ in the upper bound. Overall, we will thus find a bound of order
\begin{equation}\label{lambdaupperbound}
    O_{\Psi,b}(\lambda^{\frac{2 + N}{2}} \lambda^{-\frac{1}{p} - \frac{2 + N}{2}}) = O_{\Psi,b}(\lambda^{-\frac{1}{p}}).
\end{equation}

\myindent We will developp in Section \ref{subsec4Micro}. a rigorous definition and analysis of this special type of phases, and derive a theorem for precise quantitative upper bounds, see Theorem \ref{mixedVdCABZ}.

\quad

\myindent Now, the second ingredient is to control carefully the dependency of the upper bound \eqref{lambdaupperbound} on $\Psi,b$, or, rather, on $\delta,A,B$. On the one hand, the dependency on $\delta$ will be entirely contained in the fact that, thanks to the term
\begin{equation}
	\hat{\rho}(\delta(s + s_0 +A, t + t_0 + B))
\end{equation}
in $b$ (see \eqref{defb}), then, choosing $\hat{\rho}$ with a compact support, we only sum up to $M \lesssim \delta^{-1}$, as explained in the end of Paragraph \ref{subsubsec25Micro}.

\myindent On the other hand, and this is more delicate, we need to ensure that the implicit constant in \eqref{lambdaupperbound} is at most \textit{polynomial} in $M$ (see \eqref{defM|AB|}), i.e. that we can refine \eqref{lambdaupperbound} into
\begin{equation}
	\mathcal{I}(\lambda,\delta,A,B) = O_{i,\bullet}(\lambda^{-\frac{1}{p}} M^K).
\end{equation}

\myindent Indeed, otherwise, after resummation, we would obviously not be able to obtain a bound of the integral of the form \eqref{resummation}, i.e. where the contribution of $\delta$ is at most a negative power $\delta^{-K}$. Hence, ultimately, we would not be able to reach $\delta$ polynomially small compared to $\lambda$. Now, this is a difficulty in the analysis, since we thus have to deal with countably many phases (depending on $(A,B)$), for which we need to have precise quantitative upper bounds. As we will argue in Section \ref{subsec4Micro}, this actually calls for new estimates, since, in most texts, the dependency of \eqref{lambdaupperbound} is not explicit in terms of $\Psi,b$. 

\myindent Another point of view on that difficulty is that $M= |(A,B)|$ can be as large as $\delta^{-1}$, which we want to choose as large as $\lambda^{\kappa}$ for some $\kappa > 0$. Hence, we are in the situation of an oscillatory integral \eqref{generaloscintlocal}, but where the phase itself can be of order $\lambda^{\kappa}$, hence comparable to the large parameter $\lambda$. Obviously, we thus need to be very careful in tracking the contribution of the derivatives of the phase in the analysis.

\quad

\myindent Finally, the third general ingredient in the analysis is the method for computing the \textit{stationary points} of the phase $\Psi$, which govern the behaviour of the oscillatory integral $\mathcal{I}(\lambda,\delta,A,B)$. Now, a very important structure is that $\Psi$ is the sum of two very different, and competing, parts. Namely, we can write
\begin{equation}\label{decompositionofPsi}
	\Psi(\sigma,A,B,u,s,t,\eta) = \Psi_{G}(u,A,B) + \Psi_0(\sigma,u,s,t,\eta),
\end{equation}
where $\Psi_{G}$ (resp $\Psi_0$), which we call the \textit{global} phase function, (resp the \textit{local} phase function) takes the form
\begin{equation}\label{defPsiGPsi0}
	\begin{split}
		&\Psi_{G}(u,A,B) := -\scal{h(u)}{(A,B)}\\
		&\Psi_0(\sigma,u,s,t,\eta) := -\scal{h(u)}{(s + s_0, t + t_0)} + \phi(\sigma,s,t,\eta).
	\end{split}
\end{equation}

\myindent On the one hand, $\Psi_0$ is \textit{independent} of $(A,B)$, and we think of it as representing the geometry of the problem before the bicharacteristic curves of $q_1$ return to their starting point. In that sense, $\Psi_0$ contains all the "microlocal" part of the problem. On the other hand, $\Psi_G$, which is independent of the "microlocal" variables $(s,t,\eta)$, contains the information that there has already been $\frac{1}{2\pi} A$ wraps around the bicharacteristic flow of $q_1$, and $\frac{1}{2\pi}B$ wraps around the bicharacteristic flow of $q_2$ (i.e. $\frac{1}{2\pi}B$ turns around the vertical axis). Hence, it contains all the information for \textit{long term}, or global, effects, i.e. when the bicharacteristic curves have already returned one, or many times, to their starting point.

\myindent This division of the phase illustrates a large part of the difficulty of the analysis. Indeed, because $(A,B)$ is any point of $(2\pi\Z)^2$, and since we want to be able to choose it as large as $\lambda^{\kappa}$ for some $\kappa > 0$, there is a competition in the analysis between $\Psi_{G}$ and $\Psi_0$, which are somehow independent, and the resulting oscillatory behaviour can be expected to be quite delicate and largely dependent on $(A,B)$. Of course, this reflects the point made in the introduction that the main difficulty to reach polynomial scales on $\delta$ is the long-term behaviour of the geodesic flow on $\mathcal{S}$, especially at times beyond the injectivity radius of $\mathcal{S}$.

\quad

\myindent From an analytical perspective, and, precisely, from the perspective of \textit{computing} the stationary points of $\Psi$, the aforementioned structure of the phase $\Psi$ has the following consequence : observe that the gradient of $\Psi$ (which, again, is the most important object in the analysis), takes the form
\begin{equation}\label{gradientseparation}
	\nabla_{u,s,t,\eta} \Psi = \begin{pmatrix}
	-\scal{h'(u)}{(s + s_0 + A,t + t_0 + B)} \\
	\nabla_{s,t,\eta} \Psi_0 (\sigma,u,s,t,\eta)
	\end{pmatrix}.
\end{equation}

\myindent On the one hand, the set of $(s,t,\eta)$ stationary points of $\Psi$ (i.e. of zeros of $\nabla_{s,t,\eta} \Psi_0$) is thus \textit{independent} of $(A,B)$. More generally, one can guess that any analysis made only on the variables $(s,t,\eta)$ will be independent of $(A,B)$. Hence, in the last paragraph of this section, we focus on the set  $\mathcal{O}_{\sigma}$ of zeros of $\nabla_{s,t,\eta}\Psi_0$, which has naturally a "microlocal" meaning, and we indeed be able to compute it precisely. Moreover, since we have in mind the special type of phase function with a finite type dgeneracy, it is natural to isolate the variable $u$, and, thus, to try and understand $\mathcal{O}_{\sigma}$, as far as possible, as \textit{parameterized} by $u$.

\myindent On the other hand, it harder to compute the set of points where $\partial_u \Psi$ vanishes, say $\mathcal{M}_{\sigma,A,B}$. Indeed, from \eqref{gradientseparation}, we see that this set \textit{depends} on $(A,B)$. In particular, while is it quite simple to prove that $\mathcal{M}_{\sigma,A,B}$ is a smooth surface, it seems difficult to compute its points of intersection with $\mathcal{O}_{\sigma}$ (which are the stationary points of $\Psi$), and the order to which $\mathcal{O}_{\sigma}$ and $\mathcal{M}_{\sigma,A,B}$ are tangential at their intersection points (which obviously governs the total oscillatory behavior). Hence, in the following, in particular in Section \ref{subsec4Micro}, we will present useful tools which make it possible to go around this difficulty and to work in a framework where we do not know exactly where $\partial_u \Psi$ vanishes, but rather we have some lower bounds on the $u$ derivatives of $\Psi$.

\quad

\myindent We conclude this paragraph by mentioning that the stationary points of $\Psi$ in the full $(u,s,t,\eta)$ variables correspond actually exactly to the \textit{periodic geodesics} of $\mathcal{S}$, as analyzed in Remark \ref{nablaPsieq0interpretation}. This is not surprising since, in the usual microlocal approach roughly introduced in Section \ref{subsec3Intro}, the stationary points of the oscillatory integrals are exactly given by the geodesic loops, which are the same than the periodic geodesics since $\mathcal{S}$ is symmetric, see Section \ref{subsec2CdV}. Hence, the condition to have a stationary point of $\Psi$ are actually quite restrictive. However, we won't use that fact much in the following.

\subsubsection{The zero set $\mathcal{O}_{\sigma}$ of $\nabla_{s,t,\eta}\Psi_0$}\label{subsubsec33Micro}

\myindent In this paragraph, we give a semi-rigorous presentation of the structure of the set 
\begin{equation}\label{gendefOsigma}
	\mathcal{O}_{\sigma} := \{(u,s,t,\eta) \ \text{such that} \ \nabla_{s,t,\eta} \Psi_0 (\sigma,u,s,t,\eta) = 0 \},
\end{equation}
and, in particular, of its dependency on $i$ and on $\bullet$. 

\quad

\myindent First, from \eqref{defPsiGPsi0}, observe that
\begin{equation}
    \begin{pmatrix}
        \nabla_{s,t}\Psi_0 \\
        \nabla_{\eta}\Psi_0
    \end{pmatrix} = \begin{pmatrix}
        -h(u) + \begin{pmatrix}
            \partial_s \phi(\sigma,s,t,\eta) \\
            \partial_t \phi(\sigma,s,t,\eta) 
        \end{pmatrix}\\
        \nabla_{\eta}\phi(\sigma,s,t,\eta)
    \end{pmatrix}.
\end{equation}

\myindent Now, we already know the structure of the solution of 
\begin{equation}\label{nablaetaphieq0}
    \nabla_{\eta}\phi(\sigma,s,t,\eta) = 0,
\end{equation}
seen as an equation on $\left(s,t, \frac{\eta}{|\eta|}\right)$ (it is homogeneous of degree zero). Indeed, thanks to Section \ref{subsec2Micro}, we know that this equality holds if and only if there is a bicharacteristic curve of $q_1$ joining $(t + t_0,\sigma)$ to $(0,\sigma)$ of length $s + s_0$, and its direction is $\nabla_x \phi(\sigma,s,t,\eta)$, as in Figure \ref{FigureBicharacteristiccurve}. Hence, we call this equation the \textit{geometric} equation, since it has a geometric interpretation.

\begin{figure}[h]
\includegraphics[scale = 0.4]{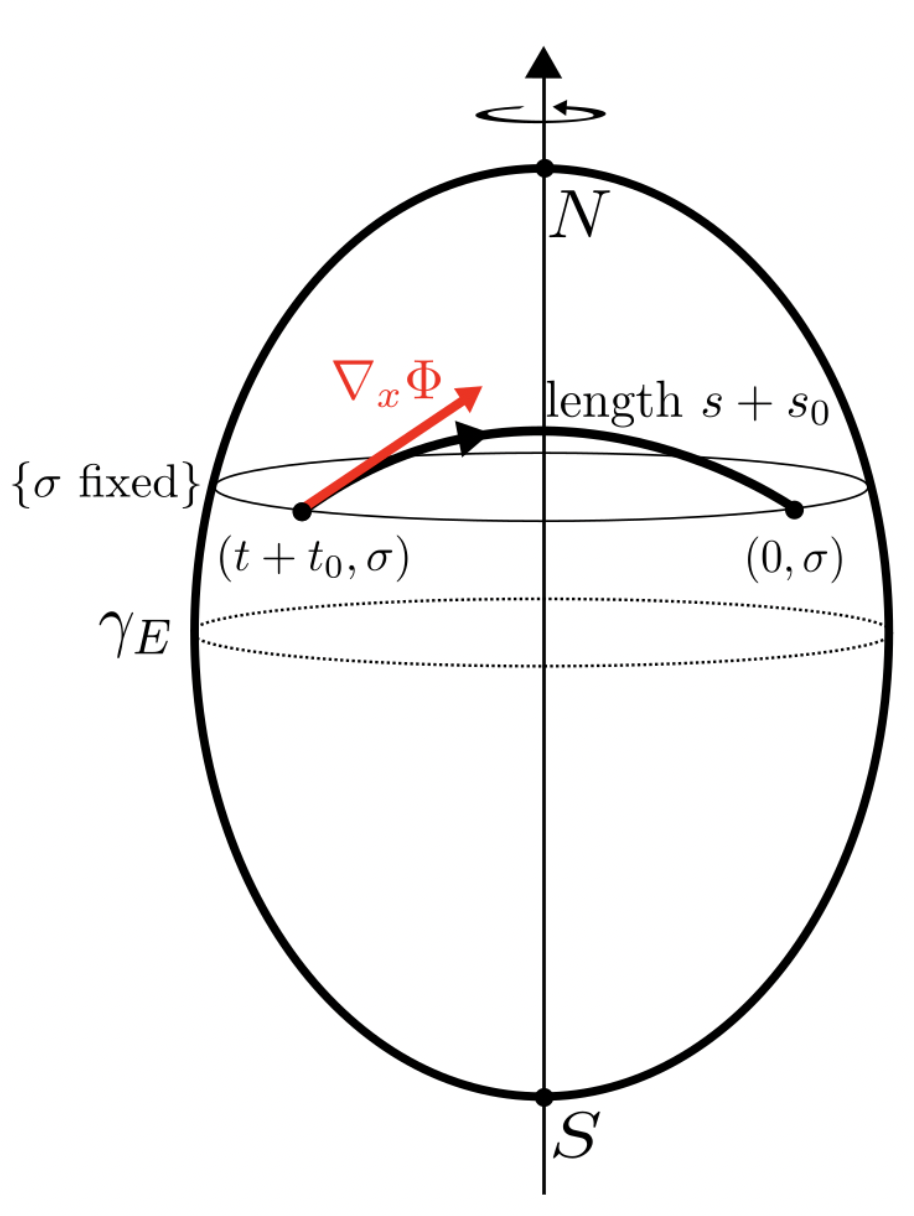}
\centering
\caption{Geometric meaning of the geometric equation}
\label{FigureBicharacteristiccurve}
\end{figure}

\quad

Hence,

\myindent i. In the case $\bullet = (H)$, i.e. when we use the Hörmander parametrix, then $\frac{\eta}{|\eta|} \in S^1$. For $t = 0$, we consider the bicharacteristic curves joining $(0,\sigma)$ to itself of small length $s$. Necessarily, $s = 0$ and any direction works, hence there is a full circle of zeros of $\nabla_{\eta}\Psi_0$. For $t \neq 0$, there is one, and only one, bicharacteristic curve joining $(t,\sigma)$ and $(0,\sigma)$, as seen on Figure \ref{FigureBicharacteristiccurve}. However, since there are two possible orientations, this bicharacteristic curve will give rise to \textit{two} zeros of $\nabla_{\eta}\Psi_0$. In order to visualize the structure of the set of solutions of \eqref{nablaetaphieq0} in that case, we give a representation of a projection of this set in $(t,\xi)$ coordinates, see Figure \ref{txistructurenear0}.

\begin{figure}[h]
\includegraphics[scale = 0.4]{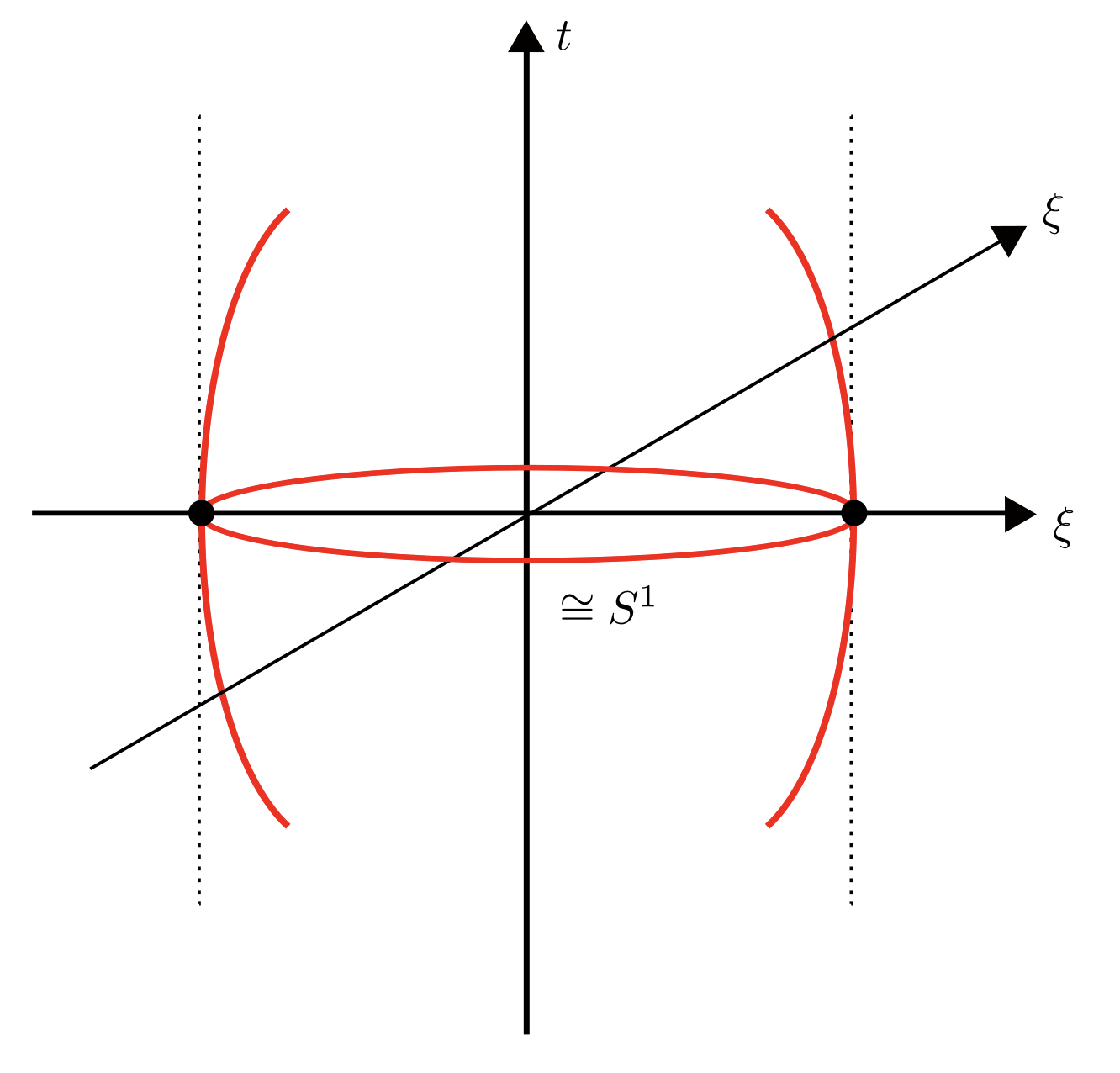}
\centering
\caption{Projection of $\mathcal{O}_{\sigma}$ on $(t,\xi)$ coordinates in the case $
\bullet = (H)$}
\label{txistructurenear0}
\end{figure}

\myindent ii. If $\bullet = (1),...,(J_i)$, i.e. when we can use the bicharacteristic length parametrix, then $\frac{\eta}{|\eta|} =1 $ so the equation \eqref{nablaetaphieq0} is really an equation on $(s,t)$. Now, thanks to Hypothesis \ref{nonintersectHyp}, for all $t$ in $\mathcal{K}_i^{\bullet}$, there is one, and exactly one, bicharacteristic curve joining $\left(t + t_i^{\bullet},\sigma\right)$ to $\left(0,\sigma\right)$. Furthermore, there is only \textit{one} possible solution $s(t)$. Indeed, \eqref{nablaetaphieq0} holds if and only if the bicharacteristic curve joining $\left(t + t_i^{\bullet},\sigma\right)$ to $\left(0,\sigma\right)$ is of length $s_i^{\bullet} + s$ (see Figure \ref{FigureBicharacteristiccurve}). Now, even if there are two possible orientations, the other orientation gives rise to a bicharacteristic curve of length $-s_0 - s$, which is \textit{not} an element in $\mathcal{J}_i^{(j)}$. Hence, the solutions of \eqref{nablaetaphieq0} form a curve in $\tilde{\mathcal{R}}_i^{\bullet}$.

\myindent iii. If $\bullet = (\pi)$, i.e. when we use the antipodal Hörmander parametrix, then, again, $\frac{\eta}{|\eta|}\in S^1$. Now, assume that $\sigma = 0$. Then, from the definition of the antipodal Hörmander parametrix \eqref{defantipodparam}, we find that $\phi$ \textit{equals} the Hörmander parametrix taken at $(\sigma,s,t,\eta) = (0,s,t,\eta)$. In particular, the structure of the solutions of \eqref{nablaetaphieq0} is exactly the same as for the first case i., which reduce, when $\sigma = 0$, to the union of two vertical lines and a circle (figure not included). When $\sigma \neq 0$, the set is "smoothed" into the two curves in Figure \ref{txistructureofOsigmaantipod}, reflecting, as in case ii., the fact that there are exactly two solutions of \eqref{nablaetaphieq0} for each $t$ (even for $t = 0$). In this figure, the blue curve is behind the red curve.

\begin{figure}[h]
\includegraphics[scale = 0.4]{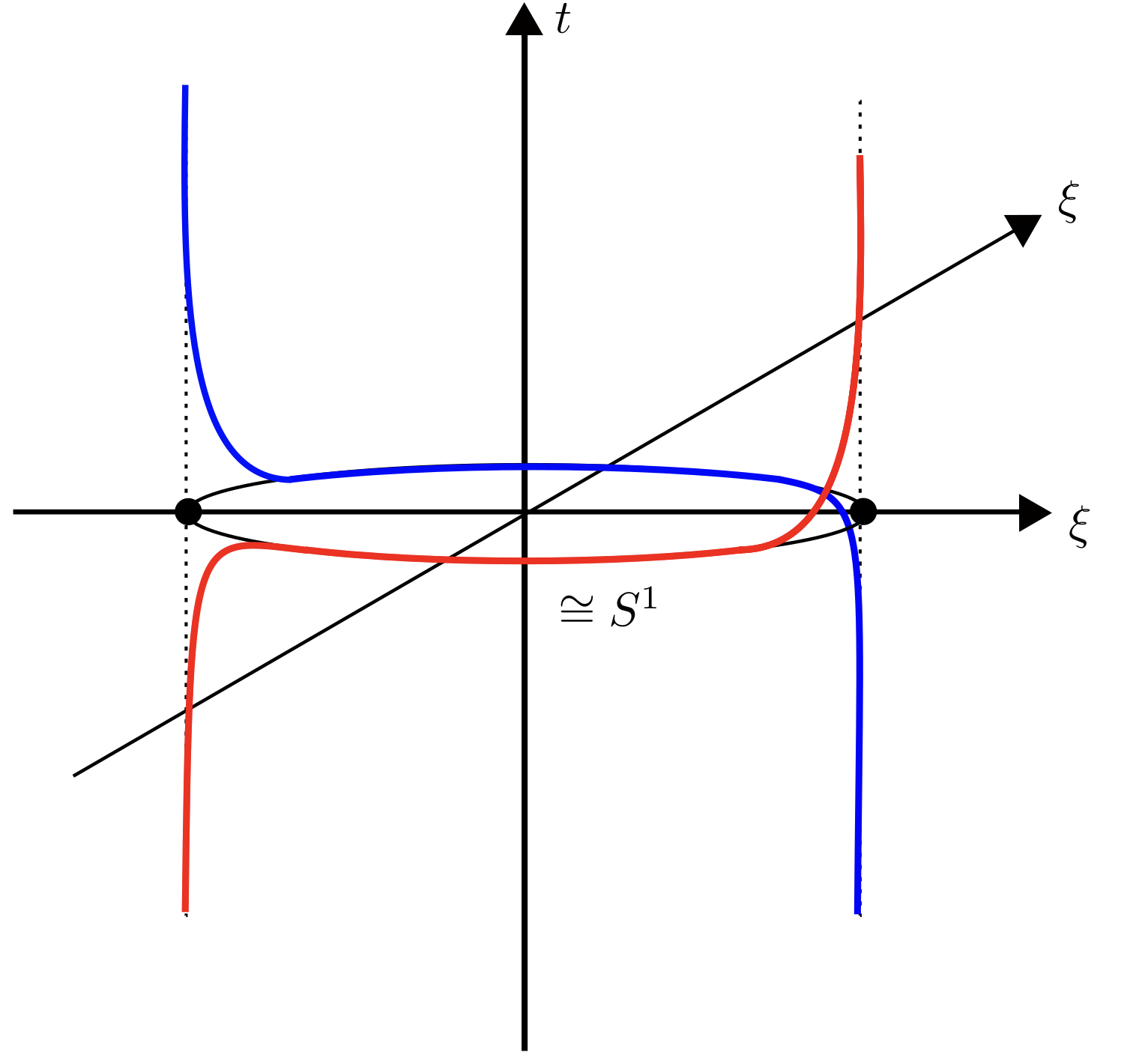}
\centering
\caption{Projection of $\mathcal{O}_{\sigma}$ on $(t,\xi)$ coordinates in the case $
\bullet = (\pi)$}
\label{txistructureofOsigmaantipod}
\end{figure}

\quad

\myindent Now that we have given the structure of the zero set of $\nabla_{\eta}\Phi$, we claim that it fully determines the zero set $\mathcal{O}_{\sigma}$ of $\nabla_{s,t,\eta}\Psi_0$.

\myindent We first observe that the following two-dimensional eikonal equation always holds.
\begin{equation}
	\begin{pmatrix}
		\partial_s \phi(\sigma,s,t,\eta) \\
		\partial_t \phi(\sigma,s,t,\eta)
	\end{pmatrix} = \begin{pmatrix}
		q_1(\sigma,\nabla_x \phi(\sigma,s,t,\eta)) \\
		q_2 (\sigma, \nabla_x \phi(\sigma,s,t,\eta)) 
	\end{pmatrix}.
\end{equation}

\myindent Indeed, for the first coordinate, this is merely the general eikonal equation which is satisfied for \textit{any} adapted phase function for a local parametrix of $e^{isQ_1}$, which follows from \eqref{whyeikonaleqholds} and \cite{sogge2017fourier}[Theorem 3.2.3]. For the second coordinate, this is merely a question of notations. Recall that $\phi(\sigma,s,t,\eta)$ is an abuse of notation for
\begin{equation}
	\Phi : (\sigma,s,t,\eta) \mapsto \phi(s,(t,\sigma),(0,\sigma),\eta),
\end{equation}
where $\phi(s,x,y,\eta)$ is the phase function which is adapted locally to the canonical relation of $e^{isQ_1}$ (see Section \ref{subsec2Micro}). Now, there holds
\begin{equation}
	\partial_t \Phi(\sigma,s,t,\eta) = \partial_{\theta} \phi (s,(t,\sigma),(0,\sigma),\eta) = q_2(\sigma, \nabla_x \phi(s,(t,\sigma),(0,\sigma),\eta)),
\end{equation}
where the last equality follows from the definition of $q_2$, see \eqref{defp2}. 

\myindent Thus, there holds
\begin{equation}\label{nablastPsi0}
    \begin{split}
        \nabla_{s,t}\Psi_0 &= -h(u) + \begin{pmatrix}
            \partial_s \Phi(\sigma,s,t,\eta)\\
            \partial_t \Phi(\sigma,s,t,\eta)
        \end{pmatrix} \\
        &= -h(u) + |\eta| \begin{pmatrix}
            q_1\left(\sigma,\nabla_x \Phi\left(\sigma,s,t,\frac{\eta}{|\eta|}\right)\right) \\
            q_2\left(\sigma,\nabla_x \Phi\left(\sigma,s,t,\frac{\eta}{|\eta|}\right)\right)
        \end{pmatrix}.
    \end{split}
\end{equation}

\myindent Now, by definition, $u \mapsto h(u)$ is the curve $\gamma$ defined by $\{\sqrt{F_2} = 1\}$ (see Definition \ref{deflilgammas}). Now, since by definition (see Theorem \ref{CdVthm})
\begin{equation}
    F_2(q_1(\sigma,\xi),q_2(\sigma,\xi)) = p_1(\sigma,\xi),
\end{equation}
we see that, for all $(\sigma,\xi)$, the equation (on $u$ and $|\xi|$)
\begin{equation}
    h(u) = |\xi|\begin{pmatrix}
        q_1\left(\sigma,\frac{\xi}{|\xi|}\right)\\
        q_2\left(\sigma,\frac{\xi}{|\xi|}\right)
    \end{pmatrix}
\end{equation}
can be interpreted as the equation for the intersection point of the curve $\gamma$ and the half-line $\R_+ \begin{pmatrix}
        q_1\left(\sigma,\frac{\xi}{|\xi|}\right)\\
        q_2\left(\sigma,\frac{\xi}{|\xi|}\right)
    \end{pmatrix}$. There is thus one and only one solution, determined by
    \begin{equation}\label{eqmodulusxi}
        |\xi| = \frac{1}{\sqrt{p_1\left(\sigma,\frac{\xi}{|\xi|}\right)}},
    \end{equation}
and $u$ is uniquely determined. We observe moreover that $u$ is a point of the segment $\mathcal{U}_{\sigma}$ such that $h(\mathcal{U}_{\sigma})$ is the intersection of the curve $\gamma_0$ defined by \eqref{deflilgamma0} and of $(q_1,q_2)(T_{(0,\sigma)}^* \mathcal{S})$. We will come back to this set in Definition \ref{defUsigma}, and denote it as $[u_{\sigma}(0),u_{\sigma}(\pi)]$ without justification for the moment.

\myindent Coming back to \eqref{nablastPsi0}, we thus find that, for \textit{any} fixed $\sigma,s,t,\frac{\eta}{|\eta|}$, the equation
\begin{equation}\label{nablastPsi0eq0}
    \nabla_{s,t}\Psi_0\left(\sigma,u,s,t,|\eta| \frac{\eta}{|\eta|}\right) = 0
\end{equation}
has one and only one solution $u$ and $|\eta|$. We claim that one can see this equation a \textit{correspondence} equation, in the sense that, first it merely fixed the modulus of the angle variable $\eta$, and secondly, it fixes the value of $u$ which \textit{corresponds} to the direction $\nabla_x \Phi(\sigma,s,t,\eta)$ when read through the "coordinates" $(q_1,q_2)$. 
Obviously, $(q_1,q_2)$ are \textit{not} coordinates, so one should rather think that we only fix the \textit{projection} of the direction $\nabla_x \Phi(s,t,\eta)$ onto the first axis. We will detail more this point of view in Paragraph \ref{subsubsec22HormPhase}.

\quad

\myindent Hence, combining the \textit{geometric} equation \eqref{nablaetaphieq0} and the \textit{correspondence} equation \eqref{nablastPsi0eq0}, we can describe the set $\mathcal{O}_{\sigma}$ of zeros of $\nabla_{s,t,\eta}\Psi_0$ as a subset of $\R/\ell\Z \times 2\tilde{\mathcal{R}_{i,j}} \times \R^N$. Keeping in mind that we want to parameterize this set by $u$, we give the following vizualization of $\mathcal{O}_{\sigma}$ as a "function" of $u$, where the vertical axis is the $t$ variable. It seems to us the easiest way to combine intuition on the geometry and on the analysis, since it ultimately represents the direction of the unique bicharacteristic curve joining $(t,\sigma)$ to $(0,\sigma)$ ; with the flaw that this direction is read through $(q_1,q_2)$, i.e. it is \textit{projected} onto the first axis. Indeed, if both \eqref{nablaetaphieq0} and \eqref{nablastPsi0eq0} hold, then $\nabla_x \Phi(\sigma,s,t,\eta)$ is the direction of the unique bicharacteristic curve of $q_1$ joining $(t + t_i^{\bullet},\sigma)$ and $(0,\sigma)$.

\myindent i. First, in the case $\bullet = (H)$, the picture is given by Figure \ref{Hormanderutawayfromeq} when $\sigma \neq 0$, and this degenerates to Figure \ref{Hormanderutateq} when $\sigma = 0$.

\begin{figure}[h]
\includegraphics[scale = 0.5]{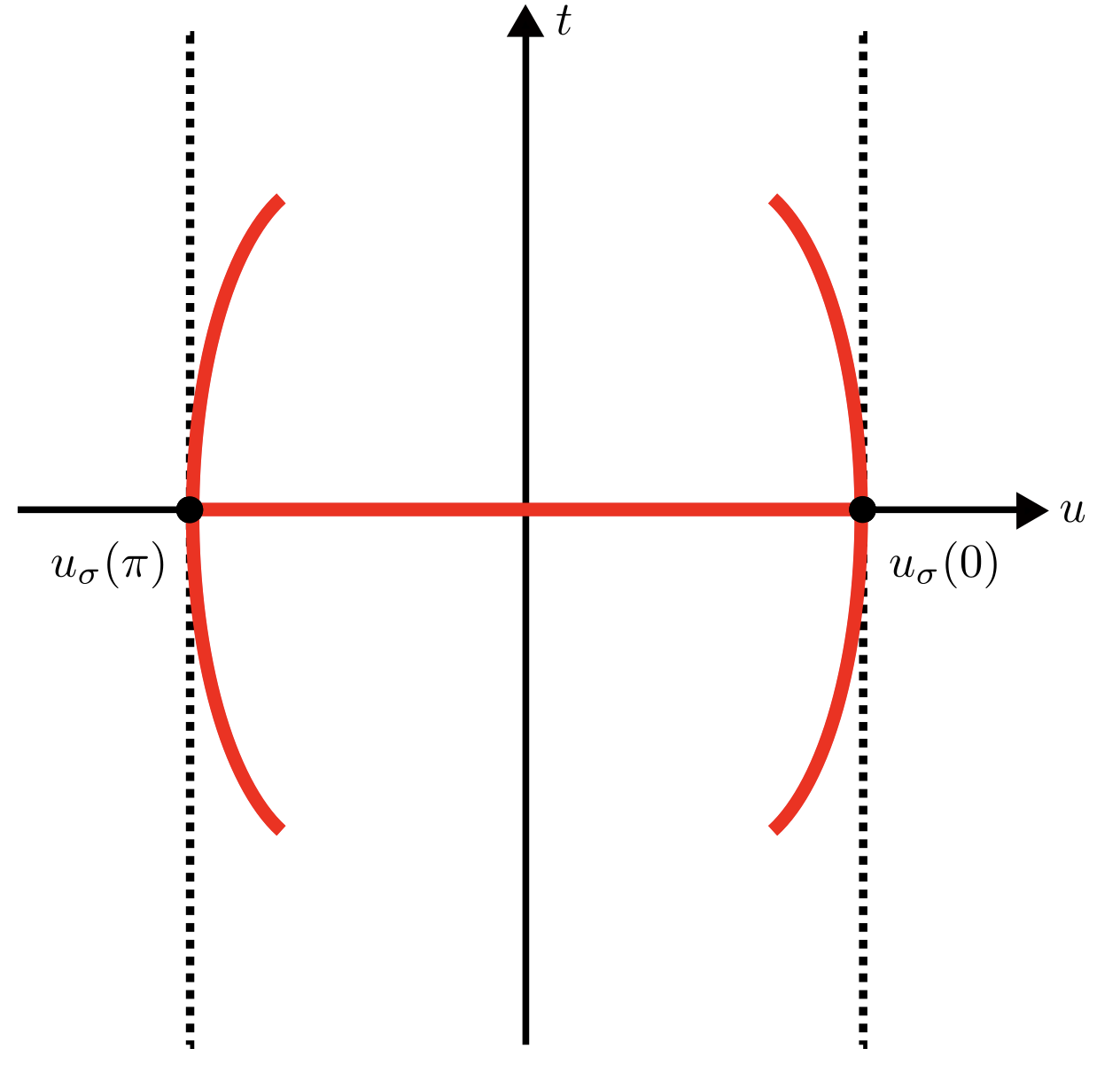}
\centering
\caption{The visualization of $\mathcal{O}_{\sigma}$ when $\bullet = (H)$ and $\sigma \neq 0$}
\label{Hormanderutawayfromeq}
\end{figure}

\begin{figure}[h]
\includegraphics[scale = 0.5]{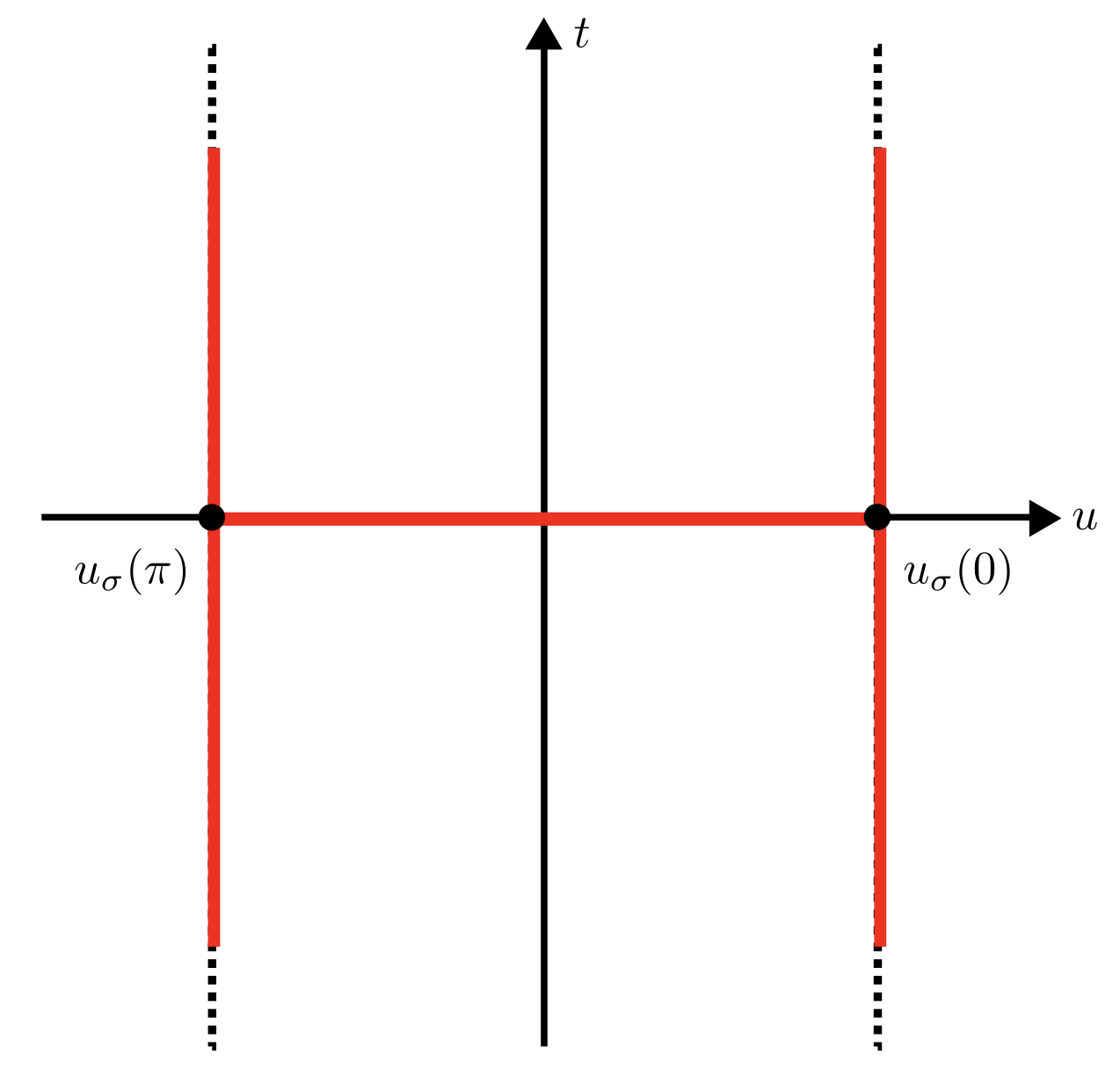}
\centering
\caption{The visualization of $\mathcal{O}_{\sigma}$ when $\bullet = (H)$ and $\sigma = 0$}
\label{Hormanderutateq}
\end{figure}

\myindent ii. In the case $\bullet = (1),...,(J_i)$, the picture is given by Figure \ref{bichutawayfromeq} when $\sigma \neq 0$, and this degenerates to a vertical line above $u_{\sigma}(0)$ or $u_{\sigma}(\pi)$ when $\sigma = 0$ (figure not included).

\begin{figure}[h]
\includegraphics[scale = 0.5]{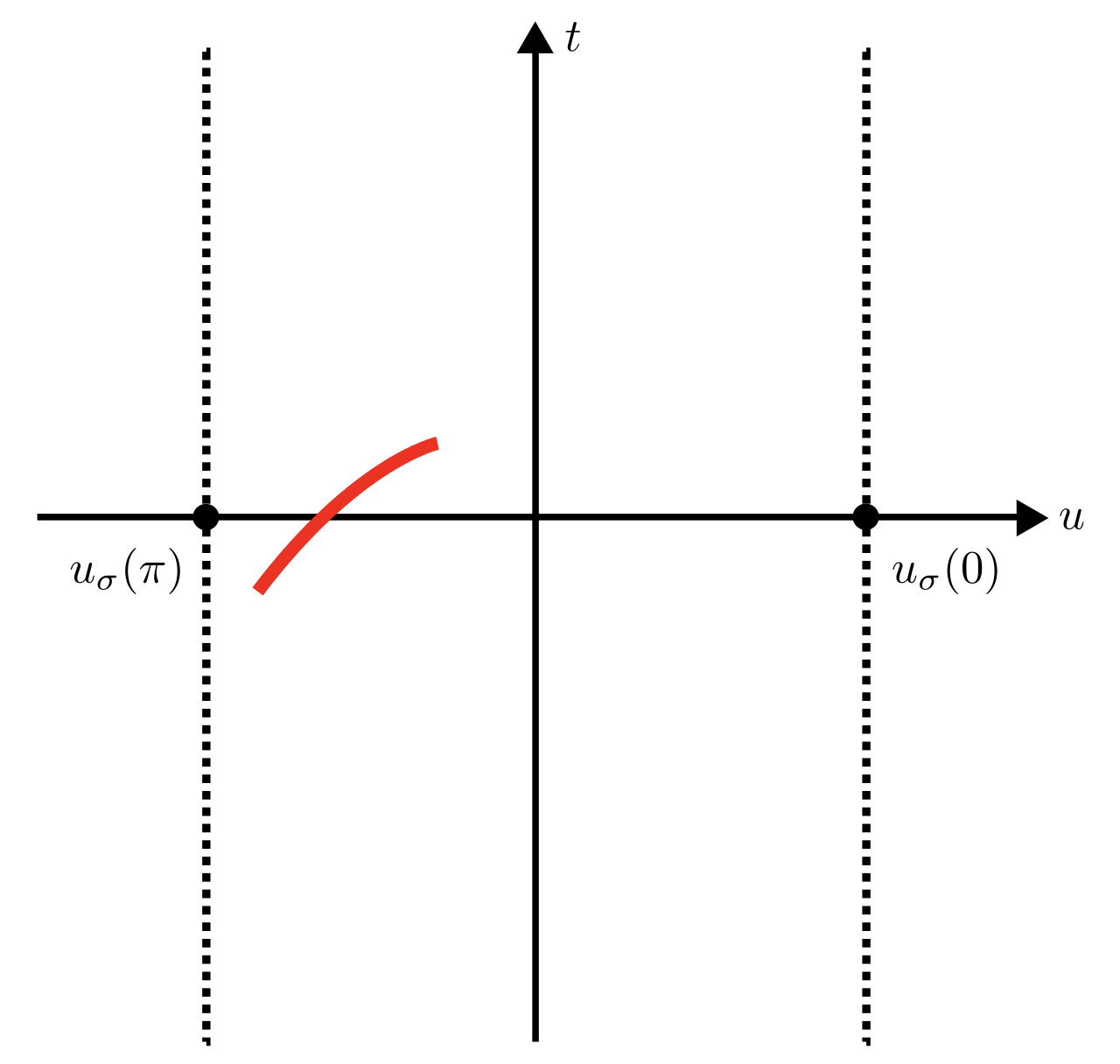}
\centering
\caption{The visualization of $\mathcal{O}_{\sigma}$ when $\bullet = (1),...,(J_i)$ and $\sigma \neq 0$}
\label{bichutawayfromeq}
\end{figure}

\myindent iii. In the case $\bullet = (\pi)$, the picture is given by Figure \ref{antipodut}, where the sharp turn degenerates to a right angle when $\sigma = 0$, at which the picture is again the same as the "H" of Figure \ref{Hormanderutateq}.

\begin{figure}[!h]
\includegraphics[scale = 0.5]{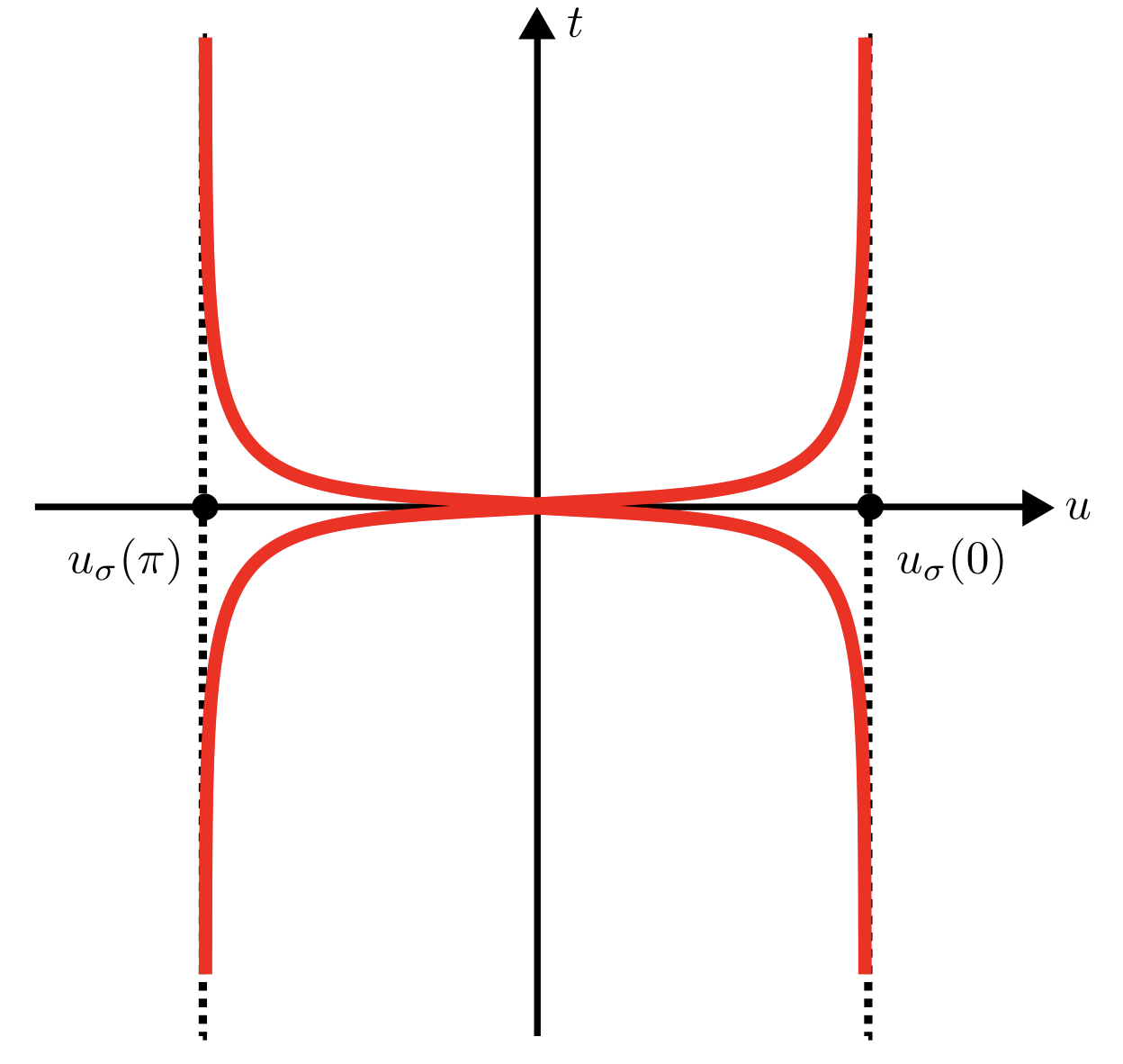}
\centering
\caption{The visualization of $\mathcal{O}_{\sigma}$ when $\bullet = (\pi)$ and $\sigma \neq 0$}
\label{antipodut}
\end{figure}

\quad

\myindent Those pictures yields the strategy that we will use : on each, a portion of $\mathcal{O}_{\sigma}$ is locally a curve parameterized by $u$. For those portions, the strategy of isolating the variable $u$ and applying the analysis of phase functions with a finite type degeneracy (cf the following section) will nicely apply, with the order of the degeneracy $p$ depending on the particular portion. However, there are some unavoidable exceptional cases. First, when $\sigma = 0$ is on the equator, or near the equator, there are vertical or almost vertical lines in the pictures, for which we will need a different argument. Secondly, for \textit{all} $\sigma$, there are exceptional points at which $\mathcal{O}_{\sigma}$ is \textit{not} a smooth curve, namely the points $P_{\pi}$ and $P_0$ depicted in Figure \ref{Hormanderutawayfromeq}. As we will see, those points play a highly nontrivial role in the asymptotics, and we will need a specific analysis near those points.

\quad

\myindent As a conclusion of this descriptive paragraph, let us observe that the above pictures are all a zoom on different parts of the following pictures, which represent the projected direction of the bicharacteristic curve joining $(t,\sigma)$ to $(0,\sigma)$ in terms of $u$. Those pictures are well-defined for \textit{all} $t$ since they don't involve the microlocal variable $\eta$. We still call the curve obtained $\mathcal{O}_{\sigma}$, even if it is rather obtained by gluing together the projections of each curve $\mathcal{O}_{\sigma}$ locally defined in the $(u,t)$ plan.

\myindent i. In the generic case, when $\sigma$ is away from the equator, the picture is given by Figure \ref{globalOsignoteq}.

\begin{figure}[h]
\includegraphics[scale = 0.5]{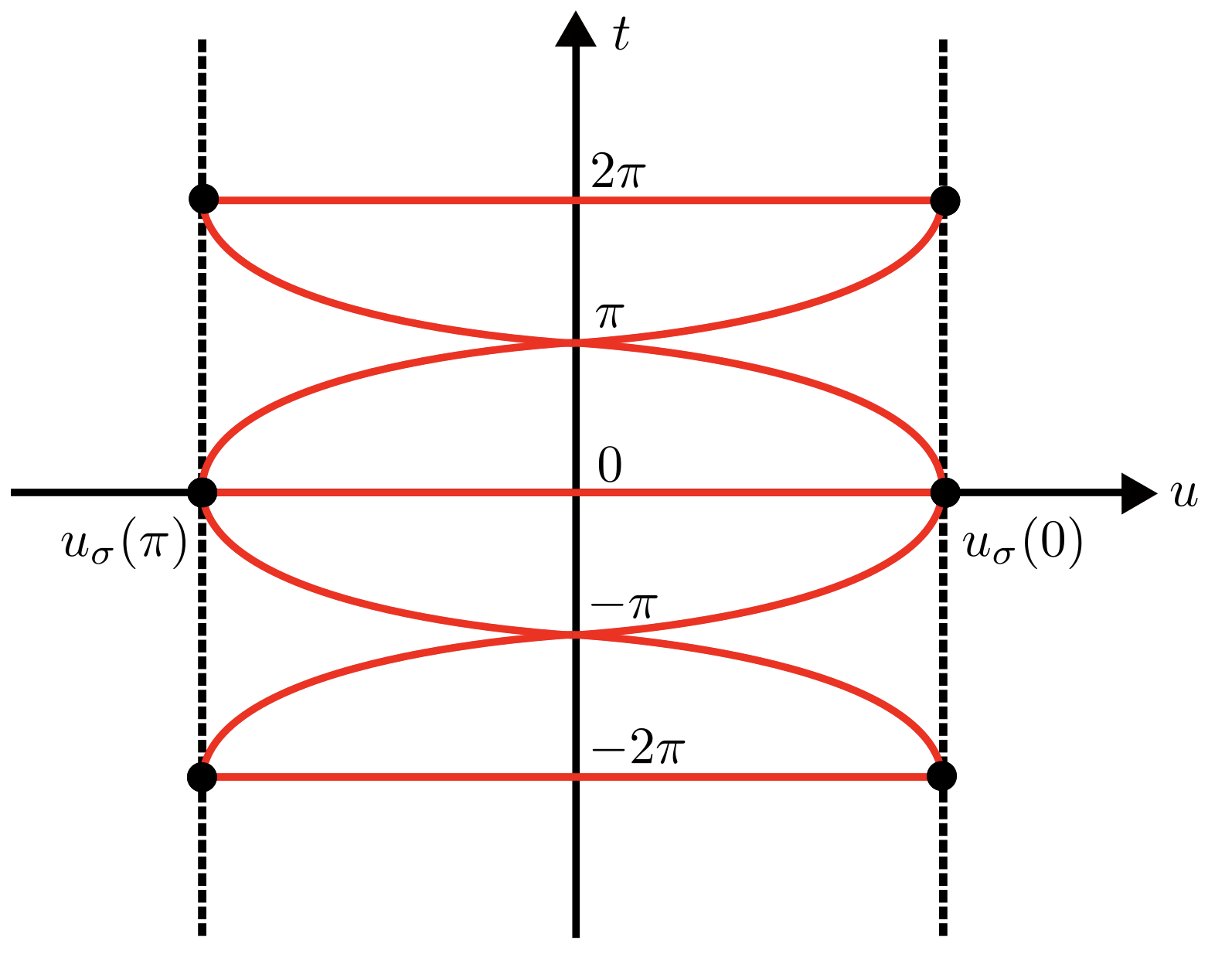}
\centering
\caption{The global visualization of $\mathcal{O}_{\sigma}$  when $\sigma \neq 0$}
\label{globalOsignoteq}
\end{figure}

\myindent ii. When $\sigma = 0$ is on the equator, in a sharp contrast, the picture is given by Figure \ref{globalOsigoneq}.

\begin{figure}[h]
\includegraphics[scale = 0.5]{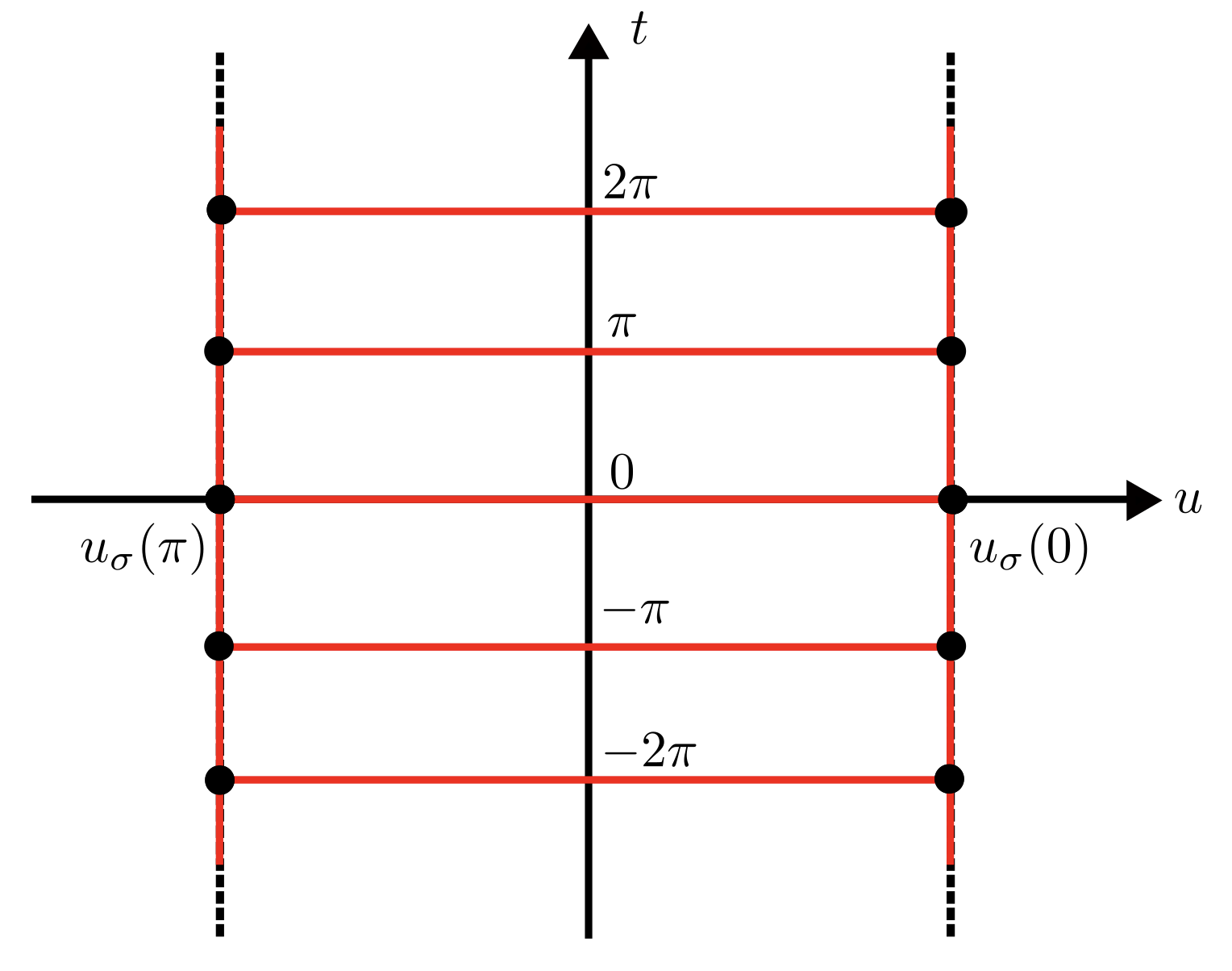}
\centering
\caption{The global visualization of $\mathcal{O}_{\sigma}$  when $\sigma = 0$}
\label{globalOsigoneq}
\end{figure}

\myindent iii. Since this picture is \textit{continuous} in $\sigma$, there is necessarily a transitional regime, when $\sigma$ is close to the equator but not at the equator, which is given by Figure \ref{globalOsigtransitional}. 

\begin{figure}[h]
\includegraphics[scale = 0.5]{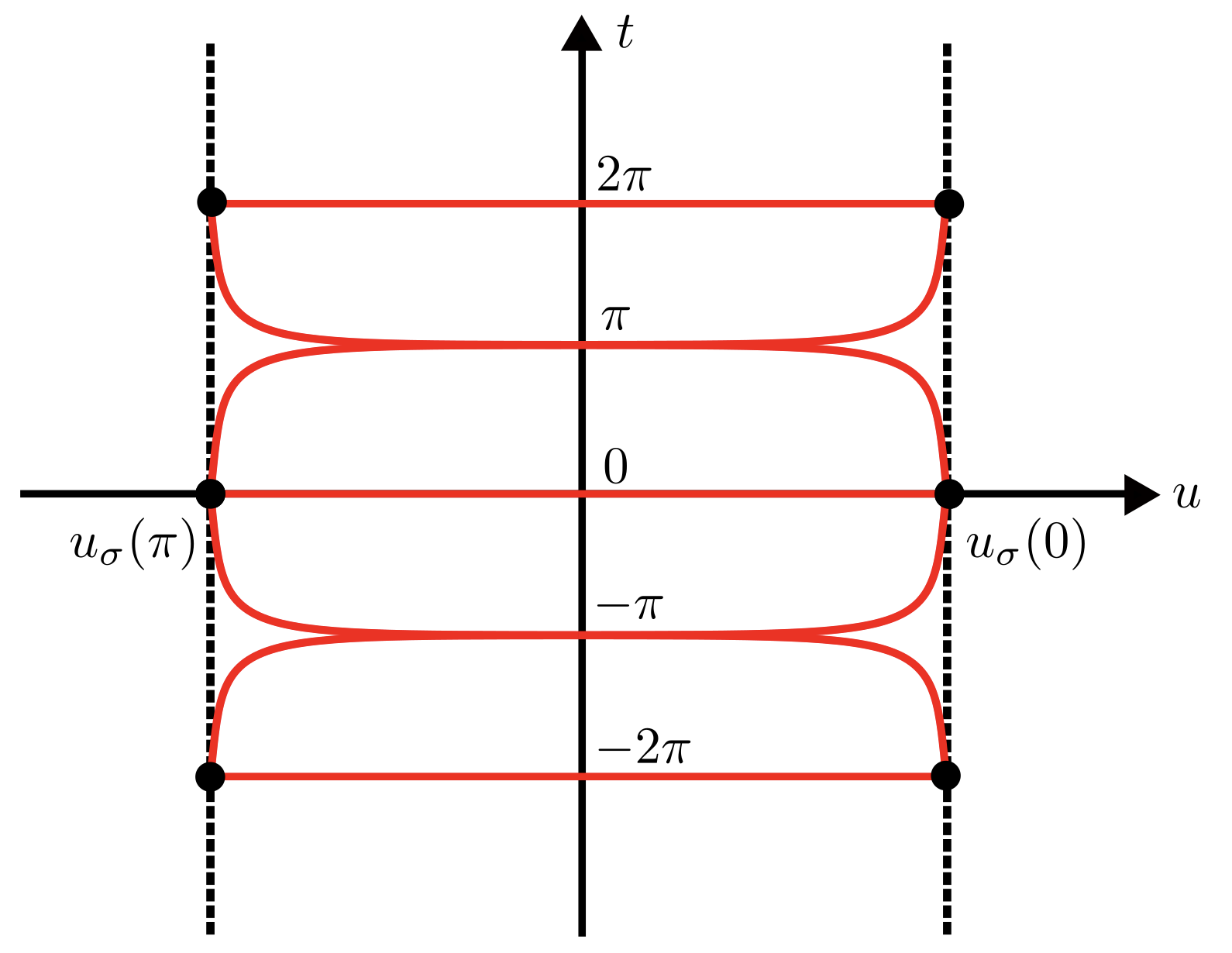}
\centering
\caption{The transitional regime for $\mathcal{O}_{\sigma}$}
\label{globalOsigtransitional}
\end{figure}

\myindent The existence of two regimes introduces a difficulty in the analysis, since it needs to be quantified via the introduction of a threshold of $\sigma$. Obviously, this will depend on $(A,B)$, making the analysis quite technical.

\subsection{Stationary phase estimates}\label{subsec4Micro}

\myindent In this section, we introduce rigorously the special type of phase functions which will appear in the analysis, namely the stationary phase functions with a finite type degeneracy. We moreover give a general useful quantitative bound for oscillatory integrals with this type of phase function, namely Theorem \ref{mixedVdCABZ}.

\subsubsection{On Van der Corput's lemma}\label{subsubsec41Micro}

\myindent First, we recall the important Van der Corput lemma, with a little twist compared to the usual version introduced by Elias Stein (see \cite{stein1993harmonic}[Section VIII, Proposition 2]).

\begin{theorem}[Van der Corput's lemma]\label{VdCthm}
    Let $[a,b]$ be a segment, and $\phi : [a,b] \to \R$ be a smooth function. Assume that there exists a constant $c > 0$ and an integer $p\geq 1$ such that
    \begin{equation}\label{eqvdc}
        \forall x \in [a,b] \qquad |\phi^{(p)}(x)| \geq c.
    \end{equation}
    
    \myindent Then there holds
    \begin{equation}
        \left|\int_a^b e^{i\lambda \phi(x)} dx \right| \leq \begin{cases}
            c_p (c\lambda)^{-\frac{1}{p}} \qquad \text{if} \ p\geq 2 \\
            c_p(c\lambda)^{-1} \left(1 + c^{-1}|b-a| \|\phi''\|_{L^{\infty}}\right) \qquad \text{if} \ p = 1
        \end{cases},
    \end{equation}
    where $c_k$ is a constant independent of $\lambda,\phi,a,b$.
\end{theorem}

\myindent Now, this theorem still holds if we suppose, more generally, that \eqref{eqvdc} holds not necessarily for the \textit{same} $p$ for all $x\in [a,b]$, but if it holds for \textit{some} $p$, as long as the value of $p$ is uniformly bounded. Precisely, we define

\begin{definition}\label{VdCN}
    Let $[a,b]$ be a segment of $\R$ and $\phi \in \mathcal{C}^{\infty}([a,b])$. Let $p\geq 1$ be an integer. We say that $\phi$ satisfies the property $(VdC)_p$ with constants $C,c$ if there exists a partition of $[a,b]$ into a finite family of intervals, say $I_1,...,I_K$, such that for all $i = 1,...,K$, there exists $p_i \in\{1,...,p\}$ such that
    \begin{equation}
        \forall x \in I_i \qquad |\phi^{(p_i)}(x)|\geq c,
    \end{equation}
    and, moreover, if $p_i = 1$ then
    \begin{equation}
        \|\phi''\|_{L^{\infty}(I_i)} \leq C.
    \end{equation}
\end{definition}

\myindent Then, an immediate consequence of Theorem \ref{VdCthm} is the following.

\begin{corollary}\label{intermediateVdCp}
    Let $[a,b]$ be a segment of $\R$ and $\phi \in \mathcal{C}^{\infty}([a,b])$. Let $p\geq 1$. Assume that $\phi$ satisfies the property $(VdC)_p$ with constants $C,c$ (see Definition \ref{VdCN}). Then there holds
    \begin{equation}
        \left|\int_a^b e^{i\lambda \phi(x)} dx\right| \leq c_{p,K} (c\lambda)^{-\frac{1}{p}} \left(1 + |b-a| \frac{C}{c}\right),
    \end{equation}
    where $c_{p,K}$ is a universal constant, and $K$ is the number of intervals in Definition \ref{VdCN}.
\end{corollary}

\begin{remark}
    The constant $c_{p,k}$ in Corollary \ref{intermediateVdCp} depends a priori on the number of intervals $K \geq 1$. However, in the following, the number of intervals can always be uniformly bounded. Hence, we won't further mention, or note, the dependency on $K$.
\end{remark}

\myindent We may finally deduce the following useful proposition.

\begin{proposition}\label{usefulVdC}
    Let $I$ be a segment of $\R$, $\phi \in \mathcal{C}^{\infty}([a,b])$, $f$ a continuous function on $[a,b]$ such that $f' \in L^1([a,b])$. Assume that $\phi$ satisfies the property $(VdC)_p$ for some $p\geq 1$, with constants $C,c > 0$. Then there holds
    \begin{equation}
        \left|\int_a^be^{i\lambda \phi(x)} f(x) dx \right| \leq c_p (c\lambda)^{-\frac{1}{p}} \left(1 + |b-a| \frac{C}{c}\right) \left(\|f\|_{L^{\infty}([a,b])} + \|f'\|_{L^{1}([a,b])} \right).
    \end{equation}
\end{proposition}

\subsubsection{Oscillatory integrals with a nondegenerate stationary phase}\label{subsubsec42Micro}

\myindent In this paragraph, we recall standard results regarding the behavior of oscillatory integrals with a non degenerate stationary phase. Precisely, let $d\geq 1$ be an integer, let $U \subset \R^d$ be an open set and let $K\subset U$ be a compact set. Let $\phi \in \mathcal{C}^{\infty}(U)$ be a real valued phase function, and let $a \in \mathcal{C}^{\infty}_0(K)$ be a test function. We are concerned with the asymptotic behavior of the oscillatory integral
\begin{equation}\label{defoscintstatphase}
    \mathcal{I}(\lambda) := \int_U e^{i\lambda \phi(y)} a(y) dy
\end{equation}
when $\lambda \to \infty$.

\myindent Now, it is well-known that this behavior is governed by the \textit{stationary points} of $\phi$, i.e. the $y \in U$ for which
\begin{equation}\label{statpointeq}
    \nabla \phi(y) = 0.
\end{equation}
\myindent For simplicity, we assume that $\phi$ has a unique stationary point $y_0 \in U$. Moreover, we assume that, at this stationary point, $\phi$ is \textit{non degenerate}, i.e. that
\begin{equation}\label{nondegeq}
    \nabla^2 \phi(y_0) \in GL_d(\R).
\end{equation}

\myindent Observe that this assumption can be seen as the "generic" kind of stationary phase $\phi$. Now, the stationary phase Lemma is the general statement that, under those hypotheses, then
\begin{equation}
    \mathcal{I}(\lambda) = \left(\frac{2\pi}{\lambda}\right)^{\frac{d}{2}}\frac{e^{i\lambda \phi(y_0)}}{ |\det( \nabla^2 \phi(y_0))|^{\frac{1}{2}}} e^{i\frac{\pi}{4} sgn(\nabla^2\phi(y_0))} a(y_0) + O_{\phi,a}(\lambda^{-\frac{d}{2}- 1}).
\end{equation}

\myindent In the reference textbook presentation of this formula by Hörmander in \cite{hormanderanalysis}[Theorem 7.7.5.], the remainder is quantified as follows\footnote{Actually, Hörmander's Theorem is that one has an asymptotic expansion for $\mathcal{I}(\lambda)$ in terms of decreasing powers of $\lambda$, whose coefficients are semi-explicit, and this up to any order with a bound on the remainder. However, we won't use the full asymptotic expansion}.

\begin{theorem}[Stationary phase approximation]\label{hormstatphase}
    Under the hypotheses written above, there holds
    \begin{equation}
        \left|\mathcal{I}(\lambda) - \left(\frac{2\pi}{\lambda}\right)^{\frac{d}{2}}\frac{e^{i\lambda \phi(y_0)}}{ |\det( \nabla^2 \phi(y_0))|^{\frac{1}{2}}} e^{i\frac{\pi}{4} sgn(\nabla^2\phi(y_0))} a(y_0)\right| \leq C \lambda^{-1} \sum_{|\alpha| \leq 2} \sup_K |D^{\alpha} a|,
    \end{equation}
    where the constant $C$ is bounded when $\phi$ remains in a bounded set in $\mathcal{C}^4(U)$ and $\frac{|y - y_0|}{|\nabla \phi(y)|}$ has a uniform bound
\end{theorem}

\myindent Now, usually, this theorem is enough to work with, since one deals with only one phase function. However, as we have mentioned in Paragraph \ref{subsubsec32Micro}, we will have to deal with countably many different phase functions, depending on $(A,B) \in (2\pi \Z)^2$, which won't obviously be uniformly bounded since the norms of their derivatives are of order $|(A,B)|$ (see formula \eqref{abstractdefofPsi}). Hence, we need to refine Theorem \ref{hormstatphase}, and, precisely, to quantify the constant $C$ which appears in the upper bounds in terms of the derivatives of $\phi$.

\quad

\myindent To our knowledge, one of the most precise quantitative bounds for the oscillatory integral \eqref{defoscintstatphase} is the result of Alazard, Burq and Zuily in \cite{alazard2017stationary}. Before giving the result, we recall the notations in their paper.

\myindent Let $\eta > 0$ be a small constant, and let
\begin{equation}
    K_{\eta} := \{y \in \R^d : dist(y,K) \leq \eta\},
\end{equation}
where "dist" is the sup distance on $\R^d$. 

\myindent For $k\geq 2$ and $l\geq 0$, let
\begin{equation}
    \begin{split}
        &\mathcal{M}_k := \sum_{2 \leq |\alpha|\leq k} \sup_{K_{\eps_0}} |D^{\alpha} \phi|\\
        &\mathcal{N}_l := \sum_{|\alpha|\leq l} \sup_K |D^{\alpha} a|.
    \end{split}
\end{equation}

\myindent Set
\begin{equation}
    a_0 := \inf_{K_{\eps_0}} |\det(\nabla^2 \phi)|.
\end{equation}

\myindent When $\mathcal{M}_2 > 0$, set
\begin{equation}
    \begin{split}
        &\delta_{\eps_0} := \frac{a_0}{4(C_1\mathcal{M}_2)^{d-1} C_2 \mathcal{M}_3} \\
        &\delta := \min(\delta_{\eps_0}, \frac{\eps_0}{4}),
    \end{split}
\end{equation}
where $C_1,C_2$ are constants depending only on the dimension $d$, the definition of which we do not recall.

\myindent Then, we can state the following theorem.

\begin{theorem}[Alazard, Burq, Zuily]\label{ABZthm}
    Assume 
    \begin{equation}
        \begin{split}
            &\mathcal{M}_{d+2} < + \infty \\
            &\mathcal{N}_{d+1} < + \infty \\
            & a_0 > 0.
        \end{split}
    \end{equation}
    
    \myindent Then, there exists $C> 0$ depending only on the dimension $d$ such that, for all $\lambda \geq 1$,
    \begin{equation}
        |\mathcal{I}(\lambda)| \leq \frac{C |K_{\eps_0}|}{a_0 \delta^d} \left(1 + \mathcal{M}_{d+2}^{\frac{d}{2}}\right) \mathcal{N}_{d+1} \lambda^{-\frac{d}{2}}.
    \end{equation}
\end{theorem}

\subsubsection{Oscillatory integrals with a stationary phase with a finite type degeneracy}\label{subsubsec43Micro}

\myindent We now introduce the special type of oscillatory integral in $(1+d)$ dimension that we will consider. Before giving a rigorous statement, we explain the intuition of the result, and why we unfortunately can't deduce it from the above theorems.

\myindent Let $\mathcal{U} \subset \R^{1+d}$ be an open set, let $K \subset \mathcal{U}$ be a compact set. Let $\phi \in \mathcal{C}^{\infty}(\mathcal{U})$ be a real valued phase function, and let $a \in \mathcal{C}_0^{\infty}(K)$ be a test function. Let
\begin{equation}
    \mathcal{I}(\lambda) := \int_{\mathcal{U}} e^{i\lambda \phi(z)} a(z) dz.
\end{equation}

\myindent Now, let us still assume that there exists a $z_0 \in \mathcal{U}$ such that $\phi$ is stationary at $z_0$, i.e. \eqref{statpointeq} holds. However, let us assume that $\phi$ is \textit{degenerate} at $z_0$, i.e. that \eqref{nondegeq} \textit{doesn't} hold. Instead, let us assume what one can think of as the "minimal" kind of degeneracy of $\phi$, i.e. that
\begin{equation}
    rk(\nabla^2 \phi(z_0)) = d,
\end{equation}
where $rk$ is the rank. Now, without loss of generality, we can thus assume that, if we denote the points of $\R^{1+d}$ as $z = (x,y) \in \R \times \R^d$, then, if $z_0 = (x_0,y_0)$,
\begin{equation}
    \nabla_y^2 \phi(x_0,y_0) \in GL_d(\R).
\end{equation}

\myindent Now, thanks to the implicit function theorem, we know that, for any $x$ close enough to $x_0$, there is a unique $y(x)$ such that
\begin{equation}
    \nabla_y \phi(x,y(x)) = 0,
\end{equation}
and, moreover, $x\mapsto y(x)$ is a smooth function. Hence, up to reducing to a small neighborhood of $z_0$, we can assume that $\mathcal{U} = I \times U$, where $I\subset \R$ is an open interval and $U \subset \R^d$ is an open set, such that moreover $x\mapsto y(x)$ is defined on all $I$ and $y(x) \in U$, and finally
\begin{equation}
    \forall x \in I \qquad \nabla_y \phi(x,y(x)) \in GL_d(\R).
\end{equation}

\myindent It is thus natural, in order to estimate $\mathcal{I}(\lambda)$, to first factorise it as
\begin{equation}\label{firstfactorisation}
    \mathcal{I}(\lambda) = \int_I  \left(\int_U  e^{i\lambda \phi_x(y)} a_x(y) dy\right) dx
    =: \int_I I(x,\lambda) dx,
\end{equation}
where, for any function $f(x,y)$, and any fixed $x \in I$, we define
\begin{equation}\label{defphix}
    f_x : y \in U \to f(x,y).
\end{equation}

\myindent Indeed, the stationary phase Lemmalready yields that
\begin{equation}
    I(x,\lambda) = \left(\frac{2\pi}{\lambda}\right)^{\frac{d}{2}}\frac{e^{i\lambda \phi(x,y(x))}}{ |\det( \nabla^2 \phi(x,y(x)))|^{\frac{1}{2}}} e^{i\frac{\pi}{4} sgn(\nabla^2\phi(x,y(x)))} a(x,y(x)) + O_{\phi,a}(\lambda^{-\frac{d}{2}- 1}).
\end{equation}

\myindent In particular, we find that
\begin{equation}\label{firstgaininlambda}
    \mathcal{I}(\lambda) = \lambda^{-\frac{d}{2}} \int_I e^{i\lambda \phi(x,y(x))} b(x) dx + O_{\phi,a,\mathcal{U}} (\lambda^{-\frac{d}{2} - 1}),
\end{equation}
where 
\begin{equation}
    b(x) := \left(2\pi\right)^{\frac{d}{2}}\frac{e^{i\frac{\pi}{4} sgn(\nabla^2\phi(x,y(x)))}}{ |\det( \nabla^2 \phi(x,y(x)))|^{\frac{1}{2}}}  a(x,y(x)).
\end{equation}

\myindent Hence, we have reduced the analysis to the asymptotic analysis of a 1D oscillatory integral, where the phase function is the \textit{1D remaining phase function}, which we define as
\begin{equation}\label{defphi1D}
    \phi^{1D}(x) := \phi(x,y(x)).
\end{equation}

\myindent Now, we already know that this phase is \textit{stationary} at the point $x_0$. Indeed, since $\nabla_{x,y}\phi(x_0,y_0) = 0$ by hypothesis, there holds
\begin{equation}
    (\phi^{1D})'(x_0) = \partial_x \phi(x_0,y_0) + \nabla_y \phi(x_0,y_0) \cdot y'(x_0) = 0.
\end{equation}

\myindent We are interested in the case where, even if $\phi$ is degenerate at $(x_0,y_0)$, there are still some oscillations of $\phi$ in the $x$ direction, or, more precisely, of $\phi^{1D}$. Hence, we will study the case where $\phi^{1D}$ is of \textit{finite type} at $x_0$, i.e. the case where there exists an order $p \geq 2$ such that
\begin{equation}
    (\phi^{1D})^{(p)}(x_0) \neq 0.
\end{equation}

\myindent Observe that, thanks to the Van der Corput Lemma \ref{VdCthm}, this finally ensures, thanks to \ref{firstgaininlambda} that 
\begin{equation}
    |\mathcal{I}(\lambda)| = O_{\phi,a,\mathcal{U}}(\lambda^{-\frac{d}{2} - \frac{1}{p}}).
\end{equation}

\myindent Moreover, this upper bound is \textit{optimal} in terms of powers of $\lambda$, since one can actually improve the upper bound given by the Van der Corput Lemma into an asymptotic expansion (see \cite{stein1993harmonic}).

\myindent We remark that there also holds
\begin{equation}
    (\phi^{1D})''(x_0) = 0.
\end{equation}

\myindent Indeed, one can compute that
\begin{equation}
    (\phi^{1D})''(x) = \frac{\det(\nabla_{x,y}^2 \phi(x,y(x)))}{\det(\nabla_y^2 \phi(x,y(x)))},
\end{equation}

and $\phi$ is degenerate at $(x_0,y_0)$ be hypothesis. Hence, the order $p$ is actually greater than or equal to $3$.

\quad

\myindent Now, this reasoning, which can be seen as a mix between Theorem \ref{VdCthm} and Theorem \ref{hormstatphase}, gives the order of $\mathcal{I}(\lambda)$, in terms of powers of $\lambda$. However, as we have mentioned in paragraph \ref{subsubsec32Micro}, this is not enough for the analysis, since we will need to bound the implicit constant explicitly in terms of $\phi$ and $a$.

\myindent In order to obtain a more quantitative upper bound, it is natural to try to use, or adapt, the quantitative upper bound given by Theorem \ref{ABZthm}. However, using the notations introduced above, this theorem only yields a bound 
\begin{equation}\label{boundIxlambda}
    |I(x,\lambda)| \leq C(\phi,a) \lambda^{-\frac{d}{2}},
\end{equation}
where $C(\phi,a)$ is explicit. Hence, in this upper bound, we loose any information regarding the remaining phase function $\phi^{1D}$, i.e. we loose the fact that $I(x,\lambda)$ is an oscillatory function at the main order. Now, it is natural to try and fix that issue by changing the factorisation \eqref{firstfactorisation} into the factorisation
\begin{equation}
    \mathcal{I}(\lambda) = \int_I e^{i\lambda \phi^{1D}(x)} \left(\int_U e^{i\lambda(\phi_x(y) - \phi^{1D}(x))} a_x(y) dy \right) dx =: \int_I e^{i\lambda \phi^{1D}(x)} J(x,\lambda) dx,
\end{equation}
where we extract the oscillatory part of $I(x,\lambda)$. If we try and apply the Van der Corput Lemma \ref{VdCthm}, we will obtain
\begin{equation}
    |\mathcal{I}(\lambda)| \leq C(I,p) \lambda^{-\frac{1}{p}}(\sup_{x\in I} |J(x,\lambda)| + |J'(x,\lambda)|).
\end{equation}

\myindent Now, $J(x,\lambda)$ will satisfy the same bound \eqref{boundIxlambda} than $J(x,\lambda)$, so there holds
\begin{equation}
    C(I,p) \lambda^{-\frac{1}{p}}\sup_{x\in I} |J(x,\lambda)|  \leq C(I,p,\phi,a) \lambda^{-\frac{d}{2} - \frac{1}{p}},
\end{equation}
with $C(I,p,\phi,a)$ an explicit constant. However, there is a major issue regarding $J'(x,\lambda)$. Indeed, we compute
\begin{equation}
    J'(x,\lambda) = \lambda \int_U e^{i\lambda(\phi_x(y) - \phi^{1D}(x))} (\partial_x \phi(x,y) - (\phi^{1D})'(x)) a_x(y) dy + \int_U e^{i\lambda(\phi_x(y) - \phi^{1D}(x))}  \partial_x a(x,y) dy.
\end{equation}

\myindent Hence, if we directly apply Theorem \ref{ABZthm} to bound $J'(x,\lambda)$, an additional factor $\lambda$ will appear, and, even if we gain on the explicit constant, we loose on the correct order of $\lambda$ in the estimate.

\myindent Thus, the main difficulty of the analysis, and the new important feature that we add in the proof, is the estimate 
\begin{equation}
     \left|\lambda \int_U e^{i\lambda(\phi_x(y) - \phi^{1D}(x))} (\partial_x \phi(x,y) - (\phi^{1D})'(x)) a_x(y) dy\right| \leq C(\phi,a) \lambda^{-\frac{d}{2}},
\end{equation}
with $C(\phi,a)$ an explicit constant. The mains novelty of the proof is a well-chosen integration by parts, in order to recover a factor $\lambda^{-1}$, before adapting the standard techniques giving upper bounds for oscillatory integrals with a non degenerate phase function, for which we follow a similar approach that \cite{alazard2017stationary}.

\quad

\myindent We now give a precise statement, after introducing a few notations. The proof of the statement, the main novelty of it being the integration by parts mentionned above, is in the Appendix \ref{AppendixC}, since it is quite technical. 

\myindent Let us start with some definitions. Let $I \subset \R$ be an interval (not necessarily open), and let $U \subset \R^d$ be an open set.

\begin{notation}
    For any two integers $0 \leq k \leq l$, and any $f \in \mathcal{C}^{\infty}(U)$, we define
    \begin{equation}
        \mathcal{M}_{k,l}(f) := 1 + \sum_{k \leq |\alpha| \leq l} \sup_{y\in U} |D^{\alpha} f(y)|.
    \end{equation}
\end{notation}

\myindent We generalize this notation for smooth functions of $(1+d)$ variables.
\begin{notation}\label{defM1+d}
    Let $\phi \in \mathcal{C}^{\infty}(\R \times \R^d)$. We denote by $(x,y)$ the points of $\R \times \R^d$. We define for any two integers $0 \leq k \leq l$
    \begin{equation}
    \begin{split}
        \mathcal{M}_{k,l}^{(y)}(\phi) &:= \sup_{x \in I} \mathcal{M}_{k,l}(\phi_x)
        = 1 + \sum_{k\leq |\alpha| \leq l} \sup_{(x,y) \in I\times U} |D^{\alpha}_y \phi(x,y)|,
    \end{split}
    \end{equation}
    where $\phi_x$ is defined by \eqref{defphix}.
\end{notation}

\begin{definition}\label{defHandN}
    Let $\phi \in \mathcal{C}^{\infty}(\R\times \R^d)$. Assume that, for all $x\in I$, the function $\phi_x$ (defined in \eqref{defphix}) has one, and only one, stationary point $y(x) \in U$ at which its Hessian
    \begin{equation}\label{defHx}
        H(x) := (\nabla_y^2 \phi)(x,y(x))
    \end{equation}
    is nondegenerate (in particular,$ x \mapsto y(x)$ is a smooth function). We define
    \begin{equation}
    \begin{split}
        &\mathcal{D}(\phi) := \inf_{x\in I} |\det H(x)| \\
        &\mathcal{N}(\phi) := \sup_{x \in I} \|H(x)^{-1} \|
    \end{split}
    \end{equation}
\end{definition}

\begin{definition}\label{defadapted}
    Let $d\geq 1$ be an integer, and let $B = B[C,r]$ be the closed ball of center $C \in \R^d$ and of radius $r > 0$. We say that $\zeta \in \mathcal{C}^{\infty}(\R^d)$ is a smooth localizer which is adapted to the ball $B$ if there holds 
    
    \myindent i. First, $0\leq \zeta \leq 1$.
    
    \myindent ii. Second, 
    \begin{equation}
        \zeta(y) = \begin{cases}
            1 \qquad y \in B[C,\frac{r}{2}] \\
            0 \qquad y \notin B[C,r]
        \end{cases}.
    \end{equation}
    
    \myindent iii. Third, for any integer $k \geq 1$,
    \begin{equation}
        \left\| \nabla^k \zeta \right\|_{L^{\infty}(\R^d)} \lesssim_d r^{-k}.
    \end{equation}
\end{definition}

\begin{theorem}\label{mixedVdCABZ}
	Let $I$ be an interval of $\R$, and $U$ be an open subset of $\R^d$. Let $\phi(x,y)$ be a smooth phase function defined in a neighborhood of $I\times U$. Let $a \in \mathcal{C}_0^{\infty}(\R\times \R^d)$. Assume that, with the notation \eqref{defphix}, for all $x\in I$,
	\begin{equation}
		\phi_x : y \mapsto \phi(x,y)
	\end{equation}
	has a unique stationary point $y(x)\in U$, at which $H(x)$ defined by \eqref{defHx} is nondegenerate. 
	
	\myindent Assume moreover that, for some integer $p\geq 1$, the 1D remaining phase function $\phi^{1D}$ defined by \eqref{defphi1D} satisfies the Property $(VdC)_p$ (see definition \ref{VdCN}) with constants $C,c$.
	
	\myindent Let $\zeta(x,y)$ be a smooth function such that, for all $x \in I$, $\zeta(x,\cdot)$ is a smooth localizer adapted to the ball (see Definitions \ref{defadapted}, \ref{defM1+d}, and \ref{defHandN})
    \begin{equation}\label{defballmixedvdcabz}
        \mathcal{B}(x) := \left\{y\in U \ \text{such that} \qquad |y - y(x)| \leq \frac{1}{2} \left(\mathcal{M}_{3,3}^{(y)}(\phi)\mathcal{N}(\phi)\right)^{-1} \right\}.
    \end{equation}

    \myindent Define the oscillatory integral
	\begin{equation}
		\mathcal{I}(\lambda) := \int_{I\times U} e^{i\lambda\phi(x,y)} \zeta(x,y) a(x,y) dx dy.
	\end{equation}
	
	\myindent Assume finally that
	\begin{equation}\label{technicalhyp}
	    \left(\mathcal{N}(\phi)\right)^2\mathcal{M}_{2,d+4}^{(y)}(\phi) \leq \lambda.
	\end{equation}
	
	\myindent Then there holds
    \begin{multline}\label{quantitativeboundmixedABZVdC}
        |\mathcal{I}(\lambda)|\lesssim_d \lambda^{-\frac{d}{2} - \frac{1}{p}} c^{-\frac{1}{p}} \left(1 + |I|\left(1+ \frac{C}{c}\right)\right)^2 \left(\mathcal{D}(\phi)\right)^{-1}\left(\mathcal{N}(\phi)\right)^{ 2} \left(\mathcal{M}_{2,d+4}^{(y)}(\phi)\right)^{\frac{d}{2} + 1}\\
        \times \mathcal{M}_{1,d+3}^{(y)}(\partial_x \phi) \left(\mathcal{M}_{0,d+2}^{(y)}(a) + \mathcal{M}_{0,d+1}^{(y)}(\partial_x a)\right).
    \end{multline}
\end{theorem}

\myindent We put the proof of this Lemma in Appendix \ref{AppendixC} since it is quite technical. Also in the appendix, in Remark \ref{resultnotechnicalhypo}, we will further discuss the Theorem, and, in particular, the technical hypothesis \eqref{technicalhyp}.

\begin{remark}
    Observe that, in the hypotheses of the theorem, we do not ask that $\phi$ actually has a stationary point on $I\times U$. More precisely, the theorem requires to compute the $y$ stationary points of $\phi$ (i.e. the points where $\nabla_y \phi = 0$), but allows for some flexibility on the $x$ stationary points of $\phi$. Indeed, since we only ask that the 1D remaining phase function satisfies the Property $(VdC)_p$, one typically needs only prove that the first $p$ derivatives of $\phi^{1D}$ cannot vanish at the same time.
\end{remark}

\begin{remark}
    In the case where $\phi$ actually has a stationary point, this type of degenerate phase function has been extensively studied by Arnold, and actually constitute a $A_p$ singularity, see \cite{arnold1976local}.
\end{remark}

\begin{remark}\label{nonisotropicmixedABZVdC}
    We have stated the theorem in an isotropic setting, since it is easier as one needs only compute the quantities $\mathcal{M}_{3,3}^{(y)}(\phi)$ and $\mathcal{N}(\Phi)$ to know how close to the $y$ stationary point one needs to localize. However, the proof that we have given extends, with the same conclusion \eqref{quantitativeboundmixedABZVdC}, with the following : let $\zeta(x,y)$ be a smooth function such that $\zeta(x,\cdot)$ is a smooth localizer around $y(x)$ such that, on the one hand, 
    \begin{equation}\label{firstconditionnonisotropic}
        \begin{split}
            \mathcal{M}_{0,l}^{(y)}(\zeta) &\lesssim \left(\mathcal{M}_{3,3}^{(y)}(\phi)\mathcal{N}(\phi)\right)^l \\
            \mathcal{M}_{0,l}^{(y)}(\partial_x\zeta) &\lesssim \sup_x |y'(x)| \left(\mathcal{M}_{3,3}^{(y)}(\phi)\mathcal{N}(\phi)\right)^{l+1},
        \end{split}
    \end{equation}
    and, on the other hand,
    \begin{equation}\label{secondconditionnonisotropic}
        \forall (x,y) \in supp(\zeta) \qquad \|M(x,y) - H(x)\| \leq \frac{1}{2} \|H(x)^{-1}\|^{-1},
    \end{equation}

    where $M(x,y)$ is defined by \eqref{definitionofM}. Indeed, in the proof, one needs only prove that $M(x,y)$ satisfies this inequality on the support of $\zeta$, since its invertibility follows directly, and the coercivity of $\phi$ on the support of $\zeta$ as well. The computations are then the same since $\zeta$ satisfies the same estimates.

    \myindent The interest is that \eqref{secondconditionnonisotropic} can hold on larger, non isotropic, neighborhoods than the ball \eqref{defballmixedvdcabz}.
\end{remark}

\section{The case $\bullet = (H)$ : analysis of the phase}\label{secHormPhase}

\myindent In this section, we derive a rigorous analysis of the phase $\Psi$ appearing in the oscillatory integral $\mathcal{I}_{\lambda,\delta,i}^{(H)}(\sigma,A,B)$, expressed in the form \eqref{oscintexpr}, in the case $\bullet = (H)$, i.e. when $s_0 = t_0 = 0$. In particular, we will prove those of the results mentioned in paragraph \ref{subsubsec33Micro} which concerns this case. Since we consider this section as the core of the technical part of Theorem \ref{mainthm}, once it has been reduced to an estimate of oscillatory integrals as in Proposition \ref{intermediatethm}, we will detail essentially all the proofs, and try to show the important features of the geometry of the oscillatory integral analysis.

\subsection{General description}\label{subsec1HormPhase}

\myindent Following the strategy presented in Paragraph \ref{subsubsec32Micro}, let us decompose the phase into
\begin{equation}
    \Psi(\sigma,A,B,u,s,t,\xi) = -\scal{h(u)}{(A,B)} + \Psi_0(\sigma,u,s,t,\xi),
\end{equation}
where we recall that, thanks to the classification Proposition \ref{classification} there holds that
\begin{equation}
    \Psi_0(\sigma,u,s,t,\xi) = -\scal{h(u)}{(s,t)} + sq_1(\sigma,\xi) + \varphi((t,\sigma,),(0,\sigma), \xi)
\end{equation}
is independent of $A,B$. Here, $\varphi$ is defined by Theorem \ref{Hormthm}. Hence, any phase analysis using the $(s,t,\xi)$ gradient of $\Psi$ is independent of $(A,B)$. 

\myindent In order to navigate Section \ref{secHormPhase} in general, we start with a detailed presentation of the following computations in Paragraph \ref{subsubsec11HormPhase}.

\subsubsection{Outline of the analysis}\label{subsubsec11HormPhase}

\myindent In this paragraph, we give a few more details on the geometry of the set
\begin{equation}\label{defOsigma}
    \mathcal{O}_{\sigma} := \{(u,s,t,\xi)\in (\R/\ell \Z)\times 2\tilde{\mathcal{R}}_i^{(H)}\times \R^2 \qquad \text{such that} \qquad \nabla_{s,t,\xi} \Psi_0(\sigma,u,s,t,\xi) = 0\}.
\end{equation}

\myindent From the analysis of Paragraph \ref{subsubsec33Micro}, we recall that the geometry of this set is computed first by understanding the set where the \textit{geometric equation} \eqref{nablaetaphieq0} holds, i.e.
\begin{equation}
    \left\{\left(s,t,\frac{\xi}{|\xi|}\right) \qquad \nabla_{\xi} \Psi_0 \left(\sigma,s,t,\frac{\xi}{|\xi|}\right) = 0\right\},
\end{equation}
and then by applying the \textit{correspondence equation} $\nabla_{s,t}\Psi_0 = 0$ (see \eqref{nablastPsi0eq0}), which fixes the value of $u$ and $|\xi|$ as functions of $\left(s,t,\frac{\xi}{|\xi|}\right)$. As we will see, $|\xi|$ is never too small (near $0$) or too large (near infinity) on $\mathcal{O}_{\sigma}$. In particular, we may restrict the analysis to an annulus $|\xi| \in [\beta^{-1},\beta]$ provided $\beta > 0$ is large enough. The interest is double : on the one hand, this enables to restrict the $(u,s,t,\xi)$ domain of integration to a \textit{compact} domain of integration. On the other hand, and perhaps more importantly, this enables to make a \textit{polar change of coordinate} on $\xi$, which we decompose into an angle variable $w \in S^1$ and a modulus variable $r$. We will moreover use a particular polar decomposition, which is well adapted to the special form of the phase. These arguments will be presented in Paragraph \ref{subsubsec12HormPhase} In particular, in the rest of the present paragraph, we use the Notation \ref{defwrho}.

\quad

\myindent Now, regarding the computation of $\mathcal{O}_{\sigma}$, we recall that $\nabla_{w,r}\Psi_0$ vanishes at $(\sigma,s,t,w)$ if and only if there exist a bicharacteristic joining $(t,\sigma)$ and $(0,\sigma)$ in time $s$, and its direction is $\R_+^* \nabla_x \varphi((t,\sigma),(0,\sigma),w)$. Now, we recall that there are essentially two cases

\myindent i. For $t=0$, studied in Section \ref{subsec2HormPhase}, we consider bicharacteristics joining $(0,\sigma)$ to itself : they have length $s = 0$ and \textit{any} direction $w\in S^1$ is suitable. Hence, $s = t = 0$ and $w\in S^1$ arbitrary yields a zero of $\nabla_{w,r}\Psi_0$. In Paragraph \ref{subsubsec21HormPhase}, we will prove this result directly thanks to the asymptotic formula \eqref{asdevphi}, and, using moreover the \textit{correspondence equation}, we will describe a first subset of $\mathcal{O}_{\sigma}$ which is an (immersed) \textit{circle}, hence its name $\mathcal{C}_{\sigma}$. While this circle is naturally parameterized by $w$, in view of Theorem \ref{mixedVdCABZ}, we will moreover explain how to reparameterize $\mathcal{C}_{\sigma}$ by $u$ in Paragraph \ref{subsubsec22HormPhase}. Finally, we will study the $(s,t,w,r)$ Hessian of $\Psi_0$ on $\mathcal{C}_{\sigma}$ in Paragraph \ref{subsubsec23HormPhase}.

\myindent ii. For $t$ small but non zero, studied in Section \ref{subsec3HormPhase}, there is one, and exactly one, bicharacteristic joining $(t,\sigma)$ to $(0,\sigma)$. Its length is a continuous function of $t$, vanishing at $t = 0$. Moreover, the choice of direction of the bicharacteristics gives rise to two zeros. Observe that, in any case, the direction $\nabla_x \varphi$ approaches smoothly the \textit{horizontal axis} as $t \to 0$. Hence, the elements of $\mathcal{O}_{\sigma}$ with $t\neq 0$ can be seen as belonging to curves, naturally parameterized by $t$, which bifurcate from the circle $\mathcal{C}_{\sigma}$ at the points $P_0, P_{\pi} \in \mathcal{C}_{\sigma}$ corresponding respectively to $w = 0$ and $w = \pi$ (after using the correspondence equation to determine $r$ and $u$), hence their names $\mathcal{E}_{\sigma,0}$ and $\mathcal{E}_{\sigma,\pi}$. We will prove the existence and uniqueness of those exceptional branches of $\mathcal{O}_{\sigma}$ using \eqref{asdevphi} in Paragraph \ref{subsubsec31HormPhase}. While these curves are naturally parameterized by $t$, in order to apply Theorem \ref{mixedVdCABZ}, we will study how to reparameterize them by the variable $u$ when $\sigma \neq 0$ in Paragraph \ref{subsubsec32HormPhase}. We will afterwards in Paragraph \ref{subsubsec33HormPhase} study the $(s,t,w,r)$ Hessian of the phase on the branches. We will also explain how to deal with the case $\sigma$ close to the equator. Finally, when it is indeed possible to parameterize by $u$ (a portion of) the curves $\mathcal{E}_{\sigma,\alpha}$, the framework of Theorem \ref{mixedVdCABZ} requires an analysis of the \textit{1D remaining phase function}. This analysis is performed in Paragraph \ref{subsubsec34HormPhase}.

\quad

\myindent Now, as already mentioned in Paragraph \ref{subsubsec33Micro} a crucial feature of the analysis, which, as we will see, is linked to deep geometrical properties, is that $\mathcal{O}_{\sigma}$ is \textit{not} a smooth curve. Precisely, the smooth curves $\mathcal{C}_{\sigma}$ and $\mathcal{E}_{\sigma,0}$ (resp $\mathcal{E}_{\sigma,\pi}$) are \textit{transervsal} at their point of  intersection $P_0$ (resp $P_{\pi}$). One can guess that this totally prevents the strategy of Theorem \ref{mixedVdCABZ} from applying by isolating the variable $u$ near these points. Indeed, the fact that $\mathcal{O}_{\sigma}$ is not a manifold at $P_{\alpha}$ implies that the Hessian $\nabla^2_{s,t,w,r}\Psi_0$ is \textit{degenerate} at $P_{\alpha}$. Hence, one needs to resort to a very different argument near $P_{\alpha}$, and actually to isolate the variable $w$ instead of $u$. This part is probably the most delicate of the analysis, and we conjecture that it is actually responsible for the dominant term in the integral $\mathcal{I}_{\lambda,\delta,i}^{(H)}(\sigma,A,B)$, and we will explain it thoroughly in Section \ref{subsec4HormPhase}.

\begin{remark}
    From a geometrical point of view, the branching points $P_0$ and $P_{\pi}$ are directly linked to the fact that the Lagrangian torus $(\vec{p})^{-1}(\lambda,\mu)$, where $\frac{\mu}{\lambda} = f(\sigma)$ (see \eqref{deff-1lambdamu}) has a \textit{caustic} at the point $(0,\sigma)$ (see \cite{duistermaat1974oscillatory}). Hence, it is not surprising that they are responsible for the most delicate part of the analysis. Let us observe that the caustic is of fold-type, as was studied for example in \cite{colin1976regularite}.
\end{remark}

\myindent Overall, there are thus three regions of interest on $\mathcal{O}_{\sigma}$, as depicted in Figure \ref{Threeregion}.

\begin{figure}[h]
\includegraphics[scale = 0.5]{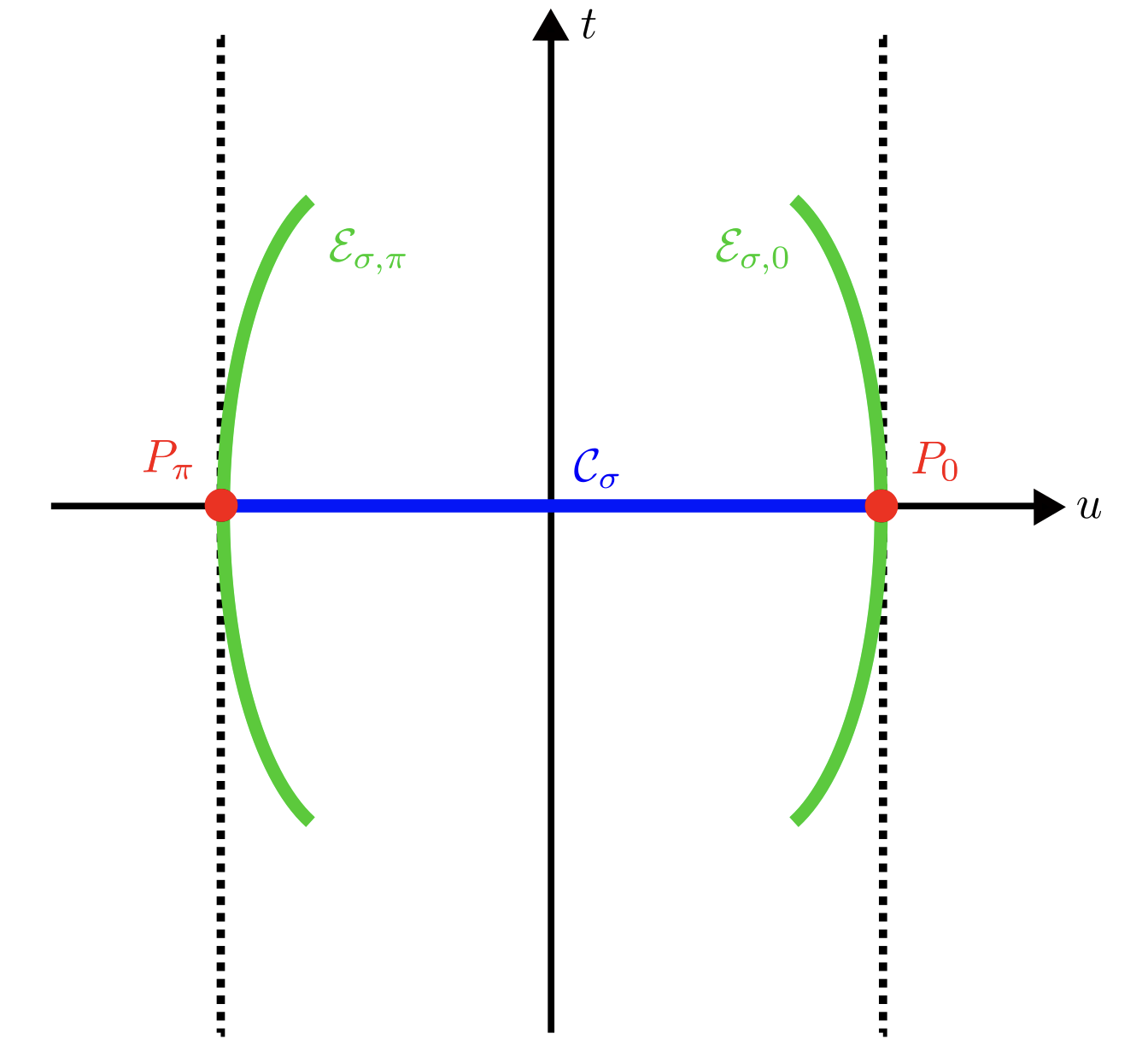}
\centering
\caption{The three region of interest of $\mathcal{O}_{\sigma}$ when $\bullet = (H)$}
\label{Threeregion}
\end{figure}

\subsubsection{Reducing to a compact domain of integration with polar coordinates}\label{subsubsec12HormPhase}

\myindent The first observation is that, thanks to the \textit{correspondence equation} \eqref{nablastPsi0eq0} we can deal with those $\xi$ which are either close to zero, or large enough, which are outside of $\mathcal{O}_{\sigma}$ i.e. at which the phase is non stationary. In particular, this will help us reduce the domain of integration on which there are singularities to a compact set where $|\xi| \sim 1$. 

\myindent Using the analysis of Paragraph \ref{subsubsec33Micro}, we know that, on the set where there are singularities, the value of $|\xi|$ is fixed by the correspondence equation \eqref{nablastPsi0eq0}. In this case, the equation \eqref{nablastPsi0} for the $(s,t)$ gradient can be written as
\begin{equation}\label{nablastpsi0expr}
    \begin{split}
    \nabla_{s,t}\Psi_0 &= -h(u) + \begin{pmatrix}
        q_1(\sigma,\xi)\\
        \partial_{\theta}\varphi((t,\sigma,),(0,\sigma),\xi)
    \end{pmatrix}\\
    &= -h(u) + \begin{pmatrix}
        q_1(\sigma,\nabla_x\varphi)\\
        q_2(\sigma,\nabla_x\varphi)
    \end{pmatrix}.
    \end{split}
\end{equation}

\myindent Now, we have already argued that this can vanish if and only if \eqref{eqmodulusxi} holds, i.e. in this case if and only if
\begin{equation}
    p_1(\sigma,\nabla_x\varphi((t,\sigma,),(0,\sigma),\xi)) = 1,
\end{equation}
and $u$ is fixed by the equation $h(u) =\begin{pmatrix}
        q_1(\sigma,\nabla_x\varphi)\\
        q_2(\sigma,\nabla_x\varphi)
    \end{pmatrix}$. However, thanks to the asymptotic expansion \eqref{asdevphi}, we know that
\begin{equation}\label{devnablaxphi}
    \nabla_x \varphi((t,\sigma,),(0,\sigma),\xi) = \xi + O(t|\xi|).
\end{equation}

\myindent Hence, since $|t| \ll 1$, there holds
\begin{equation}
    p_1(\sigma,\nabla_x \varphi ((t,\sigma,),(0,\sigma),\xi) \sim p_1(\sigma,\xi).
\end{equation}

\myindent In particular, for $\xi$ outside of a neighborhood of $\{\xi, p_1(\sigma,\xi) = 1\}$, the gradient $\nabla_{s,t}\Psi_0$ cannot vanish. Precisely, there holds the following.

\begin{lemma}\label{nicebound!}
    Let $\beta > 0$ be large enough, depending on $\mathcal{S}$. Then, there holds for any $(\sigma,s,t) \in 2\tilde{\mathcal{Q}}_i^{(H)}$, and any $(u,\xi)$,
    \begin{equation}
        |\nabla_{s,t}\Psi_0(\sigma,u,s,t,\xi)|\gtrsim \begin{cases}
            1 \qquad |\xi| \leq 2\beta^{-1} \\
            |\xi| \qquad |\xi| \geq \frac{1}{2}\beta
        \end{cases}.
    \end{equation}
\end{lemma}

\myindent Hence, outside of a neighborhood of $|\xi| \sim 1$, we will be able to integrate by parts in $(s,t)$ and find that the contribution of this region is $O(\lambda^{-\infty})$, as we will prove in Paragraph \ref{subsubsec11HormQuant} Hence, we now analyse the phase only in the region
\begin{equation}
    |\xi| \in [\beta^{-1},\beta].
\end{equation}

\myindent Observe that this means that $q_1(\sigma,\xi)$ is bounded away from zero and from $\infty$ (indeed, $q_1$ is elliptic, see Corollary \ref{q1elliptic}). In particular, we can make a \textit{polar} change of coordinates in the region $|\xi| \in [\beta^{-1},\beta]$ which is well adapted to the problem.

\begin{figure}[h]
\includegraphics[scale = 0.5]{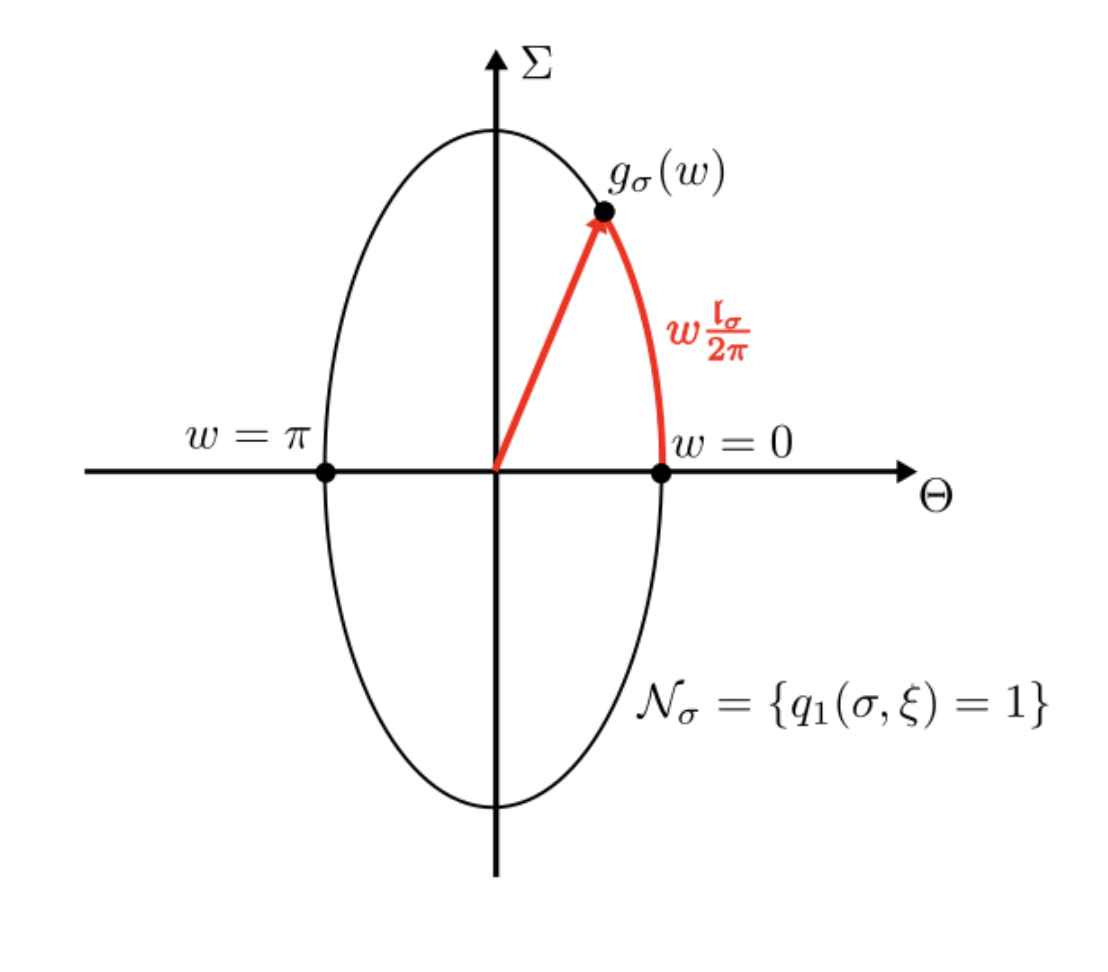}
\centering
\caption{The parameterization of the curve $\mathcal{N}_{\sigma}$}
\label{pictureNsigma}
\end{figure}

\begin{notation}\label{defwrho}
    We parameterize the curve 
    \begin{equation}\label{defNsigma}
        \mathcal{N}_{\sigma} := \{\xi \in \R^2 \qquad q_1(\sigma,\xi) = 1\} \subset \R^2 \backslash\{0\}
    \end{equation}
    by
    \begin{equation}
        w\in S^1 \mapsto g_{\sigma}(w)
    \end{equation}
    positively oriented, and such that $|g_{\sigma}'(w)| = cste =: \frac{\mathfrak{l}_{\sigma}}{2\pi}$ and $g_{\sigma}(0)\in \R_+^* (1,0)$, as in Figure \ref{pictureNsigma}. Here, $\mathfrak{l}_{\sigma}$ is the \textit{length} of the curve $\mathcal{N}_{\sigma}$ in $\R^2$.

    \myindent Vectors $\xi \in \R^2$ can thus be described via the following polar coordinates
    \begin{equation}
        \xi = r g_{\sigma}(w), \qquad r \in \R_+, w \in S^1,
    \end{equation}
    and, if $F(\xi)$ is a function of $\xi$, we use the notation $F(w,r)$ to denote the function $(w,r) \mapsto F(r g_{\sigma}(w))$.
\end{notation}

\myindent The interest of this description is that $\Psi_0$ takes the nice form
\begin{equation}
    \Psi_0(\sigma,u,s,t,w,r) = -\scal{h(u)}{(s,t)} + r\left(s + \varphi((t,\sigma,),(0,\sigma),w)\right),
\end{equation}

where we use the homogeneity of $\varphi$ in the $\xi$ variable and the notation $\varphi((t,\sigma,),(0,\sigma),w) := \varphi((t,\sigma,),(0,\sigma), g_{\sigma}(w))$. The reader may also observe that this form is very close to the expression \eqref{bicharlengthparam}.

\myindent We also define
\begin{equation}\label{defKsigma}
    \mathcal{K}_{\sigma} := \left\{(w,r) \in S^1 \times \R_+;\ \text{such that}\ |rg_{\sigma}(w)| \in [\beta^{-1},\beta]\right\}
\end{equation}

\subsection{The circle $\mathcal{C}_{\sigma}$ of $(s,t,w,r)$ stationary points of $\Psi_0$ with $s = t = 0$}\label{subsec2HormPhase}

\myindent In this section, we define rigorously the first part of the set $\mathcal{O}_{\sigma}$ (see \eqref{defOsigma}), namely we prove that those points $(u,s,t,w,r) \in \mathcal{O}_{\sigma}$ such that $s = t = 0$ form an embedded circle $\mathcal{C}_{\sigma}$, naturally parameterized by $w\in S^1$. Moreover, we study the properties of $\mathcal{C}_{\sigma}$ that we will need for Section \ref{secHormQuant}.

\subsubsection{Definition of $\mathcal{C}_{\sigma}$}\label{subsubsec21HormPhase}

\myindent Let us observe that, at a point $(u,s,t,w,r)$ such that $s = t = 0$, there holds using \eqref{asdevphi}
\begin{equation}
    \nabla_{s,t,w,r} \Psi_0(\sigma,u,0,0,w,r) = \begin{pmatrix}
        -h(u)_1 + r  \\
        -h(u)_2 + r q_2(\sigma,w,1)\\
        0 \\
        0
    \end{pmatrix},
\end{equation}
which vanishes if and only if
\begin{equation}
    h(u) = r\begin{pmatrix}
        1\\
        q_2(\sigma,w,1)
    \end{pmatrix}.
\end{equation}

\myindent Now, the curve $u\mapsto h(u)$ describes the set $F_2 = 1$, where by definition $F_2(q_1,q_2) = p_1^2$. Hence, this equality holds if and only if $p_1(\sigma,w,r) = 1$ i.e. if and only if $r = \frac{1}{p_1(\sigma,w,1)} =: r_{\sigma}(w)$ and $u = u_{\sigma}(w)$ is uniquely defined by $w$. We have thus proved the following lemma.

\begin{lemma}\label{defcirclestatpoints}
    For all $w \in S^1$, and all $\sigma$, let $u_{\sigma}(w)$ and $r_{\sigma}(w)$ be defined by
    \begin{equation}
        \begin{split}
        h(u_{\sigma}(w)) &= \begin{pmatrix}
        q_1(\sigma,w,r_{\sigma}(w))\\
        q_2(\sigma,w,r_{\sigma}(w))
        \end{pmatrix}\\
        &= \frac{1}{p_1(\sigma,w,1)}\begin{pmatrix}
            1\\
            q_2(\sigma,w,1)
        \end{pmatrix},
        \end{split}
    \end{equation}
    so that $w \mapsto u_{\sigma}(w)$ (resp $w \mapsto r_{\sigma}(w)$) is a smooth function from $S^1$ to $\R/\ell\Z$ (resp $\R_+^*$). Then for any $\sigma$ the set of $(s,t,w,r)$ stationary points of $\Psi_0$ with $s = t = 0$, defined by
    \begin{equation}
    \begin{split}
        &\bigg\{(u,s,t,w,r) \in \mathcal{O}_{\sigma} \qquad \text{such that} \qquad s = t = 0 \bigg\} \\
        = &\bigg\{(u,0,0,w,r) \qquad \ \text{such that} \ \nabla_{s,t,w,r}\Psi_0(\sigma,u,0,0,w,r) = 0\bigg\},
    \end{split}
    \end{equation}
    is the circle
    \begin{equation}
        \mathcal{C}_{\sigma} = \bigg\{\left(u_{\sigma}(w),0,0,w,r_{\sigma}(w)\right) \qquad w \in S^1\bigg\}.
    \end{equation}
\end{lemma}

\myindent We call this set a circle since it is obviously an embedded circle in a five-dimensional manifold.

\subsubsection{Parameterization of the circle $\mathcal{C}_{\sigma}$ by $u$}\label{subsubsec22HormPhase}

\myindent Since $\mathcal{C}_{\sigma}$ is a curve of $(s,t,w,r)$ stationary points of $\Psi_0$, in order to apply the strategy of Theorem \ref{mixedVdCABZ}, we need to reparameterize it, at least locally, by $u$. In particular, we will be able to do so wherever $w \mapsto u_{\sigma}(w)$ is a bijection. Now, a first observation is that, since, by Lemma \ref{formulaofq1},
\begin{equation}
    q_1(\sigma,\Theta,\Sigma) = G(p_1(\sigma,\Theta,\Sigma),\Theta),
\end{equation}
and since, by definition, $p_1(\sigma,\Theta,\Sigma)$ is an \textit{even} function of $\Sigma$ (see \eqref{normsurfrev}), then $q_1$ is an even function of $\Sigma$. In particular, the curve $\mathcal{N}_{\sigma}$ (defined by \eqref{defNsigma}) parameterized by $w \mapsto g_{\sigma}(w)$ is symmetric with respect to the $\Theta$ axis. This implies that $w\mapsto p_1(\sigma,w,1)$ and $w \mapsto q_2(\sigma,w,1)$ are \textit{even} functions of $w \in S^1$. Thus, $u_{\sigma}(w)$ is also an \textit{even} function of $w$. In particular, $u_{\sigma}'(0) = u_{\sigma}'(\pi) = 0$. This hints to the fact that the points $w = 0,\pi$ accumulate most of the singularities, which we will detail below.

\begin{definition}\label{defUsigma}
    Let $\mathcal{U}_{\sigma}$ be the segment of $h^{-1}(\gamma_0$) with extremities $u_{\sigma}(\pi)$ and $u_{\sigma}(0)$, where we recall that, for $\alpha = 0, \pi$
    \begin{equation}
        h(u_{\sigma}(\alpha)) = \frac{1}{p_1(\sigma,\alpha,1)}\begin{pmatrix}
            1\\
            q_2(\sigma,\alpha,1)
        \end{pmatrix}.
    \end{equation}

    \myindent In particular, $h(u_{\sigma}(0))$ and $h(u_{\sigma}(\pi))$ are symmetric with respect to the axis $\R(1,0)$.
\end{definition}

\myindent Geometrically, by identifying $u \in \R/\ell\Z$ and $h(u) \in \gamma$ (see Definition \ref{deflilgammas}), the segment $\mathcal{U}_{\sigma}$ can be seen as the part of the curve $\gamma$ (and even of $\gamma_0$) which is reached by  $(q_1,q_2)$ over the cotangent space at $(\theta,\sigma)$ for any $\theta \in S^1$. Observe that, when $\sigma = 0$, then $\mathcal{U}_{\sigma} = \gamma_0$ is maximal, and, when $\sigma \to \pm\frac{L}{2}$, then $\mathcal{U}_{\sigma}$ converges to a single point $\R_+^*(1,0)\cap \gamma$. In particular, $\mathcal{U}_{\sigma}$ is the only part of $\gamma$ which is relevant to the analysis at $\sigma$ fixed, since the \textit{correspondence equation} \eqref{nablastPsi0eq0} yields that the phase $\Psi(\sigma,\cdot)$ is never stationary at $u \notin \mathcal{U}_{\sigma}$. Now, an unavoidable fact is that $\mathcal{N}_{\sigma}\to \mathcal{U}_{\sigma}$ is obviously \textit{not} a bijection. Indeed, as can be seen on Figure \ref{NsigmatoUsigma}, $g_{\sigma}(w) \mapsto q_2(\sigma,w,1)$ acts like a sort of cosine. Then, $q_2(\sigma,w,1) \mapsto h(u_{\sigma}(w))$ is a bijection. Overall, there holds the following.

\begin{figure}[h]
\includegraphics[scale = 0.5]{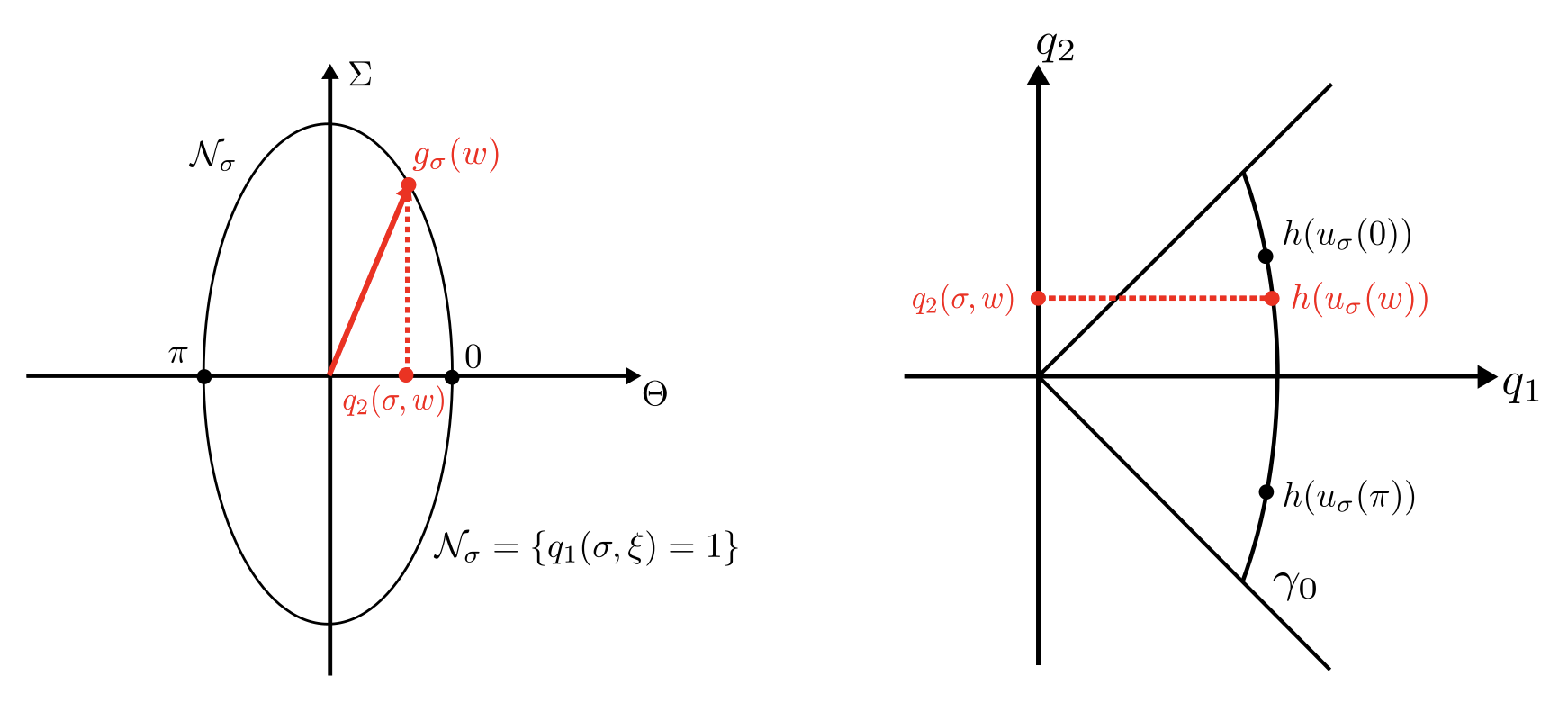}
\centering
\caption{The twofold surjection $\mathcal{N}_{\sigma} \to \mathcal{U}_{\sigma}$}
\label{NsigmatoUsigma}
\end{figure}

\begin{lemma}\label{locinverseuw}
    $w \mapsto u_{\sigma}(w)$ is a smooth even surjection from $S^1 = \R/2\pi\Z$ to $\mathcal{U}_{\sigma} = [u_{\pi}(\sigma),u_0(\sigma)]$, such that
    \begin{equation}
        u_{\sigma}'(w) = 0 \iff w \in \{0,\pi\}.
    \end{equation}
    
    \myindent In particular, $u_{\sigma}$ admits two right inverses, depending on whether we turn clockwise or counterclockwise on $S^1$, say
    \begin{equation}
        u \mapsto \begin{cases}
            w_{+}(\sigma,u) \in [0,\pi] \\
            w_{-}(\sigma,u) \in [-\pi, 0],
        \end{cases}
    \end{equation}
    such that $w_{+}(\sigma,u_{\sigma}(\pi)) = w_{-}(\sigma,u_{\sigma}(0)) = 0$. Moroever, $u_{\sigma}''(0) \neq 0$ and $u_{\sigma}''(\pi) \neq 0$, hence there holds locally near $u_{\sigma}(\alpha)$, for $\alpha = 0,\pi$
    \begin{equation}
        |w_{\pm}(\sigma,u)| \simeq \sqrt{|u - u_{\sigma}(\alpha)|}.
    \end{equation}
\end{lemma}

\begin{proof}
    We drop the $\sigma$ index in the notation in the proof. We compute 
    \begin{equation}
        h'(u(w))u'(w) = -\frac{\partial_w p_1(\sigma,w,1)}{p_1(x,w,1)} h(u(w)) + \frac{1}{p_1(\sigma,w,1)}\begin{pmatrix}
            0\\
            \partial_w q_2(\sigma,w,1)
        \end{pmatrix}.
    \end{equation}

    \myindent Taking the determinant against $h(u(w)) = \frac{1}{p_1(\sigma,w,1)}\begin{pmatrix}1 \\
    q_2(\sigma,w,1)\end{pmatrix}$ we find that
    \begin{equation}
        \begin{split}
        \det(h'(u(w)),h(u(w))) u'(w) &= \frac{1}{p_1(\sigma,w,1)^2} \det(\begin{pmatrix}
            0\\
            \partial_w q_2(\sigma,w,1)\end{pmatrix}, \begin{pmatrix}
                1 \\
                q_2(\sigma,w,1)
            \end{pmatrix})\\
        &= \frac{1}{p_1(\sigma,w,1)^2} \partial_w q_2(\sigma,w,1).
        \end{split}
    \end{equation}

    \myindent Now, using Lemma \ref{deth'hnotzero}, we find that
    \begin{equation}
        \frac{1}{p_1(\sigma,w,1)^2 \det(h'(u(w)),h(u(w)))} \simeq 1.
    \end{equation}

    \myindent Moreover, we claim that there holds
    \begin{equation}\label{partialwq2issin}
        \partial_w q_2(\sigma,w,1) \simeq \sin(w),
    \end{equation}
    in the sense that there exists a smooth function which doesn't vanish, say $f(\sigma,w)$, such that
    \begin{equation}
        \partial_w q_2(\sigma,w,1) = f(\sigma,w) \sin(w).
    \end{equation}

    \myindent We postpone the proof of this assertion to Lemma \ref{strangewayofsaying}. In particular, we find that 
    \begin{equation}
        u'(w) = 0 \iff w \in \{0,\pi\},
    \end{equation}
    which proves the first part of the lemma.

    \myindent Moreover, we find that, for small $w$,
    \begin{equation}
        u'(w) = C w + O(w^2),
    \end{equation}
    where 
    \begin{equation}
        C = \frac{f(\sigma,0)}{p_1(\sigma,0,1)\det(h'(u(0)),h(u(0)))} \neq 0.
    \end{equation}

    \myindent Thus, because $u$ is smooth, we find that $u''(0) \neq 0$. One proves similarly that $u''(\pi)\neq 0$.
\end{proof}

\myindent We now prove the identity \eqref{partialwq2issin}.

\begin{lemma}\label{strangewayofsaying}
    The function
    \begin{equation}
        w \mapsto \frac{\partial_w q_2(\sigma,w,1)}{\sin(w)}
    \end{equation}
    is well-defined on $S^1$, smooth, and doesn't vanish. In particular, it is uniformly bounded away from zero in absolute value for $(\sigma,w) \in \mathcal{L}\times S^1$.
\end{lemma}

\begin{proof}
    First, let us observe that, for any function $f(\xi) = f(w,r)$, there holds the following formula
    \begin{equation}
        \partial_w f (w,r) = \frac{\mathfrak{l}_{\sigma}}{2\pi}r \det(\frac{\nabla q_1(\sigma,w,1)}{|\nabla q_1(\sigma,w,1)|},\nabla f(w,r)),
    \end{equation}

    where we recall that the constant $\mathfrak{l}_{\sigma}$ is defined by the Notation \ref{defwrho}. Indeed, writing $\xi = r g(w)$, one can compute
    \begin{equation}
        \begin{split}
            \partial_w f(w,r) &= \nabla f(w,r)\cdot r g'(w) \\
            &= \frac{\mathfrak{l}_{\sigma}}{2\pi} r \nabla f(w,r) \cdot \left(\frac{\nabla q_1(\sigma,w,1)}{|\nabla q_1(\sigma,w,1)|}\right)^{\perp} \\
            &= \frac{\mathfrak{l}_{\sigma}}{2\pi}r \det(\frac{\nabla q_1(\sigma,w,1)}{|\nabla q_1(\sigma,w,1)|},\nabla f(w,r)),
        \end{split}
    \end{equation}
    where we use that $g'(w)$ is the tangent vector to the curve $q_1 = 1$, so it is the orthogonal of the normal direction to this curve, given by $\frac{\nabla q_1(\sigma,w,1)}{|\nabla q_1(\sigma,w,1)|}$ (the sign comes from our choice of orientation of the curve $q_1 = 1$).

    \myindent Since $w \mapsto |\nabla q_1(\sigma,w,1)|$ doesn't vanish, it is enough to study
    \begin{equation}
        w \mapsto \det(\nabla q_1(\sigma,w,1),\nabla q_2(\sigma,w,1)).
    \end{equation}

    \myindent Now, we compute in $(\Theta,\Sigma)$ coordinates
    \begin{equation}
        \det(\nabla q_1(\sigma,w,1),\nabla q_2(\sigma,w,1)) = \begin{vmatrix}
            \partial_{\Sigma} q_1 & 0 \\
            \partial_{\Theta}q_1 & 1 
        \end{vmatrix} = \partial_{\Sigma}q_1.
    \end{equation}

    \myindent Now, it is a direct consequence of Lemma \ref{dersigq1lemma} that
    \begin{equation}
        \partial_{\Sigma}q_1(\sigma,w,1) \simeq \sin(w).
    \end{equation}
\end{proof}

\subsubsection{The $(s,t,w,r)$ Hessian of $\Psi_0$ on the circle $\mathcal{C}_{\sigma}$}\label{subsubsec23HormPhase}

\myindent Since $\mathcal{C}_{\sigma}$ is a set of zeros of $\nabla_{s,t,w,r}\Psi_0$, it is natural to compute its Hessian is the variable $(s,t,w,r)$, and in particular the invertibility of its Hessian, on $\mathcal{C}_{\sigma}$. Indeed, this allows to perform a stationary phase analysis in the variables $(s,t,w,r)$. 

\begin{lemma}\label{stwrhohessofPsionCx}
    Let $\mathcal{C}_{\sigma}$ be parameterized by $w \in S^1$ as in Lemma \ref{defcirclestatpoints}. Then there holds along $\mathcal{C}_{\sigma}$
    \begin{equation}
        \det(\nabla_{s,t,w,r}^2 \Psi_0(\sigma,u_{\sigma}(w), 0,0,w,r_{\sigma}(w))) \simeq (\sin(w))^2,
    \end{equation}
    in the sense that the function
    \begin{equation}
        (\sigma,w) \mapsto \frac{\det(\nabla_{s,t,w,r}^2 \Psi_0(\sigma,u_{\sigma}(w), 0,0,w,r_{\sigma}(w)))}{(\sin(w))^2}
    \end{equation}
    is smooth and uniformly bounded from zero for $(\sigma,w) \in\mathcal{L}\times S^1$.
\end{lemma}

\begin{proof}
We drop the $\sigma$ index in the notations in the proof. Let $(u(w),0,0,w,r(w))\in \mathcal{C}_{\sigma}$, there holds
\begin{equation}\label{explicitstrhohessonCsig}
\begin{split}
    &\nabla_{s,t,w,r}^2\Psi_0(\sigma,u(w),0,0,w,r(w))\\
    &= 
    \begin{pmatrix}
        0 & 0 & 0 & 1 \\
        0 & r(w)\partial_{\theta\theta}\varphi((0,\sigma,),(0,\sigma),w) & r(w)\partial_{\theta w}\varphi((0,\sigma,),(0,\sigma),w) & \partial_{\theta}\varphi((0,\sigma,),(0,\sigma),w) \\
        0 & r(w)\partial_{\theta w}\varphi((0,\sigma,),(0,\sigma),w) & r(w)\partial_{ww}\varphi((0,\sigma,),(0,\sigma),w) & 0 \\
        1 & \partial_{\theta}\varphi((0,\sigma,),(0,\sigma),w) & 0 & 0
    \end{pmatrix} \\
    &= \begin{pmatrix}
        0 & 0 & 0 & 1 \\
        0 & r(w)\partial_{\theta\theta}\varphi((0,\sigma,),(0,\sigma),w) & r(w) \partial_w q_2(\sigma,w,1) & q_2(\sigma,w,1) \\
        0 & r(w) \partial_w q_2(\sigma,w,1) & 0 & 0 \\
        1 & q_2(\sigma,w,1) & 0 & 0
    \end{pmatrix},
\end{split}
\end{equation}

where we use \eqref{asdevphi} for the last equality. Now, the Hessian is thus, modulo an exchange of the first and last column, a block triangular matrix, hence its determinant is easy to compute. Using moreover Lemma \ref{strangewayofsaying}, we find that 
\begin{equation}
    \begin{split}
    \det(\nabla_{s,t,w,r}^2\Psi_0(\sigma,u(w),0,0,w,r(w)))  &= r(w)^2 (\partial_w q_2(\sigma,w,1))^2 \\
    &= F(\sigma,w) \sin(w)^2,
    \end{split}
\end{equation}
for some smooth positive function $F$ uniformly bounded from zero in $(\sigma,w)$.
\end{proof}

\myindent Now, this fact has two consequences

\myindent i. First, $\nabla_{s,t,w, r}^2 \Psi_0$ is \textit{degenerate} at $(u(w),0,0,w,r(w)) \in \mathcal{C}_{\sigma}$ if and only if $w = 0$ or $w = \pi$. This is natural since, if we define, for $\alpha = 0,\pi$,
\begin{equation}\label{defPalpha}
    P_{\alpha} := (u_{\sigma}(\alpha), 0,0,\alpha, r_{\sigma}(\alpha)),
\end{equation}
we have mentioned in Paragraph \ref{subsubsec11HormPhase} that there are singularities of the set $\mathcal{O}_{\sigma}$ at the branching points $P_0$ and $P_{\pi}$, hence the Hessian \textit{has} to degenerate.

\myindent ii. Second, using the two branches of inverse of $u(w)$ given by Lemma \ref{locinverseuw}, we find that, for $u \in (u_{\sigma}(\pi),u_{\sigma}(0))$, $(s,t,w, r)\mapsto \Psi_0(\sigma,u,s,t,w, r)$ has exactly two stationary points at which the Hessian is nondegenerate. Thus, half of the conditions of Theorem \ref{mixedVdCABZ} with the $u$ variable isolated are satisfied. It remains to study the remaining phase function. We will perform this analysis in Paragraph \ref{subsubsec22HormQuant}.

\quad

\myindent Now, in view of Theorem \ref{mixedVdCABZ}, we need more precisely a quantitative bound on the norm of the inverse of the $(s,t,w,r)$ Hessian of $\Psi_0$ on $\mathcal{C}_{\sigma}$. Lemma \ref{stwrhohessofPsionCx} yields the following corollary.

\begin{corollary}\label{norminversehessCsig}
    There holds on $\mathcal{C}_{\sigma}$
    \begin{equation}
        \left\|\left(\nabla_{s,t,w,r}^2 \Psi_0(\sigma, u_{\sigma}(w), 0,0, w,r_{\sigma}(w))\right)^{-1} \right\| \lesssim \frac{1}{|\sin(w)|}.
    \end{equation}
\end{corollary}

\begin{proof}
    Let us write
    \begin{equation}
        H(w) := \nabla_{s,t,w,r}^2 \Psi_0(\sigma,u_{\sigma}(w), 0,0, w,r_{\sigma}(w)).
    \end{equation}
    
    \myindent Then, we know that
    \begin{equation}
        (H(w))^{-1} = \frac{1}{\det(H(w))} ^tCom(H(w)),
    \end{equation}
    where $^t Com(H(w))$ is the matrix of cofactors of $H(w)$. Now, thanks to Lemma \ref{stwrhohessofPsionCx}, we know that
    \begin{equation}
        \left|\frac{1}{\det(H(w))}\right| \lesssim \frac{1}{\sin(w)^2}.
    \end{equation}
    
    \myindent However, looking at the explicit form of $H(w)$ given by \eqref{explicitstrhohessonCsig}, one can see that all the cofactors of $H(w)$ are $\lesssim |\sin(w)|$. Indeed, this comes from the observation that
    \begin{equation}
        \left|(r(w)\partial_{\theta\theta}\varphi((0,\sigma,),(0,\sigma),w), r(w)\partial_{\theta w}\varphi((0,\sigma,),(0,\sigma),w))\right| \lesssim |\sin(w)|,
    \end{equation}
    since the quantity on the LHS vanishes at $w= 0,\pi$ thanks to \eqref{asdevphi} and Lemma \ref{miracleequation}.
\end{proof}

\subsection{The exceptionnal branches $\mathcal{E}_{\sigma,\alpha}$ of $(s,t,w,r)$ stationary points of $\Psi_0$ with $s,t \neq 0$}\label{subsec3HormPhase}

\myindent In this section, we define rigorously the second part of the set $\mathcal{O}_{\sigma}$, which is formed of the two exceptionnal branches $\mathcal{E}_{\sigma,\alpha}$, $\alpha= 0,\pi$, of zeros with $(s,t) \neq (0,0)$, with branch from $\mathcal{C}_{\sigma}$ at $P_0$ and $P_{\pi}$.

\subsubsection{Definition of $\mathcal{E}_{\sigma,\alpha}$}\label{subsubsec31HormPhase}

\myindent In this paragraph, we will prove the following lemma.

\begin{lemma}\label{newhorriblelemma}
    For all $\sigma \in \mathcal{L}$, there exist two disjoint smooth curves of zeros of $\nabla_{s,t,w,r}\Psi_0$ parameterized by $t$ 
    \begin{equation}
        \mathcal{E}_{\sigma,\alpha} := \bigg\{\left(u_{\alpha}(\sigma,t), s_{\alpha}(\sigma,t), t, w_{\alpha}(\sigma,t),r_{\alpha}(\sigma,t) \right)\qquad |t| \ll 1\bigg\},
    \end{equation}
    where $\alpha \in \{0,\pi\}$, and $u_{\alpha},s_{\alpha},w_{\alpha},r_{\alpha}$ are smooth functions of $(\sigma,t)$ such that, with the definition \eqref{defPalpha}
    \begin{equation}
        (u_{\alpha}(\sigma,0), s_{\alpha}(\sigma,0), w_{\alpha}(\sigma,0), r_{\alpha}(\sigma,0)) = (u_{\sigma}(\alpha), 0,\alpha, r_{\sigma}(\alpha)) = P_{\alpha} \in C_{\sigma}.
    \end{equation}
    
    \myindent Moreover, all the zeros of $\nabla_{s,t,w,r} \Psi_0$ with $(s,t) \neq 0$ are on $\mathcal{E}_{\sigma,0}\cup \mathcal{E}_{\sigma,\pi}$ in the sense that
    \begin{equation}
        \forall (u,s,t,w,r) \ \text{such that} \ (s,t) \neq 0: \qquad \nabla_{s,t,w,r} \Psi_0(\sigma,u,s,t,w,r) = 0 \iff (u,s,t,w,r) \in \mathcal{E}_{\sigma,0}\cup \mathcal{E}_{\sigma,\pi}.
    \end{equation}
\end{lemma}

\begin{proof}
    First, we recall that the \textit{correspondence equation} $\nabla_{s,t}\Psi_0 = 0$ only fixes the value of $u$ and $r$ as functions of $\sigma,s,t,w$. Moreover, the equation $\partial_{r}\Psi_0 = 0$ only fixes $s$ as a function of $w$ and $t$ since it reads
    \begin{equation}
        s + \varphi((t,\sigma),(0,\sigma),w) = 0.
    \end{equation}
    
    \myindent Hence, the only difficulty is to compute the zero set of
    \begin{equation}
        \partial_w \Psi_0 = r \partial_w \varphi((t,\sigma,),(0,\sigma),w),
    \end{equation}
    i.e., dropping the $r$, we need only study the zero set of $\partial_w \varphi((t,\sigma),(0,\sigma),w)$. We start by proving the existence of the branch $\mathcal{E}_{\sigma,\alpha}$ around $P_{\alpha}$, i.e. we want to prove that for small $t$ there is a smooth function
    \begin{equation}\label{ttowalphat}
        t \mapsto w_{\alpha}(\sigma,t)
    \end{equation}
    such that
    \begin{equation}
        \partial_w \varphi((t,\sigma),(0,\sigma),w_{\alpha}(\sigma,t)) = 0.
    \end{equation}
    
    \myindent This is obviously reminiscent of the implicit function theorem. However, the problem is that, using \eqref{asdevphi}, there holds
    \begin{equation}
        \partial_{ww} \varphi((0,\sigma),(0,\sigma),w) = 0
    \end{equation}
   for all $w$, thus we may not apply the implicit function theorem (indeed, this would give that the \textit{only} zeros of $\partial_w \varphi((t,\sigma),(0,\sigma),w)$ are given by \eqref{ttowalphat}, which is obviously false since $t = 0, w \in S^1$ is another branch of zeros). However, since we have in mind to investigate the zero set of $\partial_w \varphi$ with nonzero $t$, we may rather look at the auxiliary function
   \begin{equation}
       F(\sigma,t,w) := \frac{1}{t} \partial_w \varphi((t,\sigma),(0,\sigma),w).
   \end{equation}
   
   \myindent The asymptotic expansion \eqref{asdevphi} ensures that this is a smooth function of $(\sigma,t,w)$ such that
   \begin{equation}
       F(\sigma,0,\alpha) = 0.
   \end{equation}
   
   \myindent Moreover, we claim that
   \begin{equation}
       \partial_w F(\sigma,0,\alpha) = \partial_{\theta w w} \varphi((0,\sigma),(0,\sigma),\alpha) \neq 0.
   \end{equation}
   
   \myindent The proof of this assertion is not difficult, but, since we will later need to use the exact value of $\partial_{\theta w w} \varphi$, we refer the reader to Lemma \ref{thirdthetawderivaticevarphi}.
   
   \myindent Thus, the implicit function theorem yields that there exists a neighborhood $\mathcal{U}$ of $t = 0, w = 0$ and a smooth function $t \mapsto w_{\alpha}(\sigma,t)$ such that
   \begin{equation}
       \forall (t,w) \in \mathcal{U}, \ F(\sigma,t,w) = 0 \iff w = w_{\alpha}(\sigma,t).
   \end{equation}
   
   \quad
   
   \myindent We also need to prove that there are no zeros of $F$ which are not close to $w = 0$ or $w = \pi$. However, we observe that, thanks to \eqref{asdevphi},
   \begin{equation}
   \begin{split}
       F(\sigma,t,w) &= \partial_{\theta w}\varphi((0,\sigma,),(0,\sigma),w) + O(t) \\
       &= \partial_w q_2(\sigma,w,1) + O(t)
    \end{split}
   \end{equation}
   uniformly for all $w\in S^1$ and all small enough $t$. Now, using Lemma \ref{partialwq2issin}, it is obvious that $F$ cannot vanish for $w$ not close to one of $0$ or $\pi$.
   
   \quad 
   
   \myindent Coming back to the equation $\nabla_{s,t,r}\Psi_0 = 0$, we finally have the useful following exact formula for the branches $\mathcal{E}_{\sigma,\alpha}$
   \begin{equation}\label{defsrhoualpha}
   \begin{cases}
       s_{\alpha}(\sigma,t) = - \varphi((t,\sigma,),(0,\sigma),w_{\alpha}(\sigma,t)) \\
       r_{\alpha}(\sigma,t) = \frac{1}{p_1(\sigma,\nabla_x\varphi((t,\sigma,),(0,\sigma),w_{\alpha}(\sigma,t)))} \\
       h(u_{\alpha}(\sigma,t)) = r_{\alpha}(\sigma,t) \begin{pmatrix} 1 \\
       \partial_{\theta}\varphi((t,\sigma,),(0,\sigma) w_{\alpha}(\sigma,t))\end{pmatrix}
    \end{cases}.
   \end{equation}
\end{proof}

\myindent As a corollary of the proof, there holds the following.
 
\begin{lemma}\label{quantitativelowerboundnablastwr}
    There holds
    \begin{equation}
        |\nabla_{s,t,w,r}\Psi_0(\sigma,u,s,t,w,r)| \gtrsim |(s,t)| \min_{\alpha = 0,\pi} |(u,s,w,r) - (u_{\alpha}(\sigma,t), s_{\alpha}(\sigma,t), w_{\alpha}(\sigma,t), r_{\alpha}(\sigma,t))|
    \end{equation}
\end{lemma}

\subsubsection{Parameterization of the branches $\mathcal{E}_{\sigma,\alpha}$ by $u$}\label{subsubsec32HormPhase}

\myindent Similarly to Paragraph \ref{subsubsec22HormPhase}, since $\mathcal{E}_{\sigma,\alpha}$ is a curve of zeros of $\nabla_{s,t,w,r}\Psi$, naturally parameterized by $t$, we need to reparameterize it, where it is possible, by the variable $u$, in order to implement the strategy of Theorem \ref{mixedVdCABZ}. We may use that description if and only if $t \mapsto u_{\alpha}(\sigma,t)$ is (locally) invertible, which, geometrically, means that the curve $\mathcal{E}_{\sigma,\alpha}$ is sufficiently curved (see Paragraph \ref{subsubsec33Micro}). For that purpose, we give the following lemma

\begin{lemma}
    Let $\alpha = 0, \pi$. Up to choosing $\tilde{\mathcal{R}}_i^{\bullet}$ smaller, the function $u_{\alpha}(\sigma,t)$ given by Lemma \ref{newhorriblelemma} has the following behaviour.
    
    \quad
    
    \myindent i. If $i \neq 0$, there holds the following Taylor expansion
    \begin{equation}
        u_{\alpha}(\sigma,t) = u_{\alpha}(\sigma,0) + t^2 F_{\alpha}(\sigma,t),
    \end{equation}
    where $F$ is a smooth function which doesn't vanish.
    
    \myindent ii. If $i = 0$, there holds the following Taylor expansion
    \begin{equation}
        u_{\alpha}(\sigma,t) = u_{\alpha}(\sigma,0) + t^2 \sigma^2 F_{\alpha}(\sigma,t),
    \end{equation}
    where $F$ is a smooth function which doesn't vanish.
\end{lemma}

\begin{proof}
    This lemma rigorously follows from the Taylor formula and the observations
    \begin{equation}
        \begin{cases}
            \partial_t u_{\alpha}(\sigma, 0) &= 0 \qquad \forall \sigma \\
            \partial_{\sigma t t} u_{\alpha}(0,t) &= 0 \qquad \forall t  \\
            \partial_{tt} u_{\alpha}(\sigma,0) &\neq 0 \qquad \forall \sigma \\
            \partial_{tt \sigma \sigma} u_{\alpha}(0,0) &\neq 0
        \end{cases},
    \end{equation}
    which themselves can be deduced from the parity properties (for the first two) or from Proposition \ref{exactstwrhohess} in the next paragraph and the implicit equation satisfied by $u_{\alpha}$ (for the last two). However, we wish to give a more geometric argument.
    
    \myindent Indeed, recall that $u_{\alpha}(\sigma,t)$ is defined by the equation (see \eqref{defsrhoualpha})
    \begin{equation}
    \begin{split}
    h(u_{\alpha}(\sigma,t)) &= r_{\alpha}(\sigma,t) \begin{pmatrix} 1 \\
       \partial_{\theta}\varphi((t,\sigma,),(0,\sigma) w_{\alpha}(\sigma,t))\end{pmatrix} \\
       &= r_{\alpha}(\sigma,t) \begin{pmatrix}
           q_1(\sigma, \nabla_x \varphi((t,\sigma,),(0,\sigma), w_{\alpha}(\sigma,t))) \\
           q_2(\sigma,\nabla_x \varphi((t,\sigma,),(0,\sigma), w_{\alpha}(\sigma,t)))
       \end{pmatrix}.
    \end{split}
    \end{equation}
    
    \myindent Now, as we have already argued in Paragraph \ref{subsubsec22HormPhase}, this means that $u_{\alpha}(\sigma,t)$ behaves like a cosine of the angle between $\nabla_x \varphi((t,\sigma,),(0,\sigma)), w_{\alpha}(\sigma,t))$ and the horizontal axis, i.e.
    \begin{equation}
        u_{\alpha}(\sigma,t) - u_{\alpha}(\sigma,0) \approx -\Omega(v_H, \nabla_x \varphi((t,\sigma,),(0,\sigma),w_{\alpha}(\sigma,t)))^2 =: -\Omega_t^2,
    \end{equation}
    where, for $v_1,v_2 \in \R^d \backslash \{0\}$, we define
    \begin{equation}
        \Omega(v_1,v_2) = \frac{\det(v_1,v_2)}{|v_1||v_2|},
    \end{equation}
    and where we define $v_H = \begin{pmatrix}
        1\\
        0
    \end{pmatrix}$.
    
    \myindent Now, we know moreover by definition that $\nabla_x \varphi((t,\sigma,),(0,\sigma),w_{\alpha}(\sigma,t))$ is. the direction of the only bicharacteristic curve joining $(t,\sigma)$ to $(0,\sigma)$, i.e. the angle $\Omega_t$ is roughly given by Figure \ref{Omegat}
    
    \begin{figure}[h]
    \includegraphics[scale = 0.5]{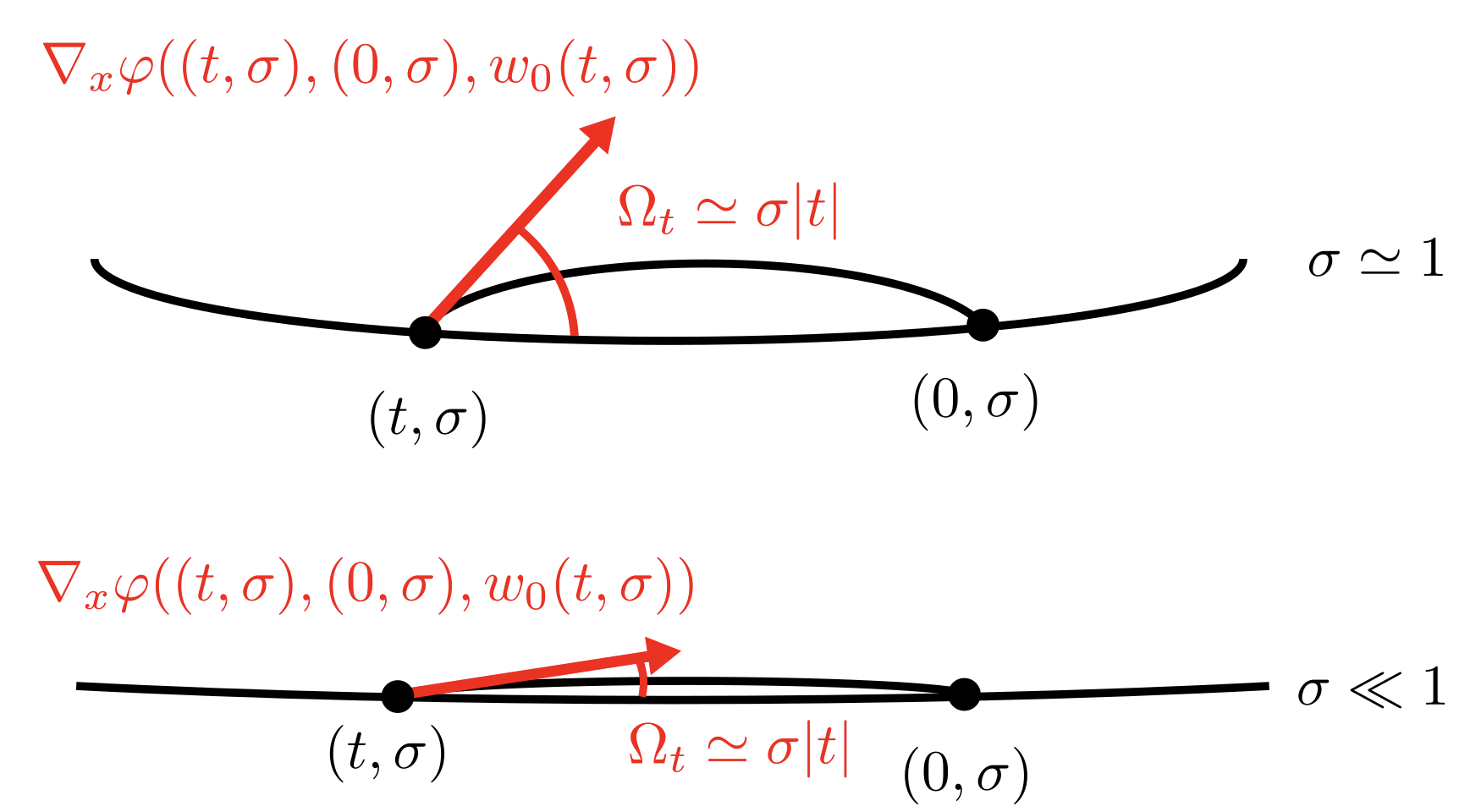}
    \centering
    \caption{The order of the angle $\Omega_t$ when $\sigma \to 0$}
    \label{Omegat}
    \end{figure}
    
    \myindent Hence, the geometric meaning of the lemma is merely that, for any $t_0$, 
    \begin{equation}
        |\Omega_{t_0}| \approx |\sigma t_0|.
    \end{equation}
    
    \myindent Now, an intuitive way to understand that relation is to come back to the definition of the Hamiltonian flow (Subsusbection \ref{subsubsec11CdV}). Indeed, if we let the bicharacteristic be defined by $t\mapsto (\theta(t),\sigma(t),\Theta(t),\Sigma(t))$, there holds
    \begin{equation}
        \Omega_{t_0} \approx \Sigma(0).
    \end{equation}
    
    \myindent Now, there holds (see Lemma \ref{partsig})
    \begin{equation}\label{dotsigmalocal}
        \dot{\sigma}(t) = - \frac{\partial q_1}{\partial \Sigma} \approx -\Sigma(t),
    \end{equation}
    and
    \begin{equation}
            \dot{\Sigma}(t) = \frac{\partial q_1}{\partial \sigma}\approx \sigma(t),
    \end{equation}
    i.e., since $\sigma(t)$ only varies of second order between the time $t = 0$ and the time $t = t_0$ thanks to \eqref{dotsigmalocal}, there holds roughly
    \begin{equation}
            \Sigma(t_0) - \Sigma(0) \approx \sigma t_0.
    \end{equation}
    
    \myindent Now, geometrically, $\Sigma(t_0) = - 
    \Sigma(0)$, hence we may conclude that
    \begin{equation}
        \Omega_{t_0} \approx \Sigma(0) \approx \sigma t_0.
    \end{equation}
\end{proof}

\myindent Thanks to this lemma, we may locally invert the function $t \mapsto u_{\alpha}(\sigma,t)$, similarly to Lemma \ref{locinverseuw}.

\begin{corollary}\label{deftpmu}
    Assume that $\sigma \neq 0$. Then the function
    \begin{equation}
        t \in \mathcal{K}_i^{(H)} \mapsto u_{\alpha}(\sigma,t)
    \end{equation}
    is a smooth and even surjection from $\mathcal{K}_i^{(H)}$ to some interval $[u_{\sigma}(\pi), u_{\sigma}(\pi) + \sigma^2 c(\sigma)]$ (if $\alpha = \pi$) or $[u_{\sigma}(0) - \sigma^2 c(\sigma), u_{\sigma}(0)]$ (if $\alpha = 0$), where $c(\sigma)$ is a smooth function which doesn't vanish when $\sigma \to 0$. Moreover, it admits two inverses
    \begin{equation}
        u \mapsto \begin{cases}
            t_+(\sigma,u) \in \mathcal{K}_i^{(H)} \\
            t_-(\sigma,u) \in \mathcal{K}_i^{(H)}
        \end{cases}.
    \end{equation}
    
    \myindent Finally, for $K = 0, 1, 2$,
    \begin{equation}
        (\partial_u)^K(t_{\pm})(u) \simeq \pm \frac{(-1)^{K-1}}{\sigma} \left(|u - u_{\sigma}(\alpha)|\right)^{\frac{1}{2} - K}.
    \end{equation}
\end{corollary}

\subsubsection{The $(s,t,w,r)$ Hessian of $\Psi_0$ on the branches $\mathcal{E}_{\sigma,\alpha}$}\label{subsubsec33HormPhase}

\myindent Similarly to Paragraph \ref{subsubsec23HormPhase}, in order to implement the strategy of Theorem \ref{mixedVdCABZ}, we now wish to study the $(s,t,w,r)$ Hessian of $\Psi_0$ along the branches $\mathcal{E}_{\sigma,\alpha}$, since they are formed of zeros of $\nabla_{s,t,w,r}\Psi_0$. We first give the following lemma.

\begin{lemma}\label{exactstwrhohess}
    Assume that $i \neq 0$ (i.e. $\sigma$ is away from the equator). Then, along the branch $\mathcal{E}_{\sigma,\alpha}$ there holds
    \begin{equation}
        \det(\nabla_{s,t,w,r}^2 \Psi_0(\sigma,s_{\alpha}(\sigma,t),t,w_{\alpha}(\sigma,t),r_{\alpha}(\sigma,t))) = t^2 F(\sigma,s,t,w,r)
    \end{equation}
    for some smooth non vanishing function $F$.

    \quad

    \myindent In the case $i = 0$, i.e. $\sigma$ is close to the equator, this is modified into
    \begin{equation}
         \det(\nabla_{s,t,w,r}^2 \Psi_0(\sigma,s_{\alpha}(\sigma,t),t,w_{\alpha}(\sigma,t),r_{\alpha}(\sigma,t))) = t^2 \sigma^2 F(\sigma,s,t,w,r),
    \end{equation}
    where $F$ is a smooth non vanishing function.
\end{lemma}

\begin{proof}
    We first give the proof in the case $i\neq 0$. Observe that, for any $(s,t,w,r)$, there holds
    \begin{equation}\label{explicitstwrhohess}
    \begin{split}
         &\det(\nabla_{s,t,w,r}^2 \Psi_0(\sigma,s,t,w,r)) \\ 
         &= \begin{vmatrix} 0 & 0 & 0 & 1 \\
         0 & r \partial_{\theta\theta} \varphi((t,\sigma,),(0,\sigma),w) & r \partial_{\theta w} \varphi((t,\sigma,),(0,\sigma),w) &  \partial_{\theta} \varphi((t,\sigma,),(0,\sigma),w) \\
         0 & r \partial_{\theta w} \varphi((t,\sigma,),(0,\sigma),w) & r \partial_{w w} \varphi((t,\sigma,),(0,\sigma),w) & \partial_{w} \varphi((t,\sigma,),(0,\sigma),w) \\
         1 &  \partial_{\theta} \varphi((t,\sigma,),(0,\sigma),w) & \partial_{w} \varphi((t,\sigma,),(0,\sigma),w)  & 0 \end{vmatrix} \\
         &= r^2 \begin{vmatrix}
          \partial_{\theta\theta} \varphi((t,\sigma,),(0,\sigma),w)  &  \partial_{\theta w} \varphi((t,\sigma,),(0,\sigma),w)  \\
           \partial_{\theta w} \varphi((t,\sigma,),(0,\sigma),w)  &  \partial_{w w} \varphi((t,\sigma,),(0,\sigma),w) 
           \end{vmatrix}.
    \end{split}
    \end{equation}
    
    \myindent Let us hence define 
    \begin{equation}\label{defAalphasigmat}
    	A_{\alpha}(\sigma,t) := \begin{pmatrix}
	 \partial_{\theta\theta} \varphi((t,\sigma,),(0,\sigma),w_{\alpha}(\sigma,t)))  &  \partial_{\theta w} \varphi((t,\sigma,),(0,\sigma),w_{\alpha}(\sigma,t))  \\
           \partial_{\theta w} \varphi((t,\sigma,),(0,\sigma),w_{\alpha}(\sigma,t))  &  \partial_{w w} \varphi((t,\sigma,),(0,\sigma),w_{\alpha}(\sigma,t)) 
           \end{pmatrix}.
     \end{equation}
            
    \myindent We will drop the explicit $\alpha,\sigma$ dependency of $A$ in the notations. Moreove, we give the proof for $\alpha = 0$ without loss of generality. Then, using Lemma \ref{newhorriblelemma}, we first observe, thanks to \eqref{miracleequation}, and \eqref{asdevphi}, that
    \begin{equation}
           	A(0) =
           	\begin{pmatrix}
		\partial_{\theta\theta}\varphi((0,\sigma,),(0,\sigma),0) & \partial_{\theta w}\varphi((0,\sigma,),(0,\sigma),0) \\
		\partial_{\theta w}\varphi((0,\sigma,),(0,\sigma),0) & \partial_{ww}\varphi((0,\sigma,),(0,\sigma),0)
		\end{pmatrix}= 0.
	\end{equation}
	
	\myindent In particular, thanks to the usual formula for differentiating a determinant, we know that
	\begin{equation}
		\det(A(t)) = \frac{t^2}{2} \det(A'(0)) + O(t^3).
	\end{equation}
	
	\myindent Hence, in order to prove the first part of the lemma, we need only prove that, for $i\neq 0$, 
	\begin{equation}
		\det(A'(0)) \neq 0.
	\end{equation}
	
	\myindent Now, from the explicit form of $A(t)$, one can compute that
	\begin{equation}
		A'(0) = \begin{pmatrix}
			\partial_{\theta \theta \theta} \varphi + w'(0) \partial_{\theta \theta w} \varphi & \partial_{\theta \theta w} \varphi + w'(0) \partial_{\theta w w} \varphi \\
			\partial_{\theta \theta w} \varphi + w'(0) \partial_{\theta w w} \varphi & \partial_{\theta w w} \varphi + w'(0) \partial_{w w w} \varphi
		\end{pmatrix},
	\end{equation}
	where all the terms depending on $\varphi$ are taken at $(0,\sigma,),(0,\sigma),0$.
	
	\myindent Now, thanks to Lemma \ref{thirdthetawderivaticevarphi}, we know that
	\begin{equation}
		\begin{cases}
			\partial_{\theta \theta \theta} \varphi((0,\sigma,),(0,\sigma),0) &= - \frac{\partial_{\Sigma \Sigma} q(\sigma,0,1) (\partial_{\sigma} q (\sigma, 0, 1))^2}{(\partial_{\Theta} q(\sigma,0,1))^3} \\
			\partial_{\theta \theta w} \varphi((0,\sigma,),(0,\sigma),0) &= - \frac{\mathfrak{l}_{\sigma}}{2\pi}\frac{\partial_{\Sigma \Sigma} q(\sigma,0,1) \partial_{\sigma} q (\sigma, 0, 1)}{(\partial_{\Theta} q(\sigma,0,1))^2}  \\
			\partial_{\theta w w} \varphi((0,\sigma,),(0,\sigma),0) &= - \left(\frac{\mathfrak{l}_{\sigma}}{2\pi}\right)^2 \frac{\partial_{\Sigma \Sigma} q_1(\sigma, g_{\sigma}(0))}{\partial_{\Theta} q_1(\sigma, g_{\sigma}(0))}  \\
			\partial_{w w w} \varphi((0,\sigma,),(0,\sigma),0) &= 0
		\end{cases}.
	\end{equation}
	
	\myindent Moreover, recall that $w(t)$ is \textit{defined} in the proof of Proposition \ref{newhorriblelemma} through the equation
	\begin{equation}\label{localFtwteqzero}
		F(t,w(t)) = 0,
	\end{equation}
	where
	\begin{equation}
		F(t,w) = \frac{1}{t} \partial_w \varphi ((t,\sigma,),(0,\sigma),w).
	\end{equation}
	
	\myindent In particular,
	\begin{equation}\label{localintexprofFtw}
		F(t,w) = \int_0^1 \partial_{\theta w} \varphi ((tv,\sigma),(0,\sigma),w) dv.
	\end{equation}
	
	\myindent Hence, on the one hand, differentiating \eqref{localFtwteqzero} in $t$, there holds
	\begin{equation}
		\partial_t F(0,0) + w'(0) \partial_w F(0,0) = 0,
	\end{equation}
	and, on the other hand, differentiating \eqref{localintexprofFtw} in $t$ and $w$, 
	\begin{equation}
		\begin{cases}
			\partial_t F(0,0) &= \frac{1}{2} \partial_{\theta \theta w} \varphi((0,\sigma,),(0,\sigma),0)  \\
			\partial_w F(0,0) &=  \partial_{\theta w w} \varphi((0,\sigma,),(0,\sigma),0) 
		\end{cases}.
	\end{equation}
	
	\myindent Thus, thanks to \eqref{localthrdderivatives} we finally find that 
	\begin{equation}
		w'(0) = -\frac{1}{2}\frac{\partial_{\sigma}q(\sigma,0,1)}{\partial_{\Theta}q(\sigma,0,1)}.
	\end{equation}
		
	\myindent Overall, there holds
	\begin{equation}\label{detA'0}
		\det(A'(0)) = \frac{1}{4}\left(\frac{\mathfrak{l}_{\sigma}}{2\pi}\right)^2\frac{(\partial_{\Sigma \Sigma} q(\sigma,0,1) )^2(\partial_{\sigma} q (\sigma, 0, 1))^2}{(\partial_{\Theta} q(\sigma,0,1))^4},
	\end{equation}
	which is nonzero for $\sigma \neq 0$ thanks to Lemma \ref{LemmapartSigSigq_1}.
	
	\quad
	
	\myindent Observe that the formula \eqref{detA'0} yields that, when $\sigma \to 0$,
	\begin{equation}
		\det(A'(0)) \simeq \sigma^2,
	\end{equation}
	in the sense that the quotient of the two is a nonzero smooth function. Unfortunately, one obviously cannot directly deduct the second part of the proposition from that fact. Instead, one should be a little more careful, and use Taylor formula with integral remainder, the key ingredient being the observation that
	\begin{equation}
		\forall t, \qquad \det(A'(\sigma = 0, t)) \equiv 0,
	\end{equation}
	which is obvious from the explicit expression of the phase at $\sigma = 0$, and that
	\begin{equation}
		\forall t, \qquad \frac{\partial}{\partial_{\sigma}} \left(\det(A'(\sigma, t)) \right) (\sigma = 0) \equiv 0,
	\end{equation}
	which comes from the \textit{parity} in $\sigma$.
\end{proof}

\myindent Now, this yields the following corollary, which is similar to Corollary \ref{norminversehessCsig}.

\begin{corollary}\label{norminversehessEsigalph}
    When $i \neq 0$, there holds along the branch $\mathcal{E}_{\sigma,\alpha}$
    \begin{equation}\label{localcaseineq1cor62}
        \left\| \left(\nabla_{s,t,w,r}^2 \Psi_0(\sigma, s_{\alpha}(\sigma,t), t, w_{\alpha}(\sigma,t), r_{\alpha}(\sigma,t))\right)^{-1} \right\| \lesssim \frac{1}{|t|}.
    \end{equation}
    
    \myindent When $i = 0$, this is modified into
    \begin{equation}\label{localcaseieq1cor62}
        \left\| \left(\nabla_{s,t,w,r}^2 \Psi_0(\sigma, s_{\alpha}(\sigma,t), t, w_{\alpha}(\sigma,t), r_{\alpha}(\sigma,t))\right)^{-1} \right\| \lesssim \frac{1}{\sigma^2|t|}.
    \end{equation}
\end{corollary}

\begin{proof}
    The proof of the estimate \eqref{localcaseineq1cor62} is really the same as the proof of Corollary \ref{norminversehessCsig}, hence we leave it to the reader. The only subtlety is that, in the estimate \eqref{localcaseieq1cor62}, there is a $\frac{1}{\sigma^2}$ in the upper bound and not a $\frac{1}{\sigma}$. The point is that, looking at the explicit formula for the $(s,t,w,r)$ Hessian along $\mathcal{E}_{\sigma,\alpha}$ given by  \eqref{explicitstwrhohess}, one sees that the cofactor $(2,2)$ is given by
    \begin{equation}
        \begin{vmatrix}
            0 & 0 & 1 \\
            0 & r_{\alpha}(\sigma,t)\partial_{ww}\varphi((t,\sigma,),(0,\sigma),w_{\alpha}(\sigma,t) & 0 \\
            1 & 0 & 0
        \end{vmatrix} = - r_{\alpha}(\sigma,t)\partial_{ww}\varphi((t,\sigma,),(0,\sigma),w_{\alpha}(\sigma,t)).
    \end{equation}
    
    \myindent Now, as $\sigma \to 0$, thanks to \eqref{asdevphi}, this coefficient is uniformly of order $t$, and \textit{not} of order $\sigma t$, hence, the coefficient $(2,2)$ in $\left(\nabla_{s,t,w,r}^2 \Psi_0(\sigma, s_{\alpha}(\sigma,t), t, w_{\alpha}(\sigma,t), r_{\alpha}(\sigma,t))\right)^{-1}$ is of order $\frac{1}{\sigma^2 |t|}$.
\end{proof}

\quad

\myindent Another consequence of Lemma \ref{exactstwrhohess} is that the determinant of the $(s,t,w,r)$ Hessian of $\Psi_0$ uniformly vanishes along $\mathcal{E}_{\sigma,\alpha}$ when $\sigma \to 0$, which is expected since the branches tend to vertical lines. Hence, as already announced, the argument for estimating the integral near $\mathcal{E}_{\sigma,\alpha}$ has to be different when $\sigma$ is very small. Now, in that case, there actually holds the following.

\begin{lemma}\label{FullHessiannearequator}
	Let $C,c> 0$. Assume that $|\sigma| \leq CM^{-\frac{1}{2}}$. Then, up to refining $\mathcal{Q}_i^{(H)}$, if
	\begin{equation}\label{defDnearequator}
	    \mathcal{D}^{\pm}_{\alpha} := \left\{(u,s,t,w,r)\ \text{such that} \ \begin{cases}
	        |u - u_{\sigma}(\alpha)| \lesssim_{c,C} 1\\
	       \pm t \geq cM^{-1} \\
	        |s| \lesssim_{c,C} 1 \\
	       |w - \alpha| \lesssim_{c,C} M^{-1} \\
	       (w,r) \in \mathcal{K}_{\sigma}
	    \end{cases}\right\},
	\end{equation}
	then for all $(u,s,t,w,r) \in \mathcal{D}^{\pm}_{\alpha}$,
	\begin{equation}
	    \left|\det(\nabla_{u,s,t,w,r}^2 \Psi(\sigma,A,B,u,s,t,w,r)) \right| \gtrsim_{c,C} M^{-1}.
	\end{equation}
\end{lemma}

\begin{proof}
    Observe that
    \begin{equation}\label{randomcomputationofdet}
        \begin{split}
            \det(\nabla_{u,s,t,w,r}^2\Psi) &= \begin{vmatrix}
                -\scal{h''(u)}{(s,t) + (A,B)} & - h'(u)_1 & -h'(u)_2 & 0 & 0 \\
                -h'(u)_1 & 0 & 0 & 0 & 1 \\
                -h'(u)_2 & 0 & r\partial_{\theta\theta}\varphi & r\partial_{\theta w} \varphi & \partial_{\theta} \varphi \\
                0 & 0 & r\partial_{w\theta} \varphi & r\partial_{ww} \varphi & \partial_{w}\phi \\
                0 & 1 & \partial_{\theta}\varphi & \partial_{w}\varphi & 0 
            \end{vmatrix} \\
            &= -\scal{h''(u)}{(s,t) + (A,B)} r^2\begin{vmatrix}
                \partial_{\theta\theta}\varphi & \partial_{\theta w}\varphi \\
                \partial_{\theta w}\varphi & \partial_{ww}\varphi
            \end{vmatrix} + r\partial_{ww}\varphi \begin{vmatrix}
                -h'(u)_1 & -h'(u)_2 \\
                1 & \partial_{\theta}\varphi
            \end{vmatrix}.
        \end{split}
    \end{equation}
    
    \myindent Now, since there holds thanks to \eqref{asdevphi}, \eqref{deth'hnotzero} and Lemma \ref{defcirclestatpoints},
    \begin{equation}
        r_{\sigma}(\alpha) \begin{vmatrix}
                -h'(u_{\sigma}(\alpha))_1 & -h'(u_{\sigma}(\alpha))_2 \\
                1 & \partial_{\theta}\varphi((0,\sigma,),(0,\sigma),\alpha)
            \end{vmatrix} = \det(h'(u_{\sigma}(\alpha)), h(u_{\sigma}(\alpha)))\neq 0,
    \end{equation}
    there obviously holds, for any $t$ small, for $|w-\alpha| \ll 1$, for $(w,r) \in \mathcal{K}_{\sigma}$ and $|u - u_{\alpha}(\sigma)| \ll 1$ that
    \begin{equation}
        \left|r\varphi \begin{vmatrix}
                -h'(u)_1 & -h'(u)_2 \\
                1 & \partial_{\theta}\varphi ((t,\sigma,),(0,\sigma),w)
            \end{vmatrix} \right| \gtrsim 1.
    \end{equation}
    
    \myindent Now, thanks to \eqref{asdevphi}, there holds moreover that
    \begin{equation}
        |\partial_{ww} \varphi((t,\sigma,),(0,\sigma),w) | \gtrsim |t|
    \end{equation}
    for any $|w| \ll |t| \ll 1$. In particular, the second term in \eqref{randomcomputationofdet} is of order $\gtrsim |t|$. 
    
    \myindent Now, regarding the first term, observe that, thanks to Lemma \ref{thirdthetawderivaticevarphi},
    \begin{equation}
        \begin{cases}
            &\partial_{\theta\theta}\varphi ((t,\sigma,),(0,\sigma),w) \lesssim \sigma^2 |t| + |\sigma| |w| \\
            &\partial_{\theta w}\varphi ((t,\sigma,),(0,\sigma),w) \lesssim |w| + |\sigma| |t| \\
            &\partial_{ww}\varphi ((t,\sigma,),(0,\sigma),w) \lesssim |t|
        \end{cases}.
    \end{equation}
    
    \myindent In particular,
    \begin{equation}
        \begin{split}
            \begin{vmatrix}
                \partial_{\theta\theta}\varphi & \partial_{\theta w}\varphi \\
                \partial_{\theta w}\varphi & \partial_{ww}\varphi
            \end{vmatrix} &= \begin{vmatrix}
                O(\sigma^2 t) + O(\sigma w) & O(w) + O(\sigma t)  \\
                O(w) + O(\sigma t) & O(t) 
            \end{vmatrix} \\
            &= O(\sigma^2 t^2) + O(\sigma w t) + O(w^2).
        \end{split}
    \end{equation}
    
    \myindent In the regime of the lemma, we may deduce that
    \begin{equation}
    \begin{split}
        \begin{vmatrix}
                \partial_{\theta\theta}\varphi & \partial_{\theta w}\varphi \\
                \partial_{\theta w}\varphi & \partial_{ww}\varphi
            \end{vmatrix} &= O(C^2 M^{-1} t^2) + O(C c' M^{-\frac{3}{2}} t) + O(c'^2 M^{-2}),
    \end{split}
    \end{equation}
    where $c'(c,C)$ is an arbitrarily small constant.  Hence, the first term is of order
    \begin{equation}
            -\scal{h''(u)}{(s,t) + (A,B)} r^2\begin{vmatrix}
                \partial_{\theta\theta}\varphi & \partial_{\theta w}\varphi \\
                \partial_{\theta w}\varphi & \partial_{ww}\varphi
            \end{vmatrix} = O\left( C^2 t^2 + Cc'M^{-\frac{1}{2}} t + c'^2 M^{-1}\right)\ll |t|,
    \end{equation}
    provided $|t|$ is small enough depending on $C$ (hence up to refining $\mathcal{Q}_i^{(H)}$), and $c'$ is small enough depending on $c$. Hence, in the regime of the lemma, the second term in \eqref{randomcomputationofdet} wins and there holds
    \begin{equation}
        |\det(\nabla_{u,s,t,w,r}^2 \Psi(\sigma,A,B,u,s,t,w,r))| \gtrsim |t| \gtrsim c M^{-1}.
    \end{equation}
\end{proof}

\subsubsection{The 1D remaining phase function along the branch $\mathcal{E}_{\sigma,\alpha}$}\label{subsubsec34HormPhase}

\myindent Now, thanks to Corollary \ref{deftpmu}, we may apply the strategy of Theorem \ref{mixedVdCABZ} to the curve $\mathcal{E}_{\sigma,\alpha}$, at least when $\sigma$ is not too small. Indeed, let us define
\begin{equation}\label{defesigmaalphapm}
    \mathcal{E}_{\sigma,\alpha}^{\pm} := \left\{ (u_{\alpha}(\sigma,t), s_{\alpha}(\sigma,t), t, w_{\alpha}(\sigma,t), r_{\alpha}(\sigma,t)) \qquad 0 < \pm t \ll 1\right\}.
\end{equation}

\myindent Then, for $u$ near $u_{\sigma}(\alpha)$, there is a one-to-one smooth mapping $u \mapsto Z(u) \in \mathcal{E}_{\sigma,\alpha}^{\pm}$. We know that $\mathcal{E}_{\sigma,\alpha}^{\pm}$ is composed of $(s,t,w,r)$ stationary points with nonvanishing $(s,t,w,r)$ hessian (provided that $\sigma \neq 0$). In order to implement the strategy of Theorem \ref{mixedVdCABZ}, we thus need that the remaining 1D phase function has a nonvanishing $p$th derivative for some $p$. We will only need this analysis in the case $(A,B) \neq (0,0)$, hence, we assume that $M\geq 1$.

\begin{lemma}\label{Psi1DnearEalpha}
    Let $\alpha = 0$ or $\pi$ and $|\sigma| \geq C M^{-\frac{1}{2}}$ where $C> 0$ is large enough. We define the remaining 1D phase function for $u$ close to $u_{\sigma}(\alpha)$ on the branch $\mathcal{E}_{\sigma,\alpha}^{\pm}$ by
    \begin{equation}
        \begin{split}
        \Psi^{1D} : u \mapsto &\Psi(\sigma,A,B,u,s_{\alpha}(\sigma,t_{\pm}(u)),t_{\pm}(u), w_{\alpha}(\sigma,t_{\pm}(u)), r_{\alpha}(\sigma,t_{\pm}(u)) \\
        &= -\scal{h(u)}{(s_{\alpha}(\sigma,t_{\pm}(u)),t_{\pm}(u)) + (A,B)}.
        \end{split}
    \end{equation}

    \myindent Then, there exist universal constants $c_1, c_2 > 0$ such that if $I := [u_{\sigma}(\pi),u_{\sigma}(\pi) + \sigma^2 c(\sigma)]$ (if $\alpha = \pi$), or $I =[u_{\sigma}(0) - \sigma^2 c(\sigma), u_{\sigma}(0)]$ (if $\alpha = 0$) then $\Psi^{1D}$ satisfies Property $(VdC)_3$ on $I$ with constants $c_1M, c_2M$ (see Definition \ref{VdCN}), for some universal $c_1,c_2 > 0$.
 \end{lemma}

\begin{proof}
    We drop the explicit notation of the $\alpha, \sigma, \pm$ dependency. By definition
    \begin{equation}
        \nabla_{s,t,w,r}\Psi(u,s(t(u)),t(u),w(t(u)),r(t(u)) = 0.
    \end{equation}

    \myindent Hence, we can compute
    \begin{equation}\label{partuPsi1Dlocal}
        \begin{split}
            \partial_u \Psi^{1D}(u) &= \partial_u\Psi(x,A,B,u,s(u),t(u),w(u),r(u)) \\
            &= -\scal{h'(u)}{(s(t(u)),t(u)) + (A,B)}.
        \end{split}
    \end{equation}

    \myindent From here, we can compute
    \begin{equation}\label{expressionpartial23Psi1D}
        \begin{split}
            (\partial_u)^2 \Psi^{1D}(u) &= -\scal{h''(u)}{(s(t(u)),t(u)) + (A,B)} - t'(u) \scal{h'(u)}{(\partial_t s(t(u)),1)} \\
            (\partial_u)^3 \Psi^{1D}(u) &= -t''(u)\scal{h'(u)}{(\partial_t s(t(u)),1)} \\
            &- t'(u)^2\scal{h'(u)}{((\partial_t)^2s(t(u)),0)}\\
            &-\scal{h^{(3)}(u)}{(s(t(u)),t(u)) + (A,B)} - 2t'(u) \scal{h''(u)}{(\partial_t s(t(u)),1)}.
        \end{split}
    \end{equation}
    
    \myindent We claim that
    \begin{equation}\label{maxpartu,uu,uuugeqM}
        \forall u \in I, \qquad \max(|\partial_u\Psi^{1D}(u)|, |\partial_{uu} \Psi^{1D}(u)|, |\partial_{uuu}\Psi^{1D}(u)|) \gtrsim M.
    \end{equation}
    
    \myindent Indeed, assume first that
    \begin{equation}\label{localpartpsi1Dueq0}
        |\partial_u \Psi^{1D}(u)| \ll M.
    \end{equation}
    
    \myindent Then, from \eqref{partuPsi1Dlocal}, and from the fact that $h'(u)$ and $h''(u)$ are orthogonal, and that their norms are bounded away from zero, this implies that
    \begin{equation}
        |\scal{h''(u)}{(s(t(u)),t(u)) + (A,B)}| \simeq M.
    \end{equation}
    
    \myindent In particular, if moreover we impose that 
    \begin{equation}\label{localpartpsi1Duueq0}
        |\partial_{uu} \Psi^{1D}(u)| \ll M,
    \end{equation}
    
    then, from equation \eqref{expressionpartial23Psi1D}, we can deduce that
    \begin{equation}\label{t'usimeqM}
        |t'(u)| \simeq M.
    \end{equation}
    
    \myindent Indeed, from \eqref{defsrhoualpha}, we find that 
    \begin{equation}
        \partial_t s(0) = \partial_{\theta} \varphi ((t,\sigma,),(0,\sigma),w_{\alpha}(\sigma,0)),
    \end{equation}
    thus, from \eqref{asdevphi}, Lemma \ref{defcirclestatpoints} and Lemma \ref{deth'hnotzero}, there holds
    \begin{equation}
        \begin{split}
        \scal{h'(u_{\sigma}(\alpha))}{(\partial_t s(t(u_{\sigma}(\alpha))), 1)} &= \scal{h'(u_{\sigma}(\alpha))}{(1, \partial_{\theta} \varphi((0,\sigma,),(0,\sigma),\alpha)} \\
            &= \scal{h'(u_{\sigma}(\alpha)}{\frac{1}{r_{\sigma}(\alpha)} h(u_{\sigma}(\alpha))} \\
            &\neq 0.
        \end{split}
    \end{equation}
    
    \myindent Hence, in particular, for all $u$ close to $u_{\sigma}(\alpha)$, 
    \begin{equation}\label{localusefuleqPsi1DuVdC3}
        |\scal{h'(u)}{(\partial_t s(t(u)),1)}| \simeq 1.
    \end{equation}
    
    \quad
    
    \myindent Now, we wish to prove that, if \eqref{localpartpsi1Dueq0} and \eqref{localpartpsi1Duueq0} hold, then
    \begin{equation}\label{localpartpsi1DuuugeqM}
        |\partial_{uuu} \Psi^{1D}(u)| \gtrsim M.
    \end{equation}
    
    \myindent Now, from Lemma \ref{deftpmu}, \eqref{t'usimeqM} implies that
    \begin{equation}
    \begin{split}
        \frac{1}{\sqrt{|u - u_{\sigma}(\alpha)|}} &\simeq |\sigma t'(u)| \\
        &\simeq |\sigma| M\\
        &\gtrsim C M^{-\frac{1}{2}}.
    \end{split}
    \end{equation}
    
    \myindent Hence, we find that, using again Lemma \ref{deftpmu},
    \begin{equation}
    \begin{split}
        |t''(u)| &\simeq \frac{1}{\sigma |u - u_{\sigma}(\alpha)|^{\frac{3}{2}}} \\
        &\simeq t'(u)^2 \frac{\sigma}{\sqrt{|u - u_{\sigma}(\alpha)|}} \\
        & \simeq M^2 \frac{\sigma}{\sqrt{|u - u_{\sigma}(\alpha)|}}\\
        &\gtrsim C^2M^2.
    \end{split} 
    \end{equation}
    
    \myindent In particular, $|t''(u)| \gg M$ and $|t''(u)| \gg |t'(u)|^2$ if $C$ is large enough. Hence, looking at the expression of $\partial_{uuu} \Psi^{1D}(u)$ given by \eqref{expressionpartial23Psi1D}, we see that the first term wins over the second and third terms, using once again \eqref{localusefuleqPsi1DuVdC3}, which yields \eqref{localpartpsi1DuuugeqM}.
    
    \quad
    
    \myindent To conclude the proof of the lemma, observe that equation \eqref{maxpartu,uu,uuugeqM} can be relaxed into finding a partition of $I$ by a finite number of intervals, on each of which one of the three first derivatives of $\Psi^{1D}$ is larger in modulus than $cM$ for $c$ a small enough constant. Moreover, one can do that choice so that on those intervals where it is \textit{only} the first derivative which is larger than $cM$, the the second derivative is smaller than $CM$ if $C$ is a large enough constant. Indeed, from the proof, we see that, if it is only the first derivative which is very large, then this means that 
    \begin{equation}
        |t'(u)| \lesssim M.
    \end{equation}
    
    \myindent Hence, coming back to \eqref{expressionpartial23Psi1D}, we find that $|(\partial_u)^2 \Psi^{1D}| \lesssim M$ as desired.
\end{proof}

\subsection{How to estimate near the branching points $P_0$ and $P_{\pi}$}\label{subsec4HormPhase}

\myindent In this section, we explain the method to deal with the branching points $P_0$ and $P_{\alpha}$, for which we need a different analysis than above. This section is more delicate, and technical, than Sections \ref{subsec2HormPhase} and \ref{subsec2HormPhase}, hence we will try to provide some intuition on our method. Moreover, as we will only need this method in the case $(A,B) \neq (0,0)$, we assume throughout Section \ref{subsec4HormPhase} that $M = |(A,B)| \geq 1$.

\subsubsection{Why it is natural to isolate the variable $w$ rather than $u$}\label{subsubsec41Hormphase}

\myindent The analysis of Sections \ref{subsec2HormPhase} and \ref{subsec3HormPhase}, respectively, naturally leads to apply the stationary phase analysis of Theorem \ref{mixedVdCABZ} on the circle $\mathcal{C}_{\sigma}$, and along the branches $\mathcal{E}_{\sigma,\alpha}$, as long as we are away from the two branching points, which we recall are defined by
\begin{equation}
    P_{\alpha} = \left(u_{\sigma}(\alpha),0,0,0,r_{\sigma}(\alpha)\right) \qquad \alpha = 0,\pi,
\end{equation}
where we recall that $u_{\sigma}(\alpha), r_{\sigma}(\alpha)$ are defined by
\begin{equation}
    h(u_{\sigma}(\alpha)) = r_{\sigma}(\alpha)\begin{pmatrix}
        1 \\
        q_2(\sigma,\alpha,1)
    \end{pmatrix}.
\end{equation}

\myindent Now, there holds, thanks to both Lemma \ref{stwrhohessofPsionCx} and Lemma \ref{exactstwrhohess}, that
\begin{equation}
    \det(\nabla_{s,t,w,r}^2\Psi_0(P_{\alpha})) = 0,
\end{equation}
which reflects the unavoidable geometric fact that the set $\mathcal{O}_{\sigma}$ of zeros of $\nabla_{s,t,w,r}\Psi_0$ is \textit{singular} at $P_{\alpha}$, as explained in Paragraph \ref{subsubsec11HormPhase}. In particular, the strategy of applying Theorem \ref{mixedVdCABZ} by isolating the variable $u$ fails near the points $P_{\alpha}$. Hence, in order to perform an oscillatory integral analysis, we need to take into account the role of the full $(u,s,t,w,r)$ variables, and, in particular, of
\begin{equation}\label{partupsipalpha}
    \partial_u \Psi(P_{\alpha}),
\end{equation}
where we use the notation 
\begin{equation}
    \Psi(P_{\alpha}) := \Psi\left(\sigma,A,B,u_{\sigma}(\alpha),0,0,\alpha,r_{\sigma}(\alpha)\right).
\end{equation}

\myindent Now, roughly two scenarios may happen : either the quantity \eqref{partupsipalpha} is nonzero, and, in that case, we may integrate by parts in $u$ in a neighborhood of $P_{\alpha}$ and find that the contribution of this region is $O(\lambda^{-\infty})$, or this quantity vanishes, and in that case we need to isolate another variable, namely $w$. Before quantifying this dichotomy, we justify that $w$ is the natural variable to isolate in the second case.

\myindent First, observe that this implication of $\partial_u \Psi(P_{\alpha}) = 0$ is that $P_{\alpha}$ is a stationary point in the full $(u,s,t,w,r)$ variables (since it is already a point of $\mathcal{O}_{\sigma}$). Moreover, the full Hessian of $\Psi$ in the variables of integration is given at $P_{\alpha}$ by (using Lemma \ref{miracleequation})
\begin{multline}\label{fullhessianatPalpha}
    \nabla_{u,s,t,w,r}^2 \Psi\left(\sigma,A,B,u_{\sigma}(\alpha),0,0,\alpha,r_{\sigma}(\alpha)\right) = \\
    \begin{pmatrix}
        -\scal{h''(u_{\sigma}(\alpha))}{(A,B)} & - h'(u_{\sigma}(\alpha))_1 & -h'(u_{\sigma}(\alpha))_2 & 0 & 0\\
        -h'(u_{\sigma}(\alpha))_1 & 0 & 0 & 0 & 1 \\
        -h'(u_{\sigma}(\alpha))_2 & 0 & 0 & 0 & \partial_{\theta}\varphi((0,\sigma,),(0,\sigma),\alpha) \\
        0 & 0 & 0 & 0 & 0\\
        0 & 1 & \partial_{\theta}\varphi((0,\sigma,),(0,\sigma),\alpha) & 0 & 0
    \end{pmatrix}.
\end{multline}

\myindent We observe that this matrix is degenerate (as expected), but that its kernel in $\R^5$ is exactly given by the fourth axis, corresponding to the $w$ variable. Indeed, if we isolate the $w$ variable, we find that 
\begin{equation}\label{ustrhohessPalpha}
    \begin{split}
        \det(\nabla_{u,s,t,r}^2\Psi(P_{\alpha})) &= \begin{vmatrix}
            -\scal{h''(u_{\sigma}(\alpha))}{(A,B)} & - h'(u_{\sigma}(\alpha))_1 & -h'(u_{\sigma}(\alpha))_2  & 0\\
        -h'(u_{\sigma}(\alpha))_1 & 0  & 0 & 1 \\
        -h'(u_{\sigma}(\alpha))_2 & 0  & 0 & \partial_{\theta}\varphi((0,\sigma,),(0,\sigma),\alpha)\\
        0 & 1 & \partial_{\theta}\varphi((0,\sigma,),(0,\sigma),\alpha)  & 0
        \end{vmatrix} \\
        &= \begin{vmatrix}
            -h'(u_{\sigma}(\alpha))_1 & 1 \\
            -h'(u_{\sigma}(\alpha))_2 & \partial_{\theta}\varphi((0,\sigma,),(0,\sigma),\alpha)
        \end{vmatrix}^2\\
        &= r_{\sigma}(\alpha)^{-2} \det(h'(u_{\sigma}(\alpha)),h(u_{\sigma}(\alpha)))^2 \\
        &\neq 0.
    \end{split}
\end{equation}

\myindent Hence, we are exactly in the situation of a degenerate Hessian in $1+d$ dimensions, with a rank $d$, which was introduced in Paragraph \ref{subsubsec43Micro}. Following the analysis of that paragraph, it is quite natural that, should there be a zero of $\nabla_{u,s,t,w,r}\Psi$ at, or near $P_{\alpha}$, one should isolate the $w$ variable and perform a stationary phase analysis on the remaining four variables $(u,s,t,r)$, and, afterwards, try and prove that the remaining 1D phase function will satisfy the $(VdC)_p$ hypothesis (see Definition \ref{VdCN}) for some $p \geq 3$.

\myindent The main difficulty for quantifying the previous heuristic is that, since we need to take into account the role of $(A,B)$, we need to relax the dichotomy between $\partial_u P_{\alpha} \neq 0$ and $\partial_u P_{\alpha} = 0$ into a threshold condition on how small $\partial_u \Psi(P_{\alpha})$ is, compared to (negative powers of) $|(A,B)|$. Now, the problem is that, in the case $|\partial_u \Psi(P_{\alpha})| \ll_{|(A,B)|} 1$, when trying to apply Theorem \ref{mixedVdCABZ} by isolating the variable $w$, there doesn't hold necessarily that
\begin{equation}
    \nabla_{u,s,t,r} \Psi(P_{\alpha}) = 0,
\end{equation}
but rather that this quantity is very small. Hence, we are not directly in the setting presented in Paragraph \ref{subsubsec43Micro}. Thus, a difficulty is to prove the existence of a $(u,s,t,r)$ stationary point near $P_{\alpha}$, all in a quantitative way depending on $|(A,B)|$.

\subsubsection{The case $|\partial_u\Psi(P_{\alpha})| \ll 1$ : rigorous statement and idea of the proof}\label{subsubsec42Hormphase}

\myindent First, in the most delicate case, that is when $|\partial_u \Psi(P_{\alpha})| \ll 1$, there holds the following proposition.

\begin{proposition}\label{delicatecase}
    Let $i \in \{0,...,I\}$, let $\alpha = 0$ or $\alpha = \pi$, and let $\sigma \in \mathcal{I}_i$. Let $c_{P_{\alpha}} > 0$ be a small enough constant, depending only on $\mathcal{S},\eps,\mathfrak{Q}$. Assume that
    \begin{equation}\label{smallpartuPsiPalpha}
        |\partial_u \Psi(P_{\alpha})| \lesssim_{c_{P_{\alpha}}} M^{-1}.
    \end{equation}
    
    \myindent Define
    \begin{equation}\label{defIUdelicatecase}
        \begin{split}
            &I := \{w \in S^1 \ \text{such that} \ |w - \alpha| \lesssim_{c_{P_{\alpha}}} M^{-1} \} \\
            &U := \left\{(u,s,t,r) \in (\R/\ell\Z)\times \tilde{\mathcal{R}}_i^{(H)} \times \R_+ \ \text{such that} \ \begin{cases}
                |u - u_{\sigma}(\alpha)| \leq c_{P_{\alpha}} M^{-2} \\
                |(s,t,r) - (0,0, r_{\sigma}(\alpha)| \leq c_{P_{\alpha}} M^{-1}
            \end{cases}
            \right\}.
        \end{split}
    \end{equation}
    
    \myindent Then, for all $w \in I$, the function
    \begin{equation}
        \Psi_w : (u,s,t,r) \in U \mapsto \Psi(\sigma,A,B, u,s,t,w,r) 
    \end{equation}
    has a unique stationary point (i.e. $\nabla_{u,s,t,r} \Psi_w$ has a unique zero) 
    \begin{equation}\label{defZsigAB}
        Z_{\sigma,A,B}(w) := (u,s,t,r)(w) \in U,
    \end{equation}
    at which its Hessian 
    \begin{equation}
        H(w) := \nabla_{u,s,t,r}^2 \Psi(\sigma,A,B, u(w),s(w),t(w),w,r(w))
    \end{equation}
    is nondegenerate. Moreover, the following estimates holds. First, using the Notations \ref{defM1+d} and \ref{defHandN}, for all $0\leq k \leq l$,
    \begin{equation}\label{upperboundMNphidelicatecase}
    \begin{split}
        &\mathcal{M}_{k,l}^{(w)}(\Psi) \lesssim M \\
        &\mathcal{D}(\Psi) \gtrsim 1 \\
        &\mathcal{N}(\Psi) \lesssim M.
    \end{split}
    \end{equation}
    
    \myindent Second, for all $w \in I$, for all $(u,s,t,r) \in U$, using the notation \ref{definitionofM}, there holds
    \begin{equation}\label{oscintconditiondelicatecase}
        \|M(w,(u,s,t,r)) - H(w)\| \leq \frac{1}{2}\|H(w)^{-1}\|^{-1}.
    \end{equation}
    
    \myindent Third, using the definition of the 1D remaining phase function introduced in \ref{defphi1D}, there holds
    \begin{equation}\label{4thderivativesPsi1Ddelicatecase}
        \forall w\in I \qquad |(\partial_w)^4 \Psi^{1D}(\sigma,A,B,w)| \gtrsim_{c_{P_{\alpha}}} M.
    \end{equation}
\end{proposition}

\myindent Before giving the proof of this proposition, let us insist on the two main difficulties. First, in the setting on the proposition, and in sparking contrast to the previous sections, we don't have a natural candidate for the stationary point $Z_{\sigma,A,B}(w)$ defined by \eqref{defZsigAB}. Indeed, we only know a priori that $|\nabla_{u,s,t,r}\Psi| \ll 1$ on the domain $I\times U$ (since it is close to $\mathcal{O}_{\sigma}$, and thanks to \eqref{smallpartuPsiPalpha}), and that the Hessian $\nabla_{u,s,t,r}^2 \Psi$ is nondegenerate, and we need to deduce the existence of $Z_{\sigma,A,B}$ from those facts. Second, there is an obvious anisotropy in the scaling between the variable $u$ and the other variables. As we will see in the proof, this anisotropy is an unavoidable feature of the analysis.

\quad

\myindent Now, we decompose the proof of Proposition \ref{delicatecase} into several steps. For the sake of simplicity, we define for the proof the following notations 

\myindent i. Firstly, in the spirit of Theorem \ref{mixedVdCABZ}, we define the new variable
\begin{equation}
    y := (u - u_{\sigma}(\alpha), s,t,r - r_{\sigma}(\alpha)) \in \R^4,
\end{equation}
and thus the variables $(u,s,t,w,r) \in \R^5$ are changed into $(w,y) \in \R \times \R^4$. In particular, in these variables, the points $P_{\alpha}$ is simply the origin $(0,0) \in \R \times \R^4$.

\myindent ii. Secondly, we use the notation
\begin{equation}
    \forall (w,y) = (w, u-u_{\sigma}(\alpha),s,t,r - r_{\sigma}(\alpha)) \qquad \phi(w,y) := \Psi(\sigma,A,B,u,s,t,w,r).
\end{equation}

\myindent iii. Thirdly, we \textit{adapt} the notations introduced in Definition \ref{defHandN} and \eqref{definitionofM}, and define locally
\begin{equation}\label{defmathcalHlocal}
    \begin{split}
        \mathcal{H}(w) &:= \nabla_y^2 \phi(w,0) \\
        &= \nabla_{u,s,t,r}^2 \Psi(\sigma,A,B,u_{\sigma}(\alpha),0,0,w,r_{\sigma}(\alpha)),
    \end{split}
\end{equation}
and 
\begin{equation}\label{defmathcalMlocal}
\begin{split}
    \mathcal{M}(w,y) &:= \int_0^1 \nabla_y^2 \phi(w,(1-v)y)dv \\
    &= \int_0^1 \nabla_{u,s,t,r}^2 \Psi(\sigma,A,B,vu_{\sigma}(\alpha) + (1-v)u,(1-v)s,(1-v)t,w,v r_{\sigma}(\alpha) + (1 - v)r) dv.
\end{split}
\end{equation}

In particular, there holds
\begin{equation}\label{nablayphi}
    \nabla_y\phi(w,y) = \nabla_y \phi(w,0) + \mathcal{M}(w,y) \cdot y.
\end{equation}

\quad

\myindent Now, the idea of the proof comes from very elementary analysis : indeed, the point is that, since
\begin{equation}
    \nabla_{u,s,t,r}^2 \Psi(P_{\alpha}) \in GL_4(\R),
\end{equation}
then, for any $|w| \ll 1$, the map
\begin{equation}
    |y| \ll 1 \mapsto \nabla_y \phi(w,y)
\end{equation}
is a diffeomorphism between two open sets (this is merely the inverse function theorem). Now, if we assume moreover that
\begin{equation}
    \partial_u \Psi(P_{\alpha}) \ll 1,
\end{equation}
then, since
\begin{equation}
    \nabla_{s,t,r} \Psi(P_{\alpha}) = 0,
\end{equation}
there holds in $(w,y)$ variables that
\begin{equation}
    |\nabla_y \phi(0,0)| \ll 1.
\end{equation}

\myindent Now, if $\nabla_y \phi(0,0)$ is sufficiently small, then obviously for all small enough $w$ there exists a $y(w)$ such that
\begin{equation}
    \nabla_y \phi(w,y(w)) = 0,
\end{equation}
and, on an appropriate neighborhood of $(w,y) = (0,0)$, $\phi$ will satisfy the hypotheses of Theorem \ref{mixedVdCABZ}, or rather of Remark \ref{nonisotropicmixedABZVdC} (once we check \eqref{4thderivativesPsi1Ddelicatecase}. The real subtlety of the proof is actually to quantify, in terms of exponents of $M^{-1}$, the previous "small" conditions.

\subsubsection{Quantitative estimates on $\nabla_{u,s,t,r}^2 \Psi$ near $P_{\alpha}$}\label{subsubsec43Hormphase}

\myindent We start with the following lemma, which quantifies the neighborhood of $P_{\alpha}$ on which $\nabla_{u,s,t,r}^2 \Psi$ is non degenerate.

\begin{lemma}\label{nondegenerateustrhohess}
    Let $i \in \{0,...,I\}$, let $\alpha = 0$ or $\alpha = \pi$, and  let$\sigma \in \mathcal{I}_i$. There holds
    \begin{equation}\label{lowerineqondetustrho}
        \forall |(u,s,t,w,r) - P_{\alpha}| \lesssim M^{-1}, \qquad |\det(\nabla_{u,s,t,r}^2 \Psi(\sigma,A,B,u,s,t,w,r))| \gtrsim 1.
    \end{equation}
    
    \myindent As a consequence, for all such $(u,s,t,w,r)$, there holds
    \begin{equation}\label{upperboundinversehessustrho}
        \left\|\left(\nabla_{u,s,t,r}^2 \Psi(\sigma,A,B,u,s,t,w,r)\right)^{-1} \right\| \lesssim M.
    \end{equation}
\end{lemma}

\begin{proof}
    \myindent We use the following useful algebraic fact
    \begin{equation}
        \begin{vmatrix}
            K & a & b & 0 \\
            a & 0 & 0 & 1 \\
            b & 0 & e & c \\
            0 & 1 & c & 0
        \end{vmatrix} = (ac - b)^2 - Ke.
    \end{equation}

    \myindent In particular, since
    \begin{multline}
        \nabla_{u,s,t,r}^2 \Psi(\sigma,A,B,u,s,t,w,r) =\\ \begin{pmatrix}
        -\scal{h''(u)}{(A,B) + (s,t)} & - h'(u)_1 & -h'(u)_2 & 0 \\
        -h'(u)_1 & 0 & 0 & 1 \\
        -h'(u)_2 & 0 & r \partial_{\theta\theta}\varphi((t,\sigma,),(0,\sigma),w)& \partial_{\theta}\varphi((t,\sigma,),(0,\sigma),w)  \\
        0 & 1 & \partial_{\theta}\varphi((t,\sigma,),(0,\sigma),w) & 0
    \end{pmatrix},
    \end{multline}
    there holds that
    \begin{equation}
    \begin{split}
        \det(\nabla_{u,s,t,r}^2 \Psi(\sigma,A,B,u,s,t,w,r)) &= \begin{vmatrix}
            -h'(u)_1 & 1 \\
            -h(u)_2 & \partial_{\theta}\varphi((t,\sigma,),(0,\sigma),w)
        \end{vmatrix}^2\\
        &+ r  \left(\scal{h''(u)}{(a + s, b +t)}\right) (\partial_{\theta\theta}\varphi((t,\sigma,),(0,\sigma),w)) \\
        &= D(\sigma,t,u,w)^2 + r K(A,B,u,s,t) e(\sigma,t,w),
    \end{split}
    \end{equation}
    where we use obvious notations. Now, we know on the one hand that
    \begin{equation}
    \begin{split}
        D(\sigma,u_{\alpha}(\sigma),0,\alpha) &= \begin{vmatrix}
            -h'(u_{\alpha}(\sigma))_1 & 1 \\
            -h'(u_{\alpha}(\sigma))_2 & \partial_{\theta}\varphi((0,\sigma,),(0,\sigma),\alpha)
        \end{vmatrix}\\
        &= r_{\alpha}(\sigma)^2 \det(h'(u_{\alpha}(\sigma)),h(u_{\alpha}(\sigma)))^2 \\
        &> 0.
    \end{split}
    \end{equation}
    
    \myindent Hence, there are universal constants $c,c' > 0$ such that
    \begin{equation}
        \forall (u,t,w) \ \text{such that} \ |(u,t, w) - (u_{\alpha}(\sigma), 0, \alpha)| \leq c, \qquad D(\sigma,u,t,w) \geq c'.
    \end{equation}

    \myindent On the other hand, thanks to Lemma \ref{miracleequation},
    \begin{equation}
            e(\sigma,0,\alpha) = \partial_{\theta\theta} \varphi((t,\sigma,),(0,\sigma),\alpha) = 0.
    \end{equation}

    \myindent Hence, there is a constant $c''$ such that
    \begin{equation}
        \forall (u,s,t,w,r) \ \text{such that} |(u,s,t,w,r) - (u_{\alpha}(\sigma),0,0,\alpha,r_{\alpha}(\sigma))| \leq c''M^{-1} \qquad \left|r K(A,B,u,s,t) e(\sigma,t,w) \right| \leq \frac{1}{2} c',
    \end{equation}
    which obviously concludes the proof of \eqref{lowerineqondetustrho}.

    \myindent Now, regarding \eqref{upperboundinversehessustrho}, recall that for any invertible matrix $\mathcal{M}$, there holds
    \begin{equation}
        \mathcal{M}^{-1} = \frac{1}{\det(\mathcal{M})}  Com^T(\mathcal{M}),
    \end{equation}
    where $Com(\mathcal{M})$ is the cofactor matrix of $\mathcal{M}$. Now, fix $\mathcal{M} = \nabla_{u,s,t,r)}^2 \Psi(\sigma,A,B,u,s,t,w,r)$. We have proved that $(\det \mathcal{M})^{-1}$ is bounded independently of $M$ with the hypotheses of \eqref{lowerineqondetustrho}. Moreover, since only one coefficient of $\mathcal{M}$ is $O(M)$, while all of the others are $O(1)$, the cofactors of $\mathcal{M}$ are all $O(M)$. Hence, there holds ultimately \eqref{upperboundinversehessustrho}.
\end{proof}

\myindent This lemma already yields a restriction on how close to $P_{\alpha}$ we need to choose $(u,s,t,w,r)$, that is on an isotropic ball of radius of order $M^{-1}$. Now, assuming that we can prove the existence of the stationary point $Z_{\sigma,A,B}(w)$, with the notations of Proposition \ref{delicatecase}, Lemma \ref{nondegenerateustrhohess} yields that 
\begin{equation}
    \forall |w -\alpha| \lesssim M^{-1} \qquad \|H(w)^{-1}\| \lesssim M^{-1},
\end{equation}
i.e. this already proves that
\begin{equation}
    \mathcal{N}(\Psi) \lesssim M.
\end{equation}

\quad

\myindent Now, in order for Theorem \ref{mixedVdCABZ} to apply, in the extended version given by Remark \ref{nonisotropicmixedABZVdC}, the right kind of condition is equation \eqref{oscintconditiondelicatecase}, which reads in terms of the variables $(w,y)$
\begin{equation}
    \|M(w,y) - H(w)\| \leq \frac{1}{2} \|H(w)^{-1}\|^{-1}.
\end{equation}

\myindent In that spirit, there actually holds the following stronger version.

\begin{lemma}\label{scalingofnieghb}
    Let $i \in \{0,...,I\}$, $\alpha = 0, \pi$ and $\sigma \in \mathcal{I}_i$. Let
    \begin{equation}\label{localdefmathcalU}
        \mathcal{U} := \left\{ (u,s,t,w,r) \ \text{such that} \ \begin{cases}
            |(s,t,w,r) - (0,0,\alpha, r_{\sigma}(\alpha))| \lesssim M^{-1} \\
            |u - u_{\sigma}(\alpha)| \lesssim M^{-2} 
        \end{cases} \right\}.
    \end{equation}
    
    \myindent Then, there holds for all $(u,s,t,w,r), (u',s',t',w,r') \in \mathcal{U}$ 
    \begin{equation}\label{differenceofhessians}
        \|\nabla_{u,s,t,r}^2 \Psi(\sigma,A,B,u,s,t,w,r) - \nabla_{u,s,t,r}^2 \Psi(\sigma,A,B, u',s',t',w,r') \|  \leq \frac{1}{2} \left \| \left(\nabla_{u,s,t,r}^2 \Psi(\sigma,A,B,u,s,t,w,r)\right)^{-1}\right\|^{-1}.
    \end{equation}
\end{lemma}

\myindent Observe that, with the definition of $M(w,y)$ given by \ref{definitionofM}, this Lemma implies equation \eqref{oscintconditiondelicatecase} once we prove the existence of the stationary point $Z_{\sigma,A,B}(w)$. 

\begin{proof}
    There holds
    \begin{multline}\label{localdifferenceofhess}
        \|\nabla_{u,s,t,r}^2 \Psi(\sigma,A,B,u,s,t,w,r) - \nabla_{u,s,t,r}^2 \Psi(\sigma,A,B, u',s',t',w,r') \| \leq \|(\partial_u)^3 \Psi\|_{\infty} |u - u'| \\
        + \sum_{(i,j,k) \in \{u,s,t,r\}^3 \backslash (u,u,u)} \|\partial_{ijk} \Psi\|_{\infty} |k - k'|,
    \end{multline}
    where the $\|\cdot\|_{\infty}$ norms are taken on $\mathcal{U}$. Now, the point is that, thanks to the special decomposition of the phase $\Psi$ given by \eqref{decompositionofPsi}, then whenever $i,j,k$ are three variables among $(u,s,t,r)$ which are not all equal to $u$, then $\partial_{i,j,k} \Psi$ is \textit{independent} of $(A,B)$. In particular, 
    \begin{equation}
        \|\partial_{i,j,k} \Psi\|_{\infty} \lesssim_ 1.
    \end{equation}
    
    \myindent Moreover, from the explicit expression of $\Psi$, there obviously holds that
    \begin{equation}
        \|(\partial_u)^3 \Psi\|_{\infty} \lesssim M.
    \end{equation}
    
    \myindent Hence, the point is that, provided the implicit constants in the definition \eqref{localdefmathcalU} of $\mathcal{U}$ are small enough, then the RHS of \eqref{localdifferenceofhess} can be made arbitrarily smaller than $M^{-1}$, which concludes the proof thanks to Lemma \ref{nondegenerateustrhohess}.
\end{proof}

\subsubsection{The existence and uniqueness of the $(u,s,t,r)$ stationary point}\label{subsubsec44Hormphase}

\myindent Lemma \ref{scalingofnieghb} already hints to the scaling of the natural neighborhood of $P_{\alpha}$ on which Remark \ref{nonisotropicmixedABZVdC} applies. Now, we prove that this scaling is actually the natural one to obtain the \textit{existence} of the stationary point $Z_{\sigma,A,B}(w)$. 

\begin{lemma}\label{exunicofzero}
    One can choose the implicit constants in Proposition \ref{delicatecase} so that
    \begin{equation}
        \forall w \in I\  \exists ! \ Z_{\sigma,A,B} (w) = (u,s,t,r)(w) \in U \ \text{such that} \ \nabla_{u,s,t,r} \Psi(\sigma,A,B,u(w),s(w),t(w),w,r(w)) = 0.
    \end{equation}
\end{lemma}

\begin{proof}
    We use the variables $(w,y)$ in the proof. In particular, we set $I \times \mathcal{U}$ the image of the neighborhood of $P_{\alpha}$ $I\times U$ in the coordinates $(w,y)$.
    
    \quad
    
    \myindent Let us first prove the uniqueness. We claim that, provided the implicit constants in the definition of $I$ and $U$ \eqref{defIUdelicatecase} are small enough, then, for all $w \in I$, 
    \begin{equation}\label{injectivegradient}
        y \in \mathcal{U} \mapsto \nabla_y \phi(w,y) 
    \end{equation}
    is injective. Indeed, for any $y,y' \in U$, there holds
    \begin{equation}\label{differencenablayphi}
    \begin{split}
        \nabla_y \phi(w,y') - \nabla_y\phi(w,y) &= \left(\int_0^1 \nabla_y^2 \phi(w, vy + (1-v)y') dv\right) (y-y') \\
        &= \left[\nabla_y^2 \phi(w,y) + \left(\int_0^1 \nabla_y^2 \phi(w, vy + (1-v)y') dv - \nabla_y^2 \phi(w,y) \right) \right] (y- y').
    \end{split}
    \end{equation}
    
    \myindent Now, we can choose the implicit constants small enough so that $I\times U \subset \mathcal{U}$ where $\mathcal{U}$ is defined by Lemma \ref{scalingofnieghb}. In particular, on $I\times U$, equation \eqref{differenceofhessians} applies, which, in terms of the coordinates $(w,y)$, obviously implies that 
    \begin{equation}
        \left\|\int_0^1 \nabla_y^2 \phi(w, vy + (1-v)y') dv - \nabla_y^2 \phi(w,y) \right\| \leq \frac{1}{2} \left\|\left(\nabla_y^2 \phi(w,y)\right)^{-1}\right\|^{-1}.
    \end{equation}
    
    \myindent In particular, we are in the setting of Lemma \ref{inversematrixlemma}. Thus, we can deduce that
    \begin{equation}
        \int_0^1 \nabla_y^2 \phi(w, vy + (1-v)y') dv
    \end{equation}
    is an invertible matrix. Coming back to equation \eqref{differencenablayphi}, we can conclude to the injectivity of \eqref{injectivegradient}, hence the uniqueness part of the lemma.
    
    \quad
    
    \myindent The most delicate part is the existence. Indeed, it ultimately amounts to finding a quantitative version of the inverse function theorem. Actually, it is easier to come back to the proof of this theorem, and the elementary contraction mapping theorem. Recall that, thanks to \eqref{nablayphi}, there holds 
    \begin{equation}
        \nabla_y \phi(w,y) = \nabla_y \phi(w,0) + \mathcal{M}(w,y) y,
    \end{equation}
    where $|\nabla_y \phi(w,0)| \ll 1$ by hypothesis. Now, the condition for $y$ to be a stationary point is thus that
    \begin{equation}\label{condstatpoint}
        \nabla_y \phi(w,0) + \mathcal{M}(w,y) y = 0.
    \end{equation}
    
    \myindent We want to express that as a fixed point condition. Now, in order not to invert $\mathcal{M}(w,y)$ in the fixed point argument, let us rewrite \eqref{condstatpoint} in the form
    \begin{equation}
        \nabla_y \phi(w,0) + \mathcal{H}(w) y + \left(\mathcal{M}(w,y) - \mathcal{H}(w)\right) y = 0,
    \end{equation}
    which is equivalent to 
    \begin{equation}
        y = -\mathcal{H}(w)^{-1} \nabla_y \phi(w,0) - \left(\mathcal{H}(w)^{-1} \mathcal{M}(w,y) - I_4\right) y. 
    \end{equation}
    
    \myindent Thus, if we define, for $w\in I$, the map
    \begin{equation}
        F_w : y \in \mathcal{U} \mapsto -\mathcal{H}(w)^{-1} \nabla_y \phi(w,0) - \left(\mathcal{H}(w)^{-1} \mathcal{M}(w,y) - I_4\right) y, 
    \end{equation}
    we find that
    \begin{equation}
        \nabla_y \phi(w,y) = 0 \qquad \iff \qquad F_w(y) = y.
    \end{equation}
    
    \myindent Now, in order to prove the existence part of the lemma, we thus need only prove that we can choose the implicit constant in the definitions of $I$ and $U$ so that for all $w\in I$, $F_w$ is a contraction on $U$. 
    
    \quad
    
    \myindent First, observe that
    \begin{equation}\label{localusefulequniqueness}
        F_w(y) - F_w(y') = \mathcal{H}(w)^{-1} (\mathcal{M}(w,y) - \mathcal{M}(w,y')) y + (\mathcal{H}(w)^{-1}\mathcal{M}(w,y) - I_4)(y-y').
    \end{equation}
    
    \myindent Thanks to Lemma \ref{scalingofnieghb}, we already know that
    \begin{equation}
        \|\mathcal{H}(w)^{-1} \mathcal{M}(w,y) - I_4\| \leq \frac{1}{2}.
    \end{equation}
    
    \myindent Hence,
    \begin{equation}\label{localusefulequniqueness1}
        \left|(\mathcal{H}(w)^{-1}\mathcal{M}(w,y) - I_4)(y-y')\right| \leq \frac{1}{2} |y- y'|.
    \end{equation}
    
    \myindent Now, we claim that, provided the implicit constants in the definition of $I$ and $U$ are small enough, then there holds for any $w \in I$ and $y,y' \in \mathcal{U}$ that
    \begin{equation}\label{localusefulequniqueness2}
        \left|\mathcal{H}(w)^{-1} (\mathcal{M}(w,y) - \mathcal{M}(w,y')) y\right| \leq \frac{1}{4} |y - y'|.
    \end{equation}
    
    \myindent Indeed, write
    \begin{equation}
    \begin{split}
        &y = (u, s,t,r) \\
        &y' = (u', s',t',r').
    \end{split}    
    \end{equation}
    
    \myindent Then, there holds, providing the implicit constants in the definition of $I$ and $U$ are small enough, and thanks to Lemma \ref{nondegenerateustrhohess},
    \begin{equation}
    \begin{split}
        |(\mathcal{M}(w,y) - \mathcal{M}(w,y'))y| &\leq \|(\partial_u)^3 \Psi\|_{\infty} |u - u'| |u| + \sum_{(i,j,k,l)\in\{u,s,t,r\}^3, (i,j,k) \neq (u,u,u)} \|\partial_{ijk}\Psi\|_{\infty} |k - k'||l|\\
        &\ll M |u - u'|  M^{-2} + |y - y'| M^{-1} \\
        & \leq \frac{1}{4} \|\mathcal{H}(w)^{-1}\|^{-1} |y-y'|.
    \end{split}
    \end{equation}
    
    \myindent Now, coming back to \eqref{localusefulequniqueness}, equations \eqref{localusefulequniqueness1} and \eqref{localusefulequniqueness2} thus ensures ultimately that
    \begin{equation}
        \forall w \in I, \ \forall y,y' \in \mathcal{U} \qquad |F_w(y) - F_w(y')| \leq \frac{3}{4} |y-y'|.
    \end{equation}
    
    \quad
    
    \myindent In order to apply the contraction mapping theorem, it only remains to prove that, choosing the implicit constants carefully, there holds that
    \begin{equation}\label{inclusionequation}
        F_w (\mathcal{U}) \subset \mathcal{U}.
    \end{equation}
    
    \myindent Now, this part is actually quite tedious. Indeed, up to now, the only conditions on the implicit constants were that they are small enough. In this part, we actually consider how they depend one on another. For that purpose, let us write explicitly the implicit constants in Proposition \ref{delicatecase} in the following form : let us write condition \eqref{smallpartuPsiPalpha} in the form
    \begin{equation}
        |\partial_u \Psi(P_{\alpha})| \leq c_1 M^{-1},
    \end{equation}
    and the definitions of $I$ and $U$ in the form
    \begin{equation}
        \begin{split}
            I &:= \{w \in S^1 \ \text{such that} \ |w-\alpha| \leq c_2 M^{-1} \} \\
            U &:= \left\{(u,s,t,r) \in \R / \ell \Z \times \tilde{\mathcal{R}}_i^{(H)} \times \R_+ \ \text{such that} \ \begin{cases}
                |u - u_{\sigma}(\alpha)| \leq c_3 M^{-2} \\
                |(s,t,r) - (0,0, r_{\sigma}(\alpha)| \leq c_3 M^{-1}
            \end{cases}
            \right\}.
        \end{split}
    \end{equation}
    
    \myindent First, we claim that we can impose that for all $(w,y) \in I\times \mathcal{U}$, there holds
    \begin{equation}
        (\mathcal{H}(w)^{-1} \mathcal{M}(w,y) - I_4)y \in \frac{1}{2} \mathcal{U} := \left\{(u,s,t,r) \ \text{such that} \ \begin{cases}
            |u| \leq \frac{1}{2} c_3 M^{-2} \\
            |(s,t,r)| \leq \frac{1}{2} c_3 M^{-1}
        \end{cases} \right\}.
    \end{equation}
    
    \myindent Indeed, on the one hand, Lemma \ref{scalingofnieghb} ensures that, provided the constants are small enough,
    \begin{equation}
        |(\mathcal{H}(w)^{-1} \mathcal{M}(w,y) - I_4)y| \leq \frac{1}{2} |y|.
    \end{equation}
    
    \myindent Now, this already yields the desired inequality on the $(s,t,r)$ coordinates (we recall that $M \geq 1$ so $c_2M^{-1}\leq c_2$).
    
    \myindent Regarding the $u$ coordinates, the argument is a little tedious. Actually, from direct computation, one can prove that, for any $(w,y) \in I \times U$, there holds
    \begin{equation}
        \mathcal{M}(w,y) - \mathcal{H}(w) = \begin{pmatrix}
            c_{11}M^{-1} & c_{12} M^{-2} & c_{13} M^{-2} & 0 \\
            c_{12} M^{-2} & 0 & 0 & 0 \\
            c_{13}M^{-2} & 0 & c_{33}M^{-1} & c_{34} M^{-2} \\
            0 & 0 & c_{34}M^{-2} & 0
        \end{pmatrix},
    \end{equation}
    where there holds, for some universal constant $K> 0$,
    \begin{equation}
        \begin{split}
        |c_{11}|,|c_{12}|,c_{13}|, |c_{33}|, |c_{34}| \leq K c_3.
        \end{split}
    \end{equation}
    
    \myindent Now, a direct computation of cofactors (along with \eqref{lowerineqondetustrho}) yields that
    \begin{equation}
        \mathcal{H}(w)^{-1} = \begin{pmatrix}
            O(M^{-1}) & O(1) & O(1) & O(M^{-1}) \\
           \ & * & \ & \
        \end{pmatrix}.
    \end{equation}

    \myindent Hence, the first line of $\mathcal{H}(w)^{-1}\mathcal{M}(w,y) - I_4$ is of the form
    \begin{equation}
        \begin{pmatrix}
            \eps_1 M^{-2} & \eps_2 M^{-3} & \eps_3 M^{-1} & \eps_4M^{-2}
        \end{pmatrix},
    \end{equation}
    where, up to taking $K$ bigger,
    \begin{equation}
        |\eps_1|,|\eps_2|,|\eps_3|, |\eps_4| \leq Kc_3.
    \end{equation}

    \myindent Overall, the $u$ coordinate of
    \begin{equation}
        (\mathcal{H}(w)^{-1}\mathcal{M}(w,y) - I_4) y 
    \end{equation}
    is bounded (up to taking $K$ bigger) by
    \begin{equation}
        Kc_3^2 M^{-2}.
    \end{equation}
    
    \myindent Thus, provided $c_3$ is small enough, depending only on the universal constant $K$ the $u$ coordinate is bounded by $\frac{1}{2} c_3 M^{-2}$ as desired.

    \quad

    \myindent Hence, to conclude the proof of \eqref{inclusionequation}, it is enough to have
    \begin{equation}
        H(w)^{-1} \nabla_y \phi(w,0) \in \frac{1}{2}\mathcal{U}.
    \end{equation}

    \myindent Now, observe that
    \begin{equation}
        \nabla_y\phi(w,0) = \begin{pmatrix}
            \partial_u \Psi(P_{\alpha}) \\
            0 \\
            -h'(u_{\alpha}(\sigma))_2 + r_{\alpha}(\sigma) \partial_{\theta} \varphi((0,\sigma,),(0,\sigma),w) \\
            0
        \end{pmatrix}.
    \end{equation}

    \myindent Thanks to the hypothesis \eqref{smallpartuPsiPalpha}, we know that
    \begin{equation}
        |\partial_u\Psi(P_{\alpha})| \leq c_1 M^{-1}.
    \end{equation}

    \myindent Moreover, thanks to \eqref{asdevphi}, we know that
    \begin{equation}
        \partial_{\theta w} \varphi ((0,\sigma,),(0,\sigma),0) = 0.
    \end{equation}
    
    \myindent Hence, 
    \begin{equation}
        \begin{split}
            -h'(u_{\alpha}(\sigma))_2 + r_{\alpha}(\sigma) \partial_{\theta} \varphi((0,\sigma,),(0,\sigma),w) &=-h'(u_{\alpha}(\sigma))_2 + r_{\alpha}(\sigma) \partial_{\theta} \varphi((0,\sigma,),(0,\sigma),0) + O(|w|^2) \\
            &\lesssim c_2^2M^{-2},
        \end{split}
    \end{equation}
    since, by definition, there holds
    \begin{equation}
        -h'(u_{\alpha}(\sigma))_2 + r_{\alpha}(\sigma) \partial_{\theta} \varphi((0,\sigma,),(0,\sigma),0) = 0.
    \end{equation}

    \myindent Now, a direct computation yields that
    \begin{equation}
        \mathcal{H}(w)^{-1} = \begin{pmatrix}
            O(M^{-1}) & * & O(1) & * \\
            O(1) & * & O(M) & * \\
            O(1) & * & O(M) & * \\
            O(1) & * & O(1) & *
        \end{pmatrix}.
    \end{equation}

    \myindent Hence,
    \begin{equation}
        H(w)^{-1}\nabla_y \phi(w,0) = \begin{pmatrix}
            O((c_1 + c_2)^2M^{-2}) \\
            O((c_1 + c_2^2) M^{-1}) \\
            O((c_1 + c_2^2) M^{-1})\\
            O((c_1 + c_2^2) M^{-1})
        \end{pmatrix},
    \end{equation}
    which obviously belong to $\frac{1}{2} \mathcal{U}$ if we choose $c_1$ and $c_2$ small enough depending on $c_3$.
    
    \quad
    
    \myindent To conclude, in order for the proof to work, we need to make the choices in the following order : 
    
    \myindent i. First, all constants must be very small.
    
    \myindent ii. Then $c_1$ and $c_2$ have to be chosen very small \textit{depending} on $c_3$.
\end{proof}

\subsubsection{The 1D remaining phase function}\label{subsubsec45Hormphase}

\myindent Last, but not least, there remains to prove the estimate on the 1D phase function \eqref{4thderivativesPsi1Ddelicatecase}.

\begin{proof}
    We only give the idea of the proof in the case where
\begin{equation}\label{simplifiedcase}
        \partial_u \Psi(P_{\alpha}) = 0,
    \end{equation}
    since the argument in the case that this quantity is only $\ll M^{-1}$ only differs by tracking the exponents of $M$. This allows to understand better why one \textit{needs} to go up to the \textit{fourth} derivative of $\Psi^{1D}$ in order to find a uniformly nonzero derivative. For simplicity, we assume that $\alpha = 0$.

    \myindent When \eqref{simplifiedcase} holds, observe that
    \begin{equation}\label{nicecasevalueofustrho0}
        (u(0), s(0),t(0),r(0)) = (u_{\sigma}(0),0,0,r_{\sigma}(0)).
    \end{equation}

    \myindent In particular, the first derivative of $\Psi^{1D}$ vanishes since
    \begin{equation}
        \begin{split}
            \partial_w \Psi^{1D}(0) &= \partial_w \Psi(\sigma,A,B,u_{\sigma}(0),0,0,0,r_{\sigma}(0))\\
            &= 0,
        \end{split}
    \end{equation}
    and since $P_0$ is a $(s,t,w,r)$ stationary point.

    \myindent The second derivative of $\Psi^{1D}$ vanishes as well since, thanks to \eqref{fullhessianatPalpha},
    \begin{equation}
        \begin{split}
            (\partial_w)^2 \Psi^{1D}(0) &= \frac{\det(\nabla_{u,s,t,w,r}^2\Psi(\sigma,A,B,u_{\sigma}(0),0,0,0,r_{\sigma}(0)))}{\det(\nabla_{u,s,t,r}^2\Psi(P_0))} \\
            &= 0.
        \end{split}
    \end{equation}

    \myindent Assume for a moment, moreover, that $\sigma = 0$ is on the equator. Then, using the symmetries, there holds that 
    \begin{equation}
        w \mapsto (u,s,t,r)(w) \qquad \text{is even}.
    \end{equation}

    \myindent In particular, $\Psi^{1D}$ is also an even function of $w$. Thus, the \textit{third} derivative of $\Psi^{1D}$ vanishes as well at $w = 0$. 

    \myindent These short observations explain why it is necessary to consider at least the \textit{fourth} derivative of $\Psi^{1D}$. We now prove that it doesn't vanish at $w = 0$. Observe that, since
    \begin{equation}
        \nabla_{u,s,t,r} \Psi(u(w),s(w),t(w),w,r(w)) = 0,
    \end{equation}
    there holds
    \begin{equation}\label{1derinw}
        \begin{split}
            0 &= \left(\partial_w \nabla_{u,s,t,r} \Psi\right) (u(w),s(w),t(w),w,r(w)) + \nabla_{u,s,t,r}^2\Psi(u(w),s(w),t(w),w,r(w)) \cdot \begin{pmatrix}
                u'(w)\\
                s'(w)\\
                t'(w)\\
                r'(w)
            \end{pmatrix}\\
            &= \begin{pmatrix}
                0 \\
                0 \\
                r(w)\partial_{w\theta}\varphi(x + t(w)v_T,x,w,1)\\
                \partial_w \phi((t,\sigma,),(0,\sigma),w)
            \end{pmatrix} + \nabla_{u,s,t,r}^2 \Psi \cdot \begin{pmatrix}
                u' \\
                s'\\
                t'\\
                r'
            \end{pmatrix}.
        \end{split} 
    \end{equation}

    \myindent Now, thanks to \eqref{nicecasevalueofustrho0}, at $w = 0$, this equation yields
    \begin{equation}\label{u's't'rho'eq0}
        \nabla_{u,s,t,r}^2\Psi(P_0) \cdot \begin{pmatrix}
            u'(0)\\
            s'(0)\\
            t'(0)\\
            r'(0)
        \end{pmatrix} = 0,
    \end{equation}
    i.e. $u'(0) = s'(0) = t'(0) = r'(0) = 0$.

    \myindent Now, differentiating one time in $w$ the equation \eqref{1derinw} yields
    \begin{equation}\label{der2w}
    \begin{split}
        0&= \left(\partial_{ww}\nabla_{u,s,t,r} \Psi\right) (u(w),s(w),t(w),w,r(w)) +  \left(\partial_w\nabla_{u,s,t,r}^2\right)\Psi(u(w),s(w),t(w),w,r(w)) \cdot \begin{pmatrix}
                u'(w)\\
                s'(w)\\
                t'(w)\\
                r'(w)
            \end{pmatrix}\\
        &+ \begin{pmatrix} u'(w) & s'(w) & t'(w) & r'(w) \end{pmatrix}\cdot \nabla_{u,s,t,r}^3 \Psi (u(w),s(w),t(w),w,r(w)) \cdot\begin{pmatrix} u'(w) \\ s'(w) \\ t'(w) \\ r'(w) \end{pmatrix}\\
        &+ \nabla_{u,s,t,r}^2\Psi(u(w),s(w),t(w),w,r(w)) \cdot \begin{pmatrix}
                u''(w)\\
                s''(w)\\
                t''(w)\\
                r''(w)
            \end{pmatrix}.
        \end{split}
    \end{equation}

    \myindent Now, there holds
    \begin{equation}
        \begin{split}
            \partial_{ww}\nabla_{u,s,t,r}\Psi &= \begin{pmatrix}
                0 \\
                0 \\
                r \partial_{ww\theta}\varphi \\
                \partial_{ww}\varphi
            \end{pmatrix} \\
            \partial_w \nabla_{u,s,t,r}^2 \Psi &= \begin{pmatrix}
                0 & 0 & 0 & 0 \\
                0 & 0 & 0 & 0\\
                0 & 0 & r\partial_{\theta\theta w} \varphi & \partial_{w\theta}\varphi \\
                0 & 0 & \partial_{w\theta}\varphi & 0
            \end{pmatrix}.
        \end{split}
    \end{equation}

    \myindent Hence, at $w = 0$, using \eqref{u's't'rho'eq0}, equation \eqref{der2w} reads
    \begin{equation}
        \begin{split}
        0 &= \begin{pmatrix}
            0 \\
            0 \\
            r_{\sigma}(0) \partial_{ww\theta}\varphi((0,\sigma,),(0,\sigma),0)\\
            0
        \end{pmatrix} + \nabla_{u,s,t,r}^2\Psi(P_0)\cdot \begin{pmatrix}
            u''(0)\\
            s''(0) \\
            t''(0)\\
            r''(0)
        \end{pmatrix} \\
        &= \begin{pmatrix}
            0 \\
            0 \\
            r_{\sigma}(0) \partial_{ww\theta}\varphi((0,\sigma,),(0,\sigma),0)\\
            0
        \end{pmatrix} + \begin{pmatrix}
            -\scal{h''(u_{\sigma}(0))}{(A,B)} & -h'(u_{\sigma}(0))_1 & -h'(u_{\sigma}(0))_2 & 0 \\
            -h'(u_{\sigma}(0))_1 & 0 & 0 & 1 \\
            -h'(u_{\sigma}(0))_2 & 0 & 0 & q_2(\sigma,0,1)\\
            0 & 1 & q_2(\sigma,0,1) & 0
        \end{pmatrix}\cdot \begin{pmatrix}
            u''(0)\\
            s''(0) \\
            t''(0)\\
            r''(0)
        \end{pmatrix}.
        \end{split}
    \end{equation}

    \myindent From this equation, and the observation that $(A,B)$ is colinear to $h''(u_{\sigma}(0))$ (nonzero by Hypothesis \ref{twisthypothesis}), since they are both orthogonal to $h'(u_{\sigma}(0))$ (from \eqref{simplifiedcase}),and the fact that, thanks to Lemma \ref{thirdthetawderivaticevarphi},
    \begin{equation}
        \partial_{ww\theta}\varphi((0,\sigma,),(0,\sigma),0) \neq 0,
    \end{equation}
    one can deduce that
    \begin{equation}\label{u''rho''s''t''0}
        \begin{split}
            &|u''(0)| \simeq 1 \\
            &|r''(0)|\simeq 1\\
            &|s''(0)| \simeq M\\
            &|t''(0)|\simeq M.
        \end{split}
    \end{equation}

    \myindent Now, there holds
    \begin{equation}
        \begin{split}
            \partial_w\Psi^{1D}(w) &= \partial_w \Psi(u(w),s(w),t(w),w,r(w))\\
            &= r(w)\partial_w\varphi((t,\sigma,),(0,\sigma),w).
        \end{split}
    \end{equation}

    \myindent Using all the equation on vanishing quantities, one can deduce from differentiating this equation that
    \begin{equation}
        (\partial_w)^i \Psi^{1D}(0) = 0 \qquad i = 1,2,3
    \end{equation}
    (in particular, the third derivative vanishes even if $\sigma \neq 0$). Moreover, differentiating carefully, one can see that 
    \begin{equation}
        \begin{split}
            (\partial_w)^4 \Psi^{1D}(0) &= t''(0) \partial_{ww\theta}\varphi((0,\sigma,),(0,\sigma),0) \\
            &\simeq M,
        \end{split}
    \end{equation}
    where we use \eqref{u''rho''s''t''0} and Lemma \ref{thirdthetawderivaticevarphi}. 

    \myindent The full proof of \eqref{4thderivativesPsi1Ddelicatecase} doesn't use more refined ideas than this computation, it is only more painful since one needs to check that the above reasoning still holds for $|w| \ll M^{-1}$ (and not only $w = 0$), and when there doesn't hold exactly \eqref{simplifiedcase} but rather \eqref{smallpartuPsiPalpha}.
\end{proof}

\subsubsection{The case $|\partial_u(P_{\alpha})|$ large enough}\label{subsubsec46Hormphase}

\myindent To conclude Section \ref{subsec4HormPhase}, it remains to deal with the easier case where $|\partial_u \Psi(P_{\alpha})|$ is large enough. There holds the following.

\begin{lemma}\label{easycase}
     Let $i \in \{0,...,I\}$, let $\alpha = 0$ or $\alpha = \pi$, and let $\sigma \in \mathcal{I}_i$. Assume that
    \begin{equation}\label{easiercasefirsteq}
        |\partial_u \Psi(P_{\alpha})| \geq  cM^{-1},
    \end{equation}
    where $c$ is the implicit constant given by Proposition \ref{delicatecase} in equation \eqref{smallpartuPsiPalpha}. Let
    \begin{equation}\label{defDeasycase}
        \mathcal{D} := \left\{ (u,s,t,w,r) \ \text{such that} \ \begin{cases}
            |u - u_{\sigma}(\alpha)| \lesssim_{c} M^{-2} \\
            |(s,t)|\lesssim_{c} M^{-1}
        \end{cases} \right\}.
    \end{equation}
    
    \myindent Then, there holds
    \begin{equation}\label{easiercasesecondeq}
        \forall (u,s,t,w,r) \in \mathcal{D} \qquad
        |\partial_u \Psi(\sigma,A,B,u,s,t,w,r)| \gtrsim_{c} M^{-1}.
    \end{equation}. 
\end{lemma}

\begin{proof}
    Observe that
    \begin{equation}
        \partial_u \Psi(\sigma,A,B,u,s,t,w,r) = -\scal{h'(u)}{(s,t) + (A,B)}.
    \end{equation}
    
    \myindent Hence, there holds
    \begin{equation}
        \left| \partial_u \Psi(P_{\alpha}) - \partial_u \Psi(\sigma,A,B,u,s,t,w,r)\right| \lesssim M |u - u_{\sigma}(\alpha)| + |(s,t)|,
    \end{equation}
    from which the lemma is obvious.
\end{proof}

\section{The case $\bullet = (H)$ : quantitative estimate of the oscillatory integral}\label{secHormQuant}

\subsection{Quantitative estimate of the oscillatory integral : the case where $(A,B) = (0,0)$}\label{subsec1HormQuant}

\myindent In this section, we prove the estimate \eqref{estIH00}, that is we prove that
\begin{equation}\label{exactformulalocal}
\begin{split}
    &\mathcal{I}_{\lambda,\delta,i}^{(H)}(\sigma,0,0)\\
    &= \lambda^2 \int \int \int e^{i\lambda \Psi(\sigma,0,0,u,s,t,\xi)} \chi(s,t) \hat{\rho}(\delta(s,t)) a(s,(t,\sigma),(0,\delta),\lambda\xi) |\det(h'(u),(h(u)))| dudsdtd\xi\\
    &=2\hat{\rho}(0,0) c_{W}(\sigma) + O(\lambda^{-1}).
\end{split}
\end{equation}

\myindent As we have already mentioned several times in Section \ref{secHormPhase}, the analysis is a little different in the case $(A,B) = (0,0)$, since, as we will see, the phase doesn't satisfy the hypotheses of Theorem \ref{mixedVdCABZ}. Actually, we will come back to a more intuitive and elementary stationary phase analysis, observing that, since $(A,B)$ is fixed, we don't have to worry about controlling the implicit constants in the stationary phase lemma (see Paragraphs \ref{subsubsec32Micro} and \ref{subsubsec43Micro}).

\subsubsection{Reducing to a compact domain of integration}\label{subsubsec11HormQuant}

\myindent The first difficulty in estimating $\mathcal{I}_{\lambda,\delta,i}^{(H)}(\sigma,0,0)$ is that the domain of integration is a priori
\begin{equation}
    (\R / \ell \Z) \times 2\tilde{\mathcal{R}}_i^{(H)} \times \R^2,
\end{equation}
which is non compact. In particular, the integral only converges in the sense of oscillatory integrals, and we can't directly apply the usual theorems. 

\myindent Now, Lemma \ref{nicebound!} indicates that the integral decays as $\lambda^{-\infty}$ on the zones where $\xi$ is either close to zero, or large enough, thanks to the non-existence of stationary points in $(s,t)$ of the phase $\Psi$. This allows to reduce the $\xi$ domain of integration to a compact domain. Indeed, write the following partition of unity
\begin{equation}
    1 = \chi_1(\xi) + \chi_2(\xi),
\end{equation}
where $\chi_1,\chi_2$ are smooth, and
\begin{equation}\label{defpartunitcompact}
    \begin{split}
    \chi_1(\xi) &= \begin{cases}
        &0 \qquad |\xi| \leq \frac{1}{2}\beta^{-1} \ \text{or} \ |\xi| \geq 2\beta \\
        & 1 \qquad \beta^{-1} \leq |\xi| \leq \beta
    \end{cases}\\
    \chi_2(\xi) &= \begin{cases}
        & 1 \qquad |\xi| \leq \frac{1}{2}\beta^{-1} \ \text{or} \ |\xi| \geq 2\beta \\
        & 0 \qquad \beta^{-1} \leq |\xi| \leq \beta
    \end{cases}.
    \end{split}
\end{equation}

\myindent Then, on the support of $\chi_2$, since $\nabla_{s,t} \Psi$ doesn't vanish, we may integrate by parts in $(s,t)$ thanks to Lemma \ref{nicebound!} to find that, for any $N \geq 0$,
\begin{multline}
    \lambda^2 \int \int \int e^{i\lambda\Psi(\sigma,0,0,u,s,t,\xi)} \chi(s,t) \chi_2(\xi) \hat{\rho}(\delta(s,t)) a(s, (t,\sigma,),(0,\sigma), \lambda\xi)  |\det(h'(u),h(u))| du ds dt d\xi \\
    = \lambda^{2 - N} \int \int \int e^{i\lambda\Psi} \chi_2 \left(\left(\frac{\nabla_{s,t}\Psi \cdot \nabla_{s,t}}{i|\nabla_{s,t}\Psi|^2}\right)^t\right)^N \left(\chi(s,t) \hat{\rho}(\delta(s,t)) a(s, (t,\sigma,),(0,\sigma),\lambda\xi)  |\det(h'(u),h(u))| \right) du ds dt d\xi.
\end{multline}

\myindent Now, observe that, since $\chi,\hat{\rho}$ and $a$ are bounded along with all their $(s,t)$ derivatives, and thanks to Lemma \ref{nicebound!}, the integrand is $O((1+|\xi|)^{-N})$. Hence, for $N\geq 3$, the integral is absolutely converging. As a consequence,
\begin{multline}\label{reducetocompact}
    \mathcal{I}_{\lambda,\delta,i}^{(H)}(\sigma,A,B) = \lambda^2 \int \int \int e^{i\lambda \Psi(\sigma,0,0,u,s,t,\xi)} \chi(s,t) \chi_1(\xi) \hat{\rho}(\delta(s,t)) a(s, (t,\sigma,),(0,\sigma), \lambda\xi) |\det(h'(u),h(u))|  dudsdt  d\xi\\ + O\left(\lambda^{2-N}\right),
\end{multline}
i.e. we need only estimate the integral where the $\xi$ support has been reduced to $\xi \sim 1$. On that region, we may introduce the change of variable given by \eqref{defwrho}
\begin{equation}
    \xi = r g_{\sigma}(w).
\end{equation}

\myindent In particular, there holds
\begin{multline}\label{defIlambadeltaABeqzero}
    \lambda^2 \int \int \int e^{i\lambda \Psi(\sigma,0,0,u,s,t,\xi)} \chi(s,t) \chi_1(\xi) \hat{\rho}(\delta(s,t)) a(s, (t,\sigma,),(0,\sigma), \lambda\xi)  |\det(h'(u),h(u))|  dudsdt d\xi \\
    =\lambda^2 \int \int \int e^{i\lambda \Psi(\sigma,0,0,u,s,t,w,r)} b(\lambda, \delta, u,s,t,w,r) dudsdt dw dr =: \mathcal{I}(\lambda,\delta),
\end{multline}
where 
\begin{multline}
b(\lambda,\delta,u,s,t,w,r) =\\ \chi(s,t) \chi_1(w,r) \hat{\rho}(\delta(s,t)) a(s, (t,\sigma,),(0,\sigma), \lambda r g_{\sigma}(w)) r |\det(g_{\sigma}'(w),g_{\sigma}(w))| |\det(h'(u),h(u))|.
\end{multline}

\subsubsection{The $(u,s,t,w,r)$ stationary points of the phase $\Psi$}\label{subsubsec12HormQuant}

\myindent As announced in the introduction Section \ref{subsec1HormQuant}, in order to estimate the oscillatory integral $\mathcal{I}(\lambda,\delta)$ defined by \eqref{defIlambadeltaABeqzero}, let us come back to an elementary analysis, that is let us first ask the question of the stationary points of $\Psi$. 

\myindent In Section \ref{secHormPhase}, we have computed the set $\mathcal{O}_{\sigma} = \{\nabla_{s,t,w,r}\Psi = 0\}$, which we recall is the reunion of the circle $\mathcal{C}_{\sigma}$ defined in Lemma \ref{defcirclestatpoints} and of the two branches $\mathcal{E}_{\sigma,\alpha}$ defined in Lemma \ref{newhorriblelemma}. Hence, in order to compute the stationary points of $\Psi$, we need only compute those points of $\mathcal{O}_{\sigma}$ at which $\partial_u \Psi = 0$. We give the following lemma, which is actually not necessary for proving the estimate \eqref{estIH00}, but gives some intuition on the geometry of the oscillatory integral in the case $(A,B) = (0,0)$, compared to the case $(A,B) \neq (0,0)$.

\begin{lemma}\label{zerosetnablaustrhoabeq0}
    Let $(A,B) = (0,0)$. Then, the zero set of the full gradient of $\Psi$, that is
    \begin{equation}
        \{(u,s,t,w,r) \qquad \text{such that} \qquad \nabla_{u,s,t,w,r}\Psi(\sigma,0,0,u,s,t,w,r) = 0\},
    \end{equation}
    is exactly the circle $\mathcal{C}_{\sigma}$ defined in Lemma \ref{defcirclestatpoints}.
\end{lemma}

\begin{proof}
    On one hand, since
    \begin{equation}
        \partial_u \Psi(\sigma,0,0,u,s,t,w,r) = -\scal{h'(u)}{(s,t)},
    \end{equation}
    we see that, for any
    \begin{equation}
        (u_{\sigma}(w),0,0,w,r_{\sigma}(w)) \in \mathcal{C}_{\sigma},
    \end{equation}
    there holds 
    \begin{equation}
        \partial_u \Psi(\sigma,0,0,u_{\sigma}(w),0,0,w,r_{\sigma}(w)) = 0.
    \end{equation}

    \myindent On the other hand, we claim that $\partial_u \Psi$ doesn't vanish along the branches $\mathcal{E}_{\sigma,\alpha}$ except at the branching points $P_{\alpha}$. Indeed, we compute along the branches
    \begin{equation}\label{localpartuPsialongbranch}
        \begin{split}
            \partial_u \Psi(\sigma,0,0,u_{\alpha}(\sigma,t),s_{\alpha}(\sigma,t),t,w_{\alpha}(\sigma,t),r_{\alpha}(\sigma,t)) &= -\scal{h'(u_{\alpha}(\sigma,t))}{(s_{\alpha}(\sigma,t),t)} \\
            &= -t\left(\scal{h'(u_{\alpha}(\sigma,0))}{(\partial_t s_{\alpha}(\sigma,0),1)} + O(t)\right).
        \end{split}
    \end{equation}
    
    \myindent Now, using \eqref{defsrhoualpha}, there holds
    \begin{equation}
    \begin{split}
        \begin{pmatrix}\partial_t s_{\alpha}(\sigma,0)\\
        1
        \end{pmatrix}
        &= \begin{pmatrix}
            -\partial_{\theta} \varphi((0,\sigma,),(0,\sigma),\alpha) \\
            1
        \end{pmatrix} \\
        &= \begin{pmatrix}
            q_1(\sigma,\nabla_x \varphi((0,\sigma,),(0,\sigma),\alpha)) \\
            q_2(\sigma,\nabla_x \varphi((0,\sigma,),(0,\sigma),\alpha))
        \end{pmatrix}^{\perp} \\
        &= \frac{1}{r_{\sigma}(\alpha)} \left(h(u_{\alpha}(\sigma,0))\right)^{\perp}.
    \end{split}
    \end{equation}
    
    \myindent Hence, we find that
    \begin{equation}
        \begin{split}
            -\scal{h'(u_{\alpha}(\sigma,0))}{(\partial_t s_{\alpha}(\sigma,0),1)} &= \frac{1}{r_{\sigma}(\alpha)}\det(h'(u_{\alpha}(\sigma,0)), h(u_{\alpha}(\sigma,0))) \\
            &\neq 0,
        \end{split}
    \end{equation}
    where we use Lemma \ref{deth'hnotzero} for the last equality. Hence, coming back to \eqref{localpartuPsialongbranch}, we finally find that $\partial_u \Psi$ cannot vanish along $\mathcal{E}_{\sigma,\alpha} \backslash\{P_{\alpha}\}$ provided $|t|$ is small enough.
\end{proof}

\myindent As a consequence, the set of stationary points of $\Psi$, $\{\nabla \Psi = 0\}$, is a \textit{curve}, naturally parameterized by $w$. This is in sharp contrast to the case $(A,B) \neq (0,0)$, for which the stationary points of $\Psi$ are \textit{isolated}. Now, since the stationary points of $\Psi$ are not isolated, there holds necessarily
\begin{equation}
    \forall (u_{\sigma}(w),0,0,w,r_{\sigma}(w))\in \mathcal{C}_{\sigma} \qquad \det(\nabla_{u,s,t,w,r}^2\Psi (u_{\sigma}(w),0,0,w,r_{\sigma}(w))) = 0,
\end{equation}
and one cannot apply the stationary phase lemma in the full 5 variables. However, it is natural to isolate the variable $w$, since it parameterize the set $\{\nabla \Psi = 0\}$, and to try and apply the stationary phase lemma in the remaining 4 variables. For that purpose, we prove the following

\begin{lemma}\label{statpointandhessianabzero}
    For $w \in S^1$, let 
    \begin{equation}
        \Psi_{\sigma,w} : (u,s,t,r) \in (\R/\ell\Z) \times 2 \tilde{\mathcal{R}}_i^{(H)} \times (Supp(\chi_1(w,\cdot))) \mapsto \Psi(\sigma,0,0,u,s,t,w,r).
    \end{equation}

    \myindent Then, provided $\tilde{\mathcal{R}}_i^{(H)}$ is small enough, for any $w \in S^1$ and $\sigma$ has one, and exactly one, stationary point given by
    \begin{equation}
        (u,s,t,r) = (u_{\sigma}(w),0,0,r_{\sigma}(w)).
    \end{equation}

    \myindent Moreover, at this stationary point, there holds
    \begin{equation}
        \begin{split}
            \det(\nabla_{u,s,t,r}^2\Psi_{w,\sigma}(u_{\sigma}(w),0,0,r_{\sigma}(w))) &= \det(\nabla_{u,s,t,r}^2 \Psi(\sigma,0,0,u_{\sigma}(w),0,0,w,r_{\sigma}(w))) \\
            &= r_{\sigma}(w)^{-2}\det(h'(u_{\sigma}(w)),h(u_{\sigma}(w)))^2
        \end{split}.
    \end{equation}

    \myindent In particular, this determinant is uniformly bounded away from zero.

    \myindent Finally, there holds
    \begin{equation}
        sgn(\nabla_{u,s,t,r}^2\Psi_{w,\sigma}(u_{\sigma}(w),0,0,r_{\sigma}(w))) = 0.
    \end{equation}
\end{lemma}

\begin{proof}
    The first part of the lemma is a direct consequence of Lemma \ref{zerosetnablaustrhoabeq0}. For the second part, we first compute
    \begin{equation}
        \begin{split}
            \det(\nabla_{u,s,t,r}^2 \Psi(\sigma,0,0,u_{\sigma}(w),0,0,w,r_{\sigma}(w))) &= \begin{vmatrix}
                0 & -h'(u_{\sigma}(w))_1 & -h'(u_{\sigma}(w))_2 & 0 \\
                -h'(u_{\sigma}(w))_1 & 0 & 0 & 1\\
                -h'(u_{\sigma}(w))_2 & 0 & \partial_{\theta\theta}\varphi((0,\sigma,),(0,\sigma),w) & q_2(\sigma,w,1) \\
                0 & 1 & q_2(\sigma,w,1) & 0
            \end{vmatrix} \\
            &= \begin{vmatrix}
                -h'(u_{\sigma}(w))_1 & 1 \\
                -h'(u_{\sigma}(w))_2 & q_2(\sigma,w,1)
            \end{vmatrix}^2 \\
        &= r_{\sigma}(w)^{-2}\det(h'(u_{\sigma}(w)),h(u_{\sigma}(w)))^2,
        \end{split}
    \end{equation}
    by definition of $u_{\sigma}(w)$ and $r_{\sigma}(w)$ (see Lemma \ref{defcirclestatpoints}).

    \myindent More generally, this is the determinant of the matrix
    \begin{equation}
        \begin{pmatrix}
            0 & -h'(u_{\sigma}(w))_1 & -h'(u_{\sigma}(w))_2 & 0 \\
            -h'(u_{\sigma}(w))_1 & 0 & 0 & 1 \\
            -h'(u_{\sigma}(w))_2 & 0 & X & q_2(\sigma,w,1) \\
            0 & 1 & q_2(\sigma,w,1) & 0
        \end{pmatrix}
    \end{equation}
    for any value of $X \in \R$. Hence, for any $X \in \R$, this matrix either has 4 nonzero eigenvalues of the same sign, or two negative and two positive eigenvalues. In particular, the sign of this matrix is necessarily independent of $X$. If we choose $X = 0$, we find that the trace vanishes. Hence, there are necessarily two negative eigenvalues and two positive eigenvalues i.e. the sign of the matrix is $0$.
\end{proof}

\subsubsection{Exact result for the integral}\label{subsubsec13HormQuant}

\myindent We are now in a position to estimate $\mathcal{I}(\lambda,\delta)$. Isolating the variable $w$, we need to evaluate
\begin{equation}\label{IlambdadeltawisolatedABeqzero}
\mathcal{I}(\lambda,\delta) = \int_{S^1} dw\left( \lambda^2 \int \int \int e^{i\lambda\Psi(\sigma,0,0,u,s,t,w, r)} b(\lambda,\delta, u,s,t,w,r) du ds dt dr\right).
\end{equation}

\myindent Now, the usual stationary phase Lemma \ref{hormstatphase}, and Lemmas \ref{statpointandhessianabzero} along with the very important observation that for all $\sigma,u,w, r$
\begin{equation}
    \Psi(\sigma,0,0,u,0,0,w, r) = 0,
\end{equation}
ensure that the inner integral equals
\begin{multline}
    \lambda^2 \int \int \int e^{i\lambda\Psi(\sigma,0,0,u,s,t,w, r)}  b(\lambda,\delta, u,s,t,w,r)du ds dt dr \\
    = (2\pi)^{-2} \frac{1}{\left|r_{\sigma}(w)^{-1}\det(h'(u_{\sigma}(w)),h(u_{\sigma}(w)))\right|}   b(\lambda,\delta, u_{\sigma}(w),0,0,w,r_{\sigma}(w)) + R(\lambda,\sigma, w),
\end{multline}
where
\begin{equation}
    R(\lambda,\sigma,w) = O(\lambda^{-1})
\end{equation}
uniformly in $\delta,\sigma,w$. Indeed, in order to prove this last assertion, the only subtle point is that
\begin{equation}
    \|\nabla_{u,s,t,r} b \|_{L^{\infty}} \lesssim 1
\end{equation}
uniformly in $\lambda,\delta$, which itself is a direct consequence of the fact that
\begin{equation}\label{nolosslambdadera}
    \|\nabla_{\xi}^K \left(a(s,(t,\sigma,),(0,\sigma), \lambda \xi\right)\|_{L^{\infty}(supp(\chi) \times supp(\chi_1))} \lesssim 1,
\end{equation}
uniformly in $\lambda$ on the domain of integration. Indeed, when differentiating in $\xi$, while one loses a factor $\lambda$, one gains a factor $\lambda^{-1}|\xi|^{-1}$ through the \textit{symbol estimate} \eqref{estsymb} satisfied by $a$. Now, this is why we have ensured that $\xi$ is bounded away from zero on the domain of integration, since otherwise one \textit{could not} avoid losing some negative powers of $|\xi|$ through the symbol estimate. 

\quad

\myindent Now, there holds
\begin{multline}
    b(\lambda,\delta,u_{\sigma}(w),0,0,w,r_{\sigma}(w)) \\ = \hat{\rho}(0,0) a(0,(0,\sigma,),(0,\sigma),\lambda r_{\sigma}(w) g_{\sigma}(w)) r_{\sigma}(w) |\det(g_{\sigma}'(w),g_{\sigma}(w))| |\det(h'(u_{\sigma}(w)),h(u_{\sigma}(w)))|.
\end{multline}

\myindent Since, moreover, there holds
\begin{equation}
    \forall x \in \mathcal{S}, \forall  |\xi|\gtrsim 1 \qquad a(0,x,x,\lambda \xi) = 1 + O(\lambda^{-1})
\end{equation}
uniformly (see Theorem \ref{Hormthm}), one finds in fact that for all $w\in S^1$,  
\begin{multline}\label{exactcompute}
    \lambda^2 \int \int \int e^{i\lambda\Psi(\sigma,0,0,u,s,t,w, r)} b(\lambda,\delta,u,s,t,w,r) du ds dt dr \\
    = (2\pi)^{-2} r_{\sigma}(w)^2|\det(g_{\sigma}'(w),g_{\sigma}(w))| \hat{\rho}(0,0) + O(\lambda^{-1}),
\end{multline}
where the remainder is bounded uniformly in $\sigma,w,\delta$.

\quad

\myindent Overall, from \eqref{reducetocompact}, \eqref{IlambdadeltawisolatedABeqzero}, and \eqref{exactcompute}, we find that
\begin{equation}\label{explicitI10quasicomplete}
    \mathcal{I}_{\lambda,\delta,i}^{(H)}(\sigma,0,0) = (2\pi)^{-2} \hat{\rho}(0,0) \int_{S^1} r_{\sigma}(w)^2|\det(g_{\sigma}'(w),g_{\sigma}(w))|dw + O(\lambda^{-1}).
\end{equation}

\myindent In order to conclude to the exact formula \eqref{exactformulalocal}, it remains to observe that, thanks to Lemma \ref{defcirclestatpoints},
\begin{equation}
    r_{\sigma}(w) = \frac{1}{p_1(\sigma,w,1)}.
\end{equation}

\myindent Now, this can also be expressed in the form
\begin{equation}
    q_1\left(\sigma, \frac{g_{\sigma}(w)}{p_1(\sigma,g_{\sigma}(w))}\right) = r_{\sigma}(w).
\end{equation}

\myindent In other words, the curve $\{p_1(\sigma,\xi) = 1\}$ can be parameterized, in the polar coordinates $(w,r)$, as 
\begin{equation}
    \{p_1(\sigma,\xi) = 1\} = \left\{ (w, r_{\sigma}(w)), \qquad w \in S^1 \right\}.
\end{equation}

\myindent Thus, after the polar change of coordinate $\xi \mapsto (w,r)$, there holds
\begin{equation}
    2c_W(\sigma) = \int_{\{p_1(\sigma,\xi) = 1\}} d\xi = \int_{S^1} r_{\sigma}^2(w) |\det(g_{\sigma}'(w), g_{\sigma}(w))| dw.
\end{equation}

\myindent Coming back to \eqref{explicitI10quasicomplete}, we thus find the exact constant in \eqref{exactformulalocal}. 

\begin{remark}
    If we compare with the general intuition of Paragraph \ref{subsubsec43Micro}, with the notations of this paragraph, we have used the fact that
\begin{equation}
    \forall w \in S^1 \qquad \Psi^{1D}(w) = 0.
\end{equation}

\myindent In particular, there are \textit{no} remaining oscillations in the variable $w$, i.e. the phase $\Psi$ is degenerate without a finite type of degeneracy. This last part is actually a direct consequence of Lemma \ref{zerosetnablaustrhoabeq0}, which implies that $\Psi^{1D}(w)$ is a constant. This explains, on a heuristic level, why there is no possible $O(\lambda^{-\frac{1}{p}})$ improvements in the estimate \eqref{exactformulalocal}, compared to the case $(A,B) \neq (0,0)$.
\end{remark}

\subsection{Quantitative estimate of the oscillatory integral : the case where $(A,B) \neq (0,0)$}\label{subsec2HormQuant}

\myindent In this section, using the extensive analysis of the phase developed in Section \ref{secHormPhase}, we prove the estimates \eqref{estIHAB}, that is we prove that, for $(A,B) \neq (0,0)$,
\begin{equation}\label{notexactformulalocal}
\begin{split}
    &\mathcal{I}_{\lambda,\delta,i}^{(H)}(\sigma,A,B)\\
    &= \lambda^2 \int \int \int e^{i\lambda \Psi(\sigma,A,B,u,s,t,\xi)} \chi(s,t) \hat{\rho}(\delta(s + A,t + B)) a(s,(t,\sigma),(0,\delta),\lambda\xi) |\det(h'(u),(h(u)))| dudsdtd\xi\\
    &= O\left(\lambda^{-\frac{1}{4}} M^{5} + \lambda^{-\frac{1}{3}} M^7 + \lambda^{-\frac{1}{2}}M^{14}\right),
\end{split}
\end{equation}
and we also prove the refinement \eqref{estIHABi0} in the case $i = 0$.

\subsubsection{Strategy for quantitative bounds}\label{subsubsec21HormQuant}

\myindent This paragraph is a brief introduction to the following  paragraphs.

\myindent First, that the argument in Paragraph \ref{subsubsec11HormQuant} \textit{doesn't} depend on $(A,B,\sigma)$. In particular, we may already write, similarly to \eqref{reducetocompact},
\begin{equation}\label{reducetocompactnonzeroab}
\begin{split}
    &\mathcal{I}_{\lambda,\delta,i}^{(H)}(\sigma,A,B) = \\
    &\lambda^2 \int \int \int e^{i\lambda \Psi(\sigma,A,B,u,s,t,w,r)} \chi(s,t) \chi_1(w,r) \hat{\rho}(\delta(s + A, t + B)) a(s, (t,\sigma,),(0,\sigma), \lambda\xi)r |\det(g'_{\sigma}(w), g_{\sigma}(w))| dudsdt dw dr \\
    &+ O\left(\lambda^{2-N}\right),
\end{split}
\end{equation}
where $\chi_1$ localizes in $|\xi| \in [\frac{1}{2} \beta^{-1}, 2 \beta]$. Hence, we wish to estimate the integral term in the RHS of \eqref{reducetocompactnonzeroab}. In order to have clearer notations, we define
\begin{multline}
    b(\lambda,A,B,\sigma,u,s,t,w,r) := \\\chi(s,t) \chi_1(w,r) \hat{\rho}(\delta(s + A, t + B)) a(s, (t,\sigma,),(0,\sigma), \lambda\xi)r |\det(g'_{\sigma}(w), g_{\sigma}(w))|  |\det(h'(u),h(u))|,
\end{multline}
and we already observe that for all $K\geq 1$,
\begin{equation}\label{boundallderivativesb}
    \|(\nabla_{u,s,t,w,r})^K b\|_{L^{\infty}} \lesssim_K 1.
\end{equation}

\myindent Indeed, the only subtelties are that $\hat{\rho}$ is a Schwartz function, hence it is uniformly bounded along with all its derivatives, and that no $\lambda$ appears when differentiating $a$, which we have already observed in \eqref{nolosslambdadera}. We define moreover the local notation, for this section,
\begin{equation}\label{defIlocal}
    \mathcal{I} := \lambda^2 \int \int \int e^{i\lambda \Psi(\sigma,A,B,u,s,t,w,r)} b(\lambda,A,B,\sigma,u,s,t,w,r) du ds dt dw dr,
\end{equation}
which is ultimately the interesting quantity that we need to bound, using the extensive analysis of the Phase $\Psi$ that we have presented in Section \ref{secHormPhase}. 

\quad

\myindent Thanks to Section \ref{secHormPhase}, we know that there are three regions of interest

\myindent i. First, when $(u,s,t,w,r)$ are near the circle $\mathcal{C}_{\sigma}$ but away from the branching points $P_0$ and $P_{\pi}$, Section \ref{subsec2HormPhase} yields that we can apply the strategy of Theorem \ref{mixedVdCABZ} by isolating the variable $u$. Precisely, those paragraphs help to define a \textit{quantitative} region where this analysis works, in terms of powers of $M^{-1}$. This will be presented in Paragraph \ref{subsubsec22HormQuant}.

\myindent ii. Second, when $(u,s,t,w,r)$ are near the branches $\mathcal{E}_{\sigma,\alpha}$, and away from the branching points $P_0$ and $P_{\pi}$, we can similarly use the strategy of Theorem \ref{mixedVdCABZ} by isolating the variable $u$, thanks to Section \ref{subsec3HormPhase}. However, a difficulty is that this strategy only works for $\sigma \neq 0$, as we have already seen in the above. Hence, we need to quantify moreover a threshold on $\sigma$, in terms of powers of $M^{-1}$, and to distinguish two cases : the case when $\sigma$ is far away enough from $0$, which we will present in Paragraph \ref{subsubsec23HormQuant}; and the case when $\sigma$ is close to $0$, which we will present in Paragraph \ref{subsubsec24HormQuant}.

\myindent iii. Third, in any case, the most delicate part is to deal with the branching points $P_0$ and $P_{\pi}$. We will explain in Paragraph \ref{subsubsec25HormQuant} how to use the tools of Section \ref{subsec4HormPhase} in order to apply the strategy of Theorem \ref{mixedVdCABZ}, but with the twist of Remark \ref{nonisotropicmixedABZVdC}. Indeed, as we will see, it is this region which gives the \textit{dominant term} in the upper bound \eqref{estIHABi0}, namely it is responsible for the term $\lambda^{-\frac{1}{4}}$. Hence, even if this paper does \textit{not} aim for optimality, we still want to give an example of how looking closer at the phase yields improved estimates compared to what would be expected with Theorem \ref{mixedVdCABZ}.

\myindent iv. The first three points put together yield a quantitative neighborhood of $\mathcal{O}_{\sigma}$, the contribution of which to $\mathcal{I}$ we can control. Hence, the last point is to bound the region where $(u,s,t,w,r)$ is \textit{not} in the neighborhood of $\mathcal{O}_{\sigma}$, i.e. where $\nabla_{s,t,w,r}\Psi_0$ is never zero. Though this part is easier since the contribution to the integral is $O(\lambda^{-\infty})$, we still need to be a little careful in controlling the powers of $M$ which appear in the estimates. This will be presented in Paragraph \ref{subsubsec26HormQuant}.

\subsubsection{Quantitative estimate near the Circle $\mathcal{C}_{\sigma}$, away from $P_0$ and $P_{\pi}$}\label{subsubsec22HormQuant}

Near $\mathcal{C}_{\sigma} \backslash\{P_0,P_{\pi}\}$, we can apply Theorem \ref{mixedVdCABZ}, isolating the $u$ variable (since it is a circle of $(s,t,w,r)$ stationary points). Indeed, Section \ref{subsec2HormPhase} yield that, away from both $P_{\alpha}$, for small $(s,t)$, there are exactly two $(s,t,w,r)$ stationary points for each $u$. The remaining phase function is then
\begin{equation}
    \Psi^{1D}(u) = -\scal{h(u)}{(A,B)},
\end{equation}
whose first and second derivative don't vanish together (in sharp contrast with the case $(A,B) = (0,0)$. Hence, we find a bound of the form $\lambda^{-\frac{1}{2}}$ for the integral. Precisely, observe first that $\mathcal{C}_{\sigma} \backslash\{P_0,P_{\pi}\}$ has two connected components, which we define, using the notations of Lemma \ref{defcirclestatpoints}, by
\begin{equation}
\begin{split}
    \mathcal{C}_{\sigma}^+ &:= \left\{\left( u_{\sigma}(w), 0, 0, w, r_{\sigma}(w) \right) \qquad w \in (0,\pi)\right\}\\
    \mathcal{C}_{\sigma}^- &:= \left\{\left( u_{\sigma}(w), 0, 0, w, r_{\sigma}(w) \right) \qquad w \in (-\pi, 0)\right\}.
\end{split}
\end{equation}

\myindent Then, there holds the following, using the notations of Lemma \ref{locinverseuw}.

\begin{lemma}\label{quantitativestimateICx}
    Let $i \in \{0,...,I\}$ and $\sigma \in \mathcal{I}_i$. Let $c_1 > 0$ be a small constant. Let
    \begin{equation}\label{defIsigma}
        \begin{split}
        I_{\sigma} &:= \left(u_{\sigma}(\pi) + \frac{1}{2}c_1M^{-2}, u_{\sigma}(0) - \frac{1}{2}c_1M^{-2}\right)\\
         U^{\pm} &:= \{(s,t,w,r) \ \text{such that} \ \pm w \in (0,\pi) \}.
        \end{split}
    \end{equation}

    \myindent There exists a smooth function $\zeta(u,s,t,w,r)$ such that, for all $u \in I_{\sigma}$, $\zeta(u,\cdot)$ is a smooth localizer adapted to a ball
    \begin{equation}
        |(s,t,w,r) - (0,0,w_{\pm}(u,\sigma),r_{\pm}(u,\sigma)| \lesssim M^{-1},
    \end{equation}
    and such that, if
    \begin{equation}
        \chi_{\mathcal{C}_{\sigma}^{\pm}} = \chi_{I_{\sigma}}(u) \zeta(u,s,t,w,r),
    \end{equation}
    where $\chi_{I_{\sigma}}$ is the characteristic function of $I_{\sigma}$, there holds the following  
    \begin{multline}\label{defICpmx}
        \mathcal{I}_{\mathcal{C}^{\pm}_{\sigma}} := \lambda^2 \int\int\int e^{i\lambda \Psi(\sigma,A,B,u,s,t,w,r)} \chi_{\mathcal{C}_{\sigma}^{\pm}}(u,s,t,w,r) b(\lambda,A,B,\sigma,u,s,t,w,r) dudsdtdwdr \\
        = O\left(\lambda^{-\frac{1}{2}} M^{4}\right).
    \end{multline}
\end{lemma}

\begin{proof}
    Without loss of generality, we give the proof for $\mathcal{C}^+_{\sigma}$. We drop the notations of the $\sigma$ dependency in the proof. For all $u \in I_{\sigma}$, we know, thanks to Lemmas \ref{defcirclestatpoints} and \ref{locinverseuw}, that 
    \begin{equation}
        (s,t,w,r)(u) := \left((0,0,w_{+}(u), \frac{1}{p_1(\sigma,w_{+}(u),1)}\right)
    \end{equation}
    is a (smooth) stationary point of 
    \begin{equation}
        \Psi_u : (s,t,w,r)\in U^+ \mapsto \Psi(\sigma,A,B,u,s,t,w,r).
    \end{equation}

    \myindent Moreover, thanks to Lemma \ref{stwrhohessofPsionCx}, we know that, at this stationary point, the Hessian of $\Psi_u$ is nondegenerate. Precisely, this lemma, along with Corollary \ref{norminversehessCsig} yields that, using the Definition \ref{defHandN},
    \begin{equation}
        \begin{split}
        &|\det H(u)| \gtrsim |\sin(w_+(u))|^2 \gtrsim M^{-2}\\
            &\|H(u)^{-1}\| \lesssim \frac{1}{\sin(w_{+}(u))}
            \lesssim M,
        \end{split}
    \end{equation}
    since $u \in I_{\sigma}$ is at a distance at least $c_1M^2$ from both $u_{\sigma}(0)$ and $u_{\sigma}(\pi)$, and thanks to Lemma \ref{locinverseuw}. In particular, with the notations of Definition \ref{defHandN}, 
    \begin{equation}
    \begin{split}
        &\mathcal{D}(\Psi) \gtrsim M^{-2} \\
        &\mathcal{N}(\Psi) \lesssim M.
    \end{split}
    \end{equation}
    
    \myindent Moreover, since any derivative of $\Psi$ in $(s,t,w,r)$ of order greater than $1$ doesn't involve any factor of $(A,B)$, and since we have restricted to a bounded region for $r$, there holds for all $1 \leq k \leq l$
    \begin{equation}
    \begin{split}
        &\mathcal{M}_{k,l}^{(s,t,w,r)}(\Psi) \lesssim_{k,l} 1 \\
        &\mathcal{M}_{k,l}^{(s,t,w,r)}(\partial_u \Psi) \lesssim 1.
    \end{split}
    \end{equation}

    \myindent Finally, there holds thanks to \eqref{boundallderivativesb}
    \begin{equation}
        \mathcal{M}_{0,l}^{(s,t,w,r)}\left(b\right) \lesssim 1,
    \end{equation}
    and the $u$ derivative of $b$ vanish.

    \myindent Hence, provided we check property \eqref{VdCN} for the 1D phase function, the framework of Theorem \ref{mixedVdCABZ} applies indeed to bound the integral \eqref{defICpmx}, with $\zeta$ localized as claimed. Now, the remaining phase function is
    \begin{equation}\label{psi1DonCsig}
        \begin{split}
            \Psi^{1D}(u) &= \Psi(\sigma,A,B,u,(s,t,w,r)(u)) \\
            &= -\scal{h(u)}{(A,B)}.
        \end{split}
    \end{equation}

    \myindent Now, since $h'(u)$ and $h''(u)$ never vanish (cf Hypothesis \ref{twistV2}), and since they are orthogonal, we can find a partition of $\R/ \ell\Z$ by a bounded finite number of intervals, say $I_1,...,I_K$, such that
    \begin{equation}
        \forall i = 1,..,K \qquad \begin{cases}
            \text{either} \ \forall u \in I_K, \qquad |\scal{h'(u)}{A,B}| \gtrsim M \\
            \text{or} \ \forall u \in I_K, \qquad |\scal{h''(u)}{A,B}| \gtrsim M.
        \end{cases}
    \end{equation}
    
    \myindent In particular, coming back to \eqref{psi1DonCsig},we see that $\Psi^{1D}$ satisfies the property $(VdC)_2$ (see Definition \ref{VdCN}) with constants $cM, c'M$ for some universal $c,c' > 0$. Indeed, \eqref{psi1DonCsig} yields that
    \begin{equation}
        \|\partial_{uu}\Psi^{1D}\|_{L^{\infty}} \lesssim M.
    \end{equation}
    
    \myindent Hence, Theorem \ref{mixedVdCABZ} yields that
    \begin{equation}
        |\mathcal{I}_{\mathcal{C}^{\pm}_{\sigma}}|\lesssim \lambda^{- \frac{1}{2}} M^{-\frac{1}{2}} M^2 M^2,
    \end{equation}
    which concludes the proof.
\end{proof}

\begin{remark}
    The reader may observe that, in fact, in the region that we are considering, the \textit{full} Hessian of $\Psi$, in the variables $(u,s,t,w,r)$, is always invertible. Hence, it is not surprising that we find a bound $O(\lambda^{-\frac{1}{2}})$, sign that any stationary point in $(u,s,t,w,r)$ in that region is simply non degenerate. However, we have chosen to present the estimate using Theorem \ref{mixedVdCABZ} rather than Theorem \ref{ABZthm} since, under the more general generic hypothesis of Colin de Verdière given in Remark \ref{defgeneric}, the analysis would still work but this time the remaining phase function satisfies property $(VdC)_p$ (see \ref{VdCN}) with $p=3$. Hence, Theorem \ref{mixedVdCABZ} still applies and yields a bound $O(\lambda^{-\frac{1}{3}})$ but Theorem \ref{ABZthm} doesn't apply anymore.
\end{remark}

\subsubsection{Quantitative estimate near the branches $\mathcal{E}_{\sigma,\alpha}$, away from $P_0$ and $P_{\pi}$, when $\sigma$ is away from the equator}\label{subsubsec23HormQuant}

\myindent Near $\mathcal{E}_{\sigma,\alpha}\backslash\{P_{\alpha}\}$, we can apply Theorem \ref{mixedVdCABZ}, isolating the $u$ variable (since it is a branch of $(s,t,w,r)$ stationary points). However, in order to apply this theorem, we need to be able to parameterize the connected component $\mathcal{E}_{\sigma,\alpha}^{\pm}$ (defined by \eqref{defesigmaalphapm}) by $u$, at least away from $P_{\alpha}$. Now, thanks to Lemma \ref{deftpmu}, we see that this is possible if and only if $\sigma$ is away from the equator. Precisely, there holds

\begin{lemma}\label{quantitativeestimateEsigmaalphaawayfromeq}
    Assume that either $i\neq 0$, or $|\sigma| \geq C M^{-\frac{1}{2}}$ where $C$ is given by Lemma \ref{Psi1DnearEalpha}.

    Let $c_2 > 0$ be a small constant, and let
    \begin{equation}
        \begin{split}
        I_{\sigma,0} &:= \left(u_{\sigma}(0) - c\sigma^2, u_{\sigma}(0) - \frac{1}{2}c\sigma^2M^{-2}\right)\\
        I_{\sigma,\pi} &:= \left(u_{\sigma}(\pi) + \frac{1}{2} c\sigma^2M^{-2}, u_{\sigma}(\pi) + c\sigma^2\right)\\
         U^{\pm} &:= \{(s,t,w,r) \ \text{such that} \ \pm w \in (0,\pi)\}.
        \end{split}
    \end{equation}

     \myindent For $\alpha = 0,\pi$, there exists a smooth function $\zeta(u,s,t,w,r)$ such that, for all $u \in I_{\sigma,\alpha}$, $\zeta(u,\cdot)$ is a smooth localize adapted to a ball
     \begin{equation}
        |(s,t,w,r) - (s_{\pm}(u),t_{\pm}(u),w_{\pm}(u),r_{\pm}(u)| \lesssim \sigma^2 M^{-1},
    \end{equation}
     such that, if 
    \begin{equation}
        \chi_{\mathcal{E}_{\sigma,\alpha}^{\pm}} = \chi_{I_{\sigma,\alpha}}(u) \zeta(u,s,t,w,r),
    \end{equation}
    where $\chi_{I_{\sigma,\alpha}}$ is the characteristic function of $I_{\sigma,\alpha}$, there holds the following : when $i \neq 0$,
    \begin{multline}\label{defIEalphapmx}
        \mathcal{I}_{\mathcal{E}_{\sigma,\alpha}^{\pm}} := \lambda^2 \int\int\int e^{i\lambda \Psi(\sigma,A,B,u,s,t,w,r)} \chi_{\mathcal{E}_{\sigma,\alpha}^{\pm}}(u,s,t,w,r) b(\lambda,A,B,\sigma,s,t,w,r) dudsdtdwdr \\
        = O\left(\lambda^{-\frac{1}{3}} M^{4}\right).
    \end{multline}

    \myindent When $i = 0$,
    \begin{equation}
        \mathcal{I}_{\mathcal{E}_{\sigma,\alpha}^{\pm}} = O(\lambda^{-\frac{1}{3}} M^{7}).
    \end{equation}
\end{lemma}

\begin{proof}
    Without loss of generality, we give the proof for $\mathcal{E}_{\sigma,0}^+$. We apply Theorem \ref{mixedVdCABZ}, isolating the $u$ variable.
    
    \myindent For all $u \in I_{\sigma,0}$, we know, thanks to Lemma \ref{newhorriblelemma} and Corollary \ref{deftpmu}, and using the notations of the latter, that
    \begin{equation}
        (s,t,w,r)(u) := \left(s_0(\sigma,t_+(\sigma,u)),t_+(\sigma,u), w_0(\sigma, t_+(\sigma, u)),r_0(\sigma, t_+(\sigma, u))\right)
    \end{equation}
    is a smooth stationary point of
    \begin{equation}
        \Psi_u : (s,t,w,r)\in U_+ \mapsto \Psi(\sigma,A,B,u,s,t,w,r).
    \end{equation}

    \myindent Moreover, thanks to Lemma  \ref{exactstwrhohess} and Corollary \ref{deftpmu}, we know that 
    \begin{equation}
    \begin{split}
        \det(\nabla_{s,t,w,r}^2\Psi_u((s,t,w,r)(u))) &= t_+(\sigma,u)^2\sigma^2 F(\sigma,(s,t,w,r)(u)) \\
        &\simeq |u_{\sigma}(0)-u|
    \end{split}
    \end{equation}

    \myindent In particular, $\Psi_u$ is nondegnerate at the stationary point.  Now, on the one hand, we have already noticed in the proof of Lemma \ref{quantitativestimateICx} that for all $1\leq l$
    \begin{equation}
    \begin{split}
        \mathcal{M}_{1,l}^{(s,t,w,r)}(\Psi) \lesssim 1 \\
        \mathcal{M}_{1,l}^{(s,t,w,r)}(\partial_u \Psi) \lesssim 1.
    \end{split}
    \end{equation}

    \myindent On the other hand, we know, thanks to Corollaries \ref{norminversehessEsigalph} and \ref{deftpmu}, that
    \begin{equation}
    \begin{split}
        \|H(u)^{-1}\| &\lesssim \frac{1}{\sigma^2 |t_{\pm}(u)|} \\
        &\lesssim \frac{1}{\sigma \sqrt{|u - u_{\sigma}(0)|}},
    \end{split}
    \end{equation}
    so, in particular, because $I_{\sigma,0}$ stops at $u_{\sigma}(0) - c_2\sigma^2M^{-2}$, we find that, with the notations of Definition \ref{defHandN},
    \begin{equation}
        \begin{split}
        &\mathcal{D}(\Psi) \gtrsim \sigma^2 M^{-2}\\
        &\mathcal{N}(\Psi) \lesssim \frac{M}{\sigma^2}.
        \end{split}
    \end{equation}
    
    \myindent Finally, we have proved in Lemma \ref{Psi1DnearEalpha} that the $1D$ remaining phase function
    \begin{equation}
            \Psi^{1D}(u) = \Psi(\sigma,A,B,u,s_0(\sigma,t_+(\sigma,u)),t_+(\sigma,u),w_0(\sigma,t_+(\sigma,u)),r_0(\sigma,t_+(\sigma,u)))
    \end{equation}

    satisfies the property $(VdC)_3$ (see Definition \ref{VdCN}) with constants $cM,c'M$ for some universal $c,c' > 0$. Hence, Theorem \ref{mixedVdCABZ} applies, with $\zeta$ localized as claimed, and yields that 
    \begin{equation}
        |\mathcal{I}_{\mathcal{E}_{\sigma,0}^+}| \lesssim \lambda^{-\frac{1}{3}} M^{-\frac{1}{3}} \sigma^{-2} M^2  \sigma^{-4}M^{2},
    \end{equation}
    from which we can conclude the proof using moreover, in the case $i = 0$, that $|\sigma|\gtrsim M^{-\frac{1}{2}}$.
\end{proof}

\subsubsection{Quantitative estimate near the branches $\mathcal{E}_{\sigma,\alpha}$, away from $P_0$ and $P_{\pi}$, when $\sigma$ is close to the equator}\label{subsubsec24HormQuant}

\myindent When $\sigma \in \mathcal{I}_0$ is close to the equator, we cannot apply the strategy of the previous paragraph, since the estimates lose negative powers of $\sigma$, or, more geometrically, since it becomes not possible to parameterize the branches $\mathcal{E}_{\sigma,\alpha}^{\pm}$ by $u$. However, in that case, there actually holds that the full $(u,s,t,w,r)$ Hessian of $\Psi$ is nondegenerate near $\mathcal{E}_{\sigma,\alpha}^{\pm}$, thanks to Lemma \ref{FullHessiannearequator}. Thus, we can actually directly use Theorem \ref{ABZthm} near the branches $\mathcal{E}_{\sigma,\alpha}^{\pm}$.

\begin{lemma}\label{estimateIEsigmaalphaclosetoequator}
    Assume that $|\sigma| \leq C M^{-\frac{1}{2}}$, where $C$ is given by Lemma \ref{Psi1DnearEalpha}. Let $\alpha = 0$ or $\alpha = \pi$, and let $\chi_{\mathcal{E}_{\sigma,\alpha}^{\pm}}$ be a smooth localizer adapted to $\mathcal{D}^{\pm}_{\alpha}$ defined by \eqref{defDnearequator}, so that in particular
    \begin{equation}
        \left\| \left(\frac{\partial}{\partial u}\right)^{\alpha} \left(\frac{\partial}{\partial (s,t,w,r)}\right)^{\beta}\chi_{\mathcal{E}_{\sigma,\alpha}^{\pm}} \right\|_{L^{\infty}} \lesssim M^{|\beta| + \frac{3}{2} |\alpha|}.
    \end{equation}

    \myindent Then there holds
    \begin{multline}
        \mathcal{I}_{\mathcal{E}_{\sigma,\alpha}^{\pm}} := \lambda^2 \int \int \int e^{i\lambda \Psi(\sigma,A,B,u,s,t,w,r)} \chi_{\mathcal{E}_{\sigma,\alpha}^{\pm}}(u,s,t,w,r) b(\lambda,A,B,\sigma,u,s,t,w,r) dudsdt dw dr \\
        = O(\lambda^{-\frac{1}{2}} M^{14}).
    \end{multline}
\end{lemma}

\begin{proof}
    Thanks to Lemma \ref{FullHessiannearequator}, we know that, near the support of $\chi_{\mathcal{E}_{\sigma,\alpha}^{\pm}}$ there holds
    \begin{equation}
        |\det(\nabla_{u,s,t,w,r}^2 \Psi(\sigma,A,B,u,s,t,w,r))| \gtrsim M^{-1}.
    \end{equation}
    
    \myindent Hence, we are in the familiar setting of a nondegenerate phase function. We can thus apply Theorem \ref{ABZthm}. While it is definitely possible to lose far less powers of $M$ in the estimates by being more careful, we will only do the following trick before applying Theorem \ref{ABZthm} : consider the following change of variables
    \begin{equation}
        (\tilde{u},s,t,w,r) := (M^{\frac{1}{2}} (u - u_{\sigma}(\alpha)), s,t,w,r).
    \end{equation}
    
    \myindent Then, 
    \begin{multline}\label{IEsigalphtildeu}
        \mathcal{I}_{\mathcal{E}_{\sigma,\alpha}^{\pm}} =\\ M^{-\frac{1}{2}} \lambda^2\int \int \int e^{i\lambda \tilde{\Psi}(\tilde{u},s,t,w,r)} \chi_{\mathcal{E}_{\sigma,\alpha}^{\pm}}\left(M^{-\frac{1}{2}} \tilde{u} + u_{\sigma}(\alpha), s,t,w,r\right) b\left(\lambda,A,B, \sigma, M^{-\frac{1}{2}} \tilde{u} + u_{\sigma}(\alpha), s,t,w,r\right) d\tilde{u}dsdtdwdr,
    \end{multline}
    where
    \begin{equation}
        \tilde{\Psi}(\tilde{u},s,t,w,r) = \Psi(\sigma,A,B,M^{-\frac{1}{2}} \tilde{u} + u_{\sigma}(\alpha),s,t,w,r).
    \end{equation}
    
    \myindent The interest is that, since
    \begin{equation}
        \frac{\partial}{\partial \tilde{u}} = M^{-\frac{1}{2}} \frac{\partial}{\partial u},
    \end{equation}
    we find that for all $k\geq 2$,
    \begin{equation}
        \mathcal{M}_{2,k} (\tilde{\Psi}) \lesssim 1.
    \end{equation}
    
    \myindent In particular, the interest is that we have exactly compensated the term of order $O(M)$ which occurs in $\nabla_{u,s,t,w,r}^2 \Psi$. However, there holds that
    \begin{equation}
        |\det(\nabla_{\tilde{u},s,t,w,r}^2 \tilde{\Psi}(\tilde{u},s,t,w,r))| = M^{-1} |\det(\nabla_{u,s,t,w,r}^2 \Psi(\sigma,A,B,u,s,t,w,r))| \gtrsim M^{-2}.
    \end{equation}
    
    \myindent Now, since, moreover, there holds that
    \begin{equation}
        \mathcal{M}_{0, k}\left( \chi_{\mathcal{E}_{\sigma,\alpha}^{\pm}}\left(M^{-\frac{1}{2}} \tilde{u} + u_{\sigma}(\alpha), s,t,w,r\right) b\left(\lambda,A,B, \sigma, M^{-\frac{1}{2}} \tilde{u} + u_{\sigma}(\alpha), s,t,w,r\right)\right) \lesssim M^k,
    \end{equation}
    we can ultimately apply Theorem \ref{ABZthm} to $\mathcal{I}_{\mathcal{E}_{\sigma,\alpha}}$ written in the form \eqref{IEsigalphtildeu} to find that 
    \begin{equation}
        |\mathcal{I}_{\mathcal{E}_{\sigma,\alpha}}| \lesssim \lambda^{-\frac{1}{2}}M^{-\frac{1}{2}} M^{-4} M^{12} M^{6},
    \end{equation}
    which yields the result.
\end{proof}

\subsubsection{Quantitative estimate near the branching points $P_0$ and $P_{\pi}$}\label{subsubsec25HormQuant}

\myindent Near $P_{\alpha}$, we cannot isolate the $u$ variable since the structure of $\mathcal{O}_{\sigma}$ has a singularity. We thus resort to the analysis of Section \ref{subsec4HormPhase}. Following the dichotomy between Proposition \ref{delicatecase} and Lemma \ref{easycase}, it is natural that the following holds.

\begin{lemma}
    Let $i \in \{0,...,I\}$, let $\sigma \in \mathcal{I}_i$ and $\alpha = 0$ or $\alpha = \pi$.
    
    \myindent i. In the setting of Proposition \ref{delicatecase}, i.e. when \ref{smallpartuPsiPalpha} holds, let $\chi_{P_{\alpha}}(u,s,t,w,r)$ be a smooth localizer such that
    \begin{equation}
        \chi_{P_{\alpha}}(u,s,t,w,r) = \begin{cases}
            1 \qquad (w,(u,s,t,r)) \in \frac{1}{2} I \times \frac{1}{2} U \\
            0 \qquad (w,(u,s,t,r)) \notin I \times U,
        \end{cases}
    \end{equation}
    where $I$ and $U$ are defined by \eqref{defIUdelicatecase}, and such that
    \begin{equation}\label{derivativeschiPalpha}
        \left\| \left(\frac{\partial}{\partial u}\right)^{k_1} \left(\frac{\partial}{\partial (s,t,w,r)}\right)^{k_2} \chi_{P_{\alpha}}\right\|_{L^{\infty}} \lesssim_{k_1,k_2} M^{2k_1 + k_2}.
    \end{equation}
    
    \myindent Then, there holds
    \begin{equation}\label{boundIPalphadelicatecase}
        \mathcal{I}_{P_{\alpha}} := \lambda^2 \int \int \int e^{i\lambda \Psi(x,A,B,u,s,t,w,r)} \chi_{P_{\alpha}}(u,s,t,w,r) b(\lambda,A,B,\sigma,u,s,t,w,r) dudsdt dw dr = O(\lambda^{-\frac{1}{4}} M^{5}).
    \end{equation}
    
    \myindent ii. Otherwise, in the setting of Lemma \ref{easycase}, i.e. when \eqref{easiercasefirsteq} holds, let $\chi_{P_{\alpha}}(u,s,t,w,r)$ be a smooth localizer such that
    \begin{equation}
        \chi_{P_{\alpha}}(u,s,t,w,r) = \begin{cases}
            1 \qquad (u,s,t,w,r) \in \frac{1}{2} \mathcal{D}\\
            0 \qquad (u,s,t,w,r) \notin \mathcal{D},
        \end{cases}
    \end{equation}
    where $\mathcal{D}$ is defined by \eqref{defDeasycase}, and with bounds on its derivatives similar to \eqref{derivativeschiPalpha}. Then, there holds
    \begin{equation}\label{firstboundnearPalpha}
        \mathcal{I}_{P_{\alpha}} = O(\lambda^{-1} M^{3}).
    \end{equation}
\end{lemma}

\myindent Before turning to the proof of this lemma, let us stress that it is stated on a \textit{non isotropic} neighborhood of $P_{\alpha}$. Indeed, we will see that Remark \ref{nonisotropicmixedABZVdC} naturally applies in this context.

\begin{proof}
    For the first case, Proposition \ref{delicatecase} exactly gives what we need in order to apply Remark \ref{nonisotropicmixedABZVdC}, by isolating the variable $w$. We use the notations introduced in Proposition \ref{delicatecase}. From the upper bounds \eqref{upperboundMNphidelicatecase}, and the upper bounds on the derivatives of $\chi_{P_{\alpha}}$ given by \eqref{derivativeschiPalpha} we see that $\chi_{P_{\alpha}}(w,\cdot)$ is a smooth localizer around $Z_{\sigma,A,B}(w)$ such that condition \eqref{firstconditionnonisotropic} holds. More importantly, equation \eqref{oscintconditiondelicatecase} corresponds exactly to the second condition \eqref{secondconditionnonisotropic} which is needed in order for Remark \ref{nonisotropicmixedABZVdC} to apply. Indeed, thanks to \eqref{4thderivativesPsi1Ddelicatecase}, the 1D remaining phase function $\Psi^{1D}(w)$ obviously satisfies Property $(VdC)_4$ with constants $1,cM$ (see Definition \ref{VdCN}) where $c > 0$ is the implicit constant in \eqref{4thderivativesPsi1Ddelicatecase}. Finally, with the Notation \ref{defM1+d} and Definition \ref{defHandN}, observe that, for all $\ell\geq 0$
    \begin{equation}
    \begin{split}
        &\mathcal{D}(\Psi) \gtrsim 1\\
        &\mathcal{N}(\Psi) \lesssim M \\
        &\mathcal{M}_{0,\ell}^{(u,s,t,r)}(\Psi) \lesssim M \\
        &\mathcal{M}_{0,\ell}^{(u,s,t,r)} (\partial_w \Psi) \lesssim 1,
    \end{split}
    \end{equation}
    since $\partial_w \Psi$ doesn't depend on $(A,B)$, and we have already observed that, thanks to \eqref{boundallderivativesb}, the derivatives of $b$ are bounded independently of $M$ and $\lambda$.
    
    \myindent Hence, we obtain the conclusion \eqref{quantitativeboundmixedABZVdC} of Theorem \ref{mixedVdCABZ} for $\mathcal{I}_{P_{\alpha}}$, which reads
    \begin{equation}
        |\mathcal{I}_{P_{\alpha}}| \lesssim \lambda^{-\frac{1}{4}} M^{-\frac{1}{4}} M^2 M^3,
    \end{equation}
    which yields equation \eqref{boundIPalphadelicatecase}.
    
    \quad

    \myindent For the second case, we may integrate by parts in $u$ on the support of $\chi_{P_{\alpha}}$ thanks to Lemma \ref{easycase}, and we find that
    \begin{equation}
        \lambda^2 \int \int \int e^{i\lambda \Psi}\chi_{P_{\alpha}} b dudsdt dw dr
        = \lambda^{2 - K} \int \int \int e^{i\lambda \Psi} \left(\left(\frac{\partial_u}{i\partial_u \Psi}\right)^t\right)^K (\chi_{P_{\alpha}} b) dudsdt dw dr.
    \end{equation}

    \myindent Now, there obviously holds, using \eqref{easiercasesecondeq}, \eqref{derivativeschiPalpha}, and the fact that for all $K\geq 0$
    \begin{equation}
        \|(\partial_u)^K \Psi\|_{\infty} \lesssim M,
    \end{equation}
    that
    \begin{equation}
        \left|\left(\left(\frac{\partial_u}{\partial_u \Psi}\right)^t\right)^K (\chi_{P_{\alpha}} b) \right| \lesssim_{c_{P_{\alpha}}}M^{3K}.
    \end{equation}

    \myindent Hence, ultimately, and since the support of $\chi_{P_{\alpha}}$ has a measure bounded by $M^{-6}$
    \begin{equation}
        \left|\lambda^2 \int \int \int e^{i\lambda \Psi} b dudsdt dw dr\right| 
        \lesssim \lambda^{2-K} M^{3K - 6},
    \end{equation}
    and choosing $K = 3$ yields \eqref{firstboundnearPalpha}.
\end{proof}

\subsubsection{Taking care of what remains : nonstationary phase analysis in $(s,t,w,r)$}\label{subsubsec26HormQuant}

\myindent The 4 previous paragraphs put together yield full control of the integral $\mathcal{I}$ defined by \eqref{defIlocal} near the set $\mathcal{O}_{\sigma}$ of zeros of $\nabla_{s,t,w,r}\Psi$. In particular, away from this set, one can integrate by parts in $(s,t,w,r)$ and find a contribution $O(\lambda^{-\infty})$ to $\mathcal{I}$. Precisely, there holds

\begin{lemma}
    Assume that the constants $c_1$ (defined by Lemma \ref{quantitativestimateICx}), $c_2$ (defined by Lemma \ref{quantitativeestimateEsigmaalphaawayfromeq}) and $c$ defined by Lemma \ref{FullHessiannearequator} are small enough, depending on the constants $C$ defined by Lemma \ref{Psi1DnearEalpha} and $c_{P_{\alpha}}$ defined by Proposition \ref{delicatecase}, and on the implicit constants in the definition of $\mathcal{D}$ \eqref{defDeasycase}. Let
    \begin{equation}
        \chi_{^c\mathcal{O}}(u,s,t,w,r) :=  1 - \sum \left(\chi_{P_{\alpha}}(u,s,t,w,r) + \chi_{\mathcal{C}_{\sigma}^{\pm}}(u,s,t,w,r) + \chi_{\mathcal{E}_{\sigma,\alpha}^{\pm}} (u,s,t,w,r)\right),
    \end{equation}
    where the sum is taken over all choice of $\alpha \in \{0,\pi\}$ and of sign $\pm$. Then, $\chi_{^c\mathcal{O}}$ is supported outside of $\mathcal{O}_{\sigma}$ and there holds
    \begin{equation}\label{boundawayO}
        \lambda^2 \int \int \int e^{i\lambda \Psi(\sigma,A,B,u,s,t,w,r)}\chi_{^c\mathcal{O}}(u,s,t,w,r) b(\lambda,A,B,\sigma,u,s,t,w,r) dudsdt dw dr = O(\lambda^{-1} M^{9}).
    \end{equation}
\end{lemma}

\begin{proof}
    First, we claim that, provided $c_1$ and $c_2$ are chosen small enough, then for all $(u,s,t,w,r)$ in the support of $\chi_{^c\mathcal{O}}$, there holds
    \begin{equation}\label{lowerboundnablastwrnotO}
        |\nabla_{s,t,w,r}\Psi|\gtrsim M^{-3}.
    \end{equation}

    \myindent Indeed, find $(u,s,t,w,r) \in supp(\chi_{^c\mathcal{O}})$. Then

    \myindent i. If $|(s,t)| \geq c' M^{-1}$, where $c' > 0$ is a small constant, then we may apply Lemma \ref{quantitativelowerboundnablastwr},
    \begin{equation}
        |\nabla_{s,t,w,r}\Psi(u,s,t,w,r)| \gtrsim |(s,t)| |(u,s,w,r) - (u,s,w,r)(t)|.
    \end{equation}

    \myindent Now, if $c_2$ and $c$ are small enough depending on $c'$, by definition of $\chi_{\mathcal{E}_{\alpha}^{\pm}}$, we are in the support of $1 - \chi_{\mathcal{E}_{\alpha}^{\pm}}$, hence there holds
    \begin{equation}
        |(u,s,w,r) - (u,s,w,r)(t)| \gtrsim \sigma^{2}M^{-1}
    \end{equation}
    when $|\sigma|\gtrsim M^{-\frac{1}{2}}$, or, otherwise, 
    \begin{equation}
        |(u,s,w,r) - (u,s,w,r)(t)| \gtrsim M^{-1}.
    \end{equation}

    \myindent Hence, in any case
    \begin{equation}
        |\nabla_{s,t,w,r}\Psi|\gtrsim M^{-3}.
    \end{equation}

    \quad

    \myindent ii. If $|(s,t)|\leq  c'M^{-1}$, and $u \notin \mathcal{U}_{\sigma}$ defined by Definition \ref{defUsigma}, then, observe that
    \begin{equation}\label{lowerboundunotinUsigma}
        |\nabla_{s,t}\Psi(u,s,t,w,r)| = \left| -h(u) + r\begin{pmatrix}
            1\\
            \partial_{\theta}\varphi
        \end{pmatrix}\right| \gtrsim d(u,\mathcal{U}_{\sigma}).
    \end{equation}
    
    \myindent Indeed, recall that, thanks to the eikonal equation \eqref{nablastPsi0}, there holds that
    \begin{equation}
        r \begin{pmatrix}
            1 \\
            \partial_{\theta}\varphi
        \end{pmatrix} \in \left\{ \begin{pmatrix}
            q_1(\sigma,\xi)\\
            q_2(\sigma,\xi)
        \end{pmatrix} \qquad \xi \in \R^2\right\} =: \Lambda.
    \end{equation}
    
    \myindent Then, since, by definition, $\mathcal{U}_{\sigma}$ is the intersection between the cone $\Lambda$ and the curve $\gamma$ (see Definition \ref{deflilgammas}), equation \eqref{lowerboundunotinUsigma} follows.
    
    \myindent Now, if $c_1$ is small enough depending on $c_{P_{\alpha}}$ or on $\mathcal{D}$ defined by \eqref{defDeasycase}, and if $c'$ is small enough depending on those constants, then we may ensure that, if $(u,s,t,w,r)$ is in the support of $\chi_{^c\mathcal{O}}$ and $u \notin \mathcal{U}_{\sigma}$ and $|(s,t)|\leq c'M^{-1}$, then necessarily
    \begin{equation}
        d(u,\mathcal{U}_{\sigma}) \gtrsim M^{-2},
    \end{equation}
    from which it follows that
    \begin{equation}
        |\nabla_{s,t,w,r}\Psi| \gtrsim M^{-2}
    \end{equation}
    
    \quad
    
    iii. If, $|(s,t)|\leq c'M^{-1}$ but $u \in \mathcal{U}_{\sigma}$, assume that $|w - \alpha| \leq c'' M^{-1}$, for $c''$ a small enough constant. Write
    \begin{equation}
        \nabla_{s,t} \Psi(u,s,t,w,r) = -h(u) + r\begin{pmatrix} 1 \\
        q_2(\sigma,w,1)\end{pmatrix} + r\begin{pmatrix}
            0\\
            \partial_{\theta}\varphi((t,\sigma),(0,\sigma),w) - q_2(\sigma,w,1)
        \end{pmatrix}.
    \end{equation}

    \myindent Now, from the definition of $\mathcal{C}_{\sigma}$ given by Lemma \ref{defcirclestatpoints}, there holds
    \begin{equation}
        \left|-h(u) + r\begin{pmatrix} 1 \\
        q_2(\sigma,w,1)\end{pmatrix} \right| \gtrsim |(u,r) - (u_{\sigma}(w),r_{\sigma}(w))|.
    \end{equation}
    
    \myindent If $c''$ is small enough, depending only on $c_{P_{\alpha}}$, since we are outside of the support of $\chi_{P_{\alpha}}$, we may ensure that 
    \begin{equation}\label{usefuleqnotOsig}
        |(u,r) - (u_{\sigma}(w),r_{\sigma}(w))| \gtrsim_{c_{P_{\alpha}}} M^{-2}.
    \end{equation}
    
    \myindent Now, thanks to \eqref{asdevphi}, to Lemma \ref{miracleequation}, and to Lemma \ref{thirdthetawderivaticevarphi}, there holds that
    \begin{equation}
        \left|\partial_{\theta}\varphi((t,\sigma),(0,\sigma),w) - q_2(\sigma,w,1)\right| \lesssim t^2 + |t||\sin(w)|.
    \end{equation}
    
    \myindent If $c'$ is small enough compared to $c_{P_{\alpha}}$, and thanks to \eqref{usefuleqnotOsig}, this is very small compared to $|(u,r) - (u,r)(w)|$. Ultimately, there holds
    \begin{equation}
        \left|\nabla_{s,t}\Psi\right| \gtrsim \left| -h(u) + \begin{pmatrix}
            1 \\
            q_2(\sigma,w,1)
        \end{pmatrix}\right| \gtrsim M^{-2}.
    \end{equation}
    
    iv. It remains to deal with the case $|(s,t)| \leq c'M^{-1}$, $u \in \mathcal{U}_{\sigma}$, and $|w| \geq c'' M^{-1}$. Now, thanks to the properties of $u_{\sigma}$ (see Lemma \ref{locinverseuw}, one can prove that
    \begin{equation}
        |u - u_{\sigma}(w)| \gtrsim \min_{\pm} \left(|\sin(w)| |w - w_{\pm}(u)|\right).
    \end{equation}
    
    \myindent In particular, since we are outside of the supports of $\chi_{\mathcal{C}_{\sigma}}$ and of $\chi_{P_{\alpha}}$, one may ensure that, either
    \begin{equation}
        |r - r_{\sigma}(w)| \gtrsim M^{-1},
    \end{equation}
    either
    \begin{equation}
        |u - u_{\sigma}(w)| \gtrsim M^{-1} |\sin(w)|.
    \end{equation}
    
    \myindent Now, choosing $c'$ small enough, we can, again, ensure that 
    \begin{equation}
        \left|\partial_{\theta}\varphi((t,\sigma),(0,\sigma),w) - q_2(\sigma,w,1)\right| \lesssim t^2 + |t||\sin(w)| \ll M^{-1} |\sin(w)|.
    \end{equation}
    
    \myindent Hence, we again find that
    \begin{equation}
        \left|\nabla_{s,t}\Psi\right| \gtrsim \left| -h(u) + \begin{pmatrix}
            1 \\
            q_2(\sigma,w,1)
        \end{pmatrix}\right| \gtrsim M^{-2}.
    \end{equation}
    
    \quad
    
    \myindent Now, the lemma simply follows by integration by parts thanks to \eqref{lowerboundnablastwrnotO}. Indeed, there holds, for any integer $K \geq 1$,
    \begin{equation}
        \lambda^2 \int \int \int e^{i\lambda \Psi} \chi_{^c \mathcal{O} b} du ds dt dw dr = \lambda^{2 - K} \int \int \int e^{i\lambda \Psi} \left(\left(\frac{\nabla_{s,t,w,r}\Psi \cdot \nabla_{s,t,w,r}}{|\nabla_{s,t,w,r}\Psi|^2} \right)^t \right)^K \left(\chi_{^c\mathcal{O}} b\right) du ds dt dw dr.
    \end{equation}
    
    \myindent By definition of $\chi_{^c\mathcal{O}}$, there holds
    \begin{equation}
        \left\| \nabla_{s,t,w,r}^K \chi_{^c\mathcal{O}} \right\|_{L^{\infty}} \lesssim M^{2K},
    \end{equation}
    and we have already seen many times that the derivatives of $b$ are uniformly bounded in $M$. Hence, thanks to \eqref{lowerboundnablastwrnotO}, there holds 
    \begin{equation}
        \left|\left(\left(\frac{\nabla_{s,t,w,r}\Psi \cdot \nabla_{s,t,w,r}}{|\nabla_{s,t,w,r}\Psi|^2} \right)^t \right)^K \left(\chi_{^c\mathcal{O}} b\right) \right| \lesssim M^{6K}.
    \end{equation}
    
    \myindent The lemma then follows by choosing $K = 3$. 
\end{proof}

\begin{remark}
    Observe that, for the contribution of $^c\mathcal{O}_{\sigma}$, we only integrate by parts in $(s,t,w,r)$. Hence, we don't need to have bounds on the $u$ derivatives of $\chi_{^c \mathcal{O}}$, or, even, that it is smooth in the $u$ variable. This is why we are able to define $\chi_{\mathcal{C}_{\sigma}}$ and $\chi_{\mathcal{E}_{\sigma,\alpha}}$ with an indicator function in $u$. On the contrary, for the estimate near $P_0$ and $P_{\pi}$, even if we isolate the variable $w$, we need to choose the localizer sufficiently smooth in the $w$ variable, and not directly an indicator function.
\end{remark}

\section{The case where we can use the bicharacteristic length parametrix}\label{secBicharact}

\myindent In this section, we derive the bounds on $\mathcal{I}_{\lambda,\delta,i}^{(j)}(\sigma,A,B)$, $i = 0,..,I, \ j = 1,...,J_i$, that is we prove estimates \eqref{IlowM} for $\bullet = (j)$ and \eqref{IBicharactlargeM}. We follow the same steps than for the estimate of $\mathcal{I}_{\lambda,\delta,i}^{(H)}(\sigma,A,B)$, that is we start with a detailed analysis of the phase $\Psi$ in Section \ref{subsec1Bicharact}, and give the resulting quantitative bounds on the integral in Section \ref{subsec2Bicharact}. There is not much difference with the case $\bullet = (H)$ regarding the oscillatory integral analysis, hence we will give most results without proofs, since the proofs are the same than what we have already done (they are even simpler in the case $\bullet = (j)$). The real difference with the case $\bullet = (H)$ is rather is the geometry of bicharacteristic curves at times $O(1)$. However, since the article is already long, we have chosen not to insist on this part of the analysis, and we will instead say a word about it in Appendix \ref{AppendixD}, hence this will be somehow hidden in the presentation.

\myindent Contrary to the case $\bullet = (H)$, there is no real difference between the cases $(A,B) = (0,0)$ or $(A,B) \neq (0,0)$. Hence, we adapt the definition of $M$ locally as
\begin{equation}\label{changedefM}
    M := |(A,B)| + 1.
\end{equation}

\subsection{Analysis of the phase}\label{subsec1Bicharact}

\myindent Following the strategy presented in Paragraph \ref{subsubsec33Micro}, let us, again, decompose the phase into
\begin{equation}
    \Psi(\sigma,A,B,u,s,t,r) = -\scal{h(u)}{(A,B)} + \Psi_0(\sigma,A,B,u,s,t,r),
\end{equation}
where
\begin{equation}
    \Psi_0(\sigma,A,B,u,s,t,r) = -\scal{(h(u)}{(s + s_i^{(j)},t + t_i^{(j)})} + r\left(|s + s_i^{(j)}| - \psi((t+ t_i^{(j)},\sigma),(0,\sigma)) \right)
\end{equation}
is independent of $A,B$. We recall that $\psi(x,y)$ is the bicharacteristic length function, defined by Definition \ref{defpsi} Hence, any phase analysis using the $(s,t,r)$ gradient of $\Psi$ is independent of $(A,B)$. Moreover, we give the analysis in the case where $s + s_i^{(j)} \geq 0$ on $\mathcal{R}_i^{(j)}$ without loss of generality.

\subsubsection{Reducing to a compact domain of integration}\label{subsubsec11Bicharact}

\myindent Very similarly, to Paragraph \ref{subsubsec12HormPhase}, and as we have argued in this paragraph and in Paragraph \ref{subsubsec33HormPhase}, the \textit{correspondence} equation \eqref{nablastPsi0eq0} cannot hold for those $r$ which are either close to zero, or large enough. Indeed, there holds, thanks to equation \eqref{nablastPsi0}
\begin{equation}
    \begin{split}
    \nabla_{s,t}\Psi_0 &= -h(u) + r\begin{pmatrix}
        1\\
        \partial_{\theta} \psi((t + t_i^{(j)},\sigma,),(0,\sigma))
    \end{pmatrix} \\
    &= -h(u) + r \begin{pmatrix}
        q_1(\sigma,\nabla_x \psi((t + t_i^{(j)},\sigma,),(0,\sigma)) \\
        q_2(\sigma,\nabla_x \psi((t + t_i^{(j)},\sigma,),(0,\sigma))
    \end{pmatrix}.
    \end{split}
\end{equation}

\myindent Hence, $\nabla_{s,t}\Psi_0$ can vanish if and only if 
\begin{equation}
    r = \frac{1}{p_1(\sigma,\nabla_x\psi((t + t_i^{(j)},\sigma,),(0,\sigma)))},
\end{equation}
 from which we can deduce the equivalent of Lemma \ref{nicebound!}.

\begin{lemma}\label{reducetocompactbicharactparam}
    There exists $\beta > 0$ such that for all $i= 0,...,I$, for all $\sigma \in \mathcal{I}_i$, for all $(u,s,t) \in \R/\ell \Z \times \tilde{\mathcal{R}}_i^{(j)}$, there holds
    \begin{equation}
        \forall r \notin [\beta, \beta^{-1}], \qquad |\nabla_{s,t}\Psi_0(\sigma,u,s,t,r)| \gtrsim 1 + |r|.
    \end{equation}
\end{lemma}

\myindent Thus, we may focus on a region where $r$ is bounded away from zero and from $+ \infty$.

\subsubsection{The curve of $(s,t,r)$ stationary points of $\Psi_0$}\label{subsubsec12Bicharact}

\myindent Observe that
\begin{equation}
    \nabla_{s,t,r} \Psi_0 = \begin{pmatrix}
        -h(u)_1 + r \\
        -h(u)_2 + r \partial_{\theta}\psi((t+ t_i^{(j)},\sigma,),(0,\sigma)) \\
        s - \psi((t+ t_i^{(j)},\sigma,),(0,\sigma)).
    \end{pmatrix}
\end{equation}

\myindent In a sharp difference with the previous section, it is thus straightforward to compute the $(s,t,r)$ stationary points of $\Psi_0$. Indeed, there holds

\begin{lemma}
    For $t \in \mathcal{K}_i^{(j)}$, let $(u,s,r)(\sigma,t)$ be smoothly defined by
    \begin{equation}\label{srhousigmatbichlentgh}
    \begin{split}
        s(\sigma,t) &= \psi((t+ t_i^{(j)},\sigma,),(0,\sigma)) \\
        r(\sigma,t) &= \frac{1}{p_1(\sigma,\nabla_x \psi((t+ t_i^{(j)},\sigma,),(0,\sigma))} \\
        h(u(\sigma,t)) &= \frac{1}{p_1(\sigma,\nabla_x \psi((t+ t_i^{(j)},\sigma,),(0,\sigma))}\begin{pmatrix}
            1 \\
            \partial_{\theta}\psi((t+ t_i^{(j)},\sigma,),(0,\sigma))
        \end{pmatrix}.
    \end{split}
\end{equation}

\myindent Then $\mathcal{O}_{\sigma}$, the zero set of $\nabla_{s,t,r}\Psi_0$, is exactly the curve
\begin{equation}
    \mathcal{E} := \{(u(\sigma,t),s(\sigma,t),t,r(\sigma, t)) \qquad t \in \mathcal{K}_i^{(j)}\}.
\end{equation}
\end{lemma}

\myindent As we have already argued in Paragraph \ref{subsubsec33HormPhase}, it can be expected that $\mathcal{O}_{\sigma}$ is simply a curve parameterized by $t$ from the fact that the \textit{geometric} equation $\partial_{r}\Psi_0 = 0$ can occur if and only if there is a bicharacteristic of $q_1$ of length $s + s_i^{(j)}$ joining $(t+ t_i^{(j)},\sigma)$ to $(0,\sigma)$. Now, by definition of the bicharacteristic length $\psi$, for all $t \in \mathcal{K}_i^{(j)}$, we know that there is exactly one such bicharacteristic, of length $s(t,\sigma)$. Finally, we have already observed that the \textit{correspondence} equation $\nabla_{s,t}\Psi_0 = 0$ simply fixes the value of $u$ and $r$.

\subsubsection{Parameterization of the curve $\mathcal{E}$ by $u$}\label{subsubsec13Bicharact}

\myindent Since $\mathcal{E}$ is a curve of zeros of $\nabla_{s,t,r}\Psi$, similarly to what we have done above, in view of Theorem \ref{mixedVdCABZ}, we need to parameterize it by $u$, which we can do as long as $t \mapsto u(\sigma,t)$ is invertible, i.e. as long as the curve $\mathcal{E}$ is not vertical when projected in the $(u,t)$ plan. Now, as we have argued in Paragraph \ref{subsubsec33HormPhase}, the curve $\mathcal{E}$ is vertical if and only if $\sigma = 0$. Since this curve is moreover even in $\sigma$, it is quite natural that one can prove the following.

\begin{lemma}\label{derivativeofubicharact}
    The functions $u(\sigma,t), s(\sigma,t)$ have the following behaviour. 

    \myindent i) If $i \neq 0$,
    \begin{equation}
        \partial_t u(\sigma,t) = F(\sigma,t),
    \end{equation}
    where $F$ is a smooth non vanishing function.

    \myindent ii) If $i = 0$ ,
    \begin{equation}
            \partial_t u(\sigma,t) = \sigma^2 F(\sigma,t),
    \end{equation}
    where $F$ is a smooth non vanishing function.
\end{lemma}

\begin{proof}
    We only give a sketch of the proof. 
    
    \myindent Looking at the expression of $u(\sigma,t)$ given by \eqref{srhousigmatbichlentgh}, one can see that
    \begin{equation}
        \partial_t u(\sigma,t) \approx \partial_{\theta \theta}\psi((t + t_i^{(j)},\sigma,),(0,\sigma)).
    \end{equation}
    
    \myindent Hence, all amount to studying the bicharacteristic length function $\psi$ defined by Definition \ref{defpsi}. For $\sigma$ not close to $0$, one can prove that 
    \begin{equation}
        \partial_{\theta \theta} \psi((t + t_i^{(j)},\sigma,),(0,\sigma)) \neq 0,
    \end{equation}
    which yields the first part of the lemma. We leave this fact to a separate lemma in Appendix \ref{AppendixD}, for which we will give a geometrical proof (see Lemma \ref{firstlemmabicharact}).
    
    \myindent Now, for the second part, one can rigorously prove that, as $\sigma \to 0$,
    \begin{equation}
        \psi((t + t_i^{(j)},\sigma,),(0,\sigma)) = t + t_i^{(j)}+ C\sigma^2 f(\sigma,t+ t_i^{(j)})^2,
    \end{equation}
    where $C$ is a constant depending only on $\mathcal{S}$, and $f(\sigma,t)$ is a smooth function such that
    \begin{equation}
        f(0,t) = \tan\left(\frac{t}{2} \right).
    \end{equation}
    
    \myindent We leave this fact to a separate lemma in Appendix \ref{AppendixD}, with a detailed sketch of proof (see Lemma \ref{secondlemmabicharact}). 
\end{proof}

\myindent In particular, as long as $\sigma \neq 0$, we may parameterize $\mathcal{E}$ by $u$, and there holds

\begin{corollary}\label{uparameterisationofE}
    Assume that $\sigma \neq 0$. Let $C > 0$. Up to refining the partition $\mathfrak{Q}$, the function 
    \begin{equation}
        t \in \mathcal{K}_i^{(j)} \mapsto u(\sigma,t)
    \end{equation}
    is a smooth and even bijection from $\mathcal{K}_i^{(j)}$ to some interval $[u_1(\sigma), u_2(\sigma)]$, and its inverse
    \begin{equation}
        u \mapsto 
            t(u)
    \end{equation}
    satisfies the following.

    \myindent i) If $i \neq 0$,
    \begin{equation}
        \left|\left(\partial_u \right)(t(u))\right|\lesssim 1.
    \end{equation}

    \myindent ii) If $i = 0$, provided we choose $\mathcal{K}_i^{(j)}$ small enough (hence up to refining $\mathfrak{Q}$), there holds
    \begin{equation}\label{sndderivativegreaterthanfrstquared}
        \left|\left(\partial_u \right)^2(t(u))\right|\gtrsim \left|\left(\partial_u \right)(t(u))\right|^2 \geq C.
    \end{equation}
\end{corollary}

\begin{proof}
    We give a word of the proof, since it is not straightforward.
    
    \myindent On the one hand, the case $i \neq 0$ is an immediate corollary of Lemma \ref{derivativeofubicharact}.
    
    \myindent On the other hand, in the case $i = 0$, Lemma \ref{derivativeofubicharact} yields that
    \begin{equation}
        |t'(u)| = \frac{1}{|\partial_t u(\sigma,t(u))| } \gtrsim \frac{1}{\sigma^2},
    \end{equation}
    which we can always choose larger than a fixed constant $C$ provided we choose $\mathcal{I}_0$ smaller. 
    
    \myindent The only subtlety is the claim on the second derivative, that is the left inequality in \eqref{sndderivativegreaterthanfrstquared}. Observe that 
    \begin{equation}
    \begin{split}
        \left(\partial_u \right)^2(t(u)) &= -\left(\frac{(\partial_t)^2 (u(t))}{((\partial_t)(u(t)))^3} \right) (t = t(u))\\
        \left(\partial_u \right)(t(u)) &= \left(\frac{1}{(\partial_t)(u(t))} \right) (t = t(u)).
    \end{split}
    \end{equation}

    \myindent Hence, the claim amounts to proving that for all $t$ and all $\sigma$ close to $0$, there holds
    \begin{equation}
        \left|\frac{(\partial_t)^2 u(t,\sigma)}{(\partial_t) u(t,\sigma)}\right| \gtrsim 1.
    \end{equation}

    \myindent Now, since both the first and the second derivatives of $t\mapsto u(t,\sigma)$ vanish when $\sigma = 0$ (since $t\mapsto u(t,0)$ is constant), and using the parity in $\sigma$, one can actually prove through Taylor expansion that all amounts to proving that
    \begin{equation}\label{partialtandttsigmagimauneq0}
        \begin{split}
            (\partial_t)^2 (\partial_{\sigma})^2 u(t,0) &\neq 0 \\
            (\partial_t) (\partial_{\sigma})^2 u(t,0) &\neq 0.
        \end{split}
    \end{equation}

    \myindent Now, as we have already mentioned, one can prove (see Lemma \ref{secondlemmabicharact}) that
    \begin{equation}
        \left(\partial_{\sigma} \right)^2 u(0,t) = C\tan\left(\frac{t + t_i^{(j)}}{2}\right)^2,
    \end{equation}
    where $C$ is a nonzero constant depending only on $\mathcal{S}$. Hence, \eqref{partialtandttsigmagimauneq0} follows from a simple computation.
\end{proof}

\subsubsection{The partial and full Hessians of $\Psi$ near the curve $\mathcal{E}$}\label{subsubsec14Bicharact}

\myindent Now, in order to apply Theorem \ref{mixedVdCABZ}, we need to study the $(s,t,r)$ Hessian of $\psi$ on the curve $\mathcal{E}$, as long as we are not too close to the equator, on which we have to resort to a different analysis. We give the following lemmas, which are respectively the equivalent of Lemmas \ref{exactstwrhohess} and \ref{FullHessiannearequator}.

\begin{lemma}\label{strhohessbichlentgh}
    There holds along the branch $\mathcal{E}$

    \myindent i) If $i \neq 0$,
    \begin{equation}\label{nondeghessbichparamawayfromeq}
        \det(\nabla_{s,t,r}^2 \Psi_0 (\sigma,s(\sigma,t),t,r(\sigma,t))) = F(\sigma,t),
    \end{equation}
    for some smooth nonvanishing function $F$.

    \myindent ii) If $i = 0$,
    \begin{equation}\label{nondeghessbichparamateq}
        \det(\nabla_{s,t,r}^2 \Psi_0 (\sigma,s(\sigma,t),t,r(\sigma,t))) = \sigma^2 F(\sigma,t),
    \end{equation}
    for some smooth nonvanishing function $F$.
\end{lemma}

\begin{proof}
    The lemma is straightforward once we observe that
    \begin{equation}
        \nabla_{s,t,r}^2 \Psi_0 (\sigma,s,t,r) = \begin{pmatrix}
            0 & 0 & 1 \\
            0 & r \partial_{\theta\theta} \psi((t+ t_i^{(j)},\sigma,),(0,\sigma)) & \partial_{\theta}\psi((t+ t_i^{(j)},\sigma,),(0,\sigma)) \\
            1 & \partial_{\theta}\psi((t+ t_i^{(j)},\sigma,),(0,\sigma)) & 0 \end{pmatrix}.
    \end{equation}
    
    \myindent Indeed, the lemma then follows from the properties of $(\sigma,t) \mapsto \partial_{\theta \theta} \psi((t+ t_i^{(j)},\sigma,),(0,\sigma))$ which we have already mentioned in the previous paragraph, and which are formalized in Appendix \ref{AppendixD}.
\end{proof}

\myindent When $\sigma$ is very small, similarly to what happens in Paragraph \ref{subsubsec24HormQuant}, we cannot apply Theorem \ref{mixedVdCABZ} and instead we just observe that the full $(u,s,t,r)$ Hessian of $\Psi$ is non degenerate near the branch $E$. Precisely, there holds the following lemma.

\begin{lemma}\label{fullhessneareqatorbichparam}
    Let
    \begin{equation}
        \mathcal{D} := \{(u,s,t,r) \ \text{such that} \ |u - u_0(0)| \lesssim 1\}.
    \end{equation}
    
    \myindent Then, provided
    \begin{equation}\label{smallsigmabichparam}
        |\sigma| \lesssim M^{-\frac{1}{2}},
    \end{equation}
    there holds on $\mathcal{D}$
    \begin{equation}
        \left|\det(\nabla_{u,s,t,r}^2\Psi (\sigma, A, B, u,s,t,r))\right| \gtrsim 1.
    \end{equation}
\end{lemma}

\begin{proof}
    The proof is straightforward from the formula
    \begin{multline}
        \det(\nabla_{u,s,t,r}^2\Psi (\sigma, A, B, u,s,t,r)) = -\scal{h''(u)}{(A,B) + (s + s_{0}^{(j)},t+ t_0^{(j)})} \det(\nabla_{s,t,r}^2 \Psi_0(\sigma,s,t,r)) \\+ \begin{vmatrix}
            -h'(u)_1 & - h'(u)_2 \\
            1 & \partial_{\theta} \psi((t+ t_0^{(j)},\sigma,),(0,\sigma))
            \end{vmatrix}^2.
    \end{multline}
    
    \myindent Indeed, thanks to \eqref{srhousigmatbichlentgh} and to Lemma \ref{deth'hnotzero}, the second term is bounded away from zero as long as $|\sigma| \ll 1$ and $u$ is close enough to $u_0(0)$, while the first is of order $O(M\sigma^2)$ thanks to (an immediate extension of) Lemma \ref{strhohessbichlentgh}. Hence, provided $|\sigma| \ll M^{-\frac{1}{2}}$, the conclusion of the Lemma holds.
\end{proof}

\subsubsection{The remaining 1D phase function}\label{subsubsec15Bicharact}

\myindent In order to apply Theorem \ref{mixedVdCABZ} to the curve $\mathcal{E}$, when $|\sigma| \gtrsim M^{-\frac{2}{2}}$, we finally prove that the remaining 1D phase function has a nonvanishing $p$th derivative. Now, a subtle but important difference with the case of Paragraph \ref{subsubsec34HormPhase} is that we will be able to prove that only for $M \gg 1$. Ultimately, there are thus a finite number of value of $(A,B)$ for which we won't be able to prove better estimates, at least in the general case. It is still possible to have improved estimates, as we will explain in Section \ref{subsec1Further}.

\begin{lemma}\label{remainingphasefunctbichparam}
    Let $\sigma \neq 0$, and let the remaining 1D phase function be defined by
    \begin{equation}
    \begin{split}
        \Psi^{1D} : u \mapsto &\Psi(\sigma,A,B,u,s(\sigma,t(u)),t(u),r(\sigma,t(u)) \\
        &= -\scal{h(u)}{(s,t)(u) + (s_i^{(j)}, t_i^{j})+ (A,B)},
    \end{split}
    \end{equation}
    where $t(u)$ is the $u$ parameterization of $\mathcal{E}$ given by Corollary \ref{uparameterisationofE}. Up to refining $\mathfrak{Q}$, there exists a constant $M_0 > 0$ such that the following holds
    
    \myindent i) If $i \neq 0$, for all $(A,B)$ such that
    \begin{equation}
        M = |(A,B)|  \geq M_0,
    \end{equation}
    $\Psi^{1D}$ satisfies the Property $(VdC)_2$ (see Definition \ref{VdCN}) with constants $c_1 M, c_2 M$ for some universal constants $c_1, c_2 > 0$.
    
    \myindent ii) If $i = 0$, then $\Psi^{1D}$ satisfies the Property $(VdC)_3$ (see Definition \ref{VdCN}) with constants $c_1M, c_2M$ for some universal constants $c_1,c_2$.
\end{lemma}

\begin{proof}
    We don't detail the proof since it is mostly similar to the proof of Lemma \ref{Psi1DnearEalpha}. The only difference is that, in the case i), then one needs only observe that, for $M$ large enough, $M \gg |t'(u)|$ since this last quantity is bounded, thanks to Corollary \ref{uparameterisationofE}. 

    \myindent However, for $M$ small, we cannot apply the strategy of the proof of Lemma \ref{Psi1DnearEalpha}, since it relies crucially on the fact that $|t''(u)| \gg |t'(u)|$. Now, for a given $i \neq 0$, and for small $M$, all quantities are $O(1)$, hence it is \textit{not} possible to find a simple analytical argument to find a nonzero derivative.
    
    \myindent In the case ii) however, up to refining the partition $\mathfrak{Q}$, we may ensure that, choosing the constant $C$ large enough, if $|t'(u)| \approx M$, then, first, $M$ needs to be very large, and then, $|t''(u)| \gtrsim |t'(u)^2| \gg M$ similarly to the proof of Lemma \ref{Psi1DnearEalpha}. This follows from \eqref{sndderivativegreaterthanfrstquared}.
\end{proof}

\subsection{Quantitative estimate of the oscillatory integral}\label{subsec2Bicharact}

\myindent Now, we can apply the phase analysis of the previous section to obtain a quantitative bound on $\mathcal{I}_{\lambda,\delta,i}^{(j)}(\sigma,A,B)$. The method is exactly the same than in Section \ref{subsec2HormQuant}, although it is even simpler since there is only a branch $\mathcal{E}$ of zeros of $\nabla_{s,t,r}\Psi_0$. Hence, we can directly give the following result.

\begin{lemma}
    Let $i \in \{0,...,I\}$, $\sigma \in \mathcal{I}_i$ and $j \in \{1,...,J_i\}$. There holds
    
    \myindent i. If $i\neq 0$ and $M\leq M_0$, where $M_0$ is defined by Lemma \ref{remainingphasefunctbichparam}, there holds
    \begin{equation}
        \mathcal{I}_{\lambda,\delta,i}^{(j)} (\sigma,A,B) = O_{M_0}(1).
    \end{equation}
    
    \myindent i. If $i \neq _0$ but $M\geq M_0$, there holds 
    \begin{equation}
        \mathcal{I}_{\lambda,\delta,i}^{(j)}(\sigma,A,B) = O_{M_0}\left(\lambda^{-\frac{1}{2}} M^{-\frac{1}{2}} + \lambda^{-\frac{3}{2}}\right).
    \end{equation}
    
    \myindent iii. If $i = 0$, and condition \eqref{smallsigmabichparam} doesn't hold, i.e. $|\sigma| \gtrsim^{-\frac{1}{2}}$, there holds
    \begin{equation}
        \mathcal{I}_{\lambda,\delta,i}^{(j)}(\sigma,A,B) = O\left(\lambda^{-\frac{1}{3}} M^3 + \lambda^{-\frac{3}{2}} M^6\right).
    \end{equation}
    
    \myindent iv. Finally, if $i = 0$, and condition \eqref{smallsigmabichparam} holds, i.e. $|\sigma|\lesssim M^{-\frac{1}{2}}$, there holds
    \begin{equation}
        \mathcal{I}_{\lambda,\delta,i}^{(j)} (\sigma,A,B) = O\left(\lambda^{-\frac{1}{2}} M^6\right).
    \end{equation}
\end{lemma}

\begin{proof} Before giving the proof, an important difference with the case of Section \ref{secHormQuant} is that $a$ is no longer a symbol of order zero, but a symbol of order $\frac{1}{2}$. Hence, there holds
\begin{equation}\label{nolosslambdabichpar1}
    |\nabla_{s,t}^K a(s,(t + t_i^{(j)}),\sigma,),(0,\sigma), \lambda r) |\lesssim_K \lambda^{\frac{1}{2}} r^{\frac{1}{2}},
\end{equation}
and, once we restrict to a compact domain of integration where $r \in [\beta^{-1}, \beta]$,
\begin{equation}
    \|\nabla_{s,t,r}^K a(s,(t+ t_i^{(j)},\sigma,),(0,\sigma),\lambda r)\|_{\infty} \lesssim_K \lambda^{\frac{1}{2}},
\end{equation}
where we a similar trick than for estimate \eqref{nolosslambdadera} : when one differentiates $a$ in $r$, one loses a factor $\lambda$, but gains a factor $\lambda^{-1} r^{-1}$ through the symbol estimate.

\quad

\myindent Turning to the estimates of the integral, the first step is to apply Lemma \ref{reducetocompactbicharactparam} in order to reduce to a compact domain of integration, exactly as in Paragraph \ref{subsubsec11HormQuant}, for any integer $K\geq 1$,
\begin{equation}
\begin{split}\label{reducetocompactnonzeroabbichparam}
   &\mathcal{I}_{\lambda,\delta,i}^{(j)}(\sigma,A,B) =\\ &\lambda \int \int \int e^{i\lambda \Psi(\sigma,A,B,u,s,t,r)} \chi(s,t) \chi_1(r) \hat{\rho}(\delta(s + s_i^{(j)}+ A, t + t_i^{(j)}+ B)) \\
   &a(s, (t+ t_i^{(j)},\sigma,),(0,\sigma), \lambda r) |\det(h'(u),h(u))| dudsdt dr
   \\ &+ O_K\left(\lambda^{\frac{3}{2}-K}\right),
\end{split}
\end{equation}
where $\chi_1$ localizes $r$ in $[\beta^{-1},\beta]$. Indeed, on the support of $1 - \chi_1$, Lemma \ref{reducetocompactbicharactparam} ensures that we can integrate by parts in $(s,t)$, and the count of powers of $\lambda$ follows from \eqref{nolosslambdabichpar1}. Hence, we wish to estimate the integral term in the RHS of \eqref{reducetocompactnonzeroabbichparam}. 

\myindent Let 
\begin{multline}
    b(\lambda,\delta,A,B,\sigma,u,s,t,r) := \\ \chi(s,t) \chi_1(r) \hat{\rho}(\delta(s+ s_i^{(j)} + A, t + t_i^{(j)}+ B)) a(s, (t+ t_i^{(j)},\sigma,),(0,\sigma), \lambda r) |\det(h'(u),h(u))|.
\end{multline}

\myindent Thanks to \eqref{nolosslambdabichpar1}, we know that $b$ is a smooth symbol such that
\begin{equation}
    \left\| \nabla_{u,s,t,r}^K b \right\|_{L^{\infty}} \lesssim_K \lambda^{\frac{1}{2}},
\end{equation}
where, most importantly, the upper bound is independent of $(\delta,A,B)$. 

\quad

\myindent We turn to the estimate of 
\begin{equation}
     \mathcal{I}(\lambda) := \lambda \int \int \int e^{i\lambda \Psi(\sigma,A,B,u,s,t,r)} b(\lambda,\delta,A,B,\sigma,u,s,t,r) dudsdtdr.
\end{equation}

\myindent The estimate now has to be divided into different cases, following the analysis of the phase that we gave in Section \ref{subsec1Bicharact}.

\myindent i) If $i \neq 0$ and $M \leq M_0$ where $M_0$ is defined by Lemma \ref{remainingphasefunctbichparam} (i.e. for a finite number of couples $(A,B)$), one can actually ignore the oscillations in $u$, and apply the usual stationary phase lemma. Indeed, for all $u$
\begin{equation}
    (s,t,r) \mapsto \Psi_0(\sigma,u,s,t,r)
\end{equation}
is a phase function with one stationary point at which the Hessian is uniformly nondegenerate, thanks to \eqref{nondeghessbichparamawayfromeq}. Theorem \ref{hormstatphase} thus yields that
\begin{equation}
    \begin{split}
        &\lambda \int \int \int e^{i\lambda \Psi(\sigma,A,B,u,s,t,r)}b dudsdt dr \\
        &= \int du  e^{-i\lambda \scal{h(u)}{(A,B)} }\left(\lambda\int \int e^{i\lambda \Psi_0(\sigma,u,s,t,r)} b dsdt dr\right) \\
        &= O_{M_0}(1)
    \end{split}
\end{equation}
uniformly. Indeed, the inner integral itself is $O(1)$, since the integration is over a 3D domain, yielding a factor $\lambda^{-\frac{3}{2}}$, exactly compensated by the factor $\lambda$ in front and the $\lambda^{\frac{1}{2}}$ which is hidden in $a$ being a symbol of order $\frac{1}{2}$.

\quad

\myindent ii) If $i\neq 0$ but $M \geq M_0$, we can again isolate the variable $u$, but this time rather apply Theorem \ref{mixedVdCABZ}, since its conditions hold thanks to Lemma \ref{remainingphasefunctbichparam}. Observing that, thanks to \eqref{nondeghessbichparamawayfromeq}, and with the notations introduced in Notation \ref{defM1+d} and Definition \ref{defHandN}, for all $1 \leq k \leq l$, 
\begin{equation}
\begin{split}
    &\mathcal{D}(\Psi) \gtrsim 1 \\
	&\mathcal{N}(\Psi) \lesssim 1 \\
	&\mathcal{M}_{k,l}^{(s,t,r)}(\Psi) \lesssim 1 \\
	& \mathcal{M}_{k,l}^{(s,t,r)}(\partial_u \Psi) \lesssim 1.
\end{split}
\end{equation}

\myindent In particular,  there is \textit{no loss} of powers of $M$, neither when applying Theorem \ref{mixedVdCABZ}, or when estimating the contribution of the integral of the zone where $\nabla_{s,t,r} \Psi_0$ doesn't vanish and we can integrate by parts in $(s,t,r)$ (see Paragraph \ref{subsubsec25HormQuant}). Hence, since the remaining phase function satisfies Property $(VdC)_2$ (see Definition \ref{VdCN}) with constants $c_1M,c_2M$ thanks to Lemma \ref{remainingphasefunctbichparam}, there finally holds
\begin{equation}
   \mathcal{I}(\lambda) 
    = O\left(\lambda^{-\frac{1}{2}}  M^{-\frac{1}{2}}\right) + O_N\left(\lambda^{\frac{3}{2} - N}\right).
\end{equation}

\quad

\myindent iii) If $i = 0$ and condition \eqref{smallsigmabichparam} doesn't hold i.e.  $|\sigma| \gtrsim M^{-\frac{1}{2}}$, a similar argument applies, thanks to Lemma \ref{remainingphasefunctbichparam}, where this times Property $(VdC)_3$ holds with constants $c_1M,c_2M$. However, this time, thanks to \eqref{nondeghessbichparamateq}, there holds that
\begin{equation}
\begin{split}
    &\mathcal{D}(\Psi) \gtrsim \sigma^2 \\
	&\mathcal{N}(\Psi) \lesssim \sigma^{-2},
\end{split}
\end{equation}
the bounds on the derivatives of $\Psi$ being unchanged. Hence, we actually have to be careful in tracking the powers of $\sigma$ which are lost in the estimates, both from the part of the integral controlled by Theorem \ref{mixedVdCABZ} and from the part where we can integrate by parts in $(s,t,r)$. We don't detail the computations, but, overall, one can prove that 
\begin{equation}
    \mathcal{I}(\lambda) = O\left(\lambda^{-\frac{1}{3}} M^{-\frac{1}{3}} \sigma^{-6}\right) + O_N\left(\lambda^{\frac{3}{2} - N} \sigma^{-4N}\right),
\end{equation}
and we can conclude to the upper bound using moreover that $\sigma \gtrsim M^{-\frac{1}{2}}$.

\quad

iv) Finally, if $i = 0$ and condition \eqref{smallsigmabichparam} holds, i.e. $|\sigma| \lesssim M^{-\frac{1}{2}}$, one can directly use Lemma \ref{fullhessneareqatorbichparam} along with Theorem \ref{ABZthm} to bound the integral around the branch $\mathcal{E}$. The interesting fact is that this lemma yields a neighborhood of $\mathcal{E}$ which is \textit{independent of $M$} on which Theorem \ref{ABZthm} applies. In particular, there is no problem in bounding the contribution of the integral of the zone where we can integrate by parts in $(s,t,r)$. Ultimately, if moreover one does a similar rescaling of the $u$ variable than in the proof of Lemma \ref{estimateIEsigmaalphaclosetoequator}, there holds
\begin{equation}
    \mathcal{I}(\lambda)
    = O\left(\lambda^{-\frac{1}{2}} M^6\right) + O_N\left(\lambda^{\frac{3}{2} - N}\right).
\end{equation}
\end{proof}

\section{The case $\bullet = (\pi)$ : antipodal Hörmander's parametrix}\label{secAntipod}

\myindent In this section, we derive the bounds on $\mathcal{I}_{\lambda,\delta,0}^{(\pi)}$, that is we prove estimates \eqref{IlowM} in the case $\bullet =(\pi)$ and \eqref{IAntipodlargeM}. We will follow the same steps than for the previous estimates, and, in particular, we start with a detailed analysis of the phase $\Psi$ in Section \ref{subsec1Antipod}, and we then give the resulting bounds on the integral in Section \ref{subsec2Antipod}. Now, the analysis is extremely similar the previous cases, hence we only mention what needs to be changes in the proofs. Indeed, as we mentioned in Paragraph \ref{subsubsec33Micro}, this case can be seen as a transition regime between the case $\bullet = (j)$ and the case $\bullet = (H)$, hence the only difficulty is to properly quantify threshold conditions in terms of $M$. Here, we again use the convention \eqref{changedefM}, that is we define
\begin{equation}
    M := |(A,B)| + 1.
\end{equation}

\subsection{Analysis of the phase}\label{subsec1Antipod}

\myindent Following the strategy presented in Paragraph \ref{subsubsec33HormPhase}, let us decompose the phase into
\begin{equation}
    \Psi(\sigma,A,B,u,s,t,w,r) = -\scal{h(u)}{(A,B)} + \Psi_0(\sigma,u,s,t,\xi),
\end{equation}
where
\begin{equation}
    \Psi_0(\sigma,u,s,t,w,r) = -\scal{h(u)}{(s + \pi,t + \pi)} + s q_1(\sigma,\xi) +  \varphi((t,-\sigma,),(0,\sigma),\xi),
\end{equation}
where we recall that $\varphi(x,y,w)$ is an abuse of notation for $\varphi(x,y,g(w))$, see Notation \ref{defwrho}.

\myindent Observe, in particular, that we have already obtained a bound for $\sigma = 0$ since in that case the integral exactly equals $\mathcal{I}_{\lambda,\delta,0}^{(H)}(0,A,B)$, which we analyzed in Sections \ref{secHormPhase} and \ref{secHormQuant}. Hence, we can fix $\sigma > 0$. We will follow the same steps as in the previous sections.

\subsubsection{Reducing to a compact domain of integration}\label{subsubsec11Antipod}

\myindent The first step is nearly identical to Paragraph \ref{subsubsec12HormPhase}, and close to Paragraph \ref{subsubsec11Bicharact} :  we observe that
\begin{equation}
    \begin{split}
    \nabla_{s,t}\Psi_0 &= -h(u) +  \begin{pmatrix}
        q_1(\sigma,\xi) \\
        \partial_{\theta}\varphi((t,-\sigma,),(0,\sigma),\xi)
    \end{pmatrix} \\
    &= -h(u) +  \begin{pmatrix}
         q_1(\sigma,\nabla_x\varphi)\\
        q_2(\sigma,\nabla_x \varphi)
    \end{pmatrix} ,
    \end{split}
\end{equation}
where the difference with \eqref{nablastpsi0expr} is that $\nabla_x \varphi$ is taken at $((t,-\sigma,),(0,\sigma),\xi)$ rather than $((t,\sigma,),(0,\sigma),\xi)$. Hence, \eqref{devnablaxphi} has to be changed a little into
\begin{equation}
    \nabla_x\varphi((t,-\sigma,),(0,\sigma),\xi) = \xi + O((|t| + |\sigma|)|\xi|),
\end{equation}
which doesn't affect the reasoning since $|\sigma| \ll 1$. Overall, there holds an equivalent of Lemma \ref{nicebound!}.

\begin{lemma}
    Let $\beta > 0$ be large enough. Then, there holds for any $(u,s,t,w,r) \in \R/\ell \Z \times \tilde{\mathcal{R}}_0^{(\pi)} \times \R^2$,
    \begin{equation}
        \forall |\xi| \notin \left[2\beta^{-1},\frac{1}{2}\beta\right] \qquad |\nabla_{s,t}\Psi_0(\sigma,u,s,t,w, r)| \gtrsim 1 + |\xi|.
    \end{equation}
\end{lemma}

\myindent In particular, since we will be able to integrate by parts in $(s,t)$ outside of a neighborhood of $|\xi| \sim 1$, and find a contribution $O(\lambda^{-\infty})$ of this region, we may restrict the analysis to the region where
\begin{equation}
	|\xi| \in [\beta^{-1},\beta].
\end{equation}

\myindent In this region, we may use the \textit{same} polar change of coordinates that was introduced in Notation \ref{defwrho} and write
\begin{equation}
	\xi = r g_{\sigma}(w),
\end{equation}
so that, in particular, the phase $\Psi_0$ takes the nice form 
\begin{equation}
	\Psi_0 (\sigma,u,s,t,w,r) = -\scal{h'(u)}{(s + \pi,t + \pi)} + r(s + \varphi((t,-\sigma,),(0,\sigma),w)).
\end{equation}

\myindent We again define
\begin{equation}
    \mathcal{K}_{\sigma} :=\{(w,r), \ \text{such that} \ |rg_{\sigma}(w)| \in [\beta^{-1}, \beta]\}.
\end{equation}

\subsubsection{The set $\mathcal{O}_{\sigma}$ of $(s,t,w,r)$ stationary points of $\Psi_0$}\label{subsubsec12Antipod}

\myindent In this paragrah, we detail the geometry of the set $\mathcal{O}_{\sigma}$ of $(s,t,w,r)$ stationary points of $\Psi_0$, which was introduced in \ref{gendefOsigma}.

\myindent Following the presentation of Paragraph \ref{subsubsec33Micro}, we must first compute the solution of the \textit{geometric} equation
\begin{equation}\label{nablawrhopsieq0antipod}
    \nabla_{w,r}\Psi_0(\sigma,u,s,t,w,r) = 0,
\end{equation}

\myindent We recall that the meaning of this equality, which depends only on $(\sigma,s,t,w)$, is that there exists a bicharacteristic curve of $q_1$ joining $(t,-\sigma)$ to $(0,\sigma)$ in a time $s$, and its direction is $\nabla_x \varphi$, as in Figure \ref{FigureAntipod}.

\begin{figure}[h]
\includegraphics[scale = 0.5]{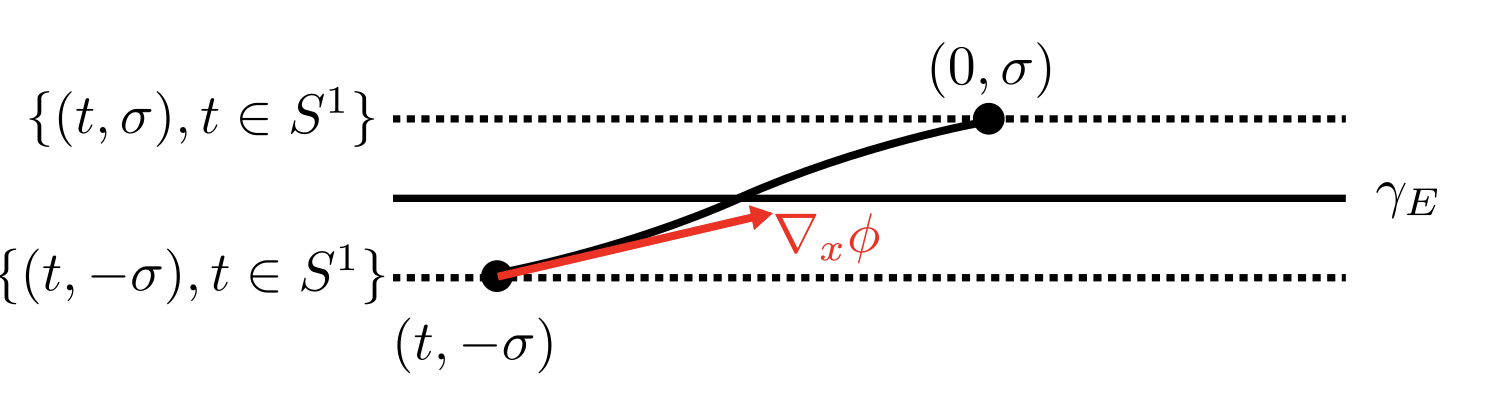}
\centering
\caption{The geometric equation for the Antipodal case}
\label{FigureAntipod}
\end{figure}

\myindent Now, for $\sigma > 0$, in order to have a geometric intuition of the set of those $(s,t,w)$ such that \eqref{nablawrhopsieq0antipod} holds, observe first that, for $t \ll -\sigma$, the situation is as in Figure \ref{tllsigma}, i.e., for any fixed $t$, there is one, and only one, bicharacteristic curve (of length $s \ll 1$) joining $(t,-\sigma)$ to $(0,\sigma)$, and it is \textit{nearly horizontal}. Now, this gives rise to \textit{two} solutions of \eqref{nablawrhopsieq0antipod}, say $(s_{\pm}(\sigma,t),t,w_{\pm}(\sigma,t))$, where $s_- = - s_+$. This corresponds to how we choose the orientation to travel on the curve.

\begin{figure}[h]
\includegraphics[scale = 0.5]{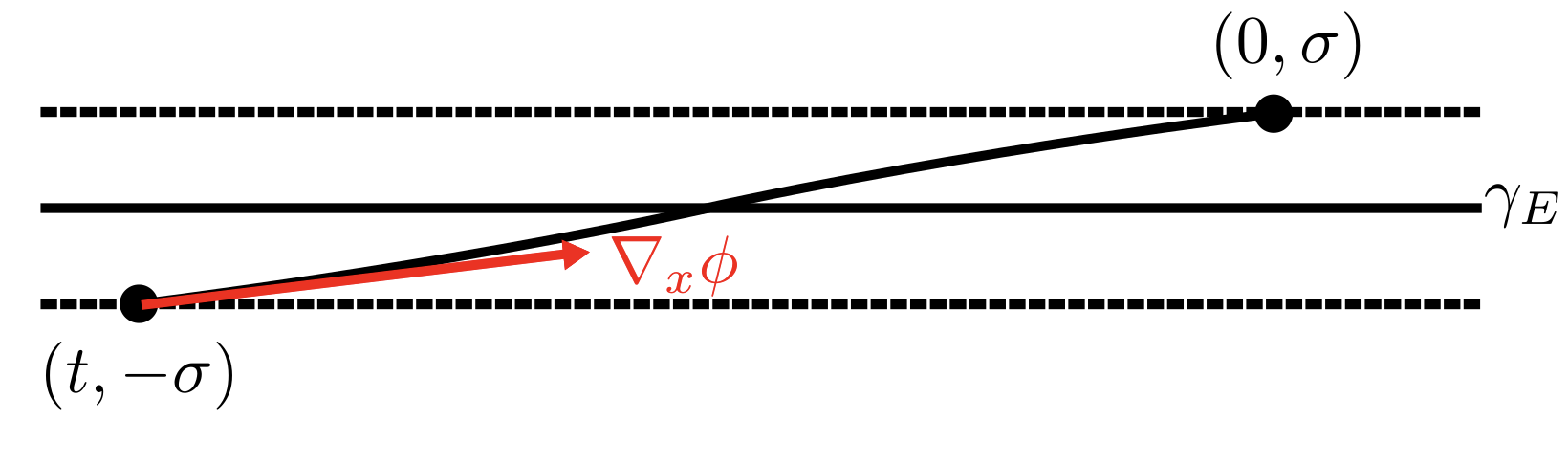}
\centering
\caption{Bicharacteristic curve joining $(t,-\sigma)$ to $(0,\sigma)$ when $t \ll -\sigma$}
\label{tllsigma}
\end{figure}

\myindent When $t$ becomes comparable to $\sigma$ and as it approaches $t = 0$ the curve becomes more and more vertical. Indeed, for $t \sim - \sigma$, the picture is as in Figure \ref{tsimsigma},

\begin{figure}[h]
\includegraphics[scale = 0.5]{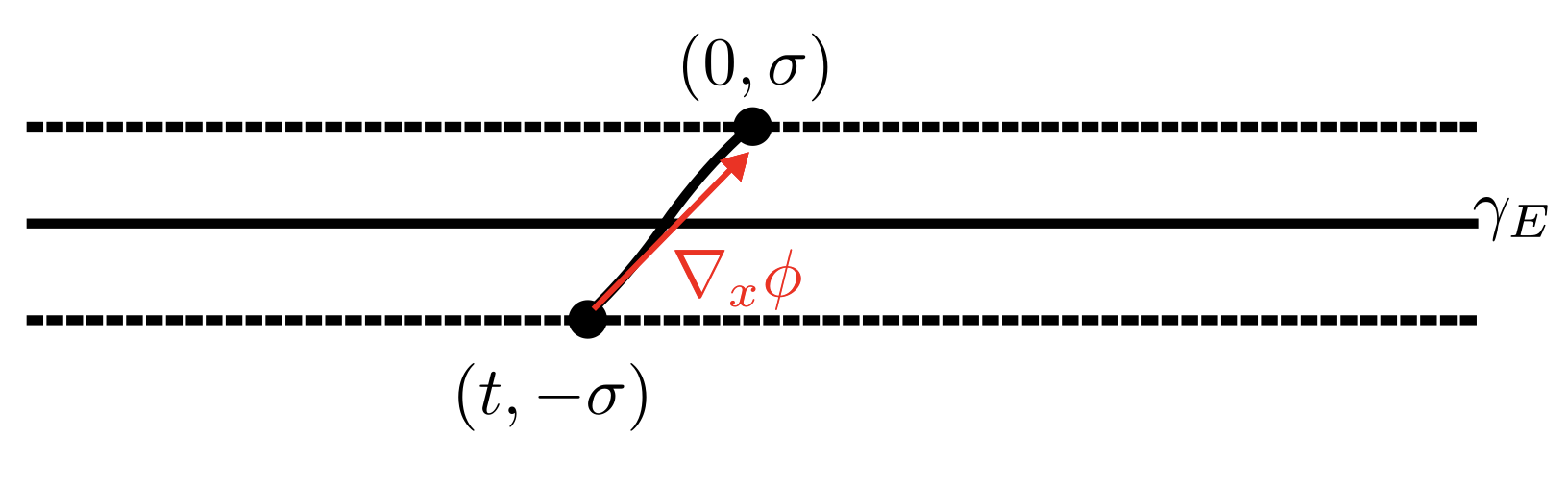}
\centering
\caption{Bicharacteristic curve joining $(t,-\sigma)$ to $(0,\sigma)$ when $t \sim -\sigma$}
\label{tsimsigma}
\end{figure}

while for $t= 0$, the situation is as in Figure \ref{teq0Figure}, since the meridian is the only bicharacteristic curve joining $(0,-\sigma)$ to $(0,\sigma)$. 

\begin{figure}[h]
\includegraphics[scale = 0.5]{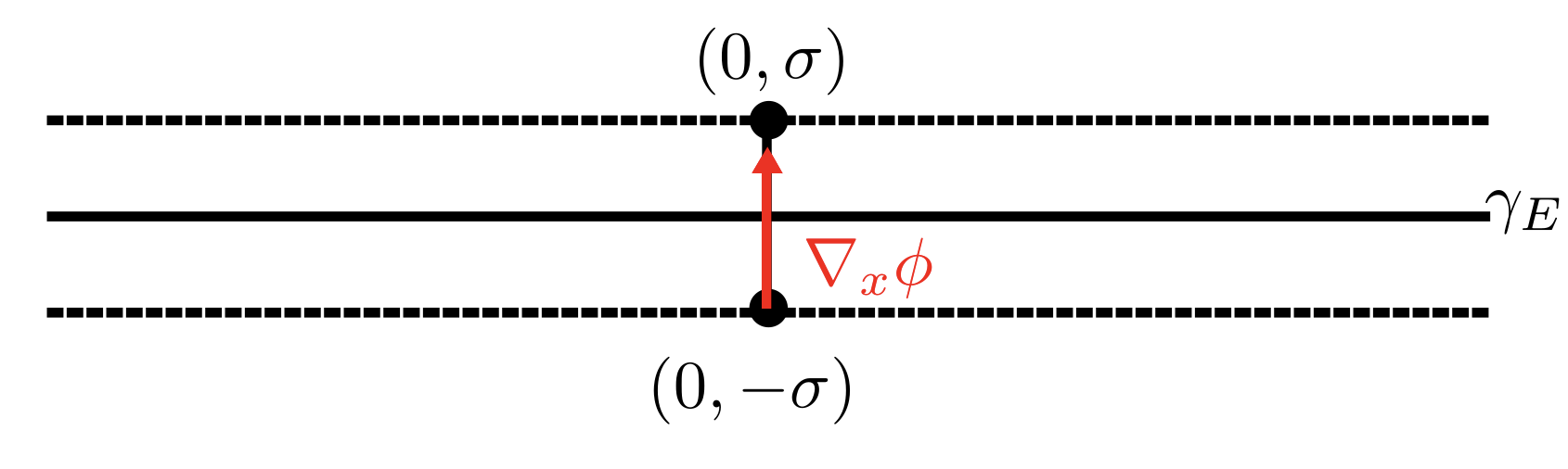}
\centering
\caption{Bicharacteristic curve joining $(0,-\sigma)$ to $(0,\sigma)$}
\label{teq0Figure}
\end{figure}

\myindent Finally, for nonnegative $t$, the same reasoning applies with an inversion of all signs. Ultimately, the zero set of $\nabla_{w,r}\Psi_0$ can be visualized, in the 3D space $(s,t,w)$, as a disjoint reunion of curves, see Figure \ref{txistructurenear0}

\myindent Still following the approach presented in Paragraph \ref{subsubsec33Micro}, we know moreover that we can view the equation
\begin{equation}
    \nabla_{s,t}\Psi_0(\sigma,u,s,t,w,r) = 0
\end{equation}
as a \textit{correspondence} equation, fixing the value of $u$ and $r$ once $(\sigma,s,t,w)$ are given.

\myindent Thanks to this correspondence, in order to visualize $\mathcal{O}_{\sigma}$, we may project it on the $(u,t)$ space and obtain, for $\sigma > 0$, Figure \ref{antipodut}.

\myindent In particular, one can prove that $\mathcal{O}_{\sigma}$ is always the disjoint reunion of two curves, which can be parameterized by $u$.

\begin{lemma}
    Let $\sigma > 0$. The set $\mathcal{O}_{\sigma}$ of zeros of $\nabla_{s,t,w,r}\Psi_0$ is the disjoint reunion of two smooth curves parameterized by $u$
    \begin{equation}
        \mathcal{E}^{\pm} := \left\{(u,s_{\pm}(\sigma,u), t(\sigma,u), w_{\pm}(\sigma,u), r_{\pm}(\sigma,u)) \qquad u \in [u_{\pi}(\sigma) + \sigma^2 u_1(\sigma), u_0(\sigma) - \sigma^2 u_1(\sigma)]\right\},
    \end{equation}
    where $u_1$ is a smooth positive function of $\sigma$, which doesn't vanish at $\sigma = 0$, and where $s_+ (\sigma,u) = -s_-(\sigma,u) \gtrsim \sigma > 0$. In particular, for all $u \in [u_{\pi}(\sigma) + \sigma^2 u_1(\sigma), u_0(\sigma) - \sigma^2 u_1(\sigma)]$, 
    \begin{equation}
        \Psi_u : (s,t,w,r) \in \{\pm s > 0\} \mapsto \Psi_0(u,s,t,w,r)
    \end{equation}
    has a unique smooth stationary point $(s_{\pm}(\sigma,u),t(\sigma,u), w_{\pm}(\sigma,u), r_{\pm}(\sigma,u))$.
\end{lemma}

\myindent Now, the analysis is very different depending on whether $\sigma$ is large or small, compared, of course, to (a power of) $M$. These two cases are respectively extensions of the analysis of Section \ref{secBicharact} and \ref{secHormPhase}. Indeed, we have already seen in Paragraph \ref{subsubsec33Micro} that this case is a transition case between the main models treated in the said sections. We will derive the precise threshold on $\sigma$ in the analysis.

\subsubsection{The case where $\sigma$ is small : back to the H structure}\label{subsubsec13Antipod}

\myindent When $\sigma$ is small, the curves $u \mapsto t_{\pm}(\sigma,u)$ look like Figure \ref{antipodut}, which is very close to Figure \ref{Hormanderutateq}.

\myindent Thus, the analysis is an extension of the analysis in the case $\bullet = (H)$ presented in Section \ref{secHormPhase}. In particular, similarly to what we presented in that section, there are three zones of interest, which are represented on Figure \ref{ThreezonesFig} with a color code.

\begin{figure}[h]
\includegraphics[scale = 0.5]{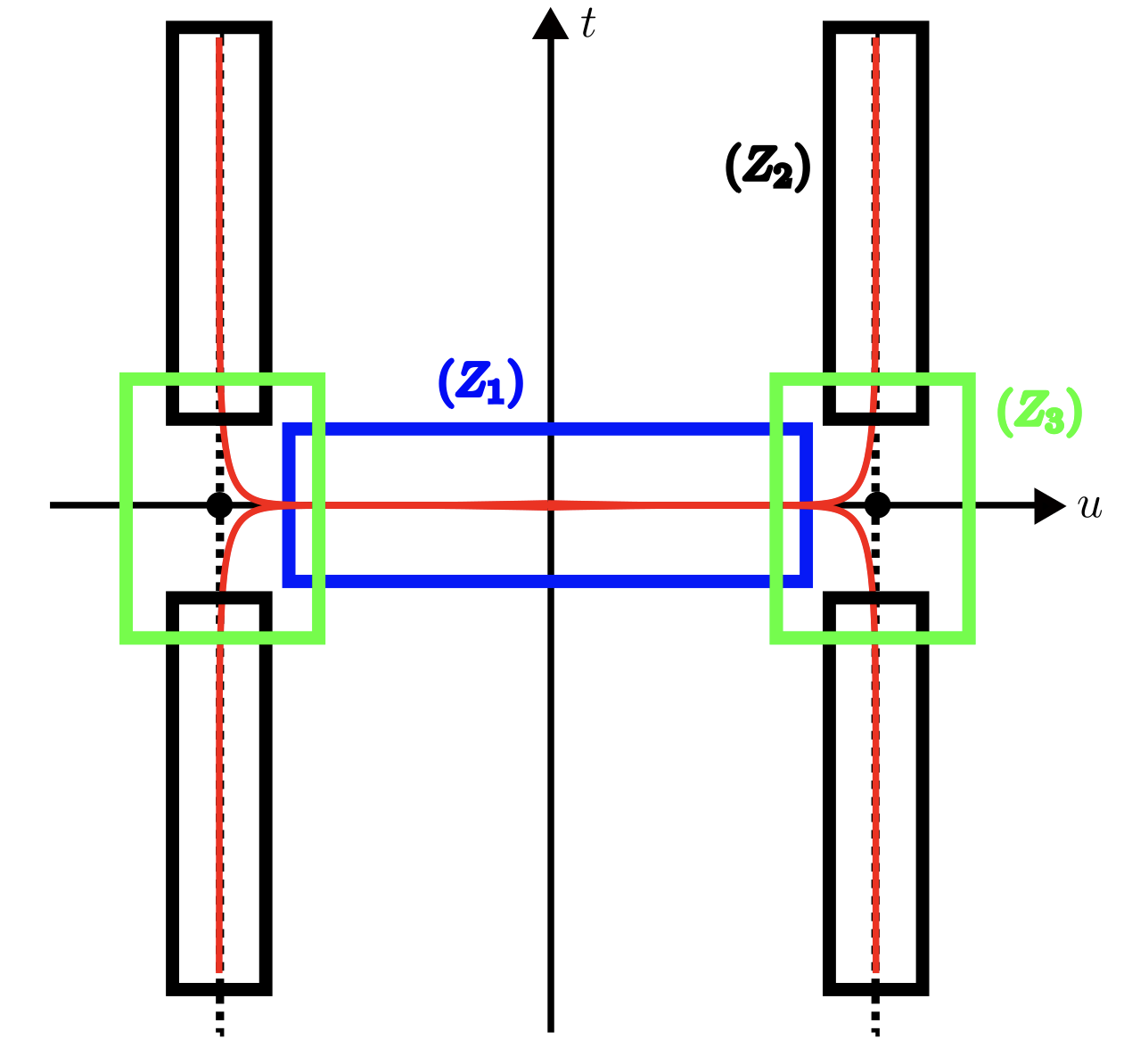}
\centering
\caption{The three zones of interest of $\mathcal{O}_{\sigma}$}
\label{ThreezonesFig}
\end{figure}

\myindent In the first zone, which is roughly represented by the blue part $(Z_1)$ in Figure \ref{ThreezonesFig}, the curve is very flat, which we will quantify by proving that the $(s,t,w,r)$ hessian of $\Psi_0$ is nondegenerate. Hence, we will be able to apply Theorem \ref{mixedVdCABZ} by isolating the $u$ variable, with exactly the same analysis than around the circle $\mathcal{C}_{\sigma}$ defined in Section \ref{subsec2HormPhase}.

\myindent In the second zone $(Z_2)$, which is represented by the black part in Figure \ref{ThreezonesFig}, the curve is nearly vertical. As we will see, there actually holds that the determinant of $\nabla_{s,t,w,r}^2\Psi_0$ is very small around that zone. Hence, we will be able to prove that the full $(u,s,t,w,r)$ hessian of $\Psi_0$ is nondegenerate around this zone, with the same analysis than when dealing with the branches $\mathcal{E}_{\alpha}$ for $\sigma$ defined in Section \ref{subsec3CdV}, in the case $\sigma$ close to zero, i.e. Lemma \ref{FullHessiannearequator}.

\myindent Lastly, for the third zone $(Z_3)$, which is represented by the green part in Figure \ref{ThreezonesFig}, we have to use to the same analysis than in Section \ref{subsec4HormPhase}. Namely, we need to isolate the $w$ variable. Since this zone is the most delicate, and since its size constraints the definition of the other zone, we start by detailing the analysis in $(Z_3)$.

\quad

\myindent The main point is that we can formulate a version of the dichotomy between Proposition \ref{delicatecase} and Lemma \ref{easycase}, which is adapted to this context. While the following proposition may seem exactly the same than the combination of Proposition \ref{delicatecase} and Lemma \ref{easycase}, it is very important to observe that, in Proposition \ref{delicatecase} and Lemma \ref{easycase}, $\Psi$ stands for the \textit{Hörmander parametrix}, while in the following proposition it stands for the \textit{antipodal Hörmander parametrix}, both defined in Section \ref{subsec2Micro}. We recall that $u_0(\alpha), r_0(\alpha)$ are defined by Lemma \ref{defcirclestatpoints}.

\begin{proposition}
    Let $\alpha = 0$ or $\alpha = \pi$. There exists a constant $c> 0$ depending on $\mathcal{S},\eps,\mathfrak{Q}$, such that, for all nonzero $(A,B)$, if
    \begin{equation}\label{smallsigmaantipod}
        \sigma \leq cM^{-1},
    \end{equation}
    then the following dichotomy holds.

    \myindent i. Either
    \begin{equation}
        |\scal{h'(u_{0}(\alpha))}{(A + \pi,B + \pi)}| \lesssim M^{-1}.
    \end{equation}

    \myindent In that case, define
    \begin{equation}
    \begin{split}
    	&I := \left\{w \in S^1 \ \text{such that} \ |w - \alpha| \lesssim M^{-1}\right\} \\
	    &U := \left\{ (u,s,t,r) \in (\R/\ell\Z) \times 
	2\tilde{\mathcal{R}}_0^{(\pi)} \times \R_+ \ \text{such that} \ \begin{cases}
	|u - u_0(\alpha)| \lesssim M^{-2} \\
	|(s,t,r) - (0,0, r_0(\alpha))| \lesssim M^{-1} \end{cases} \right\}.
	\end{split}
	\end{equation}
	
	\myindent Then, for all $w \in I$, the function
    \begin{equation}
        \Psi_w : (u,s,t,r)\in U \mapsto \Psi(\sigma,A,B,u,s,t,w,r)
    \end{equation}
    has a unique stationary point (i.e. $\nabla_{u,s,t,r}\Psi_w$ has a unique zero)
    \begin{equation}
        Z_{\sigma,A,B}(w) := (u,s,t,r)(w) \in U,
    \end{equation}
    at which its Hessian 
    \begin{equation}
    	H(w) := \nabla_{u,s,t,r}^2 \Psi(\sigma,A,B,u(w),s(w),t(w),w,r(w))
\end{equation}
	is nondegenerate. Moreover, the following estimates hold : first, using the Notations \ref{defM1+d} and \ref{defHandN}, for all integers $0\leq k \leq \ell$,
    \begin{equation}\label{controlustrhohess}
        \begin{split}
           	&\mathcal{M}_{k,l}^{(u,s,t,r)}(\Psi) \lesssim M \\
           	&\mathcal{D}(\Psi) \gtrsim 1\\
		    &\mathcal{N}(\Psi) \lesssim M.
        \end{split}
    \end{equation}

    \myindent Second, for all $w \in I$, for all $(u,s,t,r) \in U$, using the notation \eqref{definitionofM}, there holds
    \begin{equation}
    	\left\| M(w,(u,s,t,r)) - H(w) \right\| \lesssim \frac{1}{2} \|H(w)^{-1}\|^{-1} .
	\end{equation}
    
   \myindent Third, using the definition of the 1D remaining phase function introduced in \ref{defphi1D}, there holds 
    \begin{equation}\label{4thderivativesPsi1DdelicatecaseAntipod}
        \forall w \in I \qquad \left| (\partial_w)^4 \Psi^{1D} (\sigma,A,B,w) \right| \gtrsim M.
    \end{equation}
    
    \myindent ii. Or
    \begin{equation}
        |\scal{h'(u_0(\alpha))}{(A + \pi,B + \pi)}| \gtrsim  M^{-1}.
    \end{equation}

    \myindent In that case
    \begin{multline}
        \text{For all} \ (\sigma,u,s,t,w,r) \ \text{such that} \begin{cases}
            &|u - u_0(\alpha)| \lesssim M^{-2} \\
            &|(s,t)| \lesssim  M^{-1}
        \end{cases} \qquad \text{there holds}\\
        |\partial_u \Psi(\sigma,A,B,u,s,t,w,r)| \gtrsim  M^{-1}.
    \end{multline}
\end{proposition}

\begin{proof} We won't detail all the proof, but rather the main differences with the proof of Proposition \ref{delicatecase} and of Lemma \ref{easycase}. For the second case, the point is that
    \begin{equation}
        \begin{split}
            (u,s,t) \mapsto &\partial_u \Psi(\sigma,A,B,u,s,t) \\
            &= -\scal{h'(u)}{(A + \pi, B + \pi) + (s,t)}
        \end{split}
    \end{equation}
    is a smooth function which is \textit{independent} of $\sigma$. Hence, the same argument as for the proof of Lemma \ref{easycase} holds (in particular, the scaling in terms of powers of $M$ for $u$ and $(s,t)$ is the same).

    \myindent Hence, the only part to change is the first case. One can actually follow the same steps than for the proof of Proposition \ref{delicatecase}, adding $\sigma$ as a variable and tracking its contribution in terms of powers of $M$.
    
    \myindent First, for example, there holds
    \begin{equation}
        \begin{split}
            &\det(\nabla_{u,s,t,r}^2 \Psi(\sigma,A,B,u,s,t,w,r)) \\
            &= \begin{vmatrix}
                -\scal{h''(u)}{(A+\pi +s,B+\pi + t)} & -h'(u)_1 & -h'(u)_2 & 0 \\
                -h'(u)_1 & 0 & 0 & 1 \\
                -h'(u)_2 & 0 & r \partial_{\theta \theta} \varphi((t,-\sigma,),(0,\sigma),w) & \partial_{\theta} \varphi \\
                0 & 1 & \partial_{\theta}\varphi((t,-\sigma,),(0,\sigma),w) & 0 
            \end{vmatrix} \\
            &= \begin{vmatrix}
                -h'(u)_1 & 1 \\
                -h'(u)_2 & \partial_{\theta}\varphi ((t,-\sigma,),(0,\sigma),w)
            \end{vmatrix}^2 + r \left(\scal{h''(u)}{(A + \pi + s, B + \pi + t)}\right) (\partial_{\theta\theta}\phi)^2 \\
            &:= D(\sigma,t,u,w) + r K(A,B,u,s,t) e(\sigma,t,w),
        \end{split}
    \end{equation}
    where we use similar notations than in the proof of Lemma \ref{nondegenerateustrhohess}, while remaining cautious, again, that there is a difference in the definition of $\Psi$ (and more precisely $\varphi$).

    \myindent Now, there holds that
    \begin{equation}\label{firststepprf81}
        D(0,0,u_0,0) \neq 0,
    \end{equation}
    and that
    \begin{equation}\label{secondstepprf81}
        |D(\sigma, t, u,w) - D(0,0,u_0(\alpha),0)| \lesssim |(\sigma, t,u,w) -(0,0,u_0(\alpha),0)|.
    \end{equation}

    \myindent Similarly, there holds that
    \begin{equation}
        e(0,0,0) = \partial_{\theta\theta}\varphi((0,0),(0,0),0)=0,
    \end{equation}
    thanks to Lemma \ref{miracleequation}. In particular, the following scaling holds (remember that $r$ is bounded)
    \begin{equation}\label{thrdstepprf81}
        \left|r K(A,B,u,s,t)e(\sigma,t,w) \right| \lesssim M|(\sigma,u,s,t,w) - (0,u_0(\alpha),0,0,0)|.
    \end{equation}

    \myindent Ultimately, \eqref{firststepprf81}, \eqref{secondstepprf81},\eqref{thrdstepprf81} yield the following scaling 
    \begin{equation}
        \forall \begin{cases}
            |\sigma| \lesssim M^{-1} \\
            |u - u_0(\alpha)| \lesssim M^{-1} \\
            |(s,t,w)| \lesssim M^{-1} 
        \end{cases}, \qquad |\det(\nabla_{u,s,t,r}^2 \Psi(\sigma,A,B,u,s,t,w,r)) | \gtrsim 1.
    \end{equation}

    \myindent Next, an equivalent of Lemma \ref{scalingofnieghb} holds with exactly the same proof, and the scaling
    \begin{equation}
        \begin{cases}
            |(\sigma,s,t,w,r) - (0,0,0,\alpha,r_{\sigma}(\alpha))| &\lesssim  M^{-1} \\
            |u - u_0(\alpha)| \lesssim M^{-2}
        \end{cases}.
    \end{equation}

    \myindent Thirdly, the equivalent of Lemma \ref{exunicofzero} still holds. Indeed, the only subtle point is to prove that
    \begin{equation}
        \nabla_{u,s,t,r} \Psi(\sigma,A,B,u,s,t,w,r) = \begin{pmatrix}
            O(M^{-1}) \\
            0\\
            O(M^{-2}) \\
            0
        \end{pmatrix}.
    \end{equation}
    
    \myindent This follows from the fact that, thanks to Lemma \ref{miracleequation},
    \begin{equation}
        \partial_{\theta \sigma} \varphi ((0,0),(0,0),0) = 0.
    \end{equation}
    
    \myindent With similar techniques, one can check that \eqref{4thderivativesPsi1DdelicatecaseAntipod} still holds when $|\sigma| \lesssim M^{-1}$.
\end{proof}

\quad

\myindent Now, we turn to the first zone $(Z_1)$ in Figure \ref{ThreezonesFig}, that is the case where the curve is very flat. We give the following lemma, whose proof is very similar to that of Lemma \ref{explicitstrhohessonCsig} and Corollary \ref{norminversehessCsig} for the claim on the Hessian. Regarding the claim on the remaining phase function, one needs only observe that it is \textit{independent} of $\sigma$, hence we have already proved that claim (see the proof of Lemma \ref{quantitativestimateICx}).

\begin{lemma}
    Assume that condition \eqref{smallsigmaantipod} holds. Up to taking $c$ smaller, let
    \begin{equation}
        \mathcal{I} := \left\{u \in (u_0(\pi), u_0(0)) \ \text{such that} \ |u - u_0(\alpha)| \gtrsim M^{-2}, \ \alpha = 0, \pi\right\}.
    \end{equation}

    \myindent Then, on the fraction of the curves $\mathcal{E}^{\pm}$ defined by
    \begin{equation}
        \mathcal{C}^{\pm} := \{(u,s,t,r) \in \mathcal{E}^{\pm} \qquad \text{such that} \qquad u \in \mathcal{I}\},
    \end{equation}
    there holds on the one hand that the $(s,t,w,r)$ Hessian of $\Psi$, namely $\nabla_{s,t,w,r}^2\Psi$, is invertible, and the following estimates (see Notation \eqref{defM1+d} and Definition \ref{defHandN})
    \begin{equation}
        \begin{split}
            &\mathcal{M}_{3,3}^{(s,t,w,r)}(\Psi) \lesssim 1 \\
            &\mathcal{D}(\Psi) \gtrsim M^{-2}\\
            &\mathcal{N}(\Psi) \lesssim M.
        \end{split}
    \end{equation}

    \myindent Finally, the remaining phase function
    \begin{equation}
        \Psi^{1D} : u \in \mathcal{I} \mapsto \Psi(\sigma,A,B,u,s_{\pm}(\sigma,u),t(\sigma,u),w_{\pm}(\sigma,u),r_{\pm}(\sigma,u))
    \end{equation}
     satisfies the Property $(VdC)_3$ (see Definition \ref{VdCN}) with constants $c_1 M$, $c_2M$, for some universal $c_1,c_2 > 0$.
\end{lemma}

\quad

\myindent Finally, on the second zone $(Z_2)$, that is the portion of the curve with is very steep, the following equivalent of Lemma \ref{FullHessiannearequator} holds, the proof being exactly the same.
\begin{lemma}
    Assume that condition \eqref{smallsigmaantipod} holds, up to taking $c$ smaller. Let c' > 0 and
    
    \begin{equation}
	    \mathcal{D} := \left\{(u,s,t,w,r)\ \text{such that} \ \begin{cases}
	        |u - u_{0}(\alpha)| \lesssim_{c,c'} 1\\
	       |t| \geq c'M^{-1} \\
	        |s| \lesssim_{c,c'} 1 \\
	       |w - \alpha| \lesssim_{c,c'} M^{-1} \\
	       (w,r) \in \mathcal{K}_{\sigma}
	    \end{cases}\right\}.
	\end{equation}
	
	\myindent Then, for all $(u,s,t,w,r) \in \mathcal{D}$
	\begin{equation}
	    \left|\det(\nabla_{u,s,t,w,r}^2 \Psi(\sigma,A,B,u,s,t,w,r)) \right| \gtrsim_{c,c'} M^{-1}.
	\end{equation}
\end{lemma}

\subsubsection{The case where $\sigma$ is large enough : remaining phase function}\label{subsubsec14Antipod}

\myindent The previous paragraph yields, through condition \eqref{smallsigmaantipod}, the threshold which quantifies the case $\sigma$ "small" or $\sigma$ "large". Hence, in this paragraph, we assume that the converse of condition \eqref{smallsigmaantipod} holds. Intuitively, when $\sigma$ is "large", the curve $u \mapsto t(\sigma,u)$, which we recall is given by Figure \ref{antipodut}, is not too sharp, and the same analysis than in Section \ref{secBicharact} can be performed, by simply isolating the $u$ variable and apply Theorem \ref{mixedVdCABZ}. Quantitatively, one can prove the two following crucial lemmas. 

\myindent First, regarding the $(s,t,w,r)$ Hessian of $\Psi_0$, following still the same strategy as before, we need to prove that $\mathcal{E}^{\pm}$ is a branch of \textit{nondegenerate} $(s,t,w,r)$ stationary points of $\Psi_0$. One can prove the following.

\begin{lemma}
    Let $\sigma > 0$. There holds
    \begin{equation}
        \forall u \in [u_{\pi}(\sigma) + \sigma^2 u_1(\sigma), u_0(\sigma) - \sigma^2 u_1(\sigma)], \qquad |\det(\nabla_{s,t,w,r}^2 \Psi_0(u,s_{\pm}(\sigma,u),t(\sigma,u), w_{\pm}(\sigma,u),r_{\pm}(\sigma,u)))| \gtrsim \sigma^2.
    \end{equation}
    
    \myindent Moreover, with the Notation \ref{defM1+d} and the Definition \ref{defHandN}, there holds for all integers $1\leq k \leq l$,
    \begin{equation}
    \begin{split}
        &\mathcal{M}_{k,l}(\Psi) \lesssim 1\\
        &\mathcal{D}(\Psi) \gtrsim \sigma^2 \\
        &\mathcal{N}(\Psi) \lesssim \sigma^{-2}.
    \end{split}
    \end{equation}
\end{lemma}

\myindent Moreover, regarding the remaining phase function, there holds

\begin{lemma}\label{Psi1Dlargesigantipod}
    Up to refining the partition $\mathfrak{Q}$, there exists a constant $M_0 > 0$ such that for all 
    \begin{equation}
        |(A,B)| = M \geq M_0,
    \end{equation}
    then the remaining phase function on the branches $\mathcal{E}^{\pm}$, defined by
    \begin{equation}
        \Psi^{1D} : u \in [u_{\pi}(\sigma) + \sigma^2 u_1(\sigma), u_0(\sigma) - \sigma^2 u_1(\sigma)] \mapsto \Psi(\sigma,A,B,u,s_{\pm}(\sigma,u),t(\sigma,u), w_{\pm}(\sigma,u), r_{\pm}(\sigma,u)),
    \end{equation}
    satisfies the Property $(VdC)_3$ (see Definition \ref{VdCN}), with constants $c_1M,c_2M$ for some universal $c_1,c_2 > 0$.
\end{lemma}

\begin{proof}
    The proof follows the line of the proof of Lemma \ref{Psi1DnearEalpha}, with a similar twist than in the proof of Lemma \ref{remainingphasefunctbichparam}, that is we need to guarantee that, when $|t'(u)|$ is large enough, then $|t''(u)|$ is even larger. This would correspond to the very sharp part of the curve in Figure \ref{antipodut}. For the less sharp part, one observes that, since $t'(u)$ is bounded along it, as long as $M$ is large enough, then the second derivative of the 1D phase function cannot vanish (see the formulas of the proof of Lemma \ref{Psi1DnearEalpha}).
\end{proof}

\subsection{Quantitative estimate of the oscillatory integral}\label{subsec2Antipod}

\myindent From the analysis of the previous section, we know that we need to divide the quantitative estimate between two cases.

\myindent First, in the case where condition \eqref{smallsigmaantipod} holds, we have explained along Paragraph \ref{subsubsec13Antipod} that we recover the same situation than in the case $\bullet = (H)$ presented in Section \ref{secHormPhase}, and, in particular, the \textit{same} quantitative estimates (in terms of powers of $M$). Hence, we can actually perform the \textit{same} analysis than in Section \ref{secHormQuant}, and find the \textit{same} quantitative bound on the integral. Hence, we can prove the following lemma.

\begin{lemma}
    Assume that condition \eqref{smallsigmaantipod} holds. Then there holds
    \begin{equation}
        \mathcal{I}_{\lambda,\delta,0}^{(\pi)}(\sigma,A,B) = O_{c}\left(\lambda^{-\frac{1}{4}} M^5 + \lambda^{-\frac{1}{2}} M^{14} \right)
    \end{equation}
\end{lemma}

\myindent Second, in the case where the converse of condition \eqref{smallsigmaantipod} holds, we have seen in Paragraph \ref{subsubsec14Antipod} that we recover basically the same configuration than in Section \ref{secBicharact}. Hence, a similar conclusion holds, with the limitations that the threshold is here $|\sigma| \gtrsim M^{-1}$ whereas in Section \ref{secBicharact} it is $|\sigma|\gtrsim M^{-\frac{1}{2}}$. Overall, we can prove the following lemma.

\begin{lemma}
    Assume that the converse of condition \eqref{smallsigmaantipod} holds.
    
    \myindent i. If $|(A,B)| = M \leq M_0$, where $M_0$ is defined by Lemma \ref{Psi1Dlargesigantipod}, then there holds
    \begin{equation}
        \mathcal{I}_{\lambda,\delta,0}^{(\pi)}(\sigma,A,B) = O_{M_0}\left(1\right).
    \end{equation}
    
    \myindent ii. If $M \geq M_0$, then there holds
    \begin{equation}
        \mathcal{I}_{\lambda,\delta,0}^{(\pi)}(\sigma,A,B) = O_{M_0} \left(\lambda^{-\frac{1}{3}} M^6 + \lambda^{-\frac{1}{2}} M^8\right).
    \end{equation}
\end{lemma}

\section{Further results and conjectures}\label{secFurther}

\subsection{Lower bounds and asymptotic expansions}\label{subsec1Further}

\myindent In this section, we explain how to obtain lower bounds on the norm of the spectral projector \eqref{defPlambdadelta} using the \textit{same} analysis than above, see Corollary \ref{lowerbounds}.

\myindent Coming back to Paragraph \ref{subsubsec11Micro}, and to its notations, we have studied so far \textit{upper bounds} on the smoothed spectral projector $P^{\sharp}_{\lambda,\delta}$ defined by \eqref{deffld}. Precisely, let us recall that, thanks to formula \eqref{intexprflambdadelta}, we have studied the diagonal values of the Schwartz kernel of $P^{\sharp}_{\lambda,\delta}$, which we recall are given by
\begin{equation}\label{firstintexprlowerbounds}
    P^{\sharp}_{\lambda,\delta}(x,x) = \lambda \delta (2\pi)^{-2} \int_{\R^2} \hat{d\mu}(\lambda(s,t)) \hat{\rho}(\delta(s,t)) \left(e^{isQ_1} e^{itQ_2}\right)(x,x) dsdt.
\end{equation}

\myindent We recall that, in Paragraph \ref{subsubsec25Micro}, and using its notations, we have decomposed the integral \eqref{firstintexprlowerbounds} as 
\begin{equation}\label{sumexpressPsharpkernel}
    P^{\sharp}_{\lambda,\delta}(x,x) = \lambda \delta (2\pi)^{-2} \sum_{\bullet} \sum_{(A,B)\in (2\pi\Z)^2} (-1)^{\frac{A}{2\pi}} \mathcal{I}_{\lambda,\delta,i}^{(\bullet)} (\sigma,A,B) + O\left(\lambda^{-\frac{1}{3}} \delta^{-\frac{5}{3}}\right),
\end{equation}
where the sum $\sum_{\bullet}$ is taken on $\bullet = (H),(1,...,J_i)$, and additionally $\bullet = (\pi)$ in the case $i = 0$ (see Lemma \ref{smoothingcontrib}).

\myindent Now, what we have done in the previous sections is to choose $\rho$ such that $\hat{\rho}$ is compactly supported, and we have given upper bounds on the different $\mathcal{I}_{\lambda,\delta,i}^{(\bullet)} (\sigma,A,B)$ under that condition in Proposition \ref{intermediatethm}, allowing for the resummation process yielding the upper bound on Theorem \ref{mainthm} in the end of Paragraph \ref{subsubsec25Micro} (see \eqref{resummation}). However, we could twist a little this approach : indeed, $P^{\sharp}_{\lambda,\delta}$ is actually defined for any Schwartz function $\rho$. If, however, $\hat{\rho}$ is not compactly supported, one can generalize Proposition \ref{intermediatethm}, using the fact that, for any integers $k,K \geq 0$, there holds
\begin{equation}
    \forall (s,t) \in \R^2 \left| \left(\nabla^k \hat{\rho}\right)(\delta (s,t)) \right|_{L^{\infty}} \lesssim_{k,K,\rho} (1 + \delta|(s,t)|)^{-K}.
\end{equation}

\myindent In particular, when $|(A,B)| \gg \delta^{-1}$, we may always use this as an improvement for all the estimates, in the sense that, since on the support of integration of $\mathcal{I}_{\lambda,\delta,i}^{(\bullet)}(\sigma,A,B)$ there always holds
\begin{equation}
    \|\nabla^k \hat{\rho} \|_{L^{\infty}} \lesssim_{k,K,\rho} (1 + \delta M)^{-K},
\end{equation}
we can always transform bounds of the form
\begin{equation}
    O\left(\lambda^{-\tau} M^{K} \right),
\end{equation}
into bounds of the form (say)
\begin{equation}
    O_{\rho}\left(\lambda^{-\tau} \delta^{-K + 3}\right) M^{-3}.
\end{equation}

\myindent Thus, we could obtain an equivalent of the upper bounds of Proposition \ref{intermediatethm}, but which we could sum for all $(A,B) \in (2\pi \Z)^2$. The only issue is that we would need to change a little bit the computations, since, if one is careful, we have used a lot Theorem \ref{mixedVdCABZ}, with the technical assumption \eqref{technicalhyp} always satisfied since we chose $M = |(A,B)| \lesssim \delta^{-1} \ll \lambda^{-\frac{1}{3}}$. Hence, if we wished to obtain bounds for $\mathcal{I}_{\lambda,\delta,i}^{\bullet}(\sigma,A,B)$ for \textit{any} $(A,B)$, we would need to use the general version of Theorem \ref{mixedVdCABZ} given by Remark \ref{resultnotechnicalhypo}. Overall, it is possible without more efforts to prove the following.
\begin{proposition}\label{firstupperboundslowerbounds}
    There exists a constant $M_0 > 0$ such that the following holds for all $\lambda, \delta$ such that $\delta \gtrsim_{\mathcal{S},\eps,\mathfrak{Q}} \lambda^{-1}$, for $i \in \{0,...,I\}$, and for all $\sigma \in \mathcal{I}_i$.
    
    \myindent i. In the case where $\bullet, A,B = (H),0,0$, there holds
    \begin{equation}
        \mathcal{I}_{\lambda,\delta,i}^{(H)}(\sigma,0,0) = 2\hat{\rho}(0,0)c_W(\sigma) + O(\lambda^{-1}),
    \end{equation}
    where $c_W(\sigma) > 0$ is the constant of the pointwise Weyl law,i.e.
    \begin{equation}
        c_W(\sigma) = \int_{p_1(\sigma,\xi) \leq 1} d\xi.
    \end{equation}
    
    \myindent ii. In the case where $\bullet = (H)$ and $(A,B) \neq (0,0)$, or $M\geq M_0$ and $\bullet = (1,...,J_i)$ or $i = 0$ and $\bullet = (\pi)$, there holds
    \begin{equation}
         \mathcal{I}_{\lambda,\delta,i}^{(\bullet)}(\sigma,A,B) = O\left(\lambda^{-\tau}\delta^{-K}\right)M^{-3},
    \end{equation}
    where the constant $\tau \geq \frac{1}{4}$ and the integer $K \geq 0$ can be quantifiable depending on each case.
    
    \myindent iii. In the case where $M\leq M_0$, and $\bullet = (1,...,J_i)$, or $i = 0$ and $\bullet = (\pi)$, there holds
    \begin{equation}
        \mathcal{I}_{\lambda,\delta,i}^{(\bullet)}(\sigma,A,B) = O_{M_0}(1).
    \end{equation}
\end{proposition}

\myindent 

\myindent Hence, since 
\begin{equation}
    \sum_{(A,B) \in (2\pi Z)^2} M^{-3} < \infty,
\end{equation}
formula \eqref{sumexpressPsharpkernel} and Proposition \ref{firstupperboundslowerbounds} yield that, for some constants $\tau,K > 0$,
\begin{equation}\label{nearasymptotPsharplowerbounds}
    P_{\lambda,\delta}^{\sharp}(x,x) = \lambda \delta (2\pi)^{-2} \left( 2 c_W(\sigma) \hat{\rho}(0,0) + \sum_{\bullet} \sum_{|(A,B)| \leq M_0} (-1)^{\frac{A}{2\pi}} \mathcal{I}_{\lambda,\delta,i}^{\bullet} (\sigma,A,B) + O\left(\lambda^{-\tau} \delta^{-K}\right) \right).
\end{equation}

\myindent Thus, we nearly have an \textit{asymptotic expansion} of the Schwartz kernel  $P_{\lambda,\delta}^{\sharp}(x,x)$. However, the problem is that we a priori only have proved that 
\begin{equation}
    \sum_{\bullet} \sum_{|(A,B)| \leq M_0} (-1)^{\frac{A}{2\pi}} \mathcal{I}_{\lambda,\delta,i}^{\bullet} (\sigma,A,B) = O_{M_0}(1),
\end{equation}
hence the two first terms in the asymptotic expansion are of the same order. In order to remove this problem, we prove the following lemma.

\begin{lemma}\label{O1too1lowerbounds}
    Assume that $\bullet = (1),...,(J_i)$ or $i = 0$ and $\bullet = (\pi)$. Assume that $M\leq M_0$. Then there holds
    \begin{equation}
        \mathcal{I}_{\lambda,\delta,i}^{\bullet}(\sigma,A,B) = o_{M_0}(1).
    \end{equation}
\end{lemma}

\begin{proof}
    Fix $(A,B) \in (2\pi\Z)^2$. Coming back to the analysis of Sections \ref{secBicharact} and \ref{secAntipod}, what we actually proved in those sections is that, using the exact stationary phase lemma given by Theorem \ref{hormstatphase},
    \begin{equation}
    \mathcal{I}_{\lambda,\delta,i}^{\bullet}(\sigma,A,B) = \int du e^{i\lambda \Psi^{1D}(u)} f(u) + O_{A,B} (\lambda^{-1}),
    \end{equation}
    for some smooth function $f(u)$ and for the remaining 1D phase function
    \begin{equation}
        \Psi^{1D}(u) = -\scal{h(u)}{(s,t)(u) + (A, B)},
    \end{equation}
    where, by definition, $(s,t)(u)$ are such that the bicharacteristic of $q_1$ starting at $(t(u),\sigma)$ with direction $\xi$ such that
    \begin{equation}
        h(u) \ \text{is colinear to} \ \begin{pmatrix}\label{directionhulowerbounds}
            q_1(\sigma,\xi)\\
            q_2(\sigma,\xi)
        \end{pmatrix}
    \end{equation}
    reaches (the cotangent space at) $(0,\sigma)$ after a time $s(u)$\footnote{In this proof, observe that we don't note explicitly the influence of $(s_i^{\bullet},t_i^{\bullet})$. In particular, in the previous subsections, we typically defined locally some $(s(\sigma,u),t(\sigma,u))$ so that, in this proof, there would hold $(s(u),t(u)) = (s(\sigma,u),t(\sigma,u)) + (s_i^{\bullet},t_i^{\bullet})$.}.
    
    \myindent Now, the reason that we then roughly bounded the integral 
    \begin{equation}
        \int du e^{i\lambda \Psi^{1D}(u)} f(u)
    \end{equation}
    by $O(1)$ is that it is a priori not possible to have uniform useful lower bounds on the derivatives of $\Psi^{1D}$, as explained in Sections \ref{secBicharact} and \ref{secAntipod}, and thus we ultimately can't use Proposition \ref{usefulVdC} to gain a $O(\lambda^{-\tau})$ in the estimates.
    
    \myindent However, if we only wish to gain a $o(1)$, and not a $O(\lambda^{-\tau})$, we don't need uniform bounds on the derivatives. Indeed, we only need to prove that the stationary points of $u \mapsto \Psi^{1D}(u)$ are \textit{isolated}, see Lemma \ref{easyoscintlemma}. Now, this is actually true. Indeed, let us give a geometrical interpretation of the equation for a stationary point, namely 
    \begin{equation}\label{psi'eqzerolowerbounds}
        (\Psi^{1D})'(u) = 0.
    \end{equation}
    
    \myindent Using the fact that, as observed before,
    \begin{equation}
        (\Psi^{1D})'(u) = (\partial_u \Psi)(A,B,u,s(u),t(u)) = -\scal{h'(u)}{(s(u),t(u)) + (A,B)},
    \end{equation}
    we thus find that \eqref{psi'eqzerolowerbounds} is equivalent to
    \begin{equation}\label{partuPsieq0lowerbounds}
        \scal{h'(u)}{(s(u),t(u)) + (A ,B)} = 0.
    \end{equation}
    
    \myindent Now, by definition of the curve $\gamma_0$ (see Definition \ref{deflilgammas}), $h'(u)$ is the tangent to the curve $\gamma_0$ at $h(u)$. Using Proposition \ref{g'eq-omega}, and equation \eqref{directionhulowerbounds}, we thus know that
    \begin{equation}
        h'(u) \ \text{is colinear to} \ (-\omega(I),1),
    \end{equation}
    where $I$ is the Clairaut integral of the direction $\xi$ such that \eqref{directionhulowerbounds} holds. Coming back to \eqref{partuPsieq0lowerbounds}, we ultimately find that \eqref{psi'eqzerolowerbounds} is equivalent to
    \begin{equation}\label{relationstlowerbounds}
        -\omega(I)(s(u) + A) + t(u) + B= 0,
    \end{equation}
    where $I$ is the Clairaut integral of the direction of the unique bicharacteristic joining $(t(u),\sigma)$ to $(0,\sigma)$ in time $s(u)$. This last assertion means that
    \begin{equation}\label{firstgeometricallowerbounds}
        P\left(\Phi^{q_1}_{s(u)} \left(t(u),\sigma,I,\pm \sqrt{1 - I^2 f(\sigma)^{-2}}\right)\right) = (0,\sigma).
    \end{equation}
    
    \myindent Since the Hamiltonian flow of $q_1$ is $2\pi$-periodic, and since $A \in 2\pi \Z$, this implies that
    \begin{equation}\label{secondgeometricallowerbounds}
        P\left(\Phi^{q_1}_{s(u) +  A} \left(t(u),\sigma,I,\pm \sqrt{1 - I^2 f(\sigma)^{-2}}\right)\right) = (0,\sigma).
    \end{equation}
    
    \myindent Now, we know that the Hamiltonian flow $\Phi^{q_1}_s$ is given by the explicit formula of Lemma \ref{explicitbicharac}. Hence, \eqref{secondgeometricallowerbounds} is equivalent to
    \begin{equation}\label{thirdgeometricallowerbounds}
        P\left(\Phi_{\frac{(s(u) + A)}{2\pi} \tau(I)} \circ \Phi^{p_2}_{-(s(u) + A )\omega (I)} \left(t(u),\sigma,I,\pm \sqrt{1 - I^2 f(\sigma)^{-2}}\right)\right) = (0,\sigma).
    \end{equation}
    
    \myindent Now, from relation \eqref{relationstlowerbounds}, we know that
    \begin{equation}
        -(s(u) + A )\omega(I) = -t(u) - B
    \end{equation}.
    
    \myindent Hence, \eqref{thirdgeometricallowerbounds} is equivalent to 
    \begin{equation}\label{fourthgeometricallowerbounds}
        P\left(\Phi_{\frac{(s(u) + A)}{2\pi} \tau(I)} \circ \Phi^{p_2}_{-t(u) - B} \left(t(u),\sigma,I,\pm \sqrt{1 - I^2 f(\sigma)^{-2}}\right)\right) = (0,\sigma).
    \end{equation}
    
    \myindent However, $t \mapsto \Phi^{p_2}_t$ is by definition the rotation around the axis. Hence, since $B \in 2\pi \Z$, there holds that
    \begin{equation}
        \Phi^{p_2}_{-t(u) - B} \left(t(u),\sigma,I,\pm \sqrt{1 - I^2 f(\sigma)^{-2}}\right) = \left(0,\sigma,I,\pm \sqrt{1 - I^2 f(\sigma)^{-2}}\right).
    \end{equation}
    
    \myindent Thus,  \eqref{fourthgeometricallowerbounds} can finally be written as
    \begin{equation}
        P\left(\Phi_{\frac{(s(u) + A)}{2\pi} \tau(I)} \left(0,\sigma,I,\pm \sqrt{1 - I^2 f(\sigma)^{-2}}\right)\right) = (0,\sigma).
    \end{equation}
    
    \myindent Now, this equation means that the geodesic $\gamma(I)$ \textit{self-intersects}, after a time $\frac{s(u) + A}{2\pi} \tau(I)$, i.e. that it is a \textit{geodesic loop}. Since we assume that $\mathcal{S}$ is symmetric, we have already observed that this means that $\gamma(I)$ is periodic, and that a period is $2 \times \frac{s(u) + A}{2\pi} \tau(I)$.
    
    \myindent Overall, we have thus proved the very important following geometrical fact : for some universal constant $K > 0$,
    \begin{multline}
        \left(\Psi^{1D}\right)'(u) = 0 \ \text{if and only if} \\
        \text{the geodesic with Clairaut integral} \ I(u) \ \text{is \textit{periodic} of period smaller that} \ K(1+A).
    \end{multline}
    
    \myindent Here, the Clairaut integral $I(u)$ is defined by 
    \begin{equation}
        h(u) \ \text{colinear to} \ \begin{pmatrix}
            q_1(\sigma, I, \pm \sqrt{1 - I^2 f(\sigma)^{-2}})\\
            q_2(\sigma, I, \pm \sqrt{1 - I^2 f(\sigma)^{-2}})
        \end{pmatrix}.
    \end{equation}

    \myindent Now, to conclude the proof, we need only observe that, since $u \to I(u)$ is a continuous one-to-one correspondence on the interval that we consider, then the stationary points of $\Psi^{1D}$ are isolated because there is a finite number of Clairaut integrals $I \in [-1,1]$ such that the associated geodesic is periodic with period smaller that $K(1 + A)$. Indeed, this last fact follows from the fact that, thanks to the analysis of Paragraph \ref{subsec2CdV}, a geodesic with Clairaut integral $I$ is periodic, with period smaller than $T$, if and only if $\omega(I)$ is a rational number with denominator smaller that $CT$, for some universal constant $C$. Now, with the twist Hypothesis \ref{twisthypothesis}, there are necessarily a finite number of such $I$ in $[-1,1]$.
    \end{proof}
    
\begin{remark}\label{nablaPsieq0interpretation}
    With the analysis of Paragraph \ref{subsubsec33Micro} and with this proof, we are actually finally in a position to interpret the stationary points of the phase $\Psi$ in the full $(u,s,t,\eta)$, where the phase $\Psi$ is defined by \eqref{abstractdefofPsi}. Indeed, we have proved that there holds
    \begin{equation}
        \nabla_{u,s,t,\eta} \Psi(\sigma,A,B,u,s,t,w,\eta) = 0,
    \end{equation}
    if and only if
    
    \myindent i. The \textit{geometric equation} holds, i.e. the bicharacteristic of $q_1$ starting at $(t + t_i^{\bullet},\sigma)$ with direction 
    \begin{equation}
        \nabla_x \phi\left((t + t_i^{\bullet},\sigma),(0,\sigma),\frac{\eta}{|\eta|}\right) 
    \end{equation}
    reaches (the cotangent space at) $(0,\sigma)$ after a time $s + s_i^{\bullet}$.
    
    \myindent ii. The \textit{correspondence equation} holds, i.e. $u$ and $|\eta|$ are fixed by 
    \begin{equation}
        h(u) = |\eta| \begin{pmatrix}
            q_1\left(\sigma, \nabla_x \phi\left((t + t_i^{\bullet},\sigma),(0,\sigma),\frac{\eta}{|\eta|}\right) \right) \\
            q_2 \left(\nabla_x \phi\left((t + t_i^{\bullet},\sigma),(0,\sigma),\frac{\eta}{|\eta|}\right)\right)
        \end{pmatrix}.
    \end{equation}
    
    \myindent iii. If $I$ is the Clairaut integral of the direction $\nabla_x \phi\left((t + t_i^{\bullet},\sigma),(0,\sigma),\frac{\eta}{|\eta|}\right)$, then the \textit{geodesic} with Clairaut integral $\omega(I)$ is \textit{periodic}, and, precisely, that it self-intersects after a time
    \begin{equation}
        \frac{s + s_i^{\bullet} + A}{2\pi} \tau(I).
    \end{equation}
    
    \myindent iv. Finally, $B$ is fixed by the equation
    \begin{equation}
        -(s + s_i^{\bullet} + A)\omega(I) + t + t_i^{\bullet} + B = 0.
    \end{equation}
    
    \quad
    
    \myindent In other words, the stationary points of $\Psi$ correspond exactly to the \textit{periodic geodesics} (which are the same than the geodesic loops in the case of a symmetric surface of revolution, see Section \ref{subsubsec12CdV}). 
\end{remark}

\myindent Overall, Lemma \ref{O1too1lowerbounds} and formula \eqref{nearasymptotPsharplowerbounds} yield the following exact asymptotic expansion of the Schwartz kernel of $P^{\sharp}_{\lambda,\delta}$. 

\begin{proposition}\label{expansionPsharpxx}
    For all $\eps > 0$, and for all $x = (\theta,\sigma) \in K_{\eps}$, there holds
    \begin{equation}
        P_{\lambda,\delta}^{\sharp}(x,x) = 2 (2\pi)^{-2} c_W(\sigma) \hat{\rho}(0,0) \lambda \delta \left(1 + o_{M_0}(1) + O_{M_0} \left(\lambda^{-\tau} \delta^{-K} \right)\right),
    \end{equation}
    where the constant $\tau,K > 0$ are quantifiable.
\end{proposition}

\myindent We are now in a position to prove lower bounds on the spectral projector $P_{\lambda,\delta}$.

\begin{corollary}\label{lowerbounds}
    Let $\mathcal{S} \in \mathfrak{S}$ (see Definition \ref{defmathfrakS}).Let $K$ be any compact subset of $\mathcal{S}$ which doesn't contain any of the poles. There holds
    \begin{equation}
        \forall \delta \geq \lambda^{-\kappa} \qquad  \|P_{\lambda,\delta}\|_{L^2(\mathcal{S}) \to L^{\infty}(K)} \gtrsim_{\mathcal{S},\eps} \lambda^{\frac{1}{2}} \delta^{\frac{1}{2}},
    \end{equation}
    where $\eps$ is the distance between $K$ and the poles, and where $\kappa > 0$. Moreover, it is possible to find a uniform explicit lower bound for $\kappa$, and it is possible to improve the constant $\kappa$ depending on $K$.
\end{corollary}

\begin{proof}
    \myindent Using the exact result of Lemma \ref{bddnormbykernel}, let us write
    \begin{equation}
        \|P_{\lambda,\delta}\|^2_{L^2(\mathcal{S})\to L^{\infty}(K)} = \sup_{x\in K} \mathds{1}_{[\lambda-\delta, \lambda + \delta]}(\sqrt{-\Delta})(x,x)
    \end{equation}

    Now, with a similar proof than the proof of Definition-Lemma \ref{deffld}, we can prove that, if $\rho$ is a smooth nonzero bump function, which is nonnegative and compactly supported on a sufficiently small ball, then, for all $(q_1,q_2) \in \R^2$, if $F$ is defined by \eqref{defFinThmCdV},
    \begin{equation}
        \mathds{1}_{[\lambda-\delta, \lambda + \delta]}(\sqrt{F(q_1,q_2)}) \gtrsim f_{\lambda,\delta}(q_1,q_2),
    \end{equation}
    where $f_{\lambda,\delta}$ is defined by \eqref{defflambdadelta}. In particular, using Theorem \ref{CdVthm} this means that, for all $x$,
    \begin{equation}
        \mathds{1}_{[\lambda-\delta, \lambda + \delta]}(\sqrt{-\Delta})(x,x) \gtrsim P_{\lambda,\delta}^{\sharp}(x,x),
    \end{equation}
    where we recall that the smoothed spectral projector $P_{\lambda,\delta}^{\sharp}$ is defined by
    \begin{equation}
        \tilde{P}_{\lambda,\delta}^{\sharp} = \sqrt{f_{\lambda,\delta}} (Q_1,Q_2).
    \end{equation}
    
    \myindent We can then conclude the proof using Proposition \ref{expansionPsharpxx}, and $\kappa := \frac{\tau}{K}$, and the fact that
    \begin{equation}
        \hat{\rho}(0,0) = \int_{\R^2} \rho > 0.
    \end{equation}
    
    \myindent The meaning of the last sentence in the proposition is simply that the constants $\tau, K$ in Proposition \ref{expansionPsharpxx} are quantifiable, thus $\kappa$ is quantifiable as well. Moreover, local or global improvements in $\kappa$ are directly obtained from local or global improvements in the constants $\tau,K$ of Proposition \ref{expansionPsharpxx}.
\end{proof}

\begin{remark}
    Regarding lower bounds at the poles, as we mentioned in the introduction, a consequence of \cite{donnelly1978g} is that there exists a constant $R_0 > 0$ depending on $\mathcal{S}$ such that, if $K$ is a compact subset of $\mathcal{S}$ which contains at least one of the poles, then for all $\lambda_0$, there exists a $\lambda$ with $|\lambda - \lambda_0|\leq R_0$ such that
    \begin{equation}
        \|P_{\lambda,\delta} \|_{L^2(\mathcal{S}) \to L^{\infty}(K)} \geq \|P_{\lambda,0}\|_{L^2(\mathcal{S}) \to L^{\infty}(K)} \gtrsim \lambda^{\frac{1}{2}}.
    \end{equation}
\end{remark}

\subsection{Improvements, optimality, and where we see the (conjectured) dominant term}\label{subsec2Further}

\myindent In this section, we discuss possible improvements on the results, and what would be optimal results. Precisely, we ask the following question
\begin{equation}
    \text{What is the largest constant} \ \kappa > 0 \ \text{such that Theorem \ref{mainthm} holds for}\ \delta > \lambda^{-\kappa}\ ?,
\end{equation}
and we will also discuss as a corollary the similar question for lower bounds, i.e. Corollary \ref{lowerbounds}. Since we are only able to give conjectures for the moment, this section is not written very rigorously, and is intended to be rather heuristic.

\myindent For the moment, the main limiting factor for improvements is the estimates of Proposition \ref{intermediatethm}, and, precisely, the exponents appearing on $M$ in the bounds of the different $\mathcal{I}_{\lambda,\delta,i}^{\bullet}(\sigma,A,B)$, as explained in the resummation process, see \eqref{resummation}. In order to prove the different bounds of Proposition \ref{intermediatethm}, we have always ultimately used abstract theorems, such as Theorem \ref{mixedVdCABZ} or Theorem \ref{ABZthm}. The idea, to improve the estimates of Proposition \ref{intermediatethm}, is to be more careful in the estimates for each of the $\mathcal{I}_{\lambda,\delta,i}^{\bullet}(\sigma,A,B)$, i.e. to replicate the proof for example of Theorem \ref{mixedVdCABZ} but using the precise geometry in each case, to track exactly how many powers of $M$ one loses in the oscillatory integral estimates, namely when splitting the integrals into different zones, and when integrating by parts or bounding by absolute value on each. 

\myindent In order to have intuition on what one could expect from this method, we can use the a priori estimates of Theorem \ref{hormstatphase}. Indeed, let us consider, for example, the \textit{worst} term, in terms of powers of $\lambda$, that is the factor $O\left(\lambda^{-\frac{1}{4}} M^5\right)$ which appears in estimate \eqref{estIHAB}. We recall that this term comes from the case where we estimate near the branching points $P_{\alpha}$, in the case where $|\partial_u(P_{\alpha})| \lesssim M^{-1}$, see Paragraph \ref{subsubsec25HormQuant}. We recall that we isolated the variable $w$, in order to write
\begin{equation}\label{isolatewimprove}
    \mathcal{I}_{P_{\alpha}} = \int dw e^{i\lambda\Psi^{1D}(w)} \left(\lambda^2\int \int e^{i\lambda(\Psi - \Psi^{1D})} \chi_{P_{\alpha}} b du ds dt dr\right).
\end{equation}

\myindent Now, for the inner integral, we have proved in Proposition \ref{delicatecase} that the phase has exactly one stationary point $Z_{\sigma,A,B}(w)$, at which the Hessian is nondegenerate. Hence, the usual nondegenerate stationary phase Theorem \ref{hormstatphase} yields a priori that 
\begin{multline}\label{innerintegralimprove}
    \lambda^2\int \int e^{i\lambda(\Psi - \Psi^{1D})} \chi_{P_{\alpha}} b du ds dt dr \\= \frac{(2\pi)^{2}e^{i\frac{\pi}{4} sgn(\nabla_{u,s,t,r}^2\Psi(Z(w)))} }{|\det (\nabla_{u,s,t,r}^2 \Psi (Z(w)))|^{\frac{1}{2}}} b(Z(w)) + O_{\Psi} (\lambda^{-1}) =: f(w) + O_{\Psi}(\lambda^{-1}).
\end{multline}

\myindent Now, thanks to Proposition \ref{delicatecase}, and particularly to (an immediate generalization of) Lemma \ref{nondegenerateustrhohess}, there holds
\begin{equation}
    |\det (\nabla_{u,s,t,r}^2 \Psi (Z(w)))| \simeq 1.
\end{equation}

\myindent Hence, thanks to the lower bound on the fourth derivative of $\Psi^{1D}$ given by \eqref{4thderivativesPsi1Ddelicatecase}, the Van der Corput Lemma \ref{VdCthm} yields (roughly, we ignore $f'$)
\begin{equation}
    \int e^{i\lambda \Psi^{1D}(w)} f(w) dw = O(\lambda^{-\frac{1}{4}} M^{-\frac{1}{4}}).
\end{equation}

\myindent Now, since it is possible that the first three derivatives of $\Psi^{1D}$ vanish at the same time (see Paragraph \ref{subsubsec45Hormphase}), it is reasonable to conjecture that this bound cannot be improved. Overall, we thus conjecture that the $O\left(\lambda^{-\frac{1}{4}} M^5\right)$ in \eqref{estIHAB} could be improved up to
\begin{equation}\label{improvedIPalpha}
    \mathcal{I}_{P_{\alpha}}= O\left(\lambda^{-\frac{1}{4}} M^{-\frac{1}{4}}\right),
\end{equation}
and that this term appears if and only if condition \eqref{smallpartuPsiPalpha} holds, i.e. if and only if
\begin{equation}\label{1DsetABcondition}
    |\scal{h'(u_{\sigma}(\alpha))}{(A,B)}| \lesssim (1 + |(A,B)|)^{-1}.
\end{equation}

\myindent We claim that we are actually able to prove the upper bound \eqref{improvedIPalpha}. While we have not done all the computations, we conjecture that one could similarly improve most of the estimates of Proposition \ref{intermediatethm}, at least in the case $i\neq 0$. Hence, we claim the following, which we choose to keep vague
\begin{equation}
    \text{The constant} \ \frac{1}{32} \ \text{in Theorem \ref{mainthm} can be quantitatively improved, at least away from the equator}.
\end{equation}

\myindent Rather than giving the best improvement that we have been able to prove, we focus on a conjecture regarding the \textit{optimal} result. Indeed, it is not so useful to obtain quantitative improvements, if they are not optimal. Now, from similar intuition that what has been presented above, we actually believe that the dominant term in the estimates for large $M$ is exactly given by the contribution of $\mathcal{I}_{P_{\alpha}}$, in the case where \eqref{1DsetABcondition} holds, at least away from the equator. Indeed, it is the only term of order $O(\lambda^{-\frac{1}{4}})$, where all the other terms are at least of order $O(\lambda^{-\frac{1}{3}})$. Roughly, we thus conjecture that there should hold, with the notations introduced in Definition-Lemma \ref{deffld}, and using \eqref{nearasymptotPsharplowerbounds} and Lemma \ref{O1too1lowerbounds}, for all $x = (\theta, \sigma)$, at least where $\sigma \notin \mathcal{I}_0$,
\begin{multline}\label{conjecturedPsharpimprove}
    P_{\lambda,\delta}^{\sharp}(x,x) = \lambda \delta (2\pi)^{-2} \bigg(2 c_W(\sigma)\hat{\rho}(0,0) + o_{M_0}(1) +  \sum_{\alpha = 0,\pi} \sum_{0 < |(A,B)| \lesssim \delta^{-1},  |\scal{h'(u_{\sigma}(\alpha))}{(A,B)}|\lesssim |(A,B)|^{-1}} (-1)^{\frac{A}{2\pi}} \mathcal{I}_{P_{\alpha}} \\
    +\ \text{lower order terms}\bigg).
\end{multline}

\myindent Now, observing than the set of $(A,B) \in (2\pi \Z)^2$ such that \eqref{1DsetABcondition} holds is \textit{1-dimensional}, in the sense that, for all real $K > - 1$,
\begin{equation}
     \sum_{0 < |(A,B)| \lesssim \delta^{-1}  |\scal{h'(u_{\sigma}(\alpha))}{(A,B)}|\lesssim |(A,B)|^{-1}} M^{K} \lesssim \delta^{-K - 1},
\end{equation}
we recover from \eqref{improvedIPalpha} and \eqref{conjecturedPsharpimprove} the conjecture, for $x = (\theta,\sigma)$ away from both the Poles and the equator,
\begin{equation}
    P_{\lambda,\delta}^{\sharp}(x,x) = \lambda \delta (2\pi)^{-2} \left(2 c_W(\sigma)\hat{\rho}(0,0) + o_{M_0}(1) + O_{M_0}\left(\lambda^{-\frac{1}{4}} \delta^{-\frac{3}{4}}
    +\right)\ \text{lower order terms}\right),
\end{equation}
where the lower order terms should typically be of the form
\begin{equation}
    O(\lambda^{-\tau} \delta^{-K}),
\end{equation}
for some constants $\tau,K > 0$ such that $\frac{\tau}{K} \geq \frac{1}{3}$. 

\myindent As a conclusion, from this very non rigorous numerology, we conjecture the following.
\begin{conjecture}
    Let $\mathcal{S} \in \mathfrak{S}$ (see Definition \ref{defmathfrakS}), let $\eps > 0$, and let $K$ be any compact set of $\mathcal{S}$ which is at a distance at least $\eps$ from both the Poles and the equator. Then, Theorem \ref{mainthm} holds up to $\lambda^{-\frac{1}{3}}$, i.e.
    \begin{equation}
        \forall \delta \geq \lambda^{-\frac{1}{3}} \qquad \|P_{\lambda,\delta} \|_{L^2(\mathcal{S}) \to L^{\infty}(K)} \lesssim_{\mathcal{S},\eps} \lambda^{\frac{1}{2}} \delta^{\frac{1}{2}}.
    \end{equation}
    
    \myindent Furthermore, we conjecture that the lower bound of Corollary \ref{lowerbounds} also holds up to $\lambda^{-\frac{1}{3}}$.
    
    \myindent Finally, we conjecture that, for \textit{any} $\mathcal{S}$, the constant $\frac{1}{3}$ is \textit{optimal}, in the sense that Theorem \ref{mainthm} doesn't hold up to $\lambda^{-\tau}$ for any $\tau > \frac{1}{3}$.
\end{conjecture}

\myindent We now justify the last assertion of the conjecture, i.e. the (conjectured) optimality of $\frac{1}{3}$. In order to do so, recall that we have conjectured that the dominant term is given by the $\mathcal{I}_{P_{\alpha}}$, in the case where
\begin{equation}
    \scal{h'(u_{\sigma}(\alpha))}{(A,B)} \lesssim |(A,B)|^{-1}.
\end{equation}

\myindent Let us study the case where there actually holds exactly
\begin{equation}\label{exactconditionimprove}
    \scal{h'(u_{\sigma}(\alpha)}{(A,B)} = 0.
\end{equation}

\myindent Coming back to the analysis of Section \ref{secHormPhase}, and using the notations of this section, and in particular of Section \ref{subsec4HormPhase}, this means that
\begin{equation}
    \partial_u \Psi(P_{\alpha}) = 0,
\end{equation}
which ultimately means that $P_{\alpha}$ is a "true" stationary point of $\Psi$ in the full $(u,s,t,w,r)$ variable, see the discussion of Paragraph \ref{subsubsec42Hormphase}. Now, in that case, we have actually proved in Paragraph \ref{subsubsec45Hormphase} that there exactly holds
\begin{equation}\label{vanish3derivesimprove}
    \begin{split}
        &\left(\Psi^{1D}\right)'(0) = \left(\Psi^{1D}\right)''(0) = \left(\Psi^{1D}\right)^{(3)}(0) = 0 \\
        &\left|\left(\Psi^{1D}\right)^{(4)}(0)\right| \gtrsim M.
    \end{split}
\end{equation}

\myindent Now, coming back to formula \eqref{isolatewimprove}, and using formula \eqref{innerintegralimprove}, we know that, a priori
\begin{equation}\label{1Doscintimprove}
    \mathcal{I}_{P_{\alpha}} = \int dw e^{i\lambda \Psi^{1D}(w)} f(w) + O_{\Psi}(\lambda^{-1}).
\end{equation}

\myindent Hence, if we use, instead of the upper bound given by the Van der Corput Lemma (see Theorem \ref{VdCthm}), the asymptotic expansion of the oscillatory integral \eqref{1Doscintimprove} under the condition \eqref{vanish3derivesimprove}, which is given for example in \cite{stein1993harmonic}[Proposition 3, Chapter VIII], we thus find that, for some positive constant $C_{\sigma,A,B}$ wich is bounded away from zero and from $+\infty$ independently of $M$,
\begin{equation}\label{IPalphaoptimal}
    \mathcal{I}_{P_{\alpha}} = C_{\sigma,A,B}\lambda^{-\frac{1}{4}} M^{-\frac{1}{4}} e^{i\lambda \Psi^{1D}(0)} + O_{\Psi}\left(\lambda^{-\frac{1}{2}}\right),
\end{equation}
where we could use \cite{stein1993harmonic}[Remark 1.3.4, 2, Chapter VIII] to observe that the term of order $\lambda^{-\frac{1}{2}}$ vanish, and that the remainder is of order $O_{\Psi}\left(\lambda^{-\frac{3}{4}}\right)$, but thus is of no importance to the analysis since we only give conjectures.

\myindent Now, we are in position to \textit{compute} the $\sigma,\alpha$ for which the exact equality \eqref{exactconditionimprove} holds. Indeed, with exactly the same analysis which led to Remark \ref{nablaPsieq0interpretation}, this equality holds if and only if 
\begin{multline}\label{conditionimproveoptimal}
    \text{The \textit{geodesic} starting at} \ (0,\sigma) \ \text{with \textit{horizontal} direction is periodic, of period} \ \frac{A}{2\pi} \tau(I_{\sigma,\alpha}),\\
    \text{and there holds} \ \frac{A}{B} = \omega(I_{\sigma,\alpha}).
\end{multline}

\myindent Here, $I_{\sigma,\alpha}$ is the Clairaut integral of the horizontal direction in the cotangent space $T_{(0,\sigma)}^{*}(\mathcal{S})$, with a sign, i.e., thanks to \eqref{normsurfrev},
\begin{equation}
    I_{\sigma,0} = f(\sigma) \qquad I_{\sigma,\pi} = - f(\sigma).
\end{equation}

\myindent Now, the $(\sigma,A,B)$ such that \eqref{conditionimproveoptimal} holds are easy to compute. Indeed, \eqref{conditionimproveoptimal} can hold for some $(A,B) \in (2\pi \Z)^2 \backslash\{(0,0)\}$ if and only if the geodesic $\gamma(I_{\sigma,\alpha})$ (we use the notation $\gamma(I)$ from Section \ref{subsec2CdV}) starting at $(0,\sigma)$ with horizontal direction is periodic. If it is the case, then we know that $\omega(I_{\sigma,0}) = - \omega(I_{\sigma,\pi})$ is a rational number (see Section \ref{subsec2CdV}). Hence, the sets
\begin{equation}
    \mathcal{D}_{\alpha} := \left\{(A,B) \in (2\pi \Z)^2\backslash\{(0,0)\} \ \text{such that} \ \frac{A}{B} = \omega(I_{\sigma,\alpha})\right\}
\end{equation}
are exactly two straight lines, of slope $\omega(I_{\sigma,\alpha})$. Moreover, if $(A,B) \in \mathcal{D}_{\alpha}$, then there holds necessarily that $\frac{A}{2\pi}\tau(I_{\sigma,\alpha})$ is a period of the geodesic $\gamma(I_{\sigma,\alpha})$. Hence, ultimately, there exists $(A_0,B_0) \in (2\pi \Z)^2 \backslash\{(0,0)\}$ such that 
\begin{equation}
    \mathcal{D}_0 = \{k(A_0,B_0), \ k \in \Z\backslash\{0\}\} \qquad \mathcal{D}_{\pi} = \{k(A_0,-B_0), \ k \in \Z\backslash\{0\}\}.
\end{equation}

\myindent Coming back to the conjectured expansion \eqref{conjecturedPsharpimprove}, and with the formula \eqref{IPalphaoptimal}, and with the observation that since $h(u_{\sigma}(0))$ and $h(u_{\sigma}(\pi))$ are symmetric with respect to the horizontal axis there holds
\begin{equation}
    \scal{h(u_{\sigma}(0))}{(A_0,B_0)} = \scal{h(u_{\sigma}(\pi))}{(A_0,-B_0)},
\end{equation}
it is reasonable to conjecture that, if the geodesic $\gamma(I_{\sigma,0})$ starting at $(0,\sigma)$ with horizontal direction is periodic, and, finally, if $A_0 \in 4\pi \Z$, then there holds
\begin{multline}
     P_{\lambda,\delta}^{\sharp}(x,x) = \lambda \delta (2\pi)^{-2} \bigg(2 c_W(\sigma)\hat{\rho}(0,0) + o_{M_0}(1) +  2\lambda^{-\frac{1}{4}} \sum_{0 < |k| \lesssim \delta^{-1}} C_{\sigma,k} e^{-ik\lambda \scal{h(u_{\sigma}(0)}{(A_0,B_0)}} k^{-\frac{1}{4}} \\
    +\ \text{lower order terms}\bigg).
\end{multline}

\myindent Thus, ultimately, since there is at least an unbounded sequence of $\lambda$ for which
\begin{equation}
    \lambda \scal{h(u_{\sigma}(0))}{(A_0,B_0)} \in 2\pi \Z,
\end{equation}
we have thus found that, for an unbounded sequence of $\lambda$, there should hold
\begin{multline}
     P_{\lambda,\delta}^{\sharp}(x,x) = \lambda \delta (2\pi)^{-2} \bigg(2 c_W(\sigma)\hat{\rho}(0,0) + o_{M_0}(1) +  2\lambda^{-\frac{1}{4}} \sum_{0 < |k| \lesssim \delta^{-1}} C_{\sigma,k} k^{-\frac{1}{4}} \\
    +\ \text{lower order terms}\bigg).
\end{multline}

\myindent Now, since the constants $C_{\sigma,k}$ are uniformly bounded away from zero, observe that there holds
\begin{equation}
    2\lambda^{-\frac{1}{4}} \sum_{0 < |k| \lesssim \delta^{-1}} C_{\sigma,k} k^{-\frac{1}{4}} \gtrsim \lambda^{-\frac{1}{4}} \delta^{-\frac{3}{4}}.
\end{equation}

\myindent Since there is a \textit{dense} set of $\sigma$ such that the geodesic starting at $(0,\sigma)$ with horizontal direction is periodic, and such that $A_0 \in 4\pi \Z$ (it means that the geodesic wraps and even number of times around the vertical axis before passing again through the starting point for the first time), 
this is ultimately why we expect that $\frac{1}{3}$ is in general the optimal constant.

\quad

\myindent To conclude this section, we give two final remarks.

\myindent First, let us observe that, in the case of the simplest QCI geometry, which is the regular flat torus $\mathbb{T}^2$, one can go \textit{above} the constant $\frac{1}{3}$, and it is conjectured that the best constant is at least $1$ (see \eqref{conjGermain}). Now, for the flat torus $\mathbb{T}^2$, with coordinates $(x,y)$, one could similarly factorize the Laplacian as
\begin{equation}
    -\Delta = \left(i\partial_x\right)^2 + \left(i\partial_y \right)^2 =: Q_1^2 + Q_2^2,
\end{equation}
where there holds
\begin{equation}
    e^{2i\pi Q_i} = Id. 
\end{equation}

\myindent Hence, a similar kind of analysis would hold (we can't use the usual parametrices since $Q_1,Q_2$ are not elliptic, but actually everything is explicit in this case), but, ultimately, the major difference is that, in the case of the flat torus, there \textit{cannot} hold that two bicharacteristic curves of $Q_1,Q_2$ are tangent. On the contrary, we have seen that, for surfaces of revolution, it can happen that a bicharacteristic curve of $q_1$, a bicharacteristic curve of $q_2$ (i.e. a parallel), and even a \textit{periodic} geodesic can \textit{all} be tangent at the same point. Thus, we expect that it is ultimately this sort of degeneracy of the periodic bicharacteristics which would lead to different behaviors for surfaces of revolution. On a more abstract level, we thus conjecture that the particular geometry of \textit{unstable periodic geodesics} should play a key role in the estimates, especially for the question of optimal estimates.

\quad

\myindent Second, let us comment on the role of the equator, which we have excluded in the conjectures. On the one hand, we have no reason to expect that the optimal constant should be different near the equator. However, it seems more difficult to improve the bounds, since, as we have seen along Sections \ref{secHormPhase} to \ref{secAntipod}, the bounds deteriorate near the equator. Hence, we expect similar conjectures than the one we gave in this section to hold actually also near the equator, but we expect the proof to be even more technical.

\myindent On the other hand, let us note that the equator yields an upper bound on the optimal constant in Theorem \ref{mainthm}. Indeed, from the literature on quasimodes (see the references in the introduction) since the equator $\gamma_E$ is a \textit{stable} periodic geodesic of $\mathcal{S}$, for all $\eps > 0$, it should be possible to build a sequence of eigenfunctions $\phi^{k}$, corresponding to an unbounded sequence of eigenvalues $\lambda^{k}$, such that, on the equator,
\begin{equation}
    \|\phi^{k}\|_{L^{\infty}(\gamma_E)} \gtrsim \left(\lambda^{k}\right)^{\frac{1}{4} - \eps}.
\end{equation}

\myindent In particular, for any $\eps > 0$, the upper bound of Theorem \ref{mainthm} necessarily fails for $\delta \simeq \lambda^{-\frac{1}{2} - \eps}$. Hence, the optimal constant near the equator, and thus for $\mathcal{S}\backslash\{N,S\}$, is necessarily smaller than or equal to $\frac{1}{2}$.

\subsection{Removing hypotheses, and general QCI geometries}\label{subsec3Further}

\myindent In this section, we discuss the many simplifying hypotheses that we have introduced on the class of surfaces of revolution for which we are able to prove Theorem \ref{mainthm}, or Corollary \ref{lowerbounds}. In other words, we discuss how to enlarge the set $\mathfrak{S}$ on which those results still hold.

\myindent First, the non intersection of bicharacteristic curves Hypothesis \ref{nonintersectHyp} is not necessary. Indeed, if we only assume that $\mathcal{S}$ satisfies the twist Hypothesis \ref{twisthypothesis}, then one can actually prove that, even though there might be some exceptional \textit{intersections} of bicharacteristic curves, as discusses in Appendix \ref{AppendixB3}, there are no exceptional \textit{crossings} of bicharacteristic curves (see Definition \ref{crossingDef}), i.e., if two bicharacteristic curves of $q_1$ intersect outside of the antipodal refocalisation (see Lemma \ref{antipodrefocus}) at some point $y$, they do not pass through $y$ at the same time. Hence, since, ultimately, we only need to construct \textit{local} parametrices, we could still use the bicharacteristic length parametrix constructed in Paragraph \ref{subsubsec24Micro}, but we would only need to change the definition $\psi(x,y)$, which would no longer be a globally defined bicharacteristic length function (see Definition \ref{defpsi}), but rather locally defined. Indeed, the construction of Paragraph \ref{subsubsec24Micro} is actually local. Hence, the non intersection of bicharacteristic curves Hypothesis \ref{nonintersectHyp} is not necessary, and, with a little work to extend the technical results of Appendix \ref{AppendixD} to locally defined bicharacteristic length functions $\psi(x,y)$, we expect that one can extend the results to
\begin{equation}
    \mathfrak{S} := \{\mathcal{S} \qquad \text{simple, symmetric, and satisfying the twist Hypothesis \ref{twisthypothesis}} \}.
\end{equation}

\myindent Second, one could try to remove the twist Hypothesis \ref{twisthypothesis}, and to assume a less restrictive \textit{generic} hypothesis, in the sense of Colin de Verdière, see Remark \ref{defgeneric}. If, however, one was still to assume the non intersection of bicharacteristic curves Hypothesis \ref{nonintersectHyp}, most of the analysis would still apply. The main problem would be that, with the parameterization $u\mapsto h(u)$ of $\gamma$ defined by Lemma \ref{deflilgammas}, there could thus hold, for a finite number of values of $u$, $h''(u) = 0$. Looking at the quantitative estimates of Sections \ref{secHormPhase} to \ref{secAntipod}, most of our results would no longer hold, and we would need to replace them with versions where we take more derivatives of the different functions. For example, Lemma \ref{Psi1DnearEalpha} should be replaced by the equivalent lemma, but $\Psi^{1D}$ would only satisfy the property $(VdC)_4$ (see Definition \ref{VdCN}). More importantly, for the conjectured dominant term, i.e. for the estimate near the branching points $P_{\alpha}$ (see Section \ref{subsec4HormPhase} and Paragraph \ref{subsubsec25HormQuant}), the 1D remaining phase function $\Psi^{1D}$ could vanish up to the fifth derivative, and would only satisfy the property $(VdC)_6$. Hence, the number of derivatives that we need to compute would be greatly increased, and, thus, the numerology in powers of $\lambda,M,\delta$ which appear in the estimates would change. Overall, replacing the twist Hypothesis \ref{twisthypothesis} by a generic hypothesis (see Remark \ref{defgeneric}) seems possible, if we still assume for example a non intersection of bicharacteristic hypothesis, or even a non crossing of bicharacteristic hypothesis, but we believe that it would greatly increase the technicity of the proof without adding many new ideas. 

\myindent On the contrary, allowing for crossings of bicharacteristic curves seems like a new conceptual challenge. In particular, this means removing the \textit{symmetry} hypothesis (see Definition \ref{symmetrichypothesis}), or, ultimately, trying to deal with generic simple surfaces of revolution in the sense of Colin de Verdière, i.e. of Remark \ref{defgeneric}, without any other asumption. Indeed, the point is that we would no longer have explicit parametrices such as the one we built in Section \ref{subsec2Micro}. Instead, one should try to work directly with an abstract parametrix, and see if the quantitative analysis could still hold. 

\myindent Now, for this last problem, the difficulty is probably equivalent to the difficulty of dealing directly with a general  Quantum Completely Integrable (QCI) manifold. Hence, we present a possible approach for this general problem. Let $M$ be a compact smooth Riemannian manifold of dimension $d$, which is QCI, in the sense of \cite{de1980spectre}. We recall the assumptions made is this article. The setting is that there exists $d$ pseudo-differential operators of order $1$ with vanishing sub principal symbols $P_1 = \sqrt{-\Delta}, P_2,...,P_d$ such that, on the one hand,
\begin{equation}
    \forall i,j \qquad [P_i, P_j] = 0. 
\end{equation}

\myindent On the other hand, if $\mathcal{A}$ is the algebra generated by the $f(p_1,...,p_d)$, where $f \in \mathcal{C}^{\infty}(\R^d\backslash\{0\})$ is positively homogeneous, then, there exists generators $q_1,...,q_d$ of $\mathcal{A}$ such that

\myindent i. The Hamiltonian flow of the $q_j$ are $2\pi$-periodic.

\myindent ii. The fibers of $p=(p_1,...,p_d) : T^*M\backslash\{0\} \to \R^d \backslash\{0\}$ are connected. Moreover, if $z \to g\cdot z$ is the action of the torus $G := \mathbb{T}^d$ on $T^*M\backslash\{0\}$ defined by the Hamiltonian flows of the $q_j$, then there is a dense open set $\Omega \subset T^*M\backslash\{0\}$ such that, for all $z_0 \in \Omega$ $g\in G \mapsto g\cdot z_0$ is injective.

\myindent With these hypothesis, \cite{de1980spectre}[Theorem 3.1] yields that there exists $d$ commuting pseudodifferential operators of order 1 $Q_1,...,Q_d$, such that the principal symbol of $Q_j$ is $q_j$, and such that
\begin{equation}
    \forall j = 1,...,d, \qquad e^{2i\pi Q_j} = e^{2i\pi \mu_j} Id,
\end{equation}
where $\mu_j \in \frac{1}{4}\Z$ is a Maslov index (see \cite{souriau1975explicit,hormander2007symplectic}), and such that, finally,
\begin{equation}
    -\Delta= F(Q_1,...,Q_d),
\end{equation}
for
\begin{equation}
    F \sim F_2 + F_0 + F_{-1} +...
\end{equation}
a smooth classical symbol on $\R^d \backslash 0$, with vanishing 1-homogeneous component. Without loss of generality, we assume that $F_2$ doesn't vanish on $\R^d$.

\myindent In this setting, we could thus have a similar integral bound for the spectral projector $P_{\lambda,\delta}$ (see \eqref{defPlambdadelta}) than the one constructed in Section \ref{subsec1Micro}. Indeed, let $d\mu$ be the superficial measure on the hypersurface $\Sigma := \{F_2 = 1\}$. Let us assume that $\Sigma$ satisfies the following \textit{finite-type degeneracy} hypothesis.
\begin{definition}\label{finitetypedegen}
    Let $\Sigma$ be a smooth compact hypersurface of $\R^d$. We say that $\Sigma$ is degenerate at most of order $N\geq 2$ if, for any $q\in \Sigma$, there exists a unit speed curve $\gamma : (-\eps,\eps)\to \Sigma$ such that
    \begin{equation}
           \text{there exists}\ i \in \{2,...,N\} \ \text{such that}\ \gamma^{(i)}(q) \neq 0.
    \end{equation}
\end{definition}

\myindent Let $\rho$ be a smooth bump function. Let $d\mu_{\lambda} := \lambda^{d-1}d\mu(\lambda^{-1} \cdot)$ and $\rho_{\delta} := \delta^{-(d-1)} \rho(\delta^{-1} \cdot)$. Then, with similar notations than the one of Definition-Lemma \ref{deffld}, we define the smoothed projector
\begin{equation}
    P_{\lambda,\delta}^{\sharp} := (d\mu_{\lambda} \ast \rho_{\delta})(Q_1,...,Q_d).
\end{equation}

\myindent Let us define $\omega$ the set of those points $x\in M$ such that $T_x^{*}M \cap \Omega$ is of codimension at most 1. For an appropriate choice of $\rho$, for any compact subset $K$ of the set $\omega$, there holds
\begin{equation}
\begin{split}
    \|P_{\lambda,\delta}\|_{L^2(M) \to L^{\infty}(K)}^2 &\lesssim \left\|P_{\lambda,\delta}^{\sharp}\right\|_{L^2(M) \to L^{\infty}(K)} ^2\\ &\lesssim \lambda^{d-1} \delta \sup_{x\in K}  \int_{\R^d} \hat{d\mu}(\lambda(s_1,...,s_d))\hat{\rho}(\delta(s_1,...,s_d)) \left(e^{is_1 Q_1}...e^{is_d Q_d}\right)(x,x) ds_1...ds_d.
\end{split}
\end{equation}

\myindent Now, provided we are able to find a suitable parametrix for the Schwartz kernel 
\begin{equation}
    (s_1,...,s_d) \mapsto \left(e^{is_1 Q_1}...e^{is_d Q_d}\right)(x,x),
\end{equation}
we expect that our method should extend, and ultimately yield the following result.

\begin{conjecture}
    Let $M$ be a compact smooth manifold of dimension $d$, which is QCI in the sense introduced above. Assume that the hypersurface $\Sigma$ introduced above is degenerate at most of order $N \geq 2$, in the sense of Definition \ref{finitetypedegen}. Then, there exists a $\kappa = \kappa(N,d) > 0$ such that, if $K$ is any compact subset of the set $\omega$ introduced above, then
    \begin{equation}
        \forall \delta \geq \lambda^{-\kappa} \qquad \|P_{\lambda,\delta}\|_{L^2(M) \to L^{\infty}(K)} \lesssim \lambda^{\frac{d-1}{2}} \delta^{\frac{1}{2}}.
    \end{equation}
\end{conjecture}

\subsection{The pointwise Weyl law}\label{subsec4Further}

\myindent To conclude this article, we are convinced that the same methods yield the following quantitative improvement on the remainder of the pointwise Weyl law \eqref{locWeyl}, in the same setting of surfaces of revolution satisfying a number of hypotheses.

\begin{conjecture}
    Let $\mathcal{S}\in \mathfrak{S}$ (see Definition \ref{defmathfrakS}). There is a quantifiable constant $\kappa > 0$ such that, if $K_{\eps}$ is the set of those points at a distance at least $\eps$ from both poles of $\mathcal{S}$, then
    \begin{equation}
         \forall x \in K_{\eps}, \qquad N(\lambda,x) = (2\pi)^{-2} c_W(x) \lambda^2 + O_{\mathcal{S},\eps}(\lambda^{1 - \kappa}),
    \end{equation}
    where $N(\lambda,x)$ is defined by \eqref{defNlambdax}.
\end{conjecture}

\appendix

\section{Technical results}\label{AppendixA}

\subsection{Derivatives of $q_1$}\label{AppendixA1}

\myindent In this section, we give useful formulas for the derivatives of $q_1$.

\begin{lemma}[First order derivatives of $q_1$]\label{dersigq1lemma}
    For all $\xi = (\Theta,\Sigma) \in \R^2$ such that $|\xi|\sim 1$, there holds
    \begin{equation}\label{firstformulaq1}
        \partial_{\Sigma}q_1(\sigma,\Theta,\Sigma) \simeq \Sigma.
    \end{equation}
    \myindent In particular, $\partial_{\Sigma}q_1(\sigma,\xi) = 0$ if and only if $\Sigma = 0$, that is if and only if $\xi$ is horizontal.
    
    \myindent As a consequence, there holds
    \begin{equation}\label{secondformulaq1}
        \partial_{\Theta} q_1(\sigma,\Theta,0) = q_1(\sigma,1,0).
    \end{equation}
    
    \myindent Finally, for $|\Theta|$ bounded away from zero and from $+\infty$, there holds that
    \begin{equation}\label{thirdformula}
        \partial_{\sigma} q_1(\sigma,\Theta,\Sigma) \simeq \sigma.
    \end{equation}
\end{lemma}

\begin{proof}
We recall that, thanks to Lemma \ref{formulaofq1},
\begin{equation}
    q_1(\sigma,\Theta,\Sigma,) = G(p_1(\sigma,\Theta,\Sigma),\Theta),
\end{equation}
where $\partial_1G \neq 0$ thanks to Lemma \ref{dG1}. Hence,
\begin{equation}\label{partsig}
    \partial_{\Sigma}q_1 = \partial_1 G \times \partial_{\Sigma} p_1.
\end{equation}

\myindent Now, we compute, thanks to \eqref{normsurfrev},
\begin{equation}
    \partial_{\Sigma}p_1(\sigma,\Theta,\Sigma)= \frac{\Sigma}{p_1(\sigma,\Theta,\Sigma)}.
\end{equation}

\myindent This obviously concludes the proof of formula \eqref{firstformulaq1}. Formula \eqref{secondformulaq1} immediately follows, thanks to the homogeneity of $q_1$.

\myindent Regarding formula \eqref{thirdformula}, let us write, similarly,
\begin{equation}
    \partial_{\sigma} q_1 = \partial_1 G \times \partial_{\sigma} p_1.
\end{equation}

\myindent Now, in order to conclude, one needs only observe that, thanks to \eqref{normsurfrev},
\begin{equation}
    \partial_{\sigma} p_1(\sigma,\Theta,\Sigma) = \frac{-\frac{\Theta^2}{f(\sigma)^3}}{p_1(\sigma,\Theta,\Sigma)} \times f'(\sigma),
\end{equation}
and that, thanks to the simplicity of $\mathcal{S}$ (see Definition \ref{simplehypothesis}), there holds
\begin{equation}
    f'(\sigma) \simeq \sigma.
\end{equation}
\end{proof}

\begin{lemma}\label{LemmapartSigSigq_1}
    For all $\Theta \neq 0$, there holds
    \begin{equation}
        \partial_{\Sigma\Sigma} q_1(\sigma,\Theta,0) \simeq \Theta^{-1}.
    \end{equation}
\end{lemma}

\begin{proof}
    Let us again write, as in \eqref{partsig},
    \begin{equation}
        \partial_{\Sigma}q_1 = \partial_1 G \times \partial_{\Sigma} p_1.
    \end{equation}
    
    \myindent Hence,
    \begin{equation}
        \partial_{\Sigma\Sigma} q_1(\sigma,\Theta,0) = \partial_1 G \times \partial_{\Sigma \Sigma} p_1(\sigma,\Theta,0) + \partial_{11} G \times \left(\partial_{\Sigma} p_1(\sigma,\Theta,0)\right)^2.
    \end{equation}
    
    \myindent Now, we have already observed in the proof of Lemma \ref{dersigq1lemma} that
    \begin{equation}
        \partial_{\Sigma}p_1(\sigma,\Theta,0) = 0.
    \end{equation}
    
    \myindent Moreover, there holds
    \begin{equation}
        \partial_{\Sigma \Sigma}p_1(\sigma,\Theta,0) = \frac{1}{p_1(\sigma,\Theta,0)}.
    \end{equation}
    
    \myindent Along with Lemma \ref{dG1}, this concludes the proof.
\end{proof}

\subsection{Derivatives of $\varphi$}\label{AppendixA2}

\myindent In this section, we give useful formulas for the derivatives of $\varphi$, which is defined by Theorem \ref{Hormthm}.

\begin{lemma}\label{miracleequation}
    For all $x = (\theta,\sigma)$, there holds
    \begin{equation}
        \partial_{\theta\theta}\varphi(x,x,\pm 1,0) = 0.
    \end{equation}
    
    \myindent Moreover, if $x = (\theta,0)$ is on the equator, there holds
    \begin{equation}\label{secondmiracle}
        \partial_{\theta\sigma}\varphi(x,x,\pm 1,0) = 0.
    \end{equation}
\end{lemma}

\begin{proof}
    The problem is that we are not in position to compute $\nabla_x^2\varphi(x,x,0,\pm1)$ using only the expansion \eqref{asdevphi}. Instead, we really need to return to the definition of $\varphi$, that is we need to use the \textit{eikonal equation} \eqref{eikeq}, in the form
    \begin{equation}
        q_1(x,\nabla_x\varphi(x,y,\pm1, 0)) = q_1(y,\pm1,0).
    \end{equation}

    \myindent Differentiating in $x$, and using that, thanks to \eqref{asdevphi},
    \begin{equation}
        \nabla_x \varphi(x,x,\pm1,0) = \begin{pmatrix}
        \pm 1 \\ 0
        \end{pmatrix},
    \end{equation}
    we find that
    \begin{equation}\label{usefeq1}
        \begin{pmatrix}
            \partial_{\sigma} q_1(\sigma,\pm1,0) \\
            \partial_{\theta}q_1(\sigma,\pm1,0)
        \end{pmatrix} + \begin{pmatrix}
            \partial_{\sigma\sigma} \varphi(x,x,\pm 1,0) & \partial_{\sigma\theta} \varphi(x,x,\pm1,0)\\
            \partial_{\sigma\theta} \varphi(x,x,\pm1,0) & \partial_{\theta\theta} \varphi(x,x,\pm1,0)
        \end{pmatrix}\begin{pmatrix}
            \partial_{\Sigma}q_1(\sigma,\pm1,0)\\
            \partial_{\Theta}q_1(\sigma,\pm1,0)
        \end{pmatrix} = 0.
    \end{equation}

    \myindent Now, we have already proved that 
    \begin{equation}\label{usefeq2}
        \partial_{\Sigma}q_1(\sigma,\pm1,0) = 0
    \end{equation}
    in \eqref{dersigq1lemma}. Moreover, we know that, for $\Theta > 0$,
    \begin{equation}
        q_1(\sigma,\pm\Theta,0) = \Theta q_1(\sigma,\pm 1,0).
    \end{equation}
    
    \myindent Hence,
    \begin{equation}\label{usefeq3}
        \partial_{\Theta}q_1(\sigma,\pm1,0) = q_1(\sigma,\pm1,0) \neq 0,
    \end{equation}
    where we use the \textit{ellipticity} of $q_1$ given by \eqref{q1elliptic}. Finally, observe that
    \begin{equation}
        q_1(\sigma,\pm1,0) = G(p_1(\sigma,\pm1,0),\pm1)
    \end{equation}
    doesn't depend on $\theta$ since $p_1$ doesn't depend on $\theta$ (\eqref{defp1}). Hence,
    \begin{equation}\label{usefeq4}
        \partial_{\theta}q_1(\sigma,\pm1,0) = 0.
    \end{equation}

    \myindent Now, \eqref{usefeq1}, \eqref{usefeq2},\eqref{usefeq3},\eqref{usefeq4} ensure altogether that
    \begin{equation}\label{miracle3}
        \partial_{\theta\theta}\varphi(x,x,\pm1,0) = 0.
    \end{equation}
    
    \myindent Since, moreover, there holds that, thanks to \eqref{thirdformula},
    \begin{equation}
        \partial_{\sigma}q_1(0,\pm 1,0) = 0,
    \end{equation}
    \eqref{usefeq1}, \eqref{usefeq2},\eqref{usefeq3},\eqref{usefeq4} also imply \eqref{secondmiracle}.
\end{proof}

\begin{lemma}[The third $(\theta,w)$ derivatives of $\varphi$]\label{thirdthetawderivaticevarphi}
    There holds
    
    \begin{equation}\label{localthrdderivatives}
		\begin{cases}
			\partial_{\theta \theta \theta} \varphi((0,\sigma,),(0,\sigma),0) &= - \frac{\partial_{\Sigma \Sigma} q(\sigma,1,0) (\partial_{\sigma} q (\sigma, 1,0))^2}{(\partial_{\Theta} q(\sigma,1,0))^3} \\
			\partial_{\theta \theta w} \varphi((0,\sigma,),(0,\sigma),0) &= - \frac{\mathfrak{l}_{\sigma}}{2\pi}\frac{\partial_{\Sigma \Sigma} q(\sigma,1,0) \partial_{\sigma} q (\sigma, 0, 1)}{(\partial_{\Theta} q(\sigma,1,0))^2}  \\
			\partial_{\theta w w} \varphi((0,\sigma,),(0,\sigma),0) &= - \left(\frac{\mathfrak{l}_{\sigma}}{2\pi}\right)^2 \frac{\partial_{\Sigma \Sigma} q_1(\sigma, g_{\sigma}(0))}{\partial_{\Theta} q_1(\sigma, g_{\sigma}(0))}  \\
			\partial_{w w w} \varphi((0,\sigma,),(0,\sigma),0) &= 0.
		\end{cases}
	\end{equation}
\end{lemma}

\begin{proof}
    First, using \eqref{asdevphi}, it is obvious that, for all $w$, there holds
    \begin{equation}
        \partial_{www}\varphi((0,\sigma,),(0,\sigma),0) = 0.
    \end{equation}
    
    \quad
    
    \myindent Second, still using \eqref{asdevphi}, there holds
    \begin{equation}\label{localusefuleqpartthetaww1}
        \partial_{\theta w w} \varphi((0,\sigma,),(0,\sigma),0) = \scal{v_H}{g_{\sigma}''(0)},
    \end{equation}
    where $v_H = \begin{pmatrix}
        1 \\ 0
    \end{pmatrix}$.
    
    \myindent Since the curve $\mathcal{N}_{\sigma}$ defined by \eqref{defNsigma} is symmetric, $v_H$ and $g_{\sigma}''(0)$ are colinear (indeed, $g_{\sigma}''(0)$ is orthogonal to $\mathcal{N}_{\sigma}$ at $g_{\sigma}(0)$). Moreover, looking at the direction of $g_{\sigma}''(0)$, there holds 
    \begin{equation}\label{localusefuleqpartthetaww2}
        \scal{v_H}{g_{\sigma}''(0)} = -|g_{\sigma}''(0)|.
    \end{equation}
    
    \myindent Now, let us differentiate two times in $w$ the equation
    \begin{equation}
        q_1(\sigma,g_{\sigma}(w)) = 1,
    \end{equation}
    and evaluate at $w = 0$. This yields
    \begin{equation}\label{localusefuleqpartthetaww}
        \nabla_{\xi}q_1(\sigma,g_{\sigma}(0)) \cdot g_{\sigma}''(0) + g_{\sigma}'(0) \cdot \nabla_{\xi}^2 q_1(\sigma, g_{\sigma}(0)) \cdot g_{\sigma}'(0) = 0.
    \end{equation}
    
    \myindent Since moreover there holds in coordinates $\Theta,\Sigma$ that
    \begin{equation}
    \begin{split}
        &g_{\sigma}'(0) = \begin{pmatrix}
            0\\
            \frac{\mathfrak{l}_{\sigma}}{2\pi} 
        \end{pmatrix} \\ 
        &g_{\sigma}''(0) = \begin{pmatrix}
            -|g_{\sigma}''(0)| \\
            0
        \end{pmatrix},
    \end{split}
    \end{equation}
    equation \eqref{localusefuleqpartthetaww} reduces to 
    \begin{equation}
        -\partial_{\Theta}q_1(\sigma,g_{\sigma}(0)) |g_{\sigma}''(0)| + \partial_{\Sigma \Sigma}q_1 (\sigma, g_{\sigma}(0)) \left(\frac{\mathfrak{l}_{\sigma}}{2\pi}\right)^2 = 0,
    \end{equation}
    which obviously yields, using \eqref{localusefuleqpartthetaww1} and \eqref{localusefuleqpartthetaww2}, that
    \begin{equation}
        \begin{split}
            \partial_{\theta w w} \varphi((0,\sigma,),(0,\sigma),0) &= -|g_{\sigma}''(0)| \\
            &= -\left(\frac{\mathfrak{l}_{\sigma}}{2\pi}\right)^2 \frac{\partial_{\Sigma \Sigma} q_1(\sigma, g_{\sigma}(0))}{\partial_{\Theta} q_1(\sigma, g_{\sigma}(0))}.
        \end{split}
    \end{equation}
    
    \quad
    
    \myindent Third, in order to compute $\partial_{\theta\theta w} \varphi$, we can no longer rely on \eqref{asdevphi}. Instead, we have to resort to the eikonal equation \eqref{eikeq}, similarly to the proof of Lemma \ref{miracleequation}. Using \eqref{usefeq1}, we write
    \begin{equation}\label{localusefuleqpartthetathetaw}
        \partial_{\sigma \theta} \varphi((t,\sigma,),(0,\sigma),w) \partial_{\Sigma} q_1(\sigma, \nabla_x \varphi) + \partial_{\theta \theta} \varphi((t,\sigma,),(0,\sigma),g_{\sigma}(w)) \partial_{\Theta} q_1(\sigma,\nabla_x \varphi)) = 0.
    \end{equation}
    
    \myindent Now, we want to differentiate this equality in $w$, and to evaluate at $w = 0$. Since we already know, thanks to Lemmas \ref{dersigq1lemma} and \ref{miracleequation}, that
    \begin{equation}\label{localzeroderivativesppartthetathetaw}
    \begin{split}
        &\partial_{\Sigma}q_1(\sigma,g_{\sigma}(0)) = 0 \\
        &\partial_{\theta \theta} \varphi((0,\sigma,),(0,\sigma),g_{\sigma}(0)) = 0,
    \end{split}    
    \end{equation}
    we can already see, using moreover \eqref{asdevphi} that differentiating equation \eqref{localusefuleqpartthetathetaw} in $w$, and evaluating at $w = 0$, will yield
    \begin{equation}\label{localusefuleq1partthetathetaw}
        \partial_{\sigma \theta} \varphi((0,\sigma,),(0,\sigma),0) \left(\partial_{\Sigma} \nabla_{\xi} q_1 (\sigma, g_{\sigma}(0))\right) \cdot g_{\sigma}'(0) + \partial_{\theta \theta w} \varphi((0,\sigma,),(0,\sigma),0) \partial_{\Theta} q_1(\sigma, g_{\sigma}(0)) = 0.
    \end{equation}
    
    \myindent Now, using \eqref{usefeq1}, one finds that
    \begin{equation}\label{localusefuleq2partthetathetaw}
        \partial_{\sigma \theta} \varphi((0,\sigma,),(0,\sigma),0) = -\frac{\partial_{\sigma}q_1(\sigma,g_{\sigma}(0))}{\partial_{\Theta} q_1(\sigma, g_{\sigma}(0))}.
    \end{equation}
    
    \myindent Moreover, by definition, there holds in coordinates $(\Theta,\Sigma)$ that
    \begin{equation}
        g_{\sigma}'(0) = \begin{pmatrix}
        0\\
        \frac{\mathfrak{l}_{\sigma}}{2\pi} 
        \end{pmatrix}.
    \end{equation}
    
    \myindent Hence,
    \begin{equation}\label{localusefuleq3partthetathetaw}
        \left(\partial_{\Sigma} \nabla_{\xi} q_1 (\sigma, g_{\sigma}(0))\right) \cdot g_{\sigma}'(0) = \frac{\mathfrak{l}_{\sigma}}{2\pi}\partial_{\Sigma \Sigma} q_1(\sigma, g_{\sigma}(0)).
    \end{equation}
    
    \myindent Thus, equations \eqref{localusefuleq1partthetathetaw}, \eqref{localusefuleq2partthetathetaw}, and \eqref{localusefuleq3partthetathetaw}, yield together that
    \begin{equation}
        \partial_{\theta \theta w} \varphi((0,\sigma,),(0,\sigma), g_{\sigma}(0)) = -\frac{\mathfrak{l}_{\sigma}}{2\pi} \frac{\partial_{\Sigma \Sigma} q(\sigma,1,0) \partial_{\sigma} q (\sigma, 1,0)}{(\partial_{\Theta} q(\sigma,1,0))^2}.
    \end{equation}
    
    \quad
    
    \myindent Fourth, in order to compute $\partial_{\theta \theta \theta} \varphi$, we have no choice but to iterate the implicit argument of Lemma \ref{miracleequation}. We can actually start from \eqref{localusefuleqpartthetathetaw}. We want to differentiate this equation, and evaluate it at $t = 0$. Using again \eqref{localzeroderivativesppartthetathetaw}, and \eqref{asdevphi}, there holds
    \begin{equation}\label{localeqpartthetathetatheta}
        \partial_{\sigma \theta} \varphi \left(\partial_{\Sigma} \nabla_{\xi} q_1\right) \cdot \partial_{\theta} \nabla_x \varphi + \partial_{\theta \theta \theta} \varphi \partial_{\Theta} q_1 = 0,
    \end{equation}
    where all the quantities involving $\varphi$ are taken at $(0,\sigma,),(0,\sigma), g_{\sigma}(0)$, and all the quantities involving $q_1$ are taken at $(\sigma, g_{\sigma}(0))$. Since there holds moreover that, thanks to Lemma \ref{miracleequation},
    \begin{equation}
        \partial_{\theta} \nabla_x \varphi((0,\sigma,),(0,\sigma), g_{\sigma}(0)) = \begin{pmatrix}
            0\\
            \partial_{\sigma \theta} \varphi
        \end{pmatrix},
    \end{equation}
    the equation \eqref{localeqpartthetathetatheta} reduces to
    \begin{equation}
        (\partial_{\sigma \theta}\varphi)^2 \partial_{\Sigma \Sigma} q_1 + \partial_{\theta \theta \theta} \varphi \partial_{\Theta} q_1 = 0,
    \end{equation}
    where again, all the quantities involving $\varphi$ are taken at $(0,\sigma,),(0,\sigma), g_{\sigma}(0)$ and all the quantities involving $q_1$ are taken at $(\sigma, g_{\sigma}(0))$. One can then conclude, using \eqref{localusefuleq2partthetathetaw}, that 
    \begin{equation}
        \partial_{\theta \theta \theta} \varphi((0,\sigma,),(0,\sigma),0) = - \frac{\partial_{\Sigma \Sigma} q(\sigma,1,0) (\partial_{\sigma} q (\sigma, 1,0))^2}{(\partial_{\Theta} q(\sigma,1,0))^3}.
    \end{equation}
\end{proof}

\subsection{A technical lemma}\label{AppendixA3}
\myindent In this section, we state and prove the following simple lemma.

\begin{lemma}\label{easyoscintlemma}
    Let $I \subset \R$ be an interval, let $\phi \in \mathcal{C}^{\infty}(I)$, and let $a \in \mathcal{C}_0^{\infty}(I)$. Assume that
    \begin{equation}
        \text{the stationary points of}\ \phi\ \text{are isolated}.
    \end{equation}
    
    \myindent Define the oscillatory integral
    \begin{equation}
        \mathcal{I}(\lambda) := \int_I e^{i\lambda \phi(x)} a(x) dx.
    \end{equation}
    
    \myindent There holds
    \begin{equation}
        \mathcal{I}(\lambda) \to 0, \qquad \lambda \to \infty.
    \end{equation}
\end{lemma}

\begin{proof}
    Because the stationary points of $\phi$ are isolated, there is a finite number of stationary points in the support of $\phi$, say $x_1,...,x_N$. Let $\eps > 0$. Let $I_1,...,I_N$ be intervals of size $\eps$ centered at $x_1,...,x_N$ respectively. Then, on the one hand, for $i = 1,...,N$,
    \begin{equation}
        \left|\int_{I_i} e^{i\lambda \phi(x)} a(x) dx \right| \leq \|a\|_{L^{\infty}} \eps.
    \end{equation}
    
    \myindent On the other hand, since $\phi$ is not stationary on $J := I \backslash (I_1\cup... \cup I_N)$, there holds
    \begin{equation}
        \int_J e^{i\lambda \phi(x)} a(x) dx \to 0, \qquad \lambda \to \infty.
    \end{equation}
    
    \myindent Overall, there holds
    \begin{equation}
        \limsup_{\lambda \to \infty} |\mathcal{I}(\lambda)| \leq N \|a\|_{L^{\infty}} \eps.
    \end{equation}
        
    \myindent This concludes the proof since $\eps$ is arbitrary.
\end{proof}

\section{Proof of Proposition \ref{norefocusq1norm}}\label{AppendixB}

\subsection{First part of the proposition}\label{AppendixB1}

\myindent We first prove the first part of the proposition, that is the fact that, if 
\begin{equation}\label{localq1leq1}
    \{\xi \in \R^2 \qquad q_1(x,\xi) \leq 1\}
\end{equation}
is strictly convex, then $\mathcal{S}$ satisfie the non intersection of bicharacteristic curve Hypothesis \ref{nonintersectHyp}.

\begin{proof}
    Fix $x \in \mathcal{S}$, and fix two distinct directions $(x,\xi)\neq (x,\eta) \in S_x^{*}\mathcal{S}$. We want to prove that $t\mapsto \Phi^{q_1}_t(x,\xi)$ and $s \mapsto \Phi^{q_1}_s(x,\eta)$ do not intersect outside of $\{x,\bar{x}\}$. Assume that this is false, i.e. that there exists $y \notin \{x,\bar{x}\}$ at which those two curves intersect. Then, $x$ cannot be one of the Poles, since, thanks to Lemma \ref{explicitbicharac}, the bicharacteristic curves of $q_1$ passing through the Poles are exactly the meridians, which don't intersect outside of the Poles.

    \myindent Now, using the periodicity of the flow of $q_1$, and up to changing $\xi$ or $\eta$ or both into their opposite, we can assume without loss of generality that there exist $0 < s,t < \pi$ such that
    \begin{equation}
        P(\Phi_t^{q_1}(x,\xi)) = P(\Phi_s^{q_1}(x,\eta)) = y.
    \end{equation}

    \myindent From the explicit formula \eqref{explicitbicharac}, we can further reduce, up to exchanging $x$ and $\bar{x}$, and thanks to the symmetries of $\mathcal{S}$, to the following situation : on the one hand $(x,\xi)$ and $(x,\eta)$ both point towards the North, i.e. for some Clairaut integrals $I\neq J\in[-1,1]$ there holds
    \begin{equation}
        \begin{split}
            (x,\xi) &= \left(\theta, \sigma, I, \sqrt{1 - f(\sigma)^{-2}I^2}\right)\\
            (x,\eta) &= \left(\theta, \sigma, J, \sqrt{1 - f(\sigma)^{-2}J^2}\right).
        \end{split}
    \end{equation}
    
    \myindent On the other hand, this remains true along the trajectories between $x$ and $y$, i.e. the tangent to $u\mapsto \Phi_u^{q_1}(x,\xi)$ (resp $u \mapsto \Phi_u^{q_1}(x,\eta)$) points towards the North for $u\in [0,t]$ (resp $u \in [0,s]$). This is the situation of Figure \ref{inversionangle}. The important assumption is that the direction of travel (pointing the North or South Pole) don't strictly change on the portion of trajectory that we are interested in. We moreover assume that $y$ is the \textit{first} intersection of the bicharacteristic curves, i.e. they don't intersect on the portion $[x,y]$ of their trajectories.
    
    \begin{figure}[h]
\includegraphics[scale = 0.4]{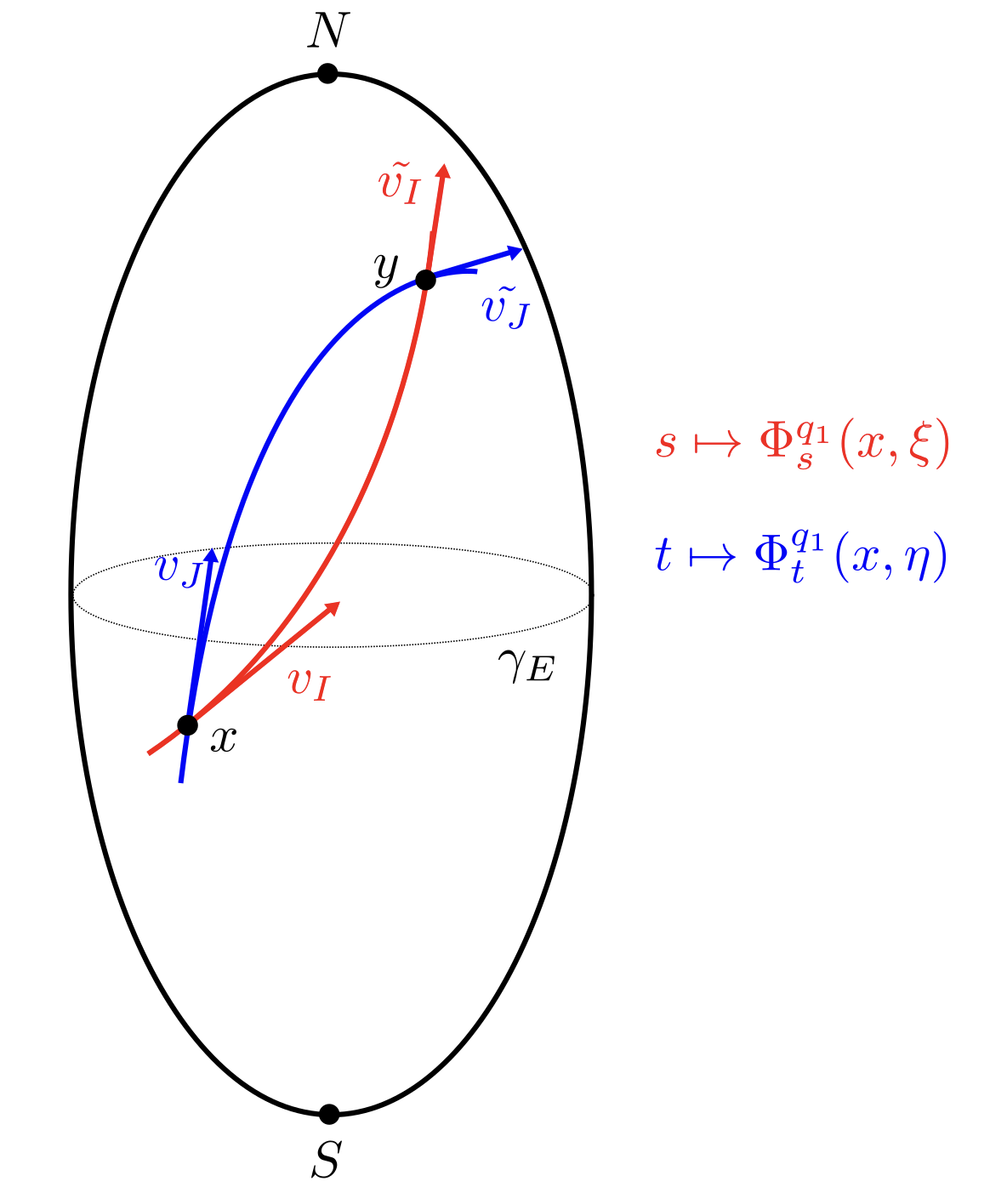}
\centering
\caption{Intersection of bicharacteristic curves outside of the antipodal refocalisation}
\label{inversionangle}
\end{figure}

    \myindent The point is that we can thus find the following explicit formula, using the conservation of $p_1$ and $p_2$ by the Hamiltonian flow, and the conservation of the sign of the $\Sigma$ coordinate, thanks to our reduction. Writing $y = (\tilde{\theta},\tilde{\sigma})$, there holds
    \begin{equation}
        \begin{split}
            \Phi_t^{q_1}(x,\xi) &= \left(\tilde{\theta},\tilde{\sigma},I,\sqrt{1 - f(\tilde{\sigma})^{-2}I^2}\right) \\
            \Phi_s^{q_1}(x,\eta) &= \left(\tilde{\theta},\tilde{\sigma},J,\sqrt{1 - f(\tilde{\sigma})^{-2}J^2}\right).
        \end{split}
    \end{equation}
    
    \myindent In particular, the angle between the directions of $\Phi_t^{q_1}(x,\xi)$ and $\Phi_s^{q_1}(x,\eta)$ is of the \textit{same sign} as the angle between $(x,\xi)$ and $(x,\eta)$. Precisely, there holds
    \begin{equation}\label{det1}
        \det(\begin{pmatrix}
            I \\
            \sqrt{1 - f(\sigma)^{-2}I^2}
        \end{pmatrix},\begin{pmatrix}
            J \\
            \sqrt{1 - f(\sigma)^{-2}J^2}
        \end{pmatrix}) \det(\begin{pmatrix}
            I \\
            \sqrt{1 - f(\tilde{\sigma})^{-2}I^2} 
        \end{pmatrix},\begin{pmatrix}
            J\\
            \sqrt{1 - f(\tilde{\sigma})^{-2}J^2} 
        \end{pmatrix}) > 0,
    \end{equation}
    or, more geometrically, seen as half-line of the upper half-plane of $\R^2$ in coordinates $(\Theta,\Sigma)$, the directions $\R_+^{*} \begin{pmatrix}
            I\\
            \sqrt{1 - f(\tilde{\sigma})^{-2}I^2}
        \end{pmatrix}$ and $\R_+^{*}\begin{pmatrix}
            J\\
            \sqrt{1 - f(\tilde{\sigma})^{-2}J^2}
        \end{pmatrix}$ are in the same order than the directions $\R_+^{*} \begin{pmatrix}
            I\\
            \sqrt{1 - f(\sigma)^{-2}I^2}
        \end{pmatrix}$ and $\R_+^{*}\begin{pmatrix}
            J\\
            \sqrt{1 - f(\sigma)^{-2}J^2}
        \end{pmatrix}$ in the sense of Figure \ref{conservationangle}.
        
        \begin{figure}[h]
\includegraphics[scale = 0.3]{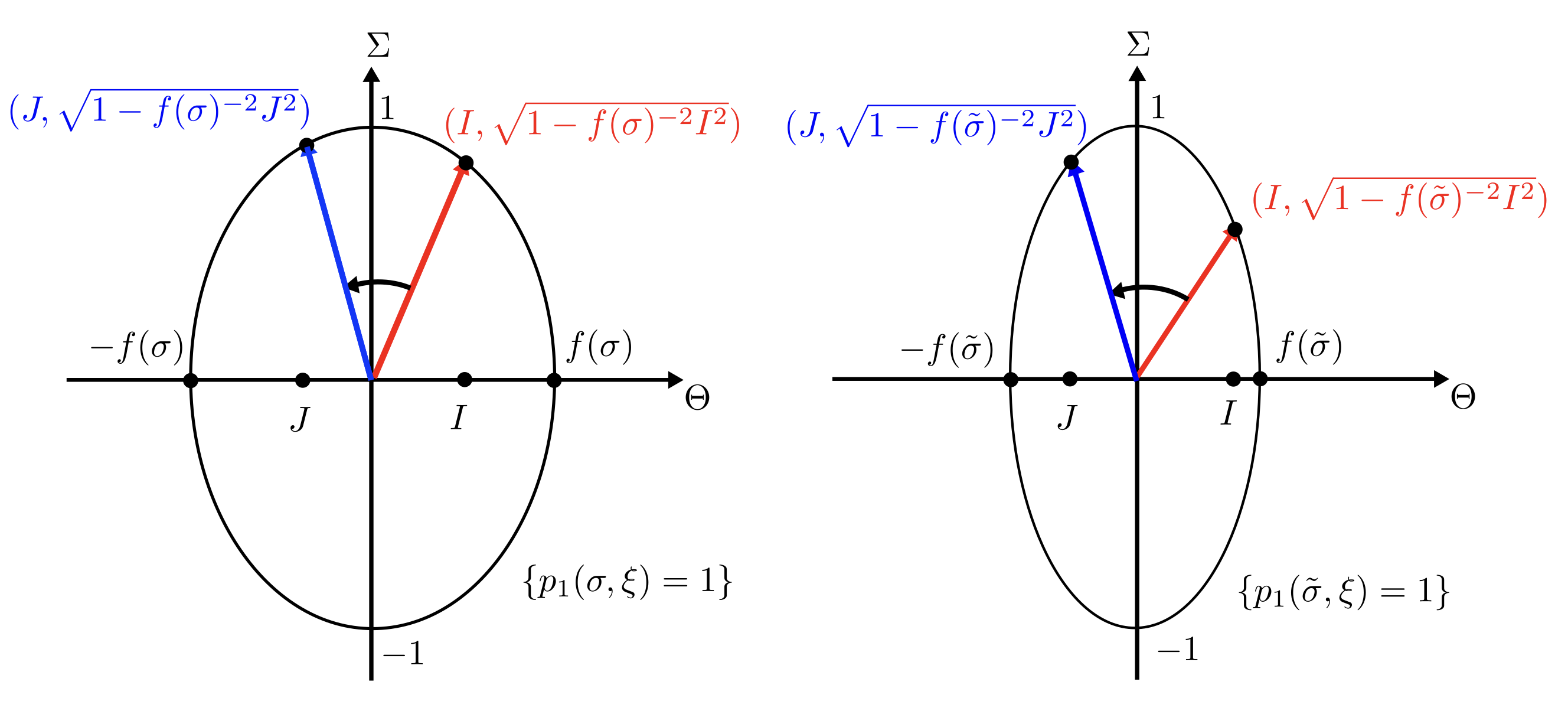}
\centering
\caption{The conservation of the sign of the angle along trajectories}
\label{conservationangle}
\end{figure}

        \myindent Now, in sharp contrast, there is an \textit{inversion} of the sign of the angle between the directions of the velocity vectors of the bicharacteristic curves between $x$ and $y$. Precisely, if we denote $v_I \in S_x\mathcal{S}$ (resp $v_J$) the velocity of the bicharacteristic curve $t \mapsto P(\Phi_t^{q_1}(x,\xi))$ (resp $s\mapsto P(\Phi_s^{q_1}(x,\eta))$) at $x$, and $\tilde{v}_I \in S_y\mathcal{S}$  (resp $\tilde{v}_J$) the velocity of the bicharacteristic curve $t \mapsto P(\Phi_t^{q_1}(x,\xi))$ (resp $s\mapsto P(\Phi_s^{q_1}(x,\eta))$) at $y$, there holds
        \begin{equation}\label{det2}
            \det(v_I, v_J) \det(\tilde{v}_I, \tilde{v}_J) < 0,
        \end{equation}
        which can be read from Figure \ref{inversionangle} (we use here the fact that $y$ is the first point of intersection along the bicharacteristic curves, and the fact that, thanks to our assumptions, the bicharacteristic curves point towards the North on all of the portion $[x,y]$).
        
        \quad
        
        \myindent Now, using the formula for the velocity vector along the Hamiltonian flow of a function on the cotangent bundle, projected on the manifold, i.e. \eqref{explicithamiltonianflow}, we can \textit{compute} the velocity vectors as follows
        \begin{equation}\label{exprofVIJ}
            \begin{split}
            &v_I = \nabla_{\xi} q_1\left(\sigma, I, \sqrt{1 - f(\sigma)^{-2} I^2}\right) \\
            &v_J = \nabla_{\xi} q_1\left(\sigma, J, \sqrt{1 - f(\sigma)^{-2} J^2}\right)\\
            &\tilde{v}_I = \nabla_{\xi} q_1\left(\tilde{\sigma}, I, \sqrt{1 - f(\tilde{\sigma})^{-2} I^2}\right)\\
            &\tilde{v}_J = \nabla_{\xi} q_1\left(\tilde{\sigma}, J, \sqrt{1 - f(\tilde{\sigma})^{-2} J^2}\right).
            \end{split}
        \end{equation}
        
        \myindent Hence, from \eqref{exprofVIJ} and \eqref{det2}, and comparing with \eqref{det1}, we can conclude to a contradiction provided we prove the fundamental fact that, for all $\sigma$ and for all $ \xi_1, \xi_2 = (\Theta,\Sigma),(\tilde{\Theta},\tilde{\Sigma})$ in the upper half-plane (i.e. $\Sigma,\tilde{\Sigma} \geq 0$) there holds that the angle between $\xi_1$ and $\xi_2$ has the \textit{same sign} than the angle between $\nabla_{\xi}q_1(\sigma,\xi_1)$ and $\nabla_{\xi} q_1(\sigma, \tilde{\xi_2})$ i.e.
        \begin{equation}\label{convexityofNsigma}
            \det(\xi_1,\xi_2) \det(\nabla_{\xi}q_1(\sigma,\xi_1), \nabla_{\xi} q_1(\sigma,\xi_2)) \geq 0.
        \end{equation}
        
        \myindent Now, for all $\xi \in \R^2$, the direction of $\nabla_{\xi} q_1(\sigma,\xi)$ is simply the normal direction to the curve $\mathcal{N}_{\sigma} := \{q_1 = 1\}$ at its intersection point with $\R_+^* \xi$, see Figure \ref{visualisationnablaxiq}.
         \begin{figure}[h]
\includegraphics[scale = 0.5]{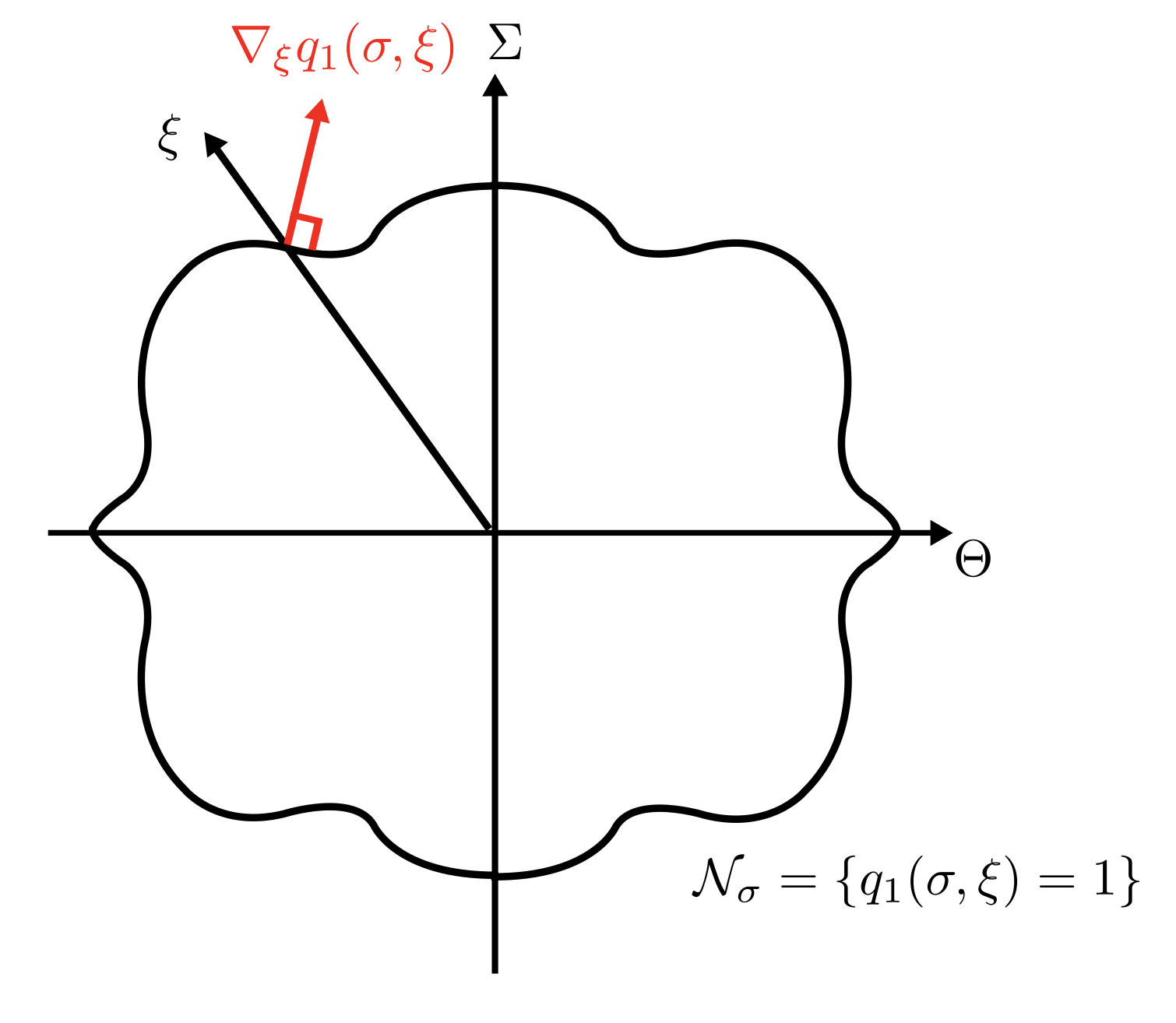}
\centering
\caption{Visualisation of $\nabla_{\xi}q_1(x,\xi)$}
\label{visualisationnablaxiq}
\end{figure}
        
        \myindent Hence, equation \eqref{convexityofNsigma} is simply a \textit{convexity} equation, i.e. it means that the domain $\{q_1 \leq 1\}$ is \textit{convex} (at least in the upper half-plane), see Figure \ref{convexityofNsigmaFigure}. This concludes the proof of the first part of Proposition \ref{norefocusq1norm}.
\end{proof}
        
\begin{figure}[h]
\includegraphics[scale = 0.5]{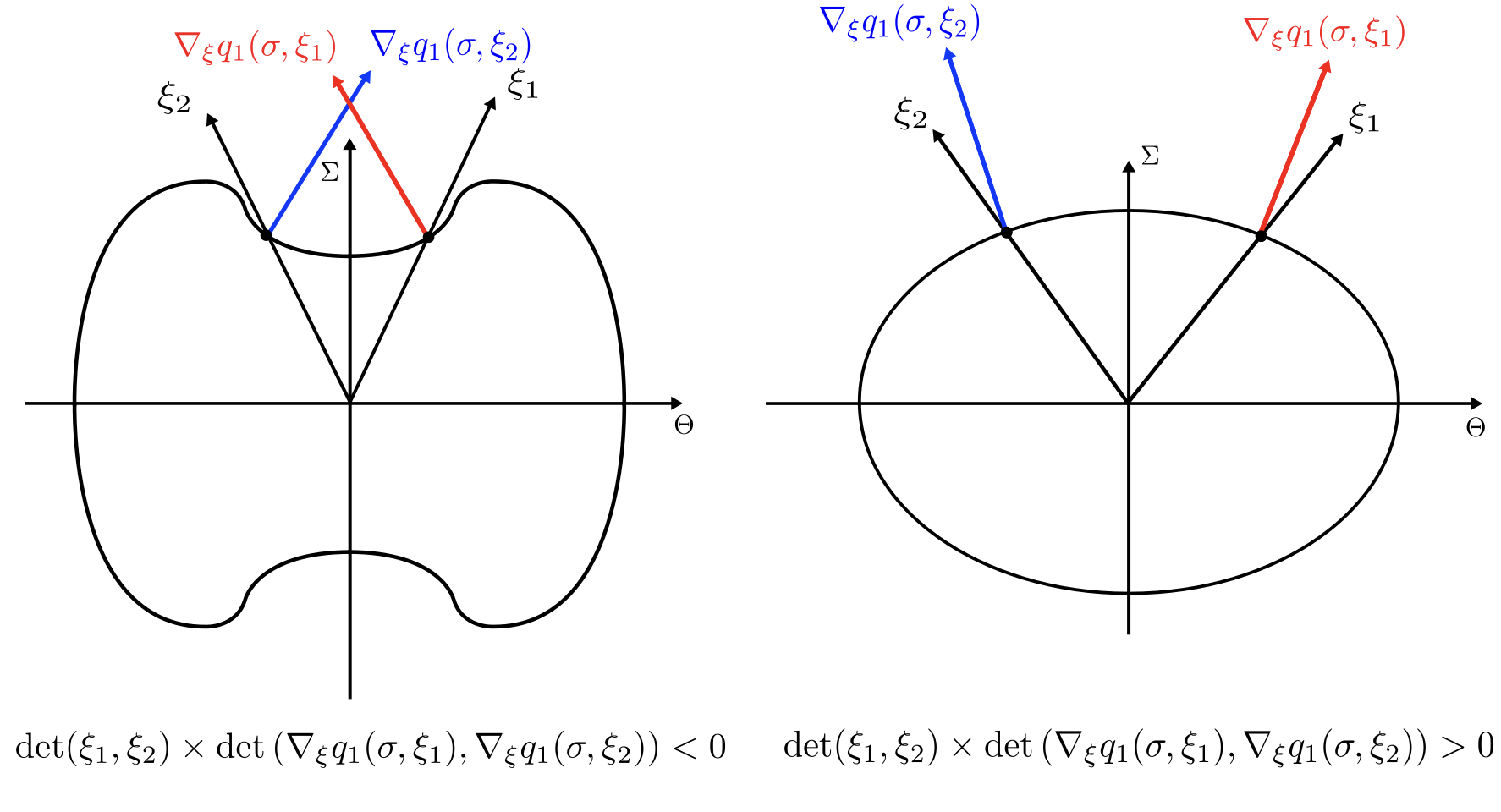}
\centering
\caption{The convexity of $\{q_1 \leq 1\}$}
\label{convexityofNsigmaFigure}
\end{figure}

\subsection{Second part of the proposition}\label{AppendixB2}

\myindent We now prove the second part of the proposition.

\begin{proof}
    Let $\mathcal{N}_{\sigma}^+$ be the intersection between the curve $\mathcal{N}_{\sigma}$ and the strict upper right quadrant of $\R^2$, i.e.
    \begin{equation}\label{defNsig+}
        \mathcal{N}^+_{\sigma} = \{(\Theta,\Sigma) \in (\R_+^*)^2 \qquad q_1(\sigma,\Theta,\Sigma) = 1\}.
    \end{equation}
    
    \myindent We first claim that this curve has an everywhere non zero curvature in the two cases of the proposition. In order to prove that fact, we introduce a parameterization of $\mathcal{N}_{\sigma}^+$ which may seem odd, but which will help greatly with the computations. 
    
    \myindent First, let $\mathcal{E}$ be the intersection of the curve $\{\xi \in \R^2, \ p_1(\sigma,\xi) = 1\}$ with the strict upper right quadrant of $\R^2$, i.e.
    \begin{equation}
        \mathcal{E} := \{\xi\in (\R_+^*)^2 \qquad p_1(\sigma,\xi) = 1\}.
    \end{equation}
    
    \myindent Then, since $q_1$ is a smooth non vanishing homogeneous function of degree 1, there is a one-to-one correspondence say
    \begin{equation}
        A\in \mathcal{E} \mapsto h(A) \in \mathcal{N}_{\sigma}^+,
    \end{equation}
    
    \myindent Hence, we need only parameterize the curve $\mathcal{E}$. Since we are in the strict upper right quadrant, we can always parameterize it by $p_2 \mapsto A(p_2)$ (recall that $p_2(\sigma,\Theta,\Sigma) = \Theta$), i.e.
    \begin{equation}
        \mathcal{E} = \left\{(p_2, \sqrt{1 - f^{-2}(\sigma)p_2^2}, \qquad p_2 \in (0, f(\sigma))\right\}.
    \end{equation}
    
    \myindent Ultimately, we have a one-to-one parameterization of the curve $\mathcal{N}_{\sigma}^+$ by $p_2 \in (0,f(\sigma))$, say $p_2 \mapsto h(A(p_2))$. Now, observe that, for all $A\in \mathcal{E}$,
    \begin{equation}
        \frac{\nabla q_1}{|\nabla q_1|} (h(A)) = \frac{\nabla q_1}{|\nabla q_1|} (A),
    \end{equation}
    since $\frac{\nabla q_1}{|\nabla q_1|}$ is a homogeneous function of degree 0 on $\R^2\backslash(0,0)$ and since $A$ and $h(A)$ are positively colinear by construction. Hence, in order to prove the claim, we need only prove that the function
    \begin{equation}
        p_2 \in (0, f(\sigma)) \mapsto \frac{\nabla q_1}{|\nabla q_1|} (A(p_2))
    \end{equation}
    has an everywhere nonzero derivative. 
    
    \myindent Now, the interest of this reduction is that, thanks to Lemma \ref{formulaofq1}, we have a simple expression of $\nabla q_1$ on $\mathcal{E}$. Indeed, observe that this lemma yields that
    \begin{equation}
        q_1(\sigma,\Theta,\Sigma) = G(p_1(\sigma,\Theta,\Sigma), \Theta) = p_1 g\left(\frac{p_2}{p_1}\right),
    \end{equation}
    where for $I \in (0,1)$, $g$ is the function introduced by Bleher in \cite{bleher1994distribution}[Equation (6.3)], i.e.
    \begin{equation}\label{defglocal}
        g(I) = I + \pi^{-1} \int_{\sigma_-(I)}^{\sigma_+(I)} \sqrt{1 - f(s)^{-2} I^2} ds.
    \end{equation}
    
    \myindent In particular, since $\partial_{\Sigma}p_2 = 0$ and $\partial_{\Theta} p_2 = 1$ we find that 
    \begin{equation}
        \nabla q_1 = \begin{pmatrix}
            \partial_{\Theta} q_1\\
            \partial_{\Sigma} q_1
        \end{pmatrix} =  \begin{pmatrix}
            \left[g\left(\frac{p_2}{p_1}\right) - \frac{p_2}{p_1} g'\left(\frac{p_2}{p_1}\right)\right] \partial_{\Theta} p_1 + g'\left(\frac{p_2}{p_1} \right) \\
            \left[g\left(\frac{p_2}{p_1}\right) - \frac{p_2}{p_1} g'\left(\frac{p_2}{p_1}\right)\right] \partial_{\Sigma} p_1 
        \end{pmatrix}.
    \end{equation}
    
    \myindent Now, on $\mathcal{E}$, there holds $p_1 = 1$. In particular, using \eqref{normsurfrev}, there holds 
    \begin{equation}
        \begin{pmatrix}
            \partial_{\Theta} p_1\\
            \partial_{\Sigma} p_1
        \end{pmatrix} = \begin{pmatrix}
            f(\sigma)^{-2} p_2 \\
            \sqrt{1 - f(\sigma)^{-2} p_2^2}
        \end{pmatrix}.
    \end{equation}
    
    \myindent Hence, on $\mathcal{E}$, there holds
    \begin{equation}\label{explicitnablaq1}
        \nabla q_1 (A(p_2)) = \begin{pmatrix}
            \left[ g(p_2) - p_2 g'(p_2)\right] f(\sigma)^{-2} p_2 + g'(p_2)  \\
            [g(p_2) - p_2 g'(p_2)] \sqrt{1 - f(\sigma)^{-2} p_2^2} 
        \end{pmatrix}.
    \end{equation}
    
    \myindent Thus, using the local notations $f := f(\sigma)$ and $p := p_2$, what we want is to prove that the function
    \begin{equation}
        p\in(0,f) \mapsto \frac{\nabla q_1}{|\nabla q_1|}(A(p)) = \frac{1}{|\nabla q_1|(A(p))}\begin{pmatrix}
            [g(p) - p g'(p)] f^{-2} p + g'(p) \\
            [g(p) - pg'(p)] \sqrt{1 - f^{-2} p^2} 
        \end{pmatrix}
    \end{equation}
    has a nonzero derivative. Now, since it is a function which takes value in the circle, this is equivalent to proving that, for all $p \in (0,f)$,
    \begin{equation}
        \det(\nabla q_1(A(p)), \frac{d}{d p} \left(\nabla q_1 (A(p))\right)) \neq 0.
    \end{equation}
    
    \myindent Using the explicit expression \eqref{explicitnablaq1}, this is equivalent to proving that 
    \begin{equation}
        \det(\begin{pmatrix}
            [g - pg'] f^{-2} p + g'  \\
            [g - pg']\sqrt{1 - f^{-2} p^2} 
        \end{pmatrix}, \begin{pmatrix}
            - p g'' f^{-2} p + [g - pg'] f^{-2} + g'' \\
            -pg''\sqrt{1 - f^{-2}p^2} -\frac{[g-pg']f^{-2} p}{\sqrt{1 - f^{-2} p^2}}
        \end{pmatrix})
    \end{equation}
    doesn't vanish, where we don't note the variable $p$ in the argument of the functions $g,g',g''$. Multiplying the second line by $\sqrt{1 - f^{-2} p^2}$, this is again equivalent to proving that, for $p\in(0,f)$
    \begin{equation}
        \begin{vmatrix}
            [g - pg'] f^{-2} p + g'&   g''(1 - f^{-2} p^2) + [g - pg'] f^{-2} \\
            [g-pg'](1 - f^{-2}p^2) & -pg''(1 - f^{-2}p^2) - [g-pg']f^{-2}p
        \end{vmatrix} \neq 0,
    \end{equation}
    or, after developing the determinant, we need only prove that, for $p\in(0,f)$,
    \begin{equation}
        -g\left(g''(1 - f^{-2} p^2) + [g - pg'] f^{-2} \right) \neq 0.
    \end{equation}
    
    \myindent From \eqref{defglocal}, we see that $g$ is positive on $(0,f)$. Hence, we need to prove that, for $p\in(0,f)$,
    \begin{equation}\label{ultimatereductionlocal}
        g''(1 - f^{-2} p^2) + [g - pg'] f^{-2} \neq 0.
    \end{equation}
    
    \myindent Now, on the one hand, we recall that, thanks to \eqref{explicitnablaq1}, 
    \begin{equation}
        [g - pg'] = \frac{1}{\sqrt{1 - f^{-2} p^2}} \partial_{\Sigma}q_1.
    \end{equation}
    
    \myindent Moreover, as we have already used many times,
    \begin{equation}
        \partial_{\Sigma}q_1 = \partial_1 G \partial_{\Sigma} p_1,
    \end{equation}
    where $\partial_1 G > 0$ (see Lemma \ref{dG1}), and $\partial_{\Sigma} p_1$ is positive since we are on the upper right quadrant (see \eqref{normsurfrev}). Overall, there holds that
    \begin{equation}
        [g - pg'] f^{-2} > 0.
    \end{equation}
    
    \myindent In particular, coming back to equation \eqref{ultimatereductionlocal}, we already see that, if
    \begin{equation}
        \forall p \in (0,f), \qquad g''(p) \geq 0,
    \end{equation}
    then we have proved the claim. Now, we recall that, thanks to Lemma \ref{g'eq-omega}, there holds
    \begin{equation}
        g'(I) = -\omega(I),
    \end{equation}
    hence,
    \begin{equation}
        g''(I) = - \omega'(I),
    \end{equation}
    which is thus nonnegative if \eqref{caseomega'neg} is satisfied. 
    
    \quad
    
    \myindent Now, since there holds, in the case of the round sphere,
    \begin{equation}
        g(I) \equiv 1 \qquad g'(I) \equiv 0 \qquad g''(I) \equiv 0,
    \end{equation}
    then we can expect that \eqref{ultimatereductionlocal} still holds if $\omega'$ is not nonnegative but $\mathcal{S}$ is close to the sphere in the sense that $\omega'$ is not too large. Indeed, if $\omega'\geq 0$, there holds first 
    \begin{equation}
        g'(p) = -\omega(p) = -\int_0^p \omega'(I) dI \leq 0.
    \end{equation}
    
    \myindent Hence, for all $p\in(0,f)$, there holds
    \begin{equation}
        g''(1 - f^{-2}p^2) + [g- pg'] f^{-2} \geq g(p) - \omega'(p) > g(p) - 1.
    \end{equation}
    
    \myindent Now, observe that, from formula \eqref{defglocal}, there holds
    \begin{equation}
        g(1) \geq 1.
    \end{equation}
    
    \myindent Since, moreover, $g'\leq 0$, we thus find that, for all $p \in (0,f)$,
    \begin{equation}
        g(p) \geq g(1) \geq 1.
    \end{equation}
    
    \myindent This ensures ultimately that \eqref{ultimatereductionlocal} holds.
    
    \quad
    
    \myindent In order to conclude the proof, we need only prove that, if the curve $\mathcal{N}_{\sigma}^+$ has an everywhere non zero curvature, then $\mathcal{D} := \{q_1(\sigma,\xi) \leq 1\}$ is strictly convex. Now, since $q_1$ is even in both $\Theta$ and $\Sigma$, thanks to Lemma \ref{formulaofq1} and to \eqref{normsurfrev}, the boundary of $\mathcal{D}$ on each of the four quadrants is obtained from $\mathcal{N}_{\sigma}^+$ by a rotation. Since $\mathcal{N}_{\sigma}$ is smooth, there holds necessarily that the curvature of $\mathcal{N}_{\sigma}^+$ is in fact everywhere positive, and thus that the curvature of $\mathcal{N}_{\sigma}$ is everywhere positive (except maybe at the four points of intersection with the axis), i.e. that $\mathcal{D}$ is strictly convex.
\end{proof}

\subsection{Discussion of the hypotheses}\label{AppendixB3}

\myindent In this section, we briefly discuss Proposition \ref{norefocusq1norm}, in order to justify its hypotheses.

\myindent First, if $\omega'(I)$ is not bounded, it is possible that the set
\begin{equation}
    \mathcal{D} := \{\xi \in \R^2 \qquad q_1(x,\xi) \leq 1\}
\end{equation}
is \textit{not} convex, in particular when $x$ is near the equator. Indeed, in the case of a very thin oblong surface of revolution, i.e., with the Definition \ref{defellipsoids}, in the case where $\mathcal{S} = \mathcal{E}(1,b)$ with $b  \gg 1$, one can prove rigorously that, if $x$ is on the equator, then the set $\mathcal{D}$ looks like the left figure in Figure \ref{convexityofNsigmaFigure} provided $b$ is large enough.

\myindent Now, the geometric meaning of this fact is that the bicharacteristic curves of $q_1$ starting at $x_0 \in \gamma_E$ with near vertical initial direction $(x_0,\xi) \in T_{x_0}^*\mathcal{S}$ will at first turn "in the wrong direction". Precisely, let us choose a direction $\xi = (\Theta,\Sigma)$ very close to the vertical direction, i.e. $0 < \Theta << \Sigma$. Then, locally, while the geodesic $t \mapsto P(\Phi_t(x_0,\xi))$ starts by turning to the right, (see the black curve in Figure \ref{intersectoblongFigure}), i.e. $\dot{\theta} > 0$, the bicharacteristic curve $t\mapsto P(\Phi^{q_1}_t(x_0,\xi))$ starts by turning to the left, (see the red curve in Figure \ref{intersectoblongFigure}), i.e. $\dot{\theta} < 0$. This can be computed by the explicit formula for the Hamiltonian flow \eqref{explicithamiltonianflow}.

\begin{figure}[h]
\includegraphics[scale = 0.5]{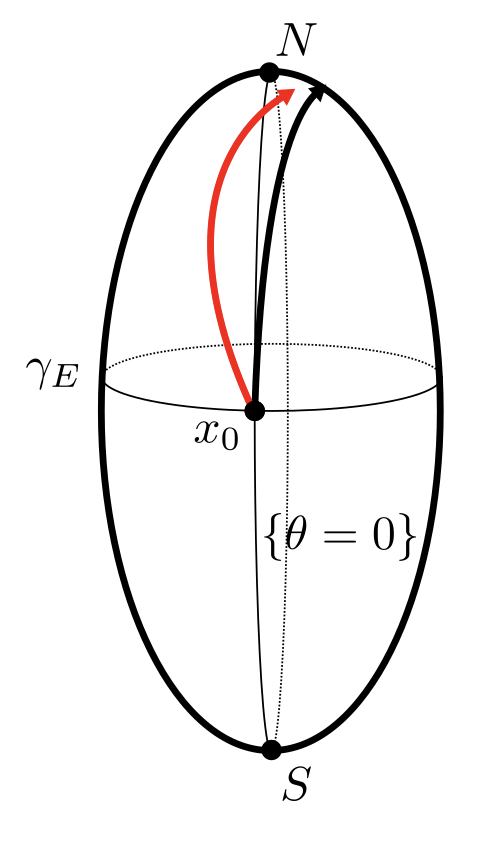}
\centering
\caption{Intersection of bicharacteristic curves for an oblong ellipsoid of revolution}
\label{intersectoblongFigure}
\end{figure}

\myindent However, after some time, when the bicharacteristic curve $t\mapsto P(\Phi_t^{q_1}(x_0,\xi))$ reaches high latitudes (i.e. $\sigma$ is close to $\frac{L}{2})$), its velocity vector rotates, and it eventually turns to the right, as depicted in Figure \ref{intersectoblongFigure} (red curve). In particular, as one can see on this figure, the bicharacteristic curves necessarily intersects the meridian passing through $x_0$ outside of the antipodal point $\bar{x_0}$. Hence, if the set $\mathcal{D}$ is \textit{not} convex, there are always some exceptional intersections of bicharacteristic curves.

\myindent This phenomenon can be interpreted by the explicit equation for the bicharacteristics of $q_1$ given by Lemma \ref{explicitbicharac}. Indeed, find $0 < I \ll 1$ the Clairaut integral of $(x_0,\xi)$. Then, the bicharacteristic $t \mapsto \Phi_t^{q_1}(x_0,\xi)$ is given by
\begin{equation}
    t\mapsto \Phi_{\frac{t}{2\pi} \tau(I)} \circ \Phi^{p_2}_{-t\omega(I)}(x_0,\xi).
\end{equation}

\myindent Now, since the geodesic $t\mapsto \Phi_{\frac{t}{2\pi}\tau(I)}(x_0,\xi)$ is very close to a meridian, it starts by turning very slowly to the right, i.e. $\dot{\theta}$ is initially very small, as seen on Figure \ref{intersectoblongFigure} (black curve). On the contrary, the influence of the rotation of angle $-t\omega(I)$ takes over since $\omega(I)$ is very large. Hence, even if the total phase shift along the bicharacteristic curve $t\mapsto P(\Phi_t^{q_1}(x_0,\xi))$ is exactly $\pi > 0$ after a time $t = \pi$, the phase shift starts by being nonpositive for $0 < t \ll 1$. In particular, there is necessarily a time $0 < t < \pi$ at which the phase shift is exactly $0$, which corresponds to an intersection with the meridian, as seen on Figure \ref{intersectoblongFigure}.

\section{Proof of Theorem \ref{mixedVdCABZ}}\label{AppendixC}

\myindent Before turning to the proof of the Theorem, let us insist that, especially for Appendix \ref{AppendixC4}, we use very similar methods than in \cite{alazard2017stationary}, which was our main source of inspiration.

\subsection{First reductions and the method of integration by parts}\label{AppendixC1}
\myindent First, we start by writing
\begin{equation}
\begin{split}
    \mathcal{I}(\lambda) = \int_I dx e^{i\lambda \phi^{1D}(x)} \times \left(\int_U dy e^{i\lambda (\phi_x(y) - \phi^{1D}(x))} \zeta(x,y) a(x,y)\right) \\
    =: \int_I dx e^{i\lambda \phi^{1D}(x)} J(\lambda,x)
\end{split}
\end{equation}

\myindent Now, since $\phi^{1D}$ satisfies the property $(VdC)_p$ with constants $C,c$, we can apply Proposition \ref{usefulVdC}. We define
\begin{equation}\label{defAB}
    \begin{split}
        &\sup_x |J'(\lambda,x)| =: A(\lambda) \\
        &\sup_x |J(\lambda,x)| =: B(\lambda).
    \end{split}
\end{equation}

\myindent Then, there holds thanks to Proposition \ref{usefulVdC} that
\begin{equation}\label{boundIbyAB}
    \mathcal{I}(\lambda) \lesssim (c\lambda)^{-\frac{1}{p}}\left(1 + |I| \frac{C}{c}\right)\left(|I|A(\lambda) + B(\lambda)\right).
\end{equation}

\myindent Hence, all amounts to bounding $A(\lambda)$ and $B(\lambda)$. We start with reducing $A(\lambda)$ to a similar form than $B(\lambda)$. Indeed, the problem is that, upon computing $J'(\lambda,x)$, using that, by definition, 
\begin{equation}\label{nablaphixyxeqzero}
        \nabla_y \phi(x,y(x)) = 0,
\end{equation}
we find that
\begin{equation}
    \begin{split}
		J'(\lambda,x) &= \int dy e^{i\lambda (\phi(x,y) - \phi(x,y(x)))} \zeta(x,y) a(x,y) \times i\lambda\left(\partial_x \phi(x,y) - \partial_x \phi(x,y(x)) \right)\\
    &+ \int dy e^{i\lambda (\phi(x,y) - \phi(x,y(x)))} \partial_x\left(\zeta(x,y) a(x,y)\right)\\
    &= J_1(\lambda,x) + J_2(\lambda,x).
    \end{split}
\end{equation}

\myindent Hence, there is a priori an undesirable power $\lambda$ in $J_1$. Now, there holds that
\begin{equation}
    \left(\zeta(x,y)a(x,y)(\partial_x \phi(x,y) - \partial_x \phi(x,y(x)))\right)(y = y(x)) = 0,
\end{equation}
so the mains order term in the usual stationary phase Lemma \ref{hormstatphase} yields that
\begin{equation}
    |J_1(\lambda,x)| = O_{\phi,\zeta,a}(\lambda^{-\frac{d}{2}}).
\end{equation}

\myindent However, the standard methods for obtaining quantitative upper bounds for an oscillatory integral of the form
\begin{equation}
    \int dy e^{i\lambda \Phi(y)} b(y) dy,
\end{equation}
where $\Phi$ has a nondegenerate stationary point $y_0$ and
\begin{equation}
    b(y) = O(|y - y_0|),
\end{equation}
usually yield only a bound of the form
\begin{equation}
    O(\lambda^{-\frac{d}{2} - \frac{1}{2}}),
\end{equation}
see for example \cite{stein1993harmonic}[Equation (23), Chapter VIII].

\myindent The idea, to recover a full $\lambda^{-1}$, is to integrate by parts in $y$. Indeed, assume for a moment that there exists a smooth vector field $v(x,y)$ such that
\begin{equation}\label{desiredformdiffpartxphi}
		\partial_x \phi(x,y) - \partial_x \phi(x,y(x)) = v(x,y) \cdot \nabla_y \phi(x,y).
	\end{equation}
	
\myindent Then, we may absorb the $\lambda$ by integrating by parts : there holds
	\begin{equation}\label{intbyparts}
		\begin{split}
			J_1(\lambda,x) &= \int dy  \left(v\cdot \nabla_y\right)\left(e^{i\lambda(\phi(x,y) - \phi(x,y(x)))}\right) \zeta(x,y) a(x,y) \\
			&= - \int dy e^{i\lambda(\phi(x,y) - \phi(x,y(x)))} \nabla_y \cdot \left(\zeta(x,y) a(x,y) v(x,y) \right) ,
		\end{split}
	\end{equation}
    to which we can apply usual stationary phase analysis to find a quantitative $O(\lambda^{-\frac{d}{2}})$ bound.
    
    \subsection{The vector field of the integration by parts $v(x,y)$ }\label{AppendixC2}

    \myindent Now, we turn to the construction of such a vector field $v$. Let us first write
	\begin{equation}\label{firstformdiffpartxphi}
		\begin{split}
			\partial_x \phi(x,y) - \partial_x \phi(x,y(x)) &= \left(\int_0^1 \left(\partial_x \nabla_y \phi\right)(x,y(x) + t(y - y(x)))dt\right) \cdot (y-y(x)) \\
			&=: w(x,y) \cdot (y-y(x)).
		\end{split}
	\end{equation}
	
	\myindent We observe moreover that, using \eqref{nablaphixyxeqzero}, there holds
	\begin{equation}\label{definitionofM}
		\begin{split}
			\nabla_y \phi(x,y) &= \left(\int_0^1 \left(\nabla_y ^2 \phi\right)(x,y(x) + t(y - y(x)))dt\right)  (y-y(x)) \\
			&=: M(x,y) (y - y(x)).
		\end{split}
	\end{equation}
	
	\myindent In particular, with the Definition \ref{defHandN},
	\begin{equation}
		M(x,y(x)) = \nabla_y^2\phi (x,y(x)) = H(x).
	\end{equation}
	
	\myindent As a consequence, for $y$ close to $y(x)$, $M(x,y)$ is invertible. Indeed, we recall the following standard lemma.
	
	\begin{lemma}[Neumann series]\label{inversematrixlemma}
    Let $A \in GL_d(\C)$ and $B \in \mathcal{M}_d(\C)$. Assume that
    \begin{equation}
        \|A - B\| \leq \frac{1}{2}\|A^{-1}\|^{-1}
    \end{equation}
    for any operator norm $\|\cdot\|$. Then there holds
    \begin{equation}
        \begin{split}
            B &\in GL_d(\C) \\
            \text{and} \qquad \|B^{-1}\| &\lesssim \|A^{-1}\|.
        \end{split}
    \end{equation}
\end{lemma}
	
	\myindent Now, with the Notation \ref{defM1+d}, observe that
    \begin{equation}
        \|M(x,y) - H(x)\| \leq \mathcal{M}^{(y)}_{3,3}(\phi)|y-y(x)|.
    \end{equation}

    \myindent Hence, if 
    \begin{equation}\label{yclosetoyofx}
        |y - y(x)| \leq \frac{1}{2} \left(\mathcal{M}_{3,3}^{(y)}(\phi)\|H(x)^{-1}\|\right)^{-1},
    \end{equation}
    then there holds, using Lemma \ref{inversematrixlemma},
    \begin{equation}\label{estnormM-1}
            M(x,y) \in GL_d(\R) \qquad
            \text{and} \qquad \|M(x,y)^{-1}\| \lesssim \|H(x)^{-1}\| \lesssim N(\phi).
    \end{equation}

    \myindent Now, condition \eqref{yclosetoyofx} is \textit{always} satisfied for $(x,y) \in supp(\zeta)$ by definition of $\zeta$. Hence, on $supp(\zeta)$, we may write
    \begin{equation}
        y-y(x) = M(x,y)^{-1}\nabla_y\phi(x,y),
    \end{equation}
    and hence, using \eqref{firstformdiffpartxphi}, and the symmetry of $M^{-1}(x,y)$, there holds
	\begin{equation}
		\begin{split}
			\partial_x \phi(x,y) - \partial_x \phi(x,y(x)) &= w(x,y) \cdot \left(M^{-1}(x,y) \nabla_y\phi(x,y)\right) \\
			&= \left(M^{-1}(x,y) w(x,y)\right) \cdot \nabla_y \phi(x,y).
		\end{split}
	\end{equation}
	
	\myindent Hence, \eqref{desiredformdiffpartxphi} holds with
	\begin{equation}\label{defvectorfield}
		v(x,y) := M^{-1}(x,y) w(x,y).
	\end{equation}
	
	\quad

    \myindent We now give bounds on the derivatives of $v(x,y)$. We find that
	\begin{equation}
		\mathcal{M}^{(y)}_{0,\ell}(v) \lesssim \sum_{m_1 + m_2 = \ell} \mathcal{M}^{(y)}_{0,m_1}(w) \mathcal{M}^{(y)}_{0,m_2}(M^{-1}).
	\end{equation}
	
	\myindent Now, on the one hand, it is obvious using \eqref{firstformdiffpartxphi} that
	\begin{equation}
		\mathcal{M}^{(y)}_{0,m}(w) \lesssim \mathcal{M}^{(y)}_{1,m+1}(\partial_x \phi).
	\end{equation}
	
	\myindent On the other hand, a quick induction argument yields that the following (rough !) bound holds
	\begin{equation}
		\mathcal{M}^{(y)}_{0,m}(M^{-1}) \lesssim \|M(x,y)^{-1}\|^{m+1} \left(\mathcal{M}_{0,m}^{(y)}(M)\right)^m.
	\end{equation}
	
	\myindent Using the definition of $M$ \eqref{definitionofM}, there holds
    \begin{equation}\label{estNmM}
		\mathcal{M}_{0,m}^{(y)}(M) \lesssim \mathcal{M}_{2,m+2}^{(y)}(\phi).
	\end{equation}
	
	\myindent Hence, using \eqref{estNmM} and \eqref{estnormM-1}, there holds
	\begin{equation}
		\mathcal{M}^{(y)}_{0,m}(M^{-1}) \lesssim \mathcal{N}(\phi)^{m+1} \left(\mathcal{M}^{(y)}_{2,m+2}(\phi)\right)^m.
	\end{equation}
	
	\myindent Thus, we ultimately find that
	\begin{equation}\label{controlNlv}
    \begin{split}
		\mathcal{M}^{(y)}_{0,\ell}(v) &\lesssim \sum_{m_1 + m_2 = \ell} \mathcal{N}(\phi)^{m_1+1}\left(\mathcal{M}^{(y)}_{2,m_1+2}(\phi)\right)^{m_1} \mathcal{M}^{(y)}_{1,m_2 + 1}(\partial_x \phi) \\
        &\lesssim \mathcal{M}_{1,\ell+1}^{(y)}(\partial_x \phi) \left(\mathcal{N}(\phi)\right)^{\ell+1} \left(\mathcal{M}_{2,l+2}^{(y)}(\phi)\right)^{\ell}.
    \end{split}
	\end{equation}

    \subsection{Uniform model $\mathcal{J}(\lambda,x)$ for the oscillatory integral which is to bound}\label{AppendixC3}

    \myindent Now, thanks to the construction of $v(x,y)$, we can write $J_1(\lambda,x)$ in the form
    \begin{equation}
        J_1(\lambda,x) = e^{-i\lambda \phi(x,y(x)))}\int dy e^{i\lambda\phi(x,y)} b_1(x,y),
    \end{equation}
    where 
    \begin{equation}
        b_1(x,y) := \nabla_y \cdot(\zeta(x,y) a(x,y) v(x,y)).
    \end{equation}

    \myindent In particular, thanks to \eqref{controlNlv}, there holds
    \begin{equation}\label{boundb1}
        \begin{split}
            \mathcal{M}_{0,\ell}^{(y)}(b_1) &\lesssim \sum_{m_1 + m_2 + m_3 = \ell+1} \mathcal{M}_{0,m_1}^{(y)}(\zeta) \mathcal{M}_{0,m_2}^{(y)}(a) \mathcal{M}_{0,m_3}^{(y)}(v) \\
            &\lesssim \sum_{m_1 + m_2 + m_3 = \ell+1} \left(\mathcal{M}_{3,3}^{(y)}(\phi) \mathcal{N}(\phi)\right)^{m_1} \mathcal{M}_{0,m_2}(a) \mathcal{M}_{1,m_3+1}^{(y)}(\partial_x \phi) \left(\mathcal{N}(\phi)\right)^{m_3+1} \left(\mathcal{M}_{2,m_3+2}^{(y)}(\phi)\right)^{m_3} \\
            &\lesssim \mathcal{M}_{0,\ell+1}(a) \mathcal{M}_{1,\ell+2}^{(y)}(\partial_x \phi) \left(\mathcal{N}(\phi)\right)^{\ell+2} \left(\mathcal{M}_{2,\ell+3}^{(y)}(\phi)\right)^{\ell+1}.
        \end{split}
    \end{equation}
    
    \quad

    \myindent Similarly, one can write
    \begin{equation}
        J_2(\lambda,x) = e^{-i\lambda \phi(x,y(x)))}\int dy e^{i\lambda\phi(x,y)} b_2(x,y),
    \end{equation}
    where 
    \begin{equation}
        \begin{split}
            b_2(x,y) &= \partial_x (\zeta(x,y) a(x,y)) \\
            &= \partial_x \zeta (x,y) a(x,y) + \zeta(x,y) \partial_x a(x,y).
        \end{split}
    \end{equation}

    \myindent Since $\zeta(x,\cdot)$ is a smooth localizer inside a ball of center $y(x)$ and of radius $\sim \left(\mathcal{M}_{3,3}^{(y)}(\phi) N(\phi)\right)^{-1}$, one can further bound
    \begin{equation}
        \mathcal{M}_{0,m}^{(y)}(\partial_x \zeta) \lesssim \left(\sup_x |y'(x)|\right) \left(\mathcal{M}_{3,3}^{(y)}(\phi) N(\phi)\right)^{m + 1}.
    \end{equation}

    \myindent Now, differentiating in $x$ the equation \eqref{nablaphixyxeqzero}, there holds
    \begin{equation}
        y'(x) = - H(x)^{-1} (\partial_x \nabla_y \phi)(x,y(x)).
    \end{equation}

    \myindent Hence there holds
    \begin{equation}
        |y'(x)| \lesssim \mathcal{N}(\phi) \mathcal{M}_{1,1}^{(y)}(\partial_x \phi).
    \end{equation}

    \myindent Thus, bounding roughly, there holds
    \begin{equation}\label{boundb2}
        \mathcal{M}_{0,m}^{(y)}(b_2) \lesssim \left(\mathcal{N}(\phi)\right)^{m+2} \mathcal{M}_{1,1}(\partial_x \phi) \left(\mathcal{M}_{3,3}^{(y)}(\phi)\right)^{m} \left(\mathcal{M}_{0,m}^{(y)}(a) + \mathcal{M}_{0,m}^{(y)}(\partial_x a) \right).
    \end{equation}
    
    \quad

    \myindent Finally,
    \begin{equation}
        J(\lambda,x) = e^{-i\lambda\phi(x,y(x))} \int e^{i\lambda \phi(x,y)} \zeta(x,y)a(x,y)
    \end{equation}
    where
    \begin{equation}\label{boundb3}
        \mathcal{M}_{0,m}^{(y)} (\zeta a) \lesssim \left(\mathcal{N}(\phi) \mathcal{M}_{3,3}^{(y)}(\phi)\right)^{m}\mathcal{M}_{0,m}^{(y)}(a).
    \end{equation}

    \quad

    \myindent Hence, bounds \eqref{boundb1}, \eqref{boundb2} and \eqref{boundb3} yield that we need only bound
    \begin{equation}
        \mathcal{J}(\lambda,x) := \int dy e^{i\lambda\phi(x,y)}b(y),
    \end{equation}
    where 
    \begin{equation}\label{boundb}
        \mathcal{M}_{0,m}(b) \lesssim \left(\mathcal{M}_{0,m+1}^{(y)}(a) + \mathcal{M}_{0,m}^{(y)}(\partial_x a)\right) \mathcal{M}_{1,m+2}^{(y)}(\partial_x \phi) \left(\mathcal{N}(\phi)\right)^{m+2} \left(\mathcal{M}_{2,m+3}^{(y)}(\phi)\right)^{m+1},
    \end{equation}
    and where $b(y)$ is supported in the ball $\mathcal{B}(x)$ defined by \eqref{defballmixedvdcabz}.

    \subsection{Bounding $\mathcal{J}(\lambda,x)$}\label{AppendixC4}
    
    \myindent Now, let $\eps > 0$ be a small constant, let $\chi \in \mathcal{C}_0^{\infty}(\R^d)$ be a smooth bump function such that $\chi \equiv 1$ on the ball $B(0,1)$ and $\chi = 0$ outside of $B(0,2)$. Consider the following partition of unity
    \begin{equation}
        1 = \chi(\eps^{-1} \nabla_y \phi(x,y)) + (1 - \chi(\eps^{-1}\nabla_y \phi(x,y))) =: \chi_1(y) + \chi_2(y),
    \end{equation}
    so that, in particular, for all two integer $k\geq 0$, and for $i = 1,2$,
    \begin{equation}\label{boundderchiilocal}
        \mathcal{M}_{0,k}^{(y)}(\chi_i) \lesssim \eps^{-k} \left(\mathcal{M}_{0,k}^{(y)}(\phi)\right)^k.
    \end{equation}
    
    \myindent Then, we can divide $\mathcal{J}(\lambda,x)$ into two parts, namely
    \begin{equation}
        \mathcal{J}(\lambda,x) = \int dy e^{i\lambda \phi(x,y)} b(y) \chi_1(y) + \int dy e^{i\lambda \phi(x,y)} b(y) \chi_2(y) =: \mathcal{J}_1(\lambda,x) + \mathcal{J}_2(\lambda,x).
    \end{equation}
    
    \myindent On the one hand, in order to bound $\mathcal{J}_1(\lambda,x)$, let us introduce the change of variable
    \begin{equation}\label{changeofvariablewy}
        w := H(x) (y-y(x)),
    \end{equation}
    where we recall that $H(x)$ is defined by Definition \ref{defHandN}. Then, using moreover the fact that, by the definition of $M$ \eqref{definitionofM}, there holds
    \begin{equation}
        \nabla_y\phi(x,y) = M(x,y)(y-y(x)),
    \end{equation}
    we may write
    \begin{equation}
        \mathcal{J}_1(\lambda,x) = |\det H(x)|^{-1} \int e^{i\lambda\phi(x,y(w))} \chi(\eps^{-1} M(x,y(w)) H(x)^{-1} w) b(H(x)^{-1}w) dw.
    \end{equation}
    
    \myindent Hence, bounding by the absolute value, there holds
    \begin{equation}\label{firstcontrolmathcalJ1}
        |\mathcal{J}_1(\lambda,x)| \leq |\det H(x)|^{-1} \mathcal{M}_{0,0}^{(y)}(b) \int \chi(\eps^{-1} M(x,y(w)) H(x)^{-1} w).
    \end{equation}
    
    \myindent Now, observe that, thanks to \eqref{estnormM-1}, there holds
    \begin{equation}
        |M(x,y(w)) H(x)^{-1} w| \geq \left\|H(x) M(x,y(w))^{-1}\right\|^{-1} |w| \gtrsim |w|.
    \end{equation}
    
    \myindent In particular, $w\mapsto \chi(\eps^{-1} M(x,y(w))^{-1}w) H(x)^{-1} w)$ is supported inside a ball 
    \begin{equation}
        \{w \in \R^d \qquad |w| \lesssim \eps\}.
    \end{equation}
    
    \myindent Hence, from \eqref{firstcontrolmathcalJ1}, we find that
    \begin{equation}\label{boundmathcalJ1}
    \begin{split}
        |\mathcal{J}_1(\lambda,x)| &\lesssim |\det H(x)|^{-1} \mathcal{M}_{0,0}^{(y)}(b) \eps^d \\
        &\lesssim \eps^d |\det H(x)^{-1}| \left(\mathcal{M}_{0,1}^{(y)}(a) + \mathcal{M}_{0,0}^{(y)}(\partial_x a)\right) \mathcal{M}_{1,2}^{(y)}(\partial_x \phi) \left(\mathcal{N}(\phi)\right)^{2}\mathcal{M}_{2,3}^{(y)}(\phi)
    \end{split}
    \end{equation}
    
    \quad
    
    \myindent On the other hand, in order to bound $\mathcal{J}_2(\lambda,x)$, we observe that, on the support of $\chi_2$, we can integrate by parts in $y$, since $\nabla_y \phi$ doesn't vanish. Indeed, define the smooth vector field
    \begin{equation}
        X := \frac{\nabla_y \phi \cdot \nabla_y}{|\nabla_y\phi|^2},
    \end{equation}
    such that
    \begin{equation}
        e^{i\lambda \phi} = \frac{1}{i\lambda} X(e^{i\lambda \phi}).
    \end{equation}
    
    \myindent Then, there holds, for any integer $K\geq 1$
    \begin{equation}
        \mathcal{J}_2(\lambda,x) = \frac{1}{(i\lambda)^K} \int e^{i\lambda \phi} \left(X^*\right)^K \left(\chi_2 b\right),
    \end{equation}
    where $X^*$ is the adjoint of $X$, namely
    \begin{equation}
        X^* f(y) = -\nabla_y \cdot\left(\frac{\nabla_y \phi}{|\nabla_y\phi|^2} f\right) = -\frac{\left(\nabla_y\phi\right)^t\cdot \nabla_y^2\phi \cdot \nabla_y \phi}{|\nabla_y\phi|^4} f - \frac{\nabla_y\phi \cdot \nabla_y f}{|\nabla_y \phi|^2}.
    \end{equation}
    
    \myindent A tedious induction, based on this formula, yields that, for all $K\geq 1$, and for all function $f(y)$,
    \begin{equation}
        \left|\left(X^*\right)^K f(y)\right| \lesssim \sum_{m_1 + m_2 = K} \mathcal{M}_{0,m_1}(f) \frac{\left(\mathcal{M}_{0,m_2+1}^{(y)}(\phi)\right)^{m_2}}{|\nabla_y\phi|^{K+m_2}}.
    \end{equation}
    
    \myindent In particular, applying this last formula to $f = \chi_2 b$, and using the a priori estimates on $\chi_2$ \eqref{boundderchiilocal}, and on $b$ \eqref{boundb}, there holds
    \begin{equation}
    \begin{split}
        \left|\left(X^*\right)^K \left(\chi_2 b\right)\right| &\lesssim \chi_2 \sum_{m_1 + m_2 + m_3 = K} \eps^{-m_1} \left(\mathcal{M}_{0,m_1}^{(y)}(\phi)\right)^{m_1} \left(\mathcal{N}(\phi)\right)^{m_2} \left(\mathcal{M}_{2,m_2 + 3}^{(y)}(\phi)\right)^{m_2} \frac{\left(\mathcal{M}^{(y)}_{0, m_3 + 1}(\phi)\right)^{m_3}}{|\nabla_y \phi|^{K + m_3}}\\
        &\times \left(\mathcal{M}_{0,K+1}^{(y)}(a) + \mathcal{M}_{0,K}^{(y)}(\partial_x a)\right) \mathcal{M}_{1,K+2}^{(y)}(\partial_x \phi) \left(\mathcal{N}(\phi)\right)^2 \mathcal{M}_{2,K+3}^{(y)}(\phi) \\
        &\lesssim \chi_2 \sum_{m_1 + m_2 + m_3 = K} \left(\mathcal{N}(\phi)\right)^{m_1} \frac{1}{\eps^{m_2} |\nabla_y \phi|^{K + m_3}}\\
        &\times \left(\mathcal{M}_{0,K+1}^{(y)}(a) + \mathcal{M}_{0,K}^{(y)}(\partial_x a)\right) \mathcal{M}_{1,K+2}^{(y)}(\partial_x \phi) \left(\mathcal{N}(\phi)\right)^2\left( \mathcal{M}_{2,K+3}^{(y)}(\phi)\right)^{K+1}.
    \end{split}
    \end{equation}
    
    \myindent Using, again, the change of variable \eqref{changeofvariablewy}, there holds that, for $K > d$,
    \begin{equation}
        \int \chi_2 \frac{dy}{|\nabla_y \phi|^{K + m_3}} \lesssim |\det H(x)|^{-1} \eps^{-K - m_3 + d},
    \end{equation}
    from which we deduce that
    \begin{equation}\label{boundmathcalJ2}
    \begin{split}
        |\mathcal{J}_2(\lambda,x)| &\lesssim \lambda^{-K} \sum_{m_1 + m_2 + m_3 = K} \left(\mathcal{N}(\phi)\right)^{m_1} \eps^{-K - m_2 - m_3 + d} \\
        &\times \left(\mathcal{M}_{0,K+1}^{(y)}(a) + \mathcal{M}_{0,K}^{(y)}(\partial_x a)\right) \mathcal{M}_{1,K+2}^{(y)}(\partial_x \phi) \left(\mathcal{N}(\phi)\right)^2 \left( \mathcal{M}_{2,K+3}^{(y)}(\phi)\right)^{K+1}|\det H(x)|^{-1} \\
        &\lesssim \lambda^{-K} \left( \mathcal{M}_{2,K+3}^{(y)}(\phi)\right)^{K+1} \left(\left(\mathcal{N}(\phi)\right)^K \eps^{-K + d} + \eps^{-2K + d}\right)\\
        &\times \left(\mathcal{M}_{0,K+1}^{(y)}(a) + \mathcal{M}_{0,K}^{(y)}(\partial_x a)\right) \mathcal{M}_{1,K+2}^{(y)}(\partial_x \phi) \left(\mathcal{N}(\phi)\right)^2  |\det H(x)|^{-1}.
    \end{split}
    \end{equation}
    
    \myindent Now, comparing the bound on $\mathcal{J}_1(\lambda,x)$ \eqref{boundmathcalJ1} and on $\mathcal{J}_2(\lambda,x)$ \eqref{boundmathcalJ2}, we are able to choose $\eps$ optimally, depending on the dominant term. It is natural to impose that
    \begin{equation}\label{conditionforeps}
        \max\left(\lambda^{-K} \left( \mathcal{M}_{2,K+3}^{(y)}(\phi)\right)^K \left(\mathcal{N}(\phi)\right)^K \eps^{-K}, \lambda^{-K} \left( \mathcal{M}_{2,K+3}^{(y)}(\phi)\right)^K\eps^{-2K}\right) = 1.
    \end{equation}
    
    \myindent In particular, we see that the term in the max which gives the condition for $\eps$ depends on whether or not there holds
    \begin{equation}
        \mathcal{N}(\phi) \left(\mathcal{M}_{2,K+3}^{(y)} (\phi) \right)^{\frac{1}{2}} \leq \lambda^{\frac{1}{2}}.
    \end{equation}
    
    \myindent This motivates hypothesis \eqref{technicalhyp} in the theorem. With this hypothesis, we may thus choose, with $K = d+1$,
    \begin{equation}
        \eps = \lambda^{-\frac{1}{2}}  \left( \mathcal{M}_{2,d+4}^{(y)}(\phi)\right)^{\frac{1}{2}},
    \end{equation}
    and ultimately find the bound
    \begin{equation}
        |\mathcal{J}(\lambda,x)| \lesssim \lambda^{-\frac{d}{2}}  \left( \mathcal{M}_{2,d+4}^{(y)}(\phi)\right)^{\frac{d}{2} + 1} |\det H(x)|^{-1} \left(\mathcal{N}(\phi)\right)^2\left(\mathcal{M}_{0,d+2}^{(y)}(a) + \mathcal{M}_{0,d+1}^{(y)}(\partial_x a)\right) \mathcal{M}_{1,d+3}^{(y)}(\partial_x \phi).
    \end{equation}
    
    \myindent In particular, $A(\lambda)$ and $B(\lambda)$ defined by \eqref{defAB} satisfy similar bounds, and, this, coming back to estimate \eqref{boundIbyAB}, concludes the proof of the theorem.

\begin{remark}\label{resultnotechnicalhypo}
    If the technical hypothesis \eqref{technicalhyp} doesn't hold, i.e. if, for some $K$,
    \begin{equation}\label{notechnicalhypo}
        \mathcal{N}(\phi) \left(\mathcal{M}_{2,K+3}^{(y)} (\phi) \right)^{\frac{1}{2}} \geq \lambda^{\frac{1}{2}},
    \end{equation}
    then one can still concludes the proof, but the winning term in the $\max$ \eqref{conditionforeps} is not always the second one. Hence, one needs to choose 
    \begin{equation}
        \eps = \lambda^{-1} \mathcal{M}_{2, d+4}^{(y)}(\phi) \mathcal{N}(\phi).
    \end{equation}
    
    \myindent Overall, this would yield the bound
    \begin{equation}
        |\mathcal{J}(\lambda,x)| \lesssim \lambda^{-d} \left(\mathcal{M}_{2,d+4}^{(y)} (\phi)\right)^{d + 1} |\det(H(x))|^{-1} \left(\mathcal{N}(\phi)\right)^{d + 2} \left(\mathcal{M}_{0,d+2}^{(y)}(a) + \mathcal{M}_{0,d+1}^{(y)}(\partial_x a)\right) \mathcal{M}_{1,d+3}^{(y)}(\partial_x \phi),
    \end{equation}
    and, ultimately, the conclusion of the theorem should be changed to
    \begin{multline}
        |\mathcal{I}(\lambda)|\lesssim_d \lambda^{-d - \frac{1}{p}} c^{-\frac{1}{p}} \left(1 + |I|\left(1+ \frac{C}{c}\right)\right)^2 \left(\mathcal{D}(\phi)\right)^{-1}\left(\mathcal{N}(\phi)\right)^{d + 2} \left(\mathcal{M}_{2,d+4}^{(y)}(\phi)\right)^{d+ 1}\\
        \times \mathcal{M}_{1,d+3}^{(y)}(\partial_x \phi) \left(\mathcal{M}_{0,d+2}^{(y)}(a) + \mathcal{M}_{0,d+1}^{(y)}(\partial_x a)\right).
    \end{multline}
    
    \myindent While the exponent $\lambda^{-d}$ in this upper bound seems artificially a better bound than the usual $\lambda^{-\frac{d}{2}}$, the bound is worse because of \eqref{notechnicalhypo}.
\end{remark}

\section{Study of the bicharacteristic length function}\label{AppendixD}

\myindent In this section, we give some useful technical results on the bicharacteristic length function $\psi$ defined by Definition \ref{defpsi}. Precisely, we study the function
\begin{equation}\label{defdsigmat}
    d(\sigma,t) := \psi((t,\sigma),(0,\sigma)),
\end{equation}
which is defined for $t\in S^1$ and $\sigma \in \left(-\frac{L}{2}, \frac{L}{2}\right)$, and smooth if moreover $t \neq 0$, thanks to Lemma \ref{propertypsi}.

\subsection{Away from the equator}\label{AppendixD1}

\myindent In this section, we prove the following lemma.

\begin{lemma}\label{firstlemmabicharact}
    The function $d(\sigma,t)$ defined by \eqref{defdsigmat} satisfies
    \begin{equation}
        \forall \sigma \left(-\frac{L}{2}, \frac{L}{2}\right) \backslash\{0\}, \ \forall t \in (-\pi,\pi)\backslash\{0\}, \qquad \partial_{tt} d(\sigma,t) \neq 0.
    \end{equation}
\end{lemma}

\begin{proof}
    Without loss of generality, we fix $\sigma \in \left(0,\frac{L}{2}\right)$, and we do the proof for $t \in (0, \pi)$. Let us denote
    \begin{equation}\label{xieqnablaxpsiAppendixD}
        \xi(t) := -\nabla_x \psi((t,\sigma),(0,\sigma)).
    \end{equation}
    
    \myindent Thanks to the proof of Proposition \ref{bicharactparamprop}, we know that, on the one hand,
    \begin{equation}
        q_1(\sigma,\xi(t)) = 1,
    \end{equation}
    i.e. $\xi(t) \in \mathcal{N}_{\sigma}$ (see \eqref{defNsigma}). On the other hand, we also know, thanks to this proof, that $\xi(t)$ is the direction of the unique bicharacteristic curve of $q_1$ joining $(t,\sigma)$ to $(0,\sigma)$ in time $s(t) = d(\sigma,t)$ by definition. We first claim
    \begin{equation}\label{xidottneq0AppendixD}
        \forall t\in(0,\pi), \ \frac{d}{dt} \xi(t) \neq 0.
    \end{equation}
    
    \myindent Indeed, let us fix $t \in (0,\pi)$, and let us introduce, as in the proof of Proposition \ref{bicharactparamprop}, the wavefront set at time $s(t)$ of the bicharacteristic flow of $q_1$ starting at $(0,\sigma)$, i.e.
    \begin{equation}
        Q := \{x \in \mathcal{S} \ \text{such that} \ \psi(x,(0,\sigma)) = s(t)\} = \{P(\Phi_{s(t)}^{q_1}(0,\sigma,\xi)), \ \xi \in \mathcal{N}_{\sigma} \}.
    \end{equation}
    
    \myindent Then, by definition, $x = (t,\sigma) \in Q$, and, around $x$, $Q$ is a smooth curve, whose normal at $x$ is $-\xi(t)$. In particular, if $\delta t$ is an infinitesimal, we see that the bicharacteristic curve joining $(0,\sigma)$ to $(t + \delta t, \sigma)$ passes through $Q$ at a point $x + \delta x$ such that
    \begin{equation}
        \delta x = \delta t \sin(\alpha) + O((\delta t)^2),
    \end{equation}
    where $\alpha$ is the angle between $-\xi(t)$ and the horizontal direction, as can be seen on Figure \ref{QtiltedFigure}.
    
    \begin{figure}[h]
\includegraphics[scale = 0.5]{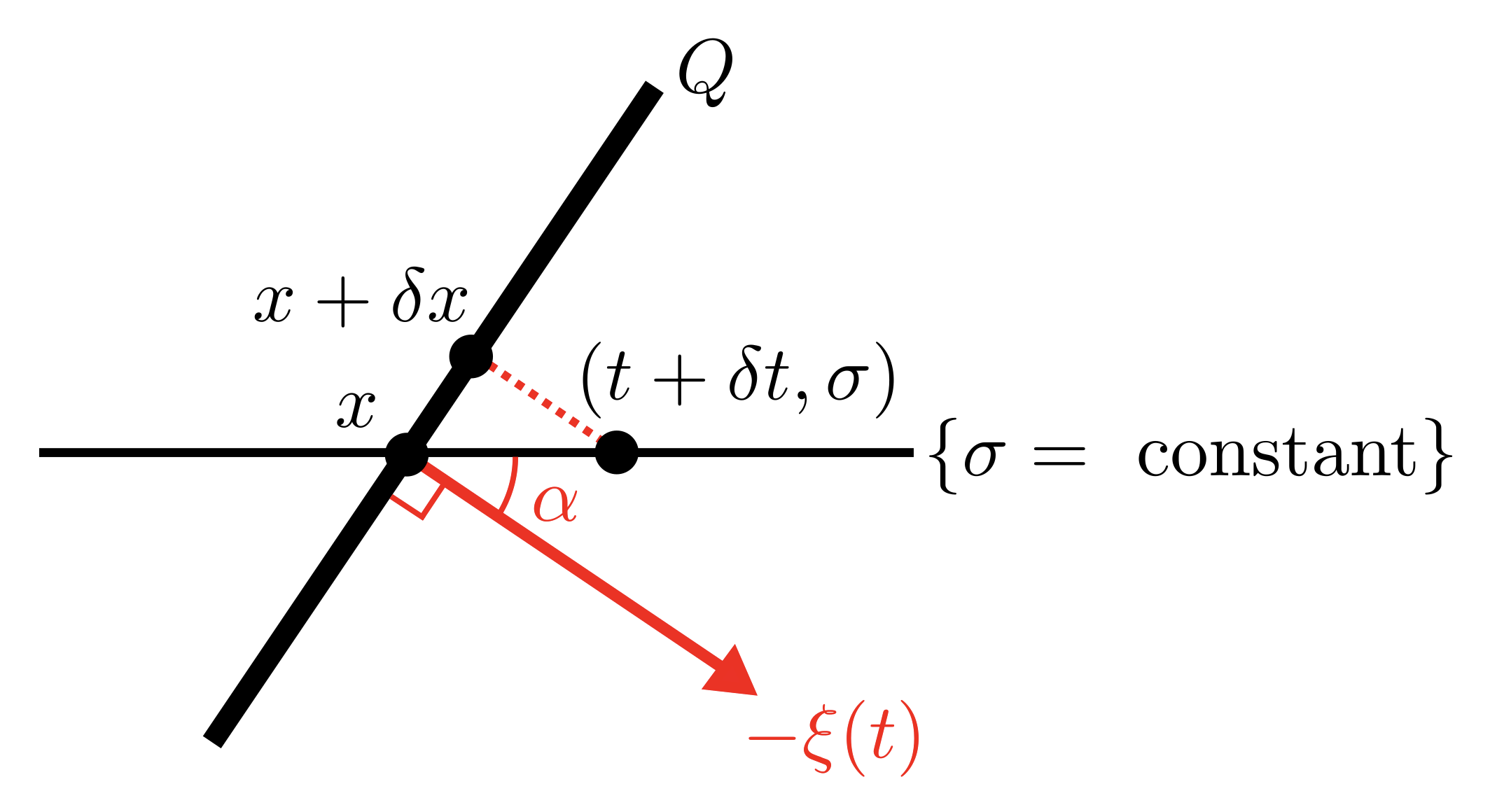}
\centering
\caption{The wavefront $Q$ around $x$}
\label{QtiltedFigure}
\end{figure}
    
    \myindent Thus, because the map 
    \begin{equation}
        \xi \in \mathcal{N}_{\sigma} \mapsto P(\Phi_{s(t)}^{q_1})\in Q
    \end{equation}
    is locally around $(\xi(t))$ a smooth bijection, we can thus conclude that the direction of the bicharacteristic curve joining $(0,\sigma)$ to $(t + \delta t,\sigma)$ is $-\xi(t) - \delta \xi \in \mathcal{N}_{\sigma}$ such that
    \begin{equation}
        -\delta \xi \ \text{is proportional to} \ \left(\delta t \sin(\alpha) + O((\delta t)^2)\right),
    \end{equation}
    or, in other words, that
    \begin{equation}
        \frac{d}{dt} \xi(t) \ \text{is proportional to} \ -\sin(\alpha).
    \end{equation}
    
    \myindent To conclude the proof of \eqref{xidottneq0AppendixD}, there only remains to observe that $\sin(\alpha) \neq 0$, because $\sigma \neq 0$ and $t\neq 0$ hence the direction $\xi(t)$ cannot be horizontal. 
    
    \myindent Now, since, again, $\xi(t)$ is not horizontal, then the projection on the first coordinate
    \begin{equation}
        \xi \in \mathcal{N}_{\sigma} \mapsto q_2(\sigma,\xi)
    \end{equation}
    is locally around $\xi(t)$ a smooth bijection (see for example the arguments of Paragraph \ref{subsubsec22HormPhase}). Thus, \eqref{xieqnablaxpsiAppendixD} and \eqref{xidottneq0AppendixD} yields ultimately that
    \begin{equation}
        \frac{d}{dt} \left(\partial_{\theta} \psi((t,\sigma),(0,\sigma))\right) \neq 0,
    \end{equation}
    or, in other words, that
    \begin{equation}
        \partial_{tt} d(\sigma,t) \neq 0.
    \end{equation}
\end{proof}

\subsection{Near the equator}\label{AppendixD2}

\myindent In this section, we give the following local expansion of the function $d$ around the equator $\sigma = 0$. We will give the main steps of the proof, without detailing all the computations.

\begin{lemma}\label{secondlemmabicharact}
    There exists a constant $C$ depending only on $\mathcal{S}$ such that
    \begin{equation}
        \partial_t d (\sigma,t) = 1 - C \sigma^2 f(\sigma,t)^2,
    \end{equation}
    where $f(\sigma,t)$ is a smooth function on $\left(-\frac{L}{2}, \frac{L}{2} \right) \times \left((-\pi,\pi)\backslash\{0\}\right)$ such that
    \begin{equation}
        f(0,t) = \tan\left(\frac{t}{2}\right).
    \end{equation}
\end{lemma}

\begin{proof}
    We give the proof for $\sigma \geq 0$ and $t\in (0,\pi)$. Using similar notations than for the proof of Lemma \ref{firstlemmabicharact}, we define, for $|\sigma| << 1$ and $t\in (0,\pi)$, 
    \begin{equation}
        \xi(\sigma,t) \in \mathcal{N}_{\sigma}
    \end{equation}
    the direction of the unique bicharacteristic joining $(0,\sigma)$ to $(t,\sigma)$ in a time $s(\sigma,t) \in (0,\pi)$ (observe that there is a sign difference with the definition of $\xi(t)$ in the proof of Lemma \ref{firstlemmabicharact}). As we have already argued in this proof, there holds
    \begin{equation}
        \partial_t d(\sigma,t) = q_2(\sigma,\xi(\sigma,t)).
    \end{equation}
    
    \myindent Now, let us write 
    \begin{equation}\label{defSigmasigmatAppendixD2}
        \xi(\sigma,t) = (\Theta(\sigma,t), \Sigma(\sigma,t)) \in \mathcal{N}_{\sigma}.
    \end{equation}
    
    \myindent Since $|\sigma| << 1$, we know a priori that $|\Sigma(\sigma,t)| << 1$. Hence, there holds (see Notation \ref{defwrho})
    \begin{equation}\label{firstexpressionAppendixD2}
        \partial_t d(\sigma,t) = \Theta(\sigma,t) = |g_{\sigma}(0)|\left(1 - \frac{1}{2}\left(\Sigma(\sigma,t)\right)^2 + O\left(\left(\Sigma(\sigma,t)\right)^4\right)\right).
    \end{equation}
    
    \myindent Now, let us study the bicharacteristic of $q_1$ starting at $x = (0,\sigma)$ with direction $\xi =(\Theta,\Sigma) \in \mathcal{N}_{\sigma}$ such that
    \begin{equation}
        0 < \Sigma \lesssim \sigma.
    \end{equation}
    
    \myindent We denote
    \begin{equation}
        \Phi_s^{q_1}(x,\xi) = (\theta(t),\sigma(t), \Theta(t),\Sigma(t)).
    \end{equation}
    
    \myindent Using, for example, Lemma \ref{explicitbicharac}, one can prove that, for all times,
    \begin{equation}\label{sigcontrol}
        |(\sigma(t),\Sigma(t))| \lesssim \sigma.
    \end{equation}
    
    \myindent Using the explicit formula for the Hamiltonian flow \eqref{explicithamiltonianflow}, we know that there holds along this bicharacteristic
    \begin{equation}
        \dot{\sigma}(t) = \partial_{\Sigma} q_1(\sigma(t),\Theta(t),\Sigma(t)) = \left(\partial_{\Sigma \Sigma} q_1(\sigma(t),\Theta(t),0) \right) \Sigma(t) + O\left(\Sigma(t)^2\right).
    \end{equation}
    
    \myindent Using similar techniques than in Appendix \ref{AppendixA1}, one can prove moreover that
    \begin{equation}
        \partial_{\Sigma \Sigma} q_1(\sigma(t),\Theta(t),0) = \partial_{\Sigma \Sigma} q_1(0,\Theta_0,0) + O\left((\sigma(t) + \Sigma(t))^2\right),
    \end{equation}
    where $\Theta_0$ is such that
    \begin{equation}
        g_0(0) = \begin{pmatrix}
            \Theta_0 \\
            0
        \end{pmatrix}.
    \end{equation}
    
    \myindent Overall, using moreover \eqref{sigcontrol}, one can prove that
    \begin{equation}
        \dot{\sigma}(t) = \partial_{\Sigma \Sigma} q_1(0,\Theta_0,0) \Sigma(t) + O(\sigma^2).
    \end{equation}
    
    \myindent Now, one can rigorously prove that this equality can be differentiated, and that there holds
    \begin{equation}
        \ddot{\sigma}(t) = \partial_{\Sigma \Sigma} q_1(0,\Theta_0,0) \dot{\Sigma}(t) + O(\sigma^2).
    \end{equation}
    
    \myindent Thanks to \eqref{explicithamiltonianflow}, we know that
    \begin{equation}
        \dot{\Sigma}(t) = -\partial_{\sigma} q_1(\sigma(t),\Theta(t),\Sigma(t)).
    \end{equation}
    
    \myindent Now, using the representation of Lemma \ref{formulaofq1}, there holds
    \begin{equation}
        \partial_{\sigma} q_1(\sigma(t),\Theta(t),\Sigma(t)) = (\partial_1 G)(p_1(\sigma(t),\Theta(t),\Sigma(t)), \Theta(t)) \partial_{\sigma} p_1(\sigma(t),\Theta(t),\Sigma(t)).
    \end{equation}
    
    \myindent There holds, using \eqref{normsurfrev} and \eqref{sigcontrol},
    \begin{equation}
        \partial_{\sigma} p_1(\sigma(t),\Theta(t),\Sigma(t)) = \frac{-\frac{\Theta(t)^2}{f(\sigma)^3} f'(\sigma)}{ p_1(\sigma(t),\Theta(t),\Sigma(t))} = \Theta_0 |f''(0)| \sigma + O(\sigma^2). 
    \end{equation}
    
    \myindent Moreover, there holds similarly
    \begin{equation}
        (\partial_1 G)(p_1(\sigma(t),\Theta(t),\Sigma(t)), \Theta(t)) = \partial_1 G (1,1) + O(\sigma^2).
    \end{equation}
    
    \myindent Overall, we thus find that
    \begin{equation}
        \begin{split}
            \ddot{\sigma}(t) &= \partial_{\Sigma\Sigma}q_1(0,\Theta_0,0) \dot{\Sigma}(t) + O(\sigma^2) \\
            &= -\left(\partial_{\Sigma\Sigma}q_1(0,\Theta_0,0) \partial_1 G(1,1) \Theta_0 |f''(0)|\right) \sigma + O(\sigma^2).
        \end{split}
    \end{equation}
    
    \myindent Now, using similar techniques than the one used in Appendix \ref{AppendixA}, one can prove that 
    \begin{equation}
        \partial_{\Sigma\Sigma}q_1(0,\Theta_0,0) \partial_1 G(1,1) \Theta_0 |f''(0)| = 1.
    \end{equation}
    
    \myindent Finally, we see that $\sigma(t)$ solves the differential equation
    \begin{equation}
        \ddot{\sigma}(t) + \sigma(t) = O(\sigma^2).
    \end{equation}
    
    \myindent Integrating explicitly this equation, and since the time interval is compact, there holds, for some constants $\alpha,\beta$.
    \begin{equation}\label{reprsigmatAppendixD2}
        \sigma(t) = \alpha \cos(t) + \beta \sin(t) + O(\sigma^2).
    \end{equation}
    \myindent In order to find $\alpha$ and $\beta$, we may use the fact 
    \begin{equation}
    \begin{split}
        &\sigma(0) = \sigma \\
        &\dot{\sigma}(0) = (\partial_{\Sigma\Sigma}q_1(0,\Theta_0,0)) \Sigma + O(\sigma^2).
    \end{split}
    \end{equation}
    \myindent Hence, we find that
    \begin{equation}
    \begin{split}
        &\alpha = \sigma + O(\sigma^2)\\
        &\beta = c^{-1}\Sigma + O(\sigma^2),
    \end{split}  
    \end{equation}
    where
    \begin{equation}
        c := \partial_{\Sigma\Sigma}q_1(0,\Theta_0,0) \neq 0.
    \end{equation}
    
    \myindent Now, if we choose $\Sigma = \Sigma(\sigma,t)$ (see \eqref{defSigmasigmatAppendixD2}), then, by definition,
    \begin{equation}
        \sigma(t) = \sigma.
    \end{equation}
    
    \myindent This equation, along with the representation \eqref{reprsigmatAppendixD2}, yields the following equation for $\Sigma(\sigma,t)$ 
    \begin{equation}
        \sigma \cos(t) + c^{-1}\Sigma(\sigma,t) \sin(t) + O(\sigma^2) = \sigma.
    \end{equation}
    
    \myindent We thus find that
    \begin{equation}
        \Sigma(\sigma,t) = c \sigma \frac{1 - \cos(t)}{\sin(t)} + O\left(\frac{\sigma^2}{\sin(t)}\right).
    \end{equation}
    
    \myindent Observe that, on any compact subset of $(0,\pi)$, $\frac{1}{\sin(t)}$ is bounded. Hence, as long as we restrict locally to a compact subset of $(0,\pi)$, we may write
    \begin{equation}
        \Sigma(\sigma,t) = c\sigma \frac{1 - \cos(t)}{\sin(t)} + O(\sigma^2).
    \end{equation}
    
    \myindent Using that
    \begin{equation}
        \frac{1 - \cos(t)}{\sin(t)} = \tan\left(\frac{t}{2}\right),
    \end{equation}
    we may further reduce to
    \begin{equation}
        \Sigma(\sigma,t) = c \sigma \tan\left(\frac{t}{2}\right) + O(\sigma^2),
    \end{equation}
    where the remainder is locally uniform on any compact subset of $(0,\pi)$. Coming back to equation \eqref{firstexpressionAppendixD2}, we finally find that
    \begin{equation}
        \partial_t d(\sigma,t) = |g_{\sigma}(0)| \left(1 - \frac{1}{2} c^2 \sigma^2 \tan^2\left(\frac{t}{2} \right) + 0(\sigma^4)\right).
    \end{equation}
    
    \myindent In order to finally complete the proof of the Lemma, there only remains to observe that
    \begin{equation}
        |g_{\sigma}(0)| = 1 + O(\sigma^2).
    \end{equation}
\end{proof}

\printbibliography

\end{document}